\documentclass[intlim,righttag,10pt]{amsart}
\usepackage{amscd}
\usepackage{amssymb}
\usepackage{stmaryrd}
\usepackage{mathrsfs}
\usepackage[all]{xy}
\usepackage[british,french]{babel}
\usepackage{hyperref}
\hypersetup{hidelinks} 
\usepackage{url}
\usepackage{xcolor}
\usepackage{tikz}
\usepackage{comment} 
\usepackage[normalem]{ulem}
\usepackage{soul}
\usetikzlibrary{arrows,matrix}
\newtheorem{theorem}{Theorem}[section]
 \newtheorem{lemma}[theorem]{Lemma}
 \newtheorem{proposition}[theorem]{Proposition}
 \newtheorem{corollary}[theorem]{Corollary}

\newtheorem*{theorem*}{Theorem}

\theoremstyle{definition}
 \newtheorem{definition}[theorem]{Definition}
 \newtheorem{construction}[theorem]{Construction}
 
 \theoremstyle{remark}
 \newtheorem{remark}[theorem]{Remark}
 \newtheorem{question}[theorem]{Question}
 \newtheorem{example}[theorem]{Example}
 \newtheorem{hypothesis}[theorem]{Hypothesis}
 \newtheorem*{acknowledgments}{Acknowledgements}
 \newtheorem*{remark*}{Remark}

\newcommand{\N}{\mathbf{N}}
\newcommand{\Z}{\mathbf{Z}}
\newcommand{\Q}{\mathbf{Q}}
\newcommand{\Var}{\mathsf{Var}}
\newcommand{\Sm}{\mathsf{Sm}}
\newcommand{\cdp}{\mathrm{cdp}}

\newcommand{\cdh}{\mathrm{cdh}}
\newcommand{\cd}{\mathrm{cd}}

\newcommand{\et}{\mathrm{\acute{e}t}}
\newcommand{\rig}{\mathrm{rig}}
\newcommand{\cris}{\mathrm{cris}}
\newcommand{\Zar}{\mathrm{Zar}}
\newcommand{\dom}{\mathsf{Dom}}
\newcommand{\Spec}{\operatorname{Spec}}
\newcommand{\Spf}{\operatorname{Spf}}  

\newcommand{\kr}{\,
\begin{picture}(-1,1)(-1,-2)\circle*{2}\end{picture}\ }

\newcommand{\Frac}{\mathrm{Frac}}

\newcommand{\salt}{\mathrm{s}\text{-}\mathrm{alt}}

\newcommand{\sm}{\mathrm{sm}}

\newcommand{\cC}{\mathcal{C}}
\newcommand{\co}{\mathcal{O}}

\newcommand{\fX}{\mathfrak{X}}
\newcommand{\fD}{\mathfrak{D}}

\DeclareMathOperator{\cosk}{cosk}
\DeclareMathOperator{\Bl}{Bl}
\DeclareMathOperator{\Sh}{Sh}

\renewcommand\labelenumi{(\roman{enumi})}
\renewcommand\theenumi\labelenumi

\begin{document}

\title[Integral $p$-adic cohomology theories]{Integral $p$-adic cohomology theories for open and singular varieties}
\author{Veronika Ertl}
\address{Université Caen - Normandie, Laboratoire de Mathématiques Nicolas Oresme, 6 boulevard Maréchal Juin, 14032 Caen, France}
\email{veronika.ertl@unicaen.fr}
%\address{Universit\"at Regensburg, Fakult\"at f\"ur Mathematik
%Universit\"atsstrasse 31,
%93053 Regensburg, Germany}
%\email{veronika.ertl@mathematik.uni-regensburg.de}
\author{Atsushi Shiho}
\address{Graduate School of Mathematical Sciences,
The University of Tokyo,
3-8-1 Komaba,  Meguro-ku,
Tokyo 153-8914, Japan}
\email{shiho@ms.u-tokyo.ac.jp}
\author{Johannes Sprang}
\address{Universität Duisburg-Essen, Fakultät für Mathematik, Thea-Leymann-Straße 9, 45127 Essen, Germany}
\email{johannes.sprang@uni-due.de}
\date{\today}
\thanks{The first named author's research was supported in part by the EPSRC grant  EP/R014604/1 and by the DFG grant SFB 1085.
The second named author was supported in part by JSPS KAKENHI (Grant Numbers JP17K05162, JP18H03667, JP18H05233, JP23K03049, and JP24H00015). 
The third named author was supported by the DFG grant SFB 1085.}
\maketitle

\selectlanguage{british}
\vspace{-.7cm}
\begin{abstract} 
For open and singular varieties in positive characteristic $p$ 
we study the existence of an integral $p$-adic cohomology theory 
which is finitely generated,  compatible with log crystalline cohomology
and rationally compatible with rigid cohomology.
We develop such a theory under certain assumptions of resolution of singularities in positive characteristic, 
 by using $\cdp$- and $\cdh$-topologies.

Without resolution of singularities in positive characteristic, we prove the existence of a good $p$-adic cohomology theory for open  and singular varieties in cohomological degree 1,  by using 
split proper generically \'etale hypercoverings. 
This is a slight generalisation of a result due to Andreatta--Barbieri-Viale.
 We also prove that this approach does not work for higher cohomological degrees.
\end{abstract}

\selectlanguage{french}
\vspace{-0.7cm}
\begin{abstract}
Pour les vari\'et\'es ouvertes et singuli\`eres en caract\'eristique positive  $p$
on \'etudie l'existence d'une th\'eorie de cohomologie $p$-adique \`a co\'efficients entiers qui est de g\'en\'eration finie,  compatible \`a la cohomologie log cristalline et compatible rationellement \`a la cohomologie rigide.
On d\'eveloppe une telle th\'eorie sous des conditions de r\'esolution de singularit\'es en caract\'eristique positive, 
en utilisant les topologies $\cdp$ et $\cdh$.

Sans conditions de r\'esolution de singularit\'es  en caract\'eristique positive, on montre l'existence d'une bonne th\'eorie de cohomologie $p$-adique pour les vari\'et\'es ouvertes  et singuli\`eres en degr\'e cohomologique 1, en  utilisant 
les hyperrecouvrements propres génériquement étales scindés.
C'est une l\'eg\`ere g\'en\'eralisation d'un r\'esultat d'Andreatta--Barbieri-Viale.
 On montre aussi que cette approche ne marche pas en degr\'es cohomologiques sup\'erieurs.\\

\end{abstract}

\selectlanguage{british}
\tableofcontents

The famous Weil conjectures can be seen as a starting point for the study of $p$-adic cohomology theories.
Already Weil has suggested to use a suitable cohomology theory to solve these conjectures for proper and smooth varieties over a field $k$ of characteristic $p$.  For $l\neq p$, this has long been solved by Grothendieck's school using $l$-adic cohomology.  The desire to fill the gap for $l=p$ motivates the search for a ``good'' $p$-adic  cohomology theory. 

In the following, let $k$ be a perfect field of positive characteristic $p$, $W(k)$ its ring of Witt vectors, and $K$ the fraction field of $W(k)$. 
The first candidate for such a theory was defined by Berthelot \cite{B_1974} 
after a suggestion due to Grothendieck in the form of crystalline cohomology $H^\ast_{\textrm{cris}}(X/W(k))$.
A drawback of crystalline cohomology is that it works well only for proper and smooth schemes -- 
for singular or non-proper schemes, the crystalline cohomology groups are not necessarily finitely generated over $W(k)$. 
In this case rigid cohomology $H^i_{\mathrm{rig}}(X/K)$ introduced by Berthelot \cite{B_1996}, \cite{B_1997} 
has become an important tool. However, rigid cohomology has coefficients in the fraction field $K$
of $W(k)$ and hence it is not an integral cohomology theory.

For a smooth variety $X$ over $k$, 
Davis--Langer--Zink \cite{DLZ_2011} introduced the overconvergent de Rham--Witt 
complex $W^{\dagger}\Omega_X^{\kr}$ as a certain subcomplex of Illusie's de Rham--Witt complex
$W\Omega_X^{\kr}$ (a complex of \'etale sheaves whose cohomology is isomorphic to the crystalline cohomology \cite{I_1979})
 and they \cite{DLZ_2011} and Lawless \cite{L_2018} proved that 
its rational cohomology $H^\ast(X, W^{\dagger}\Omega_X^{\kr}) \otimes_{ \Z} \Q$ 
is isomorphic to the rigid cohomology $H^\ast_{\mathrm{rig}}(X/K)$. 
While it is well-known that the integral overconvergent de Rham--Witt 
cohomology $H^\ast(X, W^{\dagger}\Omega_X^{\kr})$ can have infinitely generated torsion,
one might still hope that modulo torsion, it gives a  finitely generated $W(k)$-lattice in $H^\ast_{\mathrm{rig}}(X/K)$.
But in \cite{ES_2020} the first two authors have produced counterexamples to this assertion as well.

In this paper, we  would like to ask the question under which conditions one can expect the existence of a 
``good'' cohomology theory with coefficients in $W(k)$ for varieties $X$ over $k$. 
The basic requirements for ``goodness'' should include that the cohomology groups 
are finitely generated 
over $W(k)$, that they coincide with (log) crystalline cohomology for (log) smooth and proper varieties, 
and that they are rationally isomorphic to rigid cohomology.

In the first part of this paper, we construct  under certain assumptions of resolution of singularities in positive characteristic
(Hypotheses \ref{strong resolutions}, \ref{embedded resolutions}, \ref{embedded resolution with boundary} and \ref{weak factorisation}) 
a cohomology theory on the category $\Var_k$ of $k$-varieties
$$ R\Gamma_{\cdh}(X,a_{\cdh}^\ast A^{\kr}),\; X\in \Var_k$$
and show that it provides a ``good'' integral $p$-adic cohomology theory in the following sense:

\begin{theorem*}{\rm (Corollary \ref{cor: cdp-rh}, Theorems \ref{thm: fin gen}, \ref{thm:comparison})}
The cohomology theory  $R\Gamma_{\cdh}(-,a_{\cdh}^\ast A^{\kr})$ satisfies the following conditions:
\begin{enumerate}
\item
For any $X \in \Var_k$, the cohomology groups of $R\Gamma_{\cdh}(X,a_{\cdh}^\ast A^{\kr})$ are finitely generated over $W(k)$. 
\item 
If $X$ has a normal crossing compactification \textup{(}see Definitions \ref{def:normal crossing compactification}, 
\ref{def:pair} for definition\textup{)} by a proper smooth variety $\overline{X}$ and we denote by $(X,\overline{X})$ the log scheme
whose underlying scheme is  $\overline{X}$ endowed with the log structure 
associated to ``the divisor at infinity'' $\overline{X} \setminus X$ 
there exists a functorial quasi-isomorphism 
$$R\Gamma_{\cdh}(X,a_{\cdh}^\ast A^{\kr}) \simeq R\Gamma_{\cris}((X,\overline{X})/W(k)), $$ 
where $R\Gamma_{\cris}((X,\overline{X})/W(k))$ denotes the log crystalline cohomology  complex.
\item 
For any $X \in \Var_k$, there exists a functorial quasi-isomorphism 
$$R\Gamma_{\cdh}(X,a_{\cdh}^\ast A^{\kr}) \otimes_{\Z} \Q \simeq R\Gamma_{\rig}(X/K),$$
where $R\Gamma_{\rig}(X/K)$ denotes the rigid cohomology complex. 
\end{enumerate}
\end{theorem*}

\begin{remark*}
A recent preprint by  Merici \cite{M_p2024} proposes a motivic approach to the same question, 
obtaining a cohomology theory that agrees with the one presented in this paper.  
His approach allows to weaken some of the hypotheses on resolution of singularities, and to deduce some of the properties of the cohomology theory like the Künneth formula purely formally. 
This showcases nicely the interplay between explicit and more formal approaches. 
\end{remark*}

The idea of the construction is as follows: 
Let us first recall that the log crystalline cohomology provides a candidate for a good $p$-adic cohomology 
on a smooth $k$-variety $X$ if $X$ admits a normal crossing compactification $\overline{X}$. 
In the first step, we prove -- under assumption of resolution of singularities -- 
that the log crystalline cohomology of a smooth $k$-variety does not depend upon the choice of a normal crossing compactification. 
In the second step, we use -- again assuming resolution of singularities -- $\cdh$-sheafification to extend this to a cohomology theory on all $k$-varieties. Finally, we prove descent statements, which allow us to compare our new cohomology theory to the existing theories.
This construction generalises an approach considered by Mokrane for smooth open varieties in \cite{M_1993(a)}.

Since the existence of resolution of singularities has not been established in full generality in positive characteristic,
one might have the idea to use de Jong's alteration theorem \cite{dJ_1996}  instead. 
Unlike resolutions, alterations are not coverings in the $\cdh$-topology,  but in the proper topology.
In the second part of the paper, we study the question
whether $R\Gamma_{\cris}((X_{\bullet},\overline{X}_{\bullet})/W(k))$, for a given variety $X$, is independent of 
the choice of a split  proper generically \'etale hypercovering $X_{\bullet}$ of $X$ 
with normal crossing compactification $\overline{X}_{\bullet}$. 
It turns out that this is not the case in general.
We have the following theorem:

\begin{theorem*}{\rm (Theorem \ref{thm:H^1}, Proposition \ref{prop:negative}, Corollary \ref{cor:negativehighercoh})}
Let $X$ be a $k$-variety, let $\overline{X}$ be its compactification and 
let $X_{\bullet}$ be a split proper generically \'etale hypercovering of $X$ 
with a normal crossing compactification $\overline{X}_{\bullet}$ by a proper and smooth simplicial 
$k$-variety over $\overline{X}$. 
\begin{enumerate}
\item For $i=0,1$ the cohomology group $H^i_{\cris}((X_{\bullet},\overline{X}_{\bullet})/W(k))$ is independent 
of the choice of 
$(X_{\bullet},\overline{X}_{\bullet})$.

\item 
For $i\geqslant 2$ the cohomology group $H^i_{\cris}((X_{\bullet},\overline{X}_{\bullet})/W(k))$ is 
\textbf{not} in general independent 
of the choice of
$(X_{\bullet},\overline{X}_{\bullet})$.
\end{enumerate}
\end{theorem*}

The first statement has been shown by Andreatta--Barbieri-Viale for  $p\geqslant 3$  in \cite{ABV_2005}. We give an alternative proof which also works in characteristic $2$. Moreover, our argument to prove the assertion (ii) for $i=2$ implies 
that there does \textbf{not exist} a functor 
$$ A_{\et}^{\kr}: \Sm_k \rightarrow C^{\geqslant 0}(W(k)),$$
where $\Sm_k$ denotes the category of smooth $k$-varieties 
and $C^{\geqslant 0}(W(k))$ denotes the category of complexes of $W(k)$-modules of non-negative degree, 
which provides a ``good'' integral $p$-adic cohomology theory in the following sense (see Remark \ref{rem:etexplicit}):

\begin{enumerate}
\item
For any smooth $k$-variety $X$ with a normal crossing compactification $\overline{X}$ by a proper and smooth $k$-variety, 
there exists a functorial quasi-isomorphism 
$$A^{\kr}_{\et}(X) \simeq R\Gamma_{\cris}((X,\overline{X})/W(k)).$$ 
\item
It satisfies Galois descent in the sense that, for any \v{C}ech hypercovering 
$X_{\bullet} \rightarrow X$ associated to a finite \'etale Galois covering 
$X_0 \rightarrow X$, 
the induced morphism 
$$A^{\kr}_{\et}(X) \rightarrow A^{\kr}_{\et}(X_{\bullet})$$ 
is a quasi-isomorphism. 
\end{enumerate}

This is compatible with the non-existence result of Abe and Crew in \cite{AC}, which says that there is no integral $p$-adic cohomology theory which is finitely generated, coincides rationally with rigid cohomology and satisfies finite \'etale descent.

\begin{acknowledgments}
We would like to thank Shane Kelly, 
Wies\l{}awa Nizio\l{} and Alexander Schmidt for helpful discussions related to this project.
We are also very grateful to Aise Johan de Jong for his suggestions concerning several properties of our cohomology theory.
The first named author would like to thank the Isaac Newton Institute for Mathematical Sciences for support and hospitality during the programme ``$K$-theory, algebraic cycles and motivic homotopy theory'' when work on this paper was undertaken.
\end{acknowledgments}

%%%%%%%%%%%%%%%
\subsection*{Conventions}\label{conventions}
%%%%%%%%%%%%%%%

Throughout the paper, $p$ is a fixed prime, $k$ is a perfect field of characteristic $p$, 
$W(k)$ is the ring of Witt vectors of $k$ and $K=\Frac(W(k))$ is the fraction field of $W(k)$. 
By a variety over $k$ or a $k$-variety we mean a reduced separated scheme of finite type over $k$.
We denote by $\Var_k$ the category of varieties over $k$, and by $\Sm_k$ the category of smooth varieties over $k$.
Note that finite limits exist in $\Var_k$: The limit of a finite diagram $\{X_i\}_i$ in $\Var_k$ 
is given by $(\varprojlim_i X_i)_{\rm red}$, where $\varprojlim_i X_i$ is the limit of the diagram $\{X_i\}_i$ 
in the category of schemes over $k$ and $(-)_{\rm red}$ denotes the maximal reduced closed subscheme. 
On the other hand, finite limits do not always exist in $\Sm_k$.

%%%%%%%%%%%%%%%%%%%%%%
%%%
\section{Construction under assumption of resolution of singularities}
%%%
%%%%%%%%%%%%%%%%%%%%

Let $X$ be a smooth variety over $k$. 
When studying finiteness properties it is common to consider compactifications of $X$. 
While by Nagata's compactification theorem every smooth $k$-variety has a compactification (i.e., a quasi-compact open immersion into a proper $k$-variety), it might be a rather complicated one.
However, under the assumption of resolution of singularities, it is possible to reduce to the case of normal crossing compactifications. 

%%%%%
\subsection{Assumptions concerning resolution of singularities}
%%%%%

In this section, we will assume  resolution of singularities in positive characteristic.
To make this notion precise we first introduce some notations and terminology which  streamline the discussion about compactifications.

\begin{definition}\label{def:pair}
A \emph{geometric pair} is a pair $(X,\overline{X})$ of $k$-varieties 
 such that $\overline{X}$ is proper and 
equipped with an open immersion $X\hookrightarrow\overline{X}$ with dense image.
A geometric  pair is  a \emph{normal crossing  pair} (or an $nc$-\emph{pair} for short)  
if $\overline{X}$ is proper smooth and $\overline{X} \backslash X$ 
is a simple normal crossing divisor in $\overline{X}$. 

A morphism of geometric pairs $f:(X_1,\overline{X}_1) \rightarrow (X_2,\overline{X}_2)$ is a morphism of $k$-varieties $f:\overline{X}_1 \rightarrow \overline{X}_2$ such that $f(X_1)\subset X_2$.
A morphism of normal crossing pairs is a morphism  of geometric pairs.

Denote by $\Var_k^{geo}$ and $\Var_k^{nc}$ the categories of geometric and normal crossing pairs respectively.
\end{definition}

Beware that, when we say that $(X,\overline{X})$ is a normal crossing pair, we always assume that the complement $\overline{X} \setminus X$ is a \emph{simple} normal crossing divisor. 

\begin{remark}\label{rem:finiteinverselimit}
 Note that finite limits exist in $\Var_k^{geo}$: The limit of a finite diagram 
$\{(X_i,\overline{X}_i)\}_i$ in $\Var_k^{geo}$ is given by $(\varprojlim_i X_i, (\varprojlim_i \overline{X}_i)')$, 
where $\varprojlim_i X_i$, $\varprojlim_i \overline{X}_i$ are the limits of the diagrams $\{X_i\}_i, \{\overline{X}_i\}_i$ 
in the category $\Var_k$ respectively and $(\varprojlim_i \overline{X}_i)'$ is the closure of 
$\varprojlim_i X_i$ in $\varprojlim_i \overline{X}_i$ with reduced closed subscheme structure. 
On the other hand, finite limits do not always exist in $\Var_k^{nc}$. 
\end{remark}

\begin{definition}
Let $X$ be a smooth $k$-variety, $D$ a simple normal crossing divisor in $X$ and 
$Z$ a smooth closed $k$-subvariety of $X$. We say that \emph{$Z$ has normal crossing with $D$} if, 
Zariski locally on $X$, there exists an \'etale morphism 
\[ 
X \rightarrow  \Spec k[x_1,\dots,x_a,\dots,x_b,\dots, x_c,\dots,x_d],
\]
for some $0 \leqslant a \leqslant b \leqslant c \leqslant d$ such that
\begin{align*}
	Z&=\{ x_1=\cdots=x_b=0 \}, \\
	D&=\{ x_1x_2\cdots x_a x_{b+1}\cdots x_c=0 \}.
\end{align*}
\end{definition}

\begin{definition}
\begin{enumerate}
\item
A morphism of geometric pairs $f:(X_1,\overline{X}_1) \rightarrow (X_2,\overline{X}_2)$ is called \emph{strict} if $f^{-1}(X_2)=X_1$.

\item 
A morphism of geometric pairs $f:(X_1,\overline{X}_1) \rightarrow (X_2,\overline{X}_2)$ is called \emph{birational}, if the underlying morphism of $k$-varieties $f:\overline{X}_1 \rightarrow \overline{X}_2$ is birational. The morphism $f:(X_1,\overline{X}_1) \rightarrow (X_2,\overline{X}_2)$ is called \emph{a closed immersion}, if the underlying morphism $f:\overline{X}_1 \rightarrow \overline{X}_2$ is a closed immersion.
\item 
A \emph{weak factorisation} (compare \cite[\S\,1.2]{AT_2019}) of a strict birational morphism $f:(X_1,\overline{X}_1)\rightarrow (X_2,\overline{X}_2)$ of normal crossing  pairs which is an isomorphism on $ X_2$ is 
a diagram of rational maps
	$$
	\xymatrix{
	(X_1,\overline{X}_1)=(V_0,\overline{V}_0) \ar@{-->}[r]^-{f_1}& (V_1,\overline{V}_1) \ar@{-->}[r]^-{f_2} & \cdots \ar@{-->}[r] & (V_{l-1},\overline{V}_{l-1})\ar@{-->}[r]^-{f_l}  & (V_l,\overline{V}_l)= (X_2,\overline{X}_2)}
	$$
such that
\begin{enumerate}
\item 
the composition $f_l\circ f_{l-1}\circ\cdots\circ f_2\circ f_1$ gives $f$;

\item 
for $i\in\{0,\ldots,l-1\}$ the maps $\xymatrix{(V_i,\overline{V}_i)\ar@{-->}[r]&(X_2,\overline{X}_2)}$ are  strict morphisms and induce isomorphisms on $X_2$;

\item 
for every $i\in\{1,\ldots,l\} $ either $f_i$ or $f_i^{-1}$ is a blow-up along a smooth centre $Z_i$ which is a subscheme of $\overline{V}_i$ or $\overline{V}_{i-1}$ respectively disjoint from (the inverse image of) $X_2$;

\item 
for each $i\in\{1,\ldots,l\}$ the subscheme $Z_i$ of $\overline{V}_i$ or $\overline{V}_{i-1}$ respectively has normal crossing with the normal crossing divisor $\overline{V}_i\backslash V_i$ or $\overline{V}_{i-1}\backslash V_{i-1}$ respectively.
\end{enumerate}
\end{enumerate}
\end{definition}

Conceptually, normal crossing pairs are to geometric pairs 
what smooth varieties are to varieties.
In this section we assume resolution of singularities for both situations.
We will indicate clearly where each of the hypotheses is needed.

\begin{hypothesis}[Strong resolution of singularities]\label{strong resolutions}
\phantom{ }
\begin{enumerate}
\item 
 For any $k$-variety $X$ there exists a proper birational morphism $f:X'\rightarrow X$ from a smooth $k$-variety $X'$  
which is an isomorphism on the smooth locus $X_{\sm}$ of $X$.
 
\item 
For every proper birational morphism  $f:X'\rightarrow X$ of smooth $k$-varieties, there is a sequence of birational blow-ups along smooth centres $X_n \rightarrow X_{n-1} \rightarrow \cdots \rightarrow X_1 \to X$ such that the composition $X_n \rightarrow X$ factors through $f$. 
\end{enumerate}
\end{hypothesis}

\begin{remark}
Note that it would have been sufficient to assume strong resolution of singularities for irreducible varieties, i.e., integral schemes. 
Then one obtains a resolution for an arbitrary variety as the disjoint union of resolutions of its irreducible components.
\end{remark}

\begin{hypothesis}[Embedded resolution of singularities]\label{embedded resolutions}
For any geometric  pair $(X,\overline{X})$ with  $\overline{X}$ smooth, there exists a strict birational morphism $f:(X,\overline{X}')\rightarrow (X,\overline{X})$ from a normal crossing pair $(X,\overline{X}')$ which  is an isomorphism on $X$. 
\end{hypothesis}

Additionally, we will need stronger assumptions concerning certain morphisms of geometric and normal crossing 
 pairs.

\begin{hypothesis}[Embedded resolution of singularities with boundary]\label{embedded resolution with boundary}
For any strict closed immersion of geometric pairs $(Y,\overline{Y})\rightarrow (X,\overline{X})$ with $Y$ smooth and $(X,\overline{X})$ a normal crossing pair
there exists a commutative diagram
	$$
	\xymatrix{(Y,\overline{Y}') \ar@{^{(}->}[r] \ar[d] & (X,\overline{X}') \ar[d]\\
	(Y,\overline{Y}) \ar@{^{(}->}[r] & (X,\overline{X})}
	$$
such that $(Y,\overline{Y}')$ and $(X,\overline{X}')$ are normal crossing  pairs, 
the horizontal morphisms are strict  closed immersions, 
the vertical morphisms are strict birational morphisms which are isomorphisms outside $\overline{Y} \setminus Y$, and $\overline{Y}'$ has normal crossing with $\overline{X}' \setminus X$. (Compare \cite[Thm.\,1.4]{CJS}.)
\end{hypothesis}

\begin{hypothesis}[Weak factorisation]\label{weak factorisation}
For any strict birational morphism $(X,\overline{X}')\rightarrow (X,\overline{X})$ of normal crossing pairs 
which is an isomorphism on $X$, 
there  exists a weak factorisation  of it.
\end{hypothesis}

\begin{remark}
Currently, Hypothesis \ref{strong resolutions}(i) is known in the case of dimension $\leqslant 3$ (\cite{CP1}, \cite{CP2}) 
and Hypotheses \ref{strong resolutions}(ii), \ref{embedded resolutions}, 
\ref{embedded resolution with boundary}, \ref{weak factorisation} are known 
in the case of dimension $\leqslant 2$. (For Hypothesis \ref{embedded resolution with boundary}, 
see \cite[Thm.\,1.4]{CJS}.) It is proven in \cite{AT_2019} that Hypothesis \ref{weak factorisation} holds 
if a certain hypothesis on usual and embedded resolutions stronger than 
Hypotheses \ref{strong resolutions} and 
\ref{embedded resolutions} is true. 
\end{remark}

%%%%%
\subsection{The $\cdp$- and $\cdh$-topology on $k$-varieties}
%%%%%

In this subsection we will consider the completely decomposed proper topology (also called the $\cdp$-topology) 
 and the $\cdh$-topology
on the category $\Var_k$ of $k$-varieties  and the category $\Sm_k$ of smooth $k$-varieties. 
We prove that, under Hypothesis \ref{strong resolutions}, the $\cdp$-topology (resp.\ the $\cdh$-topology) on $\Sm_k$ is generated by blow-ups with smooth centre (resp.\ blow-ups with smooth centre and Nisnevich coverings).

Recall that a \emph{$\cd$-structure} on a small category $\mathcal{C}$ 
 with initial object $\emptyset$
is a class $\mathcal{Q}$ of commutative squares 
\begin{equation}\label{eq:defcdstr}
	\xymatrix{B \ar[r] \ar[d] & Y \ar^-{p}[d]\\
	A \ar^-{e}[r] & X} 
\end{equation}
which is closed under isomorphisms \cite[Def.\,2.1]{V_2010}. 
For a family $\mathcal{Q}_i \, (i \in I)$ of $\cd$-structures on $\mathcal{C}$, 
the topology generated by $\mathcal{Q}_i \, (i \in I)$ is defined to be 
the topology generated by the coverings $\{A \xrightarrow{e} X, Y \xrightarrow{p} X\}$ 
for commutative squares \eqref{eq:defcdstr} in some $\mathcal{Q}_i$ and 
the empty covering of the initial object $\emptyset$.

\begin{definition}\label{def: squares}
We consider the following $\cd$-structures on $\Var_k$:
\begin{enumerate}
\item \label{eq:blowup_cd_square} The \emph{blow-up $\cd$-structure} on $\Var_k$ consists of Cartesian squares of the form
	\begin{equation*}
	\xymatrix{Z' \ar@{^{(}->}[r] \ar[d] & X' \ar^{p}[d]\\
	Z \ar@{^{(}->}^{e}[r] & X,}
	\end{equation*}
where $e$ is a closed immersion and $p$  is a proper morphism which is an isomorphism over $X \setminus e(Z)$. 

\item \label{eq:Zariski_cd_square} The \emph{Nisnevich $\cd$-structure} on $\Var_k$ 
 consists of
Cartesian squares of the form
	\begin{equation*}
	\xymatrix{Y' \ar@{^{(}->}[r] \ar[d] & X' \ar^{p}[d]\\
	Y \ar@{^{(}->}^{e}[r] & X,}
\end{equation*}
where $e$ is an open immersion and $p$ is an \'etale morphism such that the morphism $p^{-1}(X \setminus e(Y)) \rightarrow X \setminus e(Y)$ 
(with reduced closed subscheme structure) induced by $p$ is an isomorphism.
\end{enumerate}
\end{definition}

These $\cd$-structures and their properties have been studied by Voevodsky in \cite{V_2010a,V_2010}. Note that our blow-up $\cd$-structure is called \emph{lower $\cd$-structure} in \cite{V_2010a}, while the Nisnevich $\cd$-structure is also called \emph{upper $\cd$-structure} in \cite{V_2010a}. In \cite{V_2010a} Voevodsky has proven the following:

\begin{theorem}[{\cite[Thm.\,2.2]{V_2010a}}]
	The blow-up $\cd$-structure and the Nisnevich $\cd$-structure on $\Var_k$ are complete, 
regular and bounded. \textup{(}For definitions of these terms, 
see \cite[Def.\,2.3, 2.10, 2.22]{V_2010}.\textup{)}
\end{theorem}

Using the above $\cd$-structures we define the $\cdp$-topology, the Nisnevich topology and the $\cdh$-topology on $\Var_k$:

\begin{definition}
The \emph{$\cdp$-topology on $\Var_k$} is the topology generated by the blow-up $\cd$-structure on $\Var_k$. The \emph{Nisnevich topology on $\Var_k$} is the topology generated by the Nisnevich $\cd$-structure.
The \emph{$\cdh$-topology on $\Var_k$} is the topology generated by the blow-up $\cd$-structure and the Nisnevich $\cd$-structure.

We denote by $\Var_{k,\cdp}$ and $\Var_{k,\cdh}$ the sites obtained from the category of $k$-varieties endowed with the $\cdp$- and $\cdh$-topology respectively.
\end{definition}

Let us observe that the morphism
	\[
	e\coprod p\colon Z \coprod X'\rightarrow X
	\]
induced by a blow-up  square in Definition \ref{def: squares}\ref{eq:blowup_cd_square} is a completely decomposed proper morphism in the usual sense:

\begin{definition}\label{def:cdpmorphism}
	A morphism $p\colon Y\rightarrow X$ in $\Var_k$ is a \emph{completely decomposed proper morphism} if it is proper and for all $x\in X$ there is a  point $y\in p^{-1}(\{x\})\subseteq Y$ such that the induced morphism on residue fields $\kappa(x)\rightarrow \kappa(y)$ is an isomorphism.
\end{definition}

Conversely, the following  lemma shows that every completely decomposed proper morphism is a covering in the $\cdp$-topology:

\begin{lemma}[{cf.\,\cite[Lem.\,5.8]{SV_2000}}]\label{lem:cdpvardescription}
 The topology $\tau$ on $\Var_k$ generated by the families of morphisms 
$\{f_i \colon X_i \rightarrow X\}_{i \in I}$ with $I$ finite and 
$\coprod_i f_i \colon \coprod_i X_i \rightarrow X$ a completely decomposed proper morphism 
coincides with the $\cdp$-topology on $\Var_k$. In particular, every
completely decomposed proper morphism in $\Var_k$ is 
a covering in the $\cdp$-topology on $\Var_k$. 
\end{lemma}

\begin{proof}
The observation before Definition \ref{def:cdpmorphism} implies that $\tau$ is finer than 
the $\cdp$-topology. To prove the converse, take a $\tau$-covering  
$\{f_i \colon X_i \rightarrow X\}_{i \in I}$ as in the statement of the lemma, and we will prove that 
it is a $\cdp$-covering. By Noetherian induction, we may assume that, 
for any proper closed subvariety $T \subsetneq X$, the pullback of the family to $T$ 
is a $\cdp$-covering. Note also that, for a covering $X = \bigcup_{j \in J}  Y_j$ of $X$ 
by a finite number of closed subvarieties, the family $\{ Y_j \rightarrow X\}_{j \in J}$ is a 
$\cdp$-covering: Indeed, we can reduce to the case $J=\{1,2\}$ and in this case, 
the square 
$$ 
\xymatrix{  Y_1 \cap  Y_2 \ar[r] \ar[d] &  Y_1 \ar[d] \\ 
 Y_2 \ar[r] & X }
$$ 
is a blow-up square. From these, we may assume that $X$ is integral to prove the claim. 

Let $\xi$ be the generic point of $X$. Then there exists an index $i \in I$ and 
a point $\xi' \in X_i$ over $\xi$ such that the 
the induced morphism $\kappa(\xi) \rightarrow \kappa(\xi')$ is an isomorphism. 
Then, if we define $Y$ to be the closure of $\xi'$ in $X_i$, 
$f_i|_Y: Y \rightarrow X$ is a proper birational morphism. 
Let $T \subsetneq X$ be the locus where $f_i|_Y$ is not an isomorphism. 
Then $\{Y \rightarrow X, T \rightarrow X\}$ is a $\cdp$-covering. 
The pullback of the family $\{f_i \colon X_i \rightarrow X\}_{i \in I}$ to $Y$ is a $\cdp$-covering 
because the pullback of the morphism $X_i \rightarrow X$ to $Y$ admits a section. 
Also, the pullback of the family $\{f_i \colon X_i \rightarrow X\}_{i \in I}$ 
to $T$ is a $\cdp$-covering by induction hypothesis. 
Thus the family $\{f_i \colon X_i \rightarrow X\}_{i \in I}$
is a $\cdp$-covering locally in the $\cdp$-topology and hence it is a $\cdp$-covering.
\end{proof}

From now on we use the terms $\cdp$-morphism and completely decomposed proper morphism interchangeably.

Next we recall the analogue of Lemma \ref{lem:cdpvardescription} for the Nisnevich topology, which is more or less well-known. The morphism
	\[
	e\coprod p\colon Y \coprod X'\rightarrow X
	\]
induced by a Nisnevich square in Definition \ref{def: squares}\ref{eq:Zariski_cd_square} is a Nisnevich morphism in the usual sense:

\begin{definition}\label{def:Nismorphism}
	A morphism $p\colon Y\rightarrow X$ in $\Var_k$ is a \emph{Nisnevich morphism} if it is \'etale and for all $x\in X$ there is a  point $y\in p^{-1}(\{x\})\subseteq Y$ such that the induced morphism on residue fields $\kappa(x)\rightarrow \kappa(y)$ is an isomorphism.
\end{definition}

\begin{lemma}[{cf.\,\cite[\S\,3 Prop.\,1.4]{MV_1999}}]\label{lem:Nisvardescription}
 The topology $\tau$ on $\Var_k$ generated by the families of morphisms 
$\{f_i \colon X_i \rightarrow X\}_{i \in I}$ with $I$ finite and 
$\coprod_i f_i \colon \coprod_i X_i \rightarrow X$ a Nisnevich morphism 
coincides with the Nisnevich topology on $\Var_k$. In particular, every Nisnevich morphism in $\Var_k$ is a covering in the Nisnevich topology on $\Var_k$. 
\end{lemma}

\begin{proof}
The observation before Definition \ref{def:Nismorphism} implies that $\tau$ is finer than the Nisnevich topology. 
To prove the converse, take a $\tau$-covering  
$\{f_i \colon X_i \rightarrow X\}_{i \in I}$ as in the statement of the lemma. Following \cite[p.\,97]{MV_1999}, we call a sequence of closed subvarieties of the form 
$$ \emptyset = Z_{n+1} \subseteq Z_n \subseteq \cdots \subseteq Z_0 = X $$
a splitting sequence of length $n$ if the morphism 
$(\coprod_i f_i)^{-1}(Z_j \setminus Z_{j+1}) 
\rightarrow Z_j \setminus Z_{j+1}$ induced by $\coprod_i f_i$ 
splits for any $0 \leqslant j \leqslant n$. By \cite[\S\,3 Lem.\,1.5]{MV_1999}, any 
$\tau$-covering admits a splitting sequence. We prove that 
the above $\tau$-covering is a covering in the Nisnevich topology 
on $\Var_k$ by induction on the minimal length $n$ of its splitting sequence. 

If $n=0$, the morphism $\coprod_i f_i$ splits and so it is a a covering in the Nisnevich topology on $\Var_k$. If $n>0$, choose a splitting $s$ of the morphism $(\coprod_i f_i)^{-1}(Z_n) \rightarrow Z_n$. Since $\coprod_i f_i$ is \'etale, we have $(\coprod_i f_i)^{-1}(Z_n) = {\rm Im}\,s \coprod C$ for some closed subvariety $C$ of $\coprod_i X_i$. Let $Y = X \setminus Z_n$, $X' = (\coprod_i X_i) \setminus C$, $Y' = Y \times_X X'$. Then the square 
$$ 
\xymatrix{  Y' \ar[r] \ar[d] & X' \ar[d] \\ 
 Y \ar[r] & X }
$$ 
(where the vertical arrows are induced by $\coprod_i f_i$ and horizontal arrows are canonical open immersions) is a Nisnevich square. 
Since the pullback of the $\tau$-covering  
$\{f_i \colon X_i \rightarrow X\}_{i \in I}$ to $Y$ has a splitting sequence of length $n-1$, it is a covering in the Nisnevich topology on $\Var_k$ by the induction hypothesis. 
Also, the pullback of the $\tau$-covering  
$\{f_i \colon X_i \rightarrow X\}_{i \in I}$ to $X'$ has a refinement of the form $\{X_i \setminus C \rightarrow \coprod_i (X_i \setminus C) = X'\}_{i \in I}$, which is easily seen to be a covering in the Nisnevich topology on $\Var_k$. Hence $\{f_i \colon X_i \rightarrow X\}_{i \in I}$ is a covering in 
the Nisnevich topology on $\Var_k$, as required. 
\end{proof}

Next we will prove the analogue of Lemma \ref{lem:cdpvardescription} for the $\cdh$-topology. 

\begin{lemma}[{cf.\,\cite[Prop.\,5.9]{SV_2000}}]\label{lem:rhvardescription}
The topology $\sigma$ on $\Var_k$ generated by the families of morphisms 
$\{f_i \colon X_i \rightarrow X\}_{i \in I}$ with $I$ finite such that 
$\coprod_i f_i \colon \coprod_i X_i =: Y \rightarrow X$ can be
written as the composite 
$Y \xrightarrow{g} Z \xrightarrow{h} X$ of a Nisnevich morphism $g$ and a $\cdp$-morphism $h$ 
coincides with the $\cdh$-topology on $\Var_k$. 
\end{lemma}

\begin{proof}
For a $\sigma$-covering $\{f_i \colon X_i \rightarrow X\}_{i \in I}$ and a factorisation $\coprod_i X_i = Y \xrightarrow{g} Z \xrightarrow{h} X$ of $\coprod_i f_i$ 
as in the statement of the lemma, $h$ is a $\cdp$-covering by Lemma 
\ref{lem:cdpvardescription} and $g$ is a Nisnevich covering by Lemma 
\ref{lem:Nisvardescription}. Hence the above $\sigma$-covering is a 
$\cdh$-covering and so we conclude that the $\cdh$-topology is finer than $\sigma$.

We prove the converse. Since the blow-up and the Nisnevich cd-structures are complete, it suffices to prove that a family of morphisms 
$\{f'_{i'} \colon X'_{i'} \rightarrow X\}_{i' \in I'}$ with $I'$ finite such that 
$\coprod_{i'} f'_{i'} \colon \coprod_{i'} X_{i'} =: Y' \rightarrow X$ is written as the composite 
\begin{equation}\label{eq:zarcdpzarcdp}
Y' = Y'_n \rightarrow \cdots \rightarrow Y'_1 \rightarrow Y'_0 = Y  
\end{equation}
with each arrow a cdp-morphism or a Nisnevich morphism
admits a refinement of the form $\{f_i \colon X_i \rightarrow X\}_{i \in I}$ as in the statement of the lemma (see \cite[Def.\,2.2, Def.\,2.3]{V_2010}). 

We call a morphism of the form 
$Y \xrightarrow{g} Z \xrightarrow{h} X$ with $g$ a Nisnevich morphism and $h$ a $\cdp$-morphism 
a \emph{special morphism}.  
To prove the claim in the previous paragraph, it suffices to prove 
that the morphism \eqref{eq:zarcdpzarcdp} admits a refinement by a special morphism.
Working recursively, this in turn is reduced to 
proving the claim that a morphism of the form 
\[ T \xrightarrow{p} U \xrightarrow{q} X \] 
with $p$ a $\cdp$-morphism and $q$ a Nisnevich morphism 
admits a refinement by a special morphism. 
This claim is shown in 
\cite[Prop.\,5.9]{SV_2000} and we are done. 
\end{proof}

We now consider the fully faithful functor given by inclusion $\Sm_k\hookrightarrow \Var_k$.
Resolution of singularities plays an important role after this point.

\begin{definition}
The $\cdp$-\emph{topology on} $\Sm_k$ is  the topology induced by the inclusion $\Sm_k\hookrightarrow \Var_k$ from the $\cdp$-topology on $\Var_k$.
In other words, it is the restriction of the $\cdp$-topology from $\Var_k$ to $\Sm_k$. Similarly, the $\cdh$-\emph{topology on} $\Sm_k$ is the restriction of the $\cdh$-topology from $\Var_k$ to $\Sm_k$.
\end{definition}

\begin{proposition}\label{prop: base1}
Let $\tau$ be any topology on $\Var_k$ finer than the $\cdp$-topology.
Under Hypothesis \ref{strong resolutions}, the inclusion $\Sm_k\hookrightarrow \Var_k$ induces an equivalence of topoi
	$$\Sh(\Sm_{k,\tau}) \xrightarrow{\sim} \Sh(\Var_{k,\tau}). $$
In particular, there are equivalences of topoi
		\begin{eqnarray*}
		\Sh(\Sm_{k,\cdp}) &\xrightarrow{\sim}& \Sh(\Var_{k,\cdp}),\\
		\Sh(\Sm_{k,\cdh}) &\xrightarrow{\sim}& \Sh(\Var_{k,\cdh}).
		\end{eqnarray*}
\end{proposition}
\begin{proof}
By the hypothesis of strong resolution of singularities, for every object $X\in \Var_k$ there is a $\cdp$-morphism $f:X'\rightarrow X$ where $X'$ is smooth. This can be seen by an inductive argument as follows: By strong resolution of singularities, there is a proper birational morphism $\widetilde{f}\colon \widetilde{X}\rightarrow X$ with $\widetilde{X}$ smooth. Let us denote by $U\subseteq X$ an open dense subset such that $\tilde{f}$ induces an isomorphism $\widetilde{f}^{-1}(U)\cong U$, and let $T:=X\setminus U$ denote its complement. Then $\widetilde{X}\coprod T\rightarrow X$ is a $\cdp$-covering. By induction, there is a $\cdp$-morphism $T'\rightarrow T$ with $T'$ smooth. The disjoint union $X':=\widetilde{X}\coprod T'\rightarrow X$ gives the desired $\cdp$-covering of $X$ by a smooth $k$-variety.\par  
This means that any $k$-variety can be covered in the $\cdp$-topology by smooth $k$-varieties.
As the inclusion $\Sm_k\hookrightarrow \Var_k$ is a fully faithful functor, 
the statement follows from \cite[Thm.\,4.1]{V_1972}.
\end{proof}

\begin{remark}
	Due to the lack of enough fibre products, the $\cdp$-topology (resp.\ the $\cdh$-topology) on $\Sm_k$ is not defined in terms of covering families, but in terms of covering sieves. Using \cite[Prop.\,3.2]{V_1972} and the fact that any $k$-variety can be covered in the $\cdp$-topology by smooth $k$-varieties, we can describe the covering sieves explicitly in the following way (see also \cite[Prop.\,5.11]{SV_2000}): Every $\cdp$-covering (resp.\ $\cdh$-covering) $\{ X_i\rightarrow X \}_i$ in $\Var_k$ of a smooth $k$-variety $X$  defines a covering sieve of $X$ in $\Sm_k$ which consists of all morphisms of smooth $k$-varieties $X'\rightarrow X$ factoring through one of the $X_i$.  
The covering sieves on $\Sm_k$ are the ones
which contain some sieve of the form just described.
\end{remark}

We give  now an interpretation of the $\cdp$- and  the $\cdh$-topology on $\Sm_k$ in terms of $\cd$-structures.

\begin{definition}\label{def:smoothsquare}
We consider the following $\cd$-structures on $\Sm_k$:
\begin{enumerate}
\item\label{eq:smooth blowup_cd_square}
The \emph{smooth blow-up $\cd$-structure} on $\Sm_k$ consists of Cartesian squares of smooth $k$-varieties
	\begin{equation*}
	\xymatrix{Z' \ar@{^{(}->}[r] \ar[d] & X' \ar^{p}[d]\\
	Z \ar@{^{(}->}^{e}[r] & X}
	\end{equation*}
where $e$ is a closed immersion and $p:X'\rightarrow X$ is a blow-up with centre $Z$. 

\item \label{eq:smooth Zariski_cd_square} The  \emph{smooth Nisnevich $\cd$-structure} on $\Sm_k$ 
consists of Cartesian
squares of the form
	\begin{equation*}
	\xymatrix{Y' \ar@{^{(}->}[r] \ar[d] & X' \ar^{p}[d]\\
	Y \ar@{^{(}->}^{e}[r] & X}
\end{equation*}
where $e$ is an open immersion and $p$ is an \'etale morphism such that the morphism $p^{-1}(X \setminus e(Y)) \rightarrow X \setminus e(Y)$ 
(with reduced closed subscheme structure) induced by $p$ is an isomorphism.
\end{enumerate}
\end{definition}

In \cite{V_2010a} Voevodsky has proven the following:

\begin{proposition}[{\cite[Lem.\,4.3, 4.4,  4.5, Thm.\,2.2]{V_2010a}}]
\begin{enumerate}
\item
	The smooth blow-up $\cd$-structure on $\Sm_k$ is regular and bounded. 
Assuming strong resolution of singularities, i.e., Hypothesis \ref{strong resolutions}, it is also complete.

\item
The smooth Nisnevich $\cd$-structure on $\Sm_k$ is complete, regular and bounded.
\end{enumerate}
\end{proposition}

\begin{proposition}\label{prop: scdp_smoothbu}
 	Assuming Hypothesis \ref{strong resolutions}, every birational 
blow-up $X'=\Bl_Z X\rightarrow X $ with smooth centre $Z\subseteq X$ is 
 	a covering of $X\in \Sm_k$ in the topology generated by the smooth blow-up $\cd$-structure.
 \end{proposition}
 
\begin{proof}
In this proof, we denote by $\tau$ the 
topology on $\Sm_k$ generated by the smooth blow-up $\cd$-structure.

Since $X' \rightarrow X$ is birational, 
$Z$ is a proper closed subvariety of $X$ on each connected component of $X$. 
We prove the claim by induction on $\dim X$. 
In dimension $0$ the claim is trivial. 
In general case, consider the smooth blow-up square
 \[
 \xymatrix{Z' \ar@{^{(}->}[r] \ar[d] & X' \ar^{p}[d]\\
	Z \ar@{^{(}->}^{e}[r] & X.}
 \]
We observe that $X'\coprod Z\rightarrow X$ is 
a $\tau$-covering.

Let us assume that the statement is proven for all smooth varieties 
of dimension $< \dim X$. 
The pullback of $p\colon X'\rightarrow X$ to $X'$ is a $\tau$-covering.
Hence by locality it suffices to show that the pullback
$X'\times_X Z\rightarrow Z$ of $p$ to $Z$ can be refined 
by a $\tau$-covering.   
But this is the left vertical map $Z'\rightarrow Z$ in the smooth blow-up square, 
and since $Z$ is a proper closed subvariety of $X$ on each connected component of $X$, 
it is a (nonempty) projective bundle. 
Thus, we find a closed subvariety $Z''\hookrightarrow Z'$ of $Z'$ such that 
 \[
 	Z''\hookrightarrow Z'\rightarrow Z
 \]
 is a proper birational morphism. 
By strong resolution of singularities there is a refinement of $Z''\rightarrow Z$ given by a sequence of birational blow-ups with smooth centre. Since $Z$ is a proper closed subvariety of $X$ 
on each connected component of $X$, $\dim Z < \dim X$ holds and the same inequality  
holds for varieties appearing in the sequence of blow-ups. 
So, by induction hypothesis, this sequence of blow-ups is a $\tau$-covering of $Z$ 
refining $Z'\rightarrow Z$. 
Hence $X'\rightarrow X$ is a $\tau$-covering, as required. 
 \end{proof}

\begin{lemma}\label{lem:smoothNis}
A family of morphisms 
$\{f_i \colon X_i \rightarrow X\}_{i \in I}$ with $I$ finite and 
$\coprod_i f_i \colon \coprod_i X_i \rightarrow X$ a Nisnevich morphism 
is a covering of $X \in \Sm_k$ in the topology generated by the smooth Nisnevich  $\cd$-structure. 
\end{lemma}

\begin{proof}
The same proof as that in Lemma \ref{lem:Nisvardescription} works. 
\end{proof}

\begin{proposition}\label{cor:cdprhtoponsm}
Assume Hypothesis \ref{strong resolutions}.
\begin{enumerate}
\item The $\cdp$-topology on $\Sm_k$ coincides with the topology $\tau$ generated by 
the smooth blow-up $\cd$-structure. 
\item The $\cdh$-topology on $\Sm_k$ coincides with the topology $\sigma$ generated by the smooth blow-up and the smooth Nisnevich $\cd$-structures. 
\end{enumerate}
\end{proposition}

\begin{proof}
First we prove (i). Since every smooth blow-up square in $\Sm_k$ is a  
blow-up square in $\Var_k$, $\tau$ is coarser than the $\cdp$-topology on $\Sm_k$. 

To prove the converse, let $\mathcal{S}$ be a covering sieve in the $\cdp$-topology 
defined by a $\cdp$-covering $\{f_i \colon X'_i \rightarrow X \}_{i \in I}$ in $\Var_k$ 
with $X$ smooth, $I$ finite and $\coprod_i f_i \colon \coprod_i X'_i \rightarrow X$ a 
$\cdp$-morphism. (It is a covering sieve and it is sufficient to consider such covering sieves by Lemma 
\ref{lem:cdpvardescription}.) 
For each connected component $X_j$ of $X$ with generic point $\xi_j$, choose an index 
$i(j) \in I$ and a point $\xi'_j \in X'_{i(j)}$ above $\xi_j$ such that 
$\kappa(\xi_j) \rightarrow \kappa(\xi'_j)$ is an isomorphism. 
If we define $Y_j$ to be the closure of $\xi'_j$ in $X'_{i(j)}$, 
$Y_j \rightarrow X_j$ is a proper birational morphism in $\Var_k$. By replacing $Y_j$ by a suitable refinement 
using Hypothesis \ref{strong resolutions}, we may assume that it is a composition of birational 
blow-ups with smooth centre. Then, 
the family $\{Y_j \rightarrow X_j \hookrightarrow X\}_j$ is a $\tau$-covering 
because $\{X_j \hookrightarrow X\}_j$ is realized by an iteration of smooth blow-up squares of the form 
$$
\xymatrix{
\emptyset \ar[r] \ar[d] & X_1\ar[d] \\ 
X_2 \ar[r] & X_1 \coprod X_2
}
$$
and each $Y_j \rightarrow X_j$ is a $\tau$-covering by Proposition 
\ref{prop: scdp_smoothbu}. Also, this family 
refines the original family $\{f_i \colon X'_i \rightarrow X\}_{i \in I}$ by construction. 
Thus $\mathcal{S}$ is a covering sieve 
in the $\tau$-topology and so $\tau$ is finer than the $\cdp$-topology on $\Sm_k$, as desired. 

Next we prove (ii). Since every smooth blow-up square in $\Sm_k$ 
is a  blow-up square in $\Var_k$ and every smooth Nisnevich square in $\Sm_k$ is a Nisnevich square in $\Var_k$, $\sigma$ is coarser than the $\cdh$-topology on $\Sm_k$. 

To prove the converse, let $\mathcal{S}$ be a covering sieve in the $\cdh$-topology 
defined by $\cdh$-covering $\{f_i \colon X_i \rightarrow X\}_{i \in I}$ in $\Var_k$ with $X$ smooth, $I$ finite and $\coprod_i f_i \colon \coprod_i X_i \rightarrow X$ factors as $\coprod_i X_i \rightarrow Y \rightarrow X$, where the first morphism is a Nisnevich morphism 
and the second one is a $\cdp$-morphism. (It is a covering sieve 
and it is sufficient to consider such covering sieves by Lemma \ref{lem:rhvardescription}.) By replacing $Y$ by a suitable refinement using Hypothesis \ref{strong resolutions} 
and by replacing the $X_i$'s by their pullbacks, 
we may assume that the morphism $Y \rightarrow X$ is a composition of birational blow-ups with 
smooth centre. (Repeat the argument in the proof of (i) on each connected component of $X$.) Then the family $\{f_i \colon X_i \rightarrow X\}_{i \in I}$ is a $\sigma$-covering by Proposition \ref{prop: scdp_smoothbu} and Lemma \ref{lem:smoothNis}. Thus $\mathcal{S}$ is a covering sieve 
in the $\sigma$-topology and so $\sigma$ is finer than the $\cdp$-topology on $\Sm_k$, as desired.
\end{proof}

%%%%%
\subsection{Good blow-up and good Nisnevich cd-structures on smooth $k$-varieties}
%%%%%

In this section, we consider compactifications of smooth $k$-varieties and prove propositions which roughly claim that, 
under the assumptions of Hypotheses \ref{strong resolutions}, \ref{embedded resolutions} 
and \ref{embedded resolution with boundary}, 
it suffices to consider normal crossing compactifications for our purpose.

\begin{definition}\label{def:normal crossing compactification}
Let $X\in\Sm_k$. 
A \emph{normal crossing compactification} of $X$ is an embedding $X\hookrightarrow \overline{X}$ such that $(X,\overline{X})$ is a normal crossing pair.
\end{definition}

\begin{lemma}\label{lem: nc compactification}
Under Hypotheses \ref{strong resolutions} and \ref{embedded resolutions} every smooth $k$-variety has a normal crossing compactification.
Moreover, for $X\in\Sm_k$ fixed, the category of normal crossing pairs $(X,\overline{X})$ extending $X$ is cofiltered.
For simplicity, we denote this category by $\{(X,\overline{X})/X\}$.
\end{lemma}

\begin{proof}
According to Nagata's theorem, there is a proper $k$-variety $\overline{X}$ 
such that $(X,\overline{X})$ is a geometric pair.
Because of Hypotheses \ref{strong resolutions} (i) and \ref{embedded resolutions} this can be transformed via a strict  birational morphism into a normal crossing pair 
$(X,\overline{X}')$. 

To see that the category of normal crossing pairs extending $X$ is cofiltered, let $(X,\overline{X}_1)$ and $(X,\overline{X}_2)$ be two normal crossing pairs.
Take $\overline{X}'$ to be the closure of $X\hookrightarrow \overline{X}_1 \times \overline{X}_2$. 
Then $(X,\overline{X}')$ is a geometric pair.
By Hypotheses \ref{strong resolutions} (i) and \ref{embedded resolutions}
this can again be transformed via a strict birational morphism into a normal crossing pair $(X,\overline{X}_3)$ together with strict morphisms
	$$
	\xymatrix{ & (X,\overline{X}_3) \ar[dr] \ar[dl]&\\ (X,\overline{X}_1) && (X,\overline{X}_2)}
	$$
over $X$.
\end{proof}

\begin{definition}\label{def:goodsmcdstr}
We consider the following $\cd$-structures on $\Sm_k$ which we call \emph{good}:
\begin{enumerate}
\item\label{eq: good smooth blowup_cd_square}
The \emph{good smooth blow-up $\cd$-structure} on $\Sm_k$ consists of smooth blow-up squares
	\begin{equation*}
	\xymatrix{Z' \ar@{^{(}->}[r] \ar[d] & X' \ar^{p}[d]\\
	Z \ar@{^{(}->}^{e}[r] & X}
	\end{equation*}
which can be embedded into Cartesian squares of normal crossing pairs
	$$
	\xymatrix{(Z',\overline{Z}') \ar@{^{(}->}[r] \ar[d] & (X',\overline{X}') \ar^{p}[d]\\
	(Z,\overline{Z}) \ar@{^{(}->}^{e}[r] & (X,\overline{X}),}
	$$
where all morphisms are strict, $p$ is the blow-up with centre $\overline{Z}$, and $\overline{Z}$ has normal crossings with $\overline{X}\setminus X$.

\item \label{eq: good smooth Zariski_cd_square} The  \emph{good smooth Nisnevich $\cd$-structure} on $\Sm_k$ is given by smooth Nisnevich squares 
	\begin{equation*}
	\xymatrix{Y' \ar@{^{(}->}[r] \ar[d] & X' \ar[d]^p\\
	Y \ar@{^{(}->}[r]^e & X},
	\end{equation*}
	such that its restriction to each connected component $X_0$ of $X$
	\begin{equation*}
	\xymatrix{Y'_0 \ar@{^{(}->}[r] \ar[d] & X'_0 \ar[d]^p\\
	Y_0 \ar@{^{(}->}[r]^e & X_0},
	\end{equation*}
	 can be embedded into a Cartesian square of normal crossing pairs
		$$
	\xymatrix{(Y'_0,\overline{X}') \ar@{^{(}->}[r] \ar[d] & (X'_0,\overline{X}') \ar[d]^p\\
	(Y_0,\overline{X}) \ar@{^{(}->}[r]^e & (X_0,\overline{X})}
	$$
	 of the form (a) or (b) below:
	\begin{enumerate}
	\item a diagram
	$$
	\xymatrix{(Y'_0,\overline{X}') \ar@{^{(}->}[r] \ar[d] & (X'_0,\overline{X}') \ar[d]^p\\
	(Y_0,\overline{X}) \ar@{^{(}->}[r]^e & (X_0,\overline{X})}
	$$
with $e$ identity on $\overline{X}$ such that the closure of $X_0 \setminus Y_0$ in $\overline{X}$ is a smooth divisor of $\overline{X}$.

	\item diagrams
	$$
	\xymatrix{\emptyset \ar[r] \ar[d] & \emptyset \ar[d]^p & \emptyset \ar[r] \ar[d] & (X_0,\overline{X}) \ar[d]^p \\
	(X_0,\overline{X}) \ar[r]^e& (X_0,\overline{X}), & \emptyset \ar[r]^-e &(X_0,\overline{X}), }
	$$
	where two of the entries are empty.
	\end{enumerate}

\end{enumerate}
In (i) (resp. (ii)), for a good smooth blow-up square (resp. a good smooth Nisnevich square), 
we call a square of normal crossing pairs as above in which the given square is embedded 
a \emph{good compactification} of the given square. 
\end{definition}

\begin{proposition}\label{prop:goodsquares}
	Under Hypotheses \ref{strong resolutions}, \ref{embedded resolutions} and \ref{embedded resolution with boundary}, any smooth blow-up square is a good smooth blow-up square.
\end{proposition}
\begin{proof}
Given a smooth blow-up square in $\Sm_k$
\begin{equation*}
	\xymatrix{Z' \ar@{^{(}->}[r] \ar[d] & X' \ar^{p}[d]\\
	Z \ar@{^{(}->}^{e}[r] & X,}
\end{equation*}
there exists according to Lemma \ref{lem: nc compactification} a normal crossing compactification $(X,\overline{X}_1)$ with $\overline{X}_1$ smooth. 
We denote by $\overline{Z}_1$ the closure of $Z$ in $\overline{X}_1$. 
By Hypothesis \ref{embedded resolution with boundary}, there is a commutative diagram
	$$
	\xymatrix{(Z,\overline{Z}) \ar@{^{(}->}[r] \ar[d] & (X,\overline{X}) \ar[d]\\
	(Z,\overline{Z}_1) \ar@{^{(}->}[r] & (X,\overline{X}_1)}
	$$
	such that $(Z,\overline{Z})\hookrightarrow (X,\overline{X})$ is a strict closed immersion of $nc$-pairs, the vertical morphisms are strict birational, and $\overline{Z}$ has normal crossings with $\overline{X}\setminus X$. 
	By setting $\overline{X}':=\Bl_{\overline{Z}}\overline{X}$,
	we obtain the desired good compactification  
	$$
	\xymatrix{(Z',\overline{Z}') \ar@{^{(}->}[r] \ar[d] & (X',\overline{X}') \ar^{p}[d]\\
	(Z,\overline{Z}) \ar@{^{(}->}^{e}[r] & (X,\overline{X})}
	$$
	of our given smooth blow-up square, given that $X'=\Bl_Z X$ by hypothesis.
\end{proof}

In general, not every Nisnevich square is good, but the following result shows that we can relate every Nisnevich square with good ones at least cdp-locally:

\begin{proposition}\label{prop: good simplicial Zariski}
	Let us assume Hypotheses \ref{strong resolutions}, \ref{embedded resolutions} and \ref{embedded resolution with boundary}. For a given Nisnevich square 
	\begin{equation}\label{eq:smZariskiSquare}
		\xymatrix{Y' \ar[r] \ar[d] & X' \ar[d]\\
	Y \ar[r] & X}
	\end{equation}
	in  $\Var_k$, there is a split hypercovering $X_\bullet \rightarrow X$ with 
respect to the $cdp$-topology in $\Var_k$ with $X_\bullet$ smooth such that 
the pullback
	\[
		\xymatrix{Y'_\bullet \ar[r] \ar[d] & X'_\bullet \ar[d]\\
	Y_\bullet \ar[r] & X_\bullet}
	\]
of the square \eqref{eq:smZariskiSquare} to $X_\bullet$ satisfies the following condition: For each $i$, the square 
	\[
	\xymatrix{Y'_i \ar[r] \ar[d] & X'_i \ar[d] \\
	Y_i \ar[r] & X_i}
	\]
admits a factorisation 
	\[
	\xymatrix{Y'_i \ar@{=}[r] & X'_{i,n} \ar[r] \ar[d] & X'_{i,n-1} \ar[r] \ar[d] & \cdots \ar[r] & X'_{i,1} \ar[r] \ar[d] & X'_{i,0} \ar@{=}[r] \ar[d] & X'_i \\ 
	Y_i \ar@{=}[r] & X_{i,n} \ar[r]  & X_{i,n-1} \ar[r]  & \cdots \ar[r] & X_{i,1} \ar[r] & X_{i,0} \ar@{=}[r] & X_i}
	\]
for some $n$ (which may depend on $i$) such that each square 
	\[
	\xymatrix{X'_{i,l} \ar[r] \ar[d] & X'_{i,l-1} \ar[d]\\
	X_{i,l} \ar[r] & X_{i,l-1}}
	\]
is a good smooth Nisnevich square.
\end{proposition}

\begin{proof}
In the proof, for a commutative square 
\begin{equation}\label{eq:defmodification}
\xymatrix{(B', \overline{B}') \ar[r] \ar[d] & (A',\overline{A}') \ar[d] \\
(B,\overline{B}) \ar[r] & (A,\overline{A})}
\end{equation}
in $\Var^{geo}_k$, we call a commutative square in 
$\Var^{geo}_k$ of the form 
\[ 
\xymatrix{(B', \widetilde{B}') \ar[r] \ar[d] & (A',\widetilde{A}') \ar[d] \\
(B,\widetilde{B}) \ar[r] & (A,\widetilde{A})} 
\] 
over the given square \eqref{eq:defmodification} a modification of the square \eqref{eq:defmodification}.  Also, we say that a commutative square \eqref{eq:defmodification} in $\Var^{geo}_k$ satisfies the condition $(N)$ if it admits a factorisation 
	\[
	\xymatrix{(B',\overline{B}') \ar@{=}[r] & (A'_n,\overline{A}'_n) \ar[r] \ar[d] & (A'_{n-1}, \overline{A}'_{n-1}) \ar[r] \ar[d] & \cdots \ar[r] & (A'_1,\overline{A}'_1) \ar[r] \ar[d] & (A'_0,\overline{A}'_0) \ar@{=}[r] \ar[d] & (A',\overline{A}') \\ 
	(B,\overline{B}) \ar@{=}[r] & (A_n,\overline{A}_n) \ar[r]  & (A_{n-1},\overline{A}_{n-1}) \ar[r]  & \cdots \ar[r] & (A_1,\overline{A}_1) \ar[r] & (A_0,\overline{A}_0) \ar@{=}[r] & (A,\overline{A})}
	\]
for some $n$ such that each square 
	\[
	\xymatrix{(A'_l,\overline{A}'_{l}) \ar[r] \ar[d] & (A'_{l-1},\overline{A}'_{l-1}) \ar[d]\\
	(A_l,\overline{A}_l) \ar[r] & (A_{l-1},\overline{A}_{l-1})}
	\]
admits a modification which is a good compactification of a good smooth Nisnevich square. 

Now we start the proof. By Nagata's theorem we have a geometric pair of the form $(X,\overline{X})$ and a geometric pair of the form $(X',\overline{X}')$ over $(X,\overline{X})$ which extends the given morphism $X' \rightarrow X$. If we define $\overline{Y}, \overline{Y}'$ to be the closure of $Y,Y'$ in $\overline{X}, \overline{X}'$ respectively, the resulting square 
	\begin{equation}\label{eq:geocptZariskiSquare}
		\xymatrix{(Y',\overline{Y}') \ar[r] \ar[d] & (X',\overline{X}') \ar[d]\\
	(Y,\overline{Y}) \ar[r] & (X,\overline{X})}
	\end{equation}
is a commutative square in $\Var_k^{geo}$ enclosing the Nisnevich square \eqref{eq:smZariskiSquare}.
\bigskip 

\noindent 
{\bf Claim 1.} Given a commutative square \eqref{eq:geocptZariskiSquare} in $\Var_k^{geo}$ 
enclosing a Nisnevich square \eqref{eq:smZariskiSquare}, there exists a strict birational morphism $f_1 \colon (W_1, \overline{W}_1) \rightarrow (X, \overline{X})$ with 
an $nc$-pair $(W_1, \overline{W}_1)$ such that the pullback of the square 
\eqref{eq:geocptZariskiSquare} by $f_1$ in the category $\Var_k^{geo}$ 
satisfies the condition $(N)$.

\begin{proof}[Proof of Claim 1] 
By Hypotheses \ref{strong resolutions} and \ref{embedded resolutions}, we may assume that 
$(X, \overline{X})$ is an $nc$-pair to prove the claim. 
Furthermore, by working on each connected component separately, we may assume that $(X, \overline{X})$ is connected.

If $Y'= \emptyset$, either $Y = X$, $X' = \emptyset$ or $Y = \emptyset, X' = X$.
In the former (resp.\ the latter) case, the square \eqref{eq:geocptZariskiSquare} is equal to 
the former (resp.\ the latter) square in Definition \ref{def:goodsmcdstr}(ii)(b), which is a good compactification of a good smooth Nisnevich square. In particular, it satisfies the condition $(N)$.

If $Y' \not= \emptyset$, 
$\overline{Y} = \overline{X}$ and $\overline{Y}' = \overline{X}'$. 
Applying Hypothesis \ref{embedded resolutions} to $(Y,\overline{Y})$, 
we obtain a morphism of $nc$-pairs $(Y,\widetilde{Y}) \rightarrow (X,\overline{X})$ over $(Y,\overline{Y}) \rightarrow (X,\overline{X})$ such that $(Y,\widetilde{Y}) \rightarrow (Y,\overline{Y})$ is a strict 
birational morphism.
Because the inverse image of the simple normal crossing divisor
$\overline{X} \setminus X$ in $\widetilde{Y}$ is a simple normal 
crossing subdivisor of $\widetilde{Y} \setminus Y$, we see that 
the pullback of the morphism $(Y,\overline{Y}) \rightarrow (X,\overline{X})$ by the strict birational morphism 
$(X \times_{\overline{X}} \widetilde{Y}, \widetilde{Y}) \rightarrow (X, \overline{X})$ is a morphism of $nc$-pairs 
$(Y, \widetilde{Y}) \rightarrow (X \times_{\overline{X}} \widetilde{Y}, \widetilde{Y})$. 
Hence, to prove the claim, we may assume that the morphism $(Y,\overline{Y}) = (Y,\overline{X}) \rightarrow (X,\overline{X})$ is a morphism of $nc$-pairs. It suffices to prove that, under this assumption, the square \eqref{eq:geocptZariskiSquare} satisfies the condition $(N)$. 

We denote the decomposition of the simple normal crossing divisor $X \setminus Y$ into irreducible components by $X \setminus Y = \bigcup_{m=1}^n D_m$ and denote by $\overline{D}_m$ the closure of $D_m$ in $\overline{X}$. Then each $\overline{D}_m$ is a smooth divisor of $\overline{X}$ and we have 
$\overline{X} \setminus Y = (\overline{X} \setminus X) \cup \bigcup_{m=1}^n \overline{D}_m$.

For $0 \leqslant l \leqslant n$, define 
\[ 
X_l := X \setminus \bigcup_{m=1}^l D_l = 
\overline{X} \setminus ((\overline{X} \setminus X) \cup \bigcup_{m=1}^l \overline{D}_l), \quad X'_l := X_l \times_X X'. 
\] 
Then the square \eqref{eq:geocptZariskiSquare} admits a factorisation 
	\[
	\xymatrix{(Y',\overline{Y}') \ar@{=}[r] & (X'_n,\overline{X}') \ar[r] \ar[d] & (X'_{n-1}, \overline{X}') \ar[r] \ar[d] & \cdots \ar[r] & (X'_1,\overline{X}') \ar[r] \ar[d] & (X'_0,\overline{X}') \ar@{=}[r] \ar[d] & (X',\overline{X}') \\ 
	(Y,\overline{Y}) \ar@{=}[r] & (X_n,\overline{X}) \ar[r]  & (X_{n-1},\overline{X}) \ar[r]  & \cdots \ar[r] & (X_1,\overline{X}) \ar[r] & (X_0,\overline{X}) \ar@{=}[r] & (X,\overline{X}).}
	\]
It remains to show that each of the squares in this diagram admits a modification which is a good compactification of a of a good smooth Nisnevich square. Using the factorisation $X_n\to X_{n-1}\to\dots\to X_1$ we can reduce to the case when $n = 1$, that is we consider the diagram
\begin{equation}\label{eq:ell}
\xymatrix{
(Y', \overline{X}') \ar[r] \ar[d] & (X',\overline{X}') \ar[d] \\
(Y, \overline{X}) \ar[r] & (X, \overline{X}), 
}
\end{equation}
where the lower horizontal morphism is a morphism of $nc$-pairs, the closure $\overline{D}$ of 
$ X \setminus Y$ in $\overline{X}$ is a smooth divisor and the square 
\[\xymatrix{
Y'  \ar[r] \ar[d] & X' \ar[d] \\
Y \ar[r] & X
}\]
induced by \eqref{eq:ell} is a smooth Nisnevich square. 

Let $D' := X' \setminus Y'$ and let $\overline{D}'$ be the closure of $D'$ in $\overline{X}'$. Then we have the isomorphism $D' \xrightarrow{\cong} X\setminus Y$ by definition of smooth Nisnevich square and so $D'$ is smooth. Also, we obtain a strict closed immersion 
$(D', \overline{D}') \hookrightarrow (X',\overline{X}')$ of geometric pairs. So, 
by Hypothesis \ref{embedded resolution with boundary}, 
there exists a commutative diagram 
\[ 
\xymatrix{
(D', \widetilde{D}') \ar@{^{(}->}[r] \ar[d] 
& (X',\widetilde{X}') \ar[d] \\ 
(D', \overline{D}') \ar@{^{(}->}[r] 
& (X',\overline{X}')
} \] 
such that $(D', \widetilde{D}'), (X',\widetilde{X}')$ are 
$nc$-pairs, the horizontal arrows are strict closed immersions, the vertical arrows are strict birational morphisms which are isomorphisms outside $\overline{D}' \setminus D'$, and $\widetilde{D}'$ has normal crossing with $\widetilde{X}' \setminus X'$. Then we have 
$$ Y' = X' \setminus D' = \widetilde{X}' \setminus 
((\widetilde{X}' \setminus X') \cup \widetilde{D}') $$
and $(\widetilde{X}' \setminus X') \cup \widetilde{D}'$ is a simple normal crossing divisor in $\widetilde{X}'$. So $(Y',\widetilde{X}') \rightarrow 
(X',\widetilde{X}')$ is a morphism of $nc$-pairs. Therefore, the square 
\[ 
\xymatrix{
(Y',\widetilde{X}') \ar[r] \ar[d] & (X',\widetilde{X}') \ar[d] \\ 
(Y,\overline{X}) \ar[r] & (X,\overline{X})  
}
\]
is a good compactification of a good smooth Nisnevich square which is a modification of the square \eqref{eq:ell}. So the square \eqref{eq:geocptZariskiSquare} satisfies the condition $(N)$, as required. 
\end{proof}

While Claim 1 says that the pullback of a given commutative square \eqref{eq:geocptZariskiSquare} along a \emph{strict birational morphism} satisifies condition $(N)$, we will now prove that any such diagram satisfies condition $(N)$ \emph{$\cdp$-locally} in the following precise sense:

\noindent 
{\bf Claim 2.} Given a commutative square \eqref{eq:geocptZariskiSquare} in $\Var_k^{geo}$ 
enclosing a Nisnevich square \eqref{eq:smZariskiSquare}, there exists a strict morphism $f \colon (X_0, \overline{X}_0) \rightarrow (X, \overline{X})$ with 
$(X_0, \overline{X}_0)$ an $nc$-pair and $X_0 \rightarrow X$ a $\cdp$-covering 
such that the pullback of the square 
\eqref{eq:geocptZariskiSquare} by $f$ in the category $\Var_k^{geo}$ satisfies the condition $(N)$.

\begin{proof}[Proof of Claim 2] 
Let $f_1 \colon (W_1, \overline{W}_1) \rightarrow (X, \overline{X})$ be the strict birational morphism 
constructed in Claim 1. Let $\overline{Z}_1'\subseteq \overline{X}$ be the locus where 
the morphism $f_1 \colon \overline{W}_1 \rightarrow \overline{X}$ is not an isomorphism, 
put $Z_1 := X \cap \overline{Z}'_1$ and let $\overline{Z}_1$ be the closure 
of $Z_1$ in $\overline{Z}'_1$. Then we apply Claim 1 to the pullback of the square 
\eqref{eq:geocptZariskiSquare} to $(Z_1, \overline{Z}_1)$ in the category $\Var_k^{geo}$ to 
define a strict birational morphism $f_2 \colon (W_2, \overline{W}_2) \rightarrow (Z_1, \overline{Z}_1)$. 
Let $\overline{Z}_2'\subseteq \overline{Z}_1$ be the locus where the morphism $f_2 \colon \overline{W}_2 \rightarrow \overline{Z}_1$ is not an isomorphism, set $Z_2 := X \cap \overline{Z}'_2$ and let $\overline{Z}_2$ be the closure 
of $Z_2$ in $\overline{Z}'_2$. Continuing this process, which terminates because $\overline{X}$ is Noetherian, we obtain a strict morphism $f \colon (X_0,\overline{X}_0):=\coprod_{j} (W_j,\overline{W}_j)\rightarrow (X,\overline{X})$ 
with $(X_0, \overline{X}_0)$ an $nc$-pair and $X_0 \rightarrow X$ a $\cdp$-covering 
such that the pullback of the square 
\eqref{eq:geocptZariskiSquare} by $f$ in the category $\Var_k^{geo}$ satisfies the condition $(N)$.
\end{proof}

Now we prove the proposition. It suffices to construct by induction  
an $i$-truncated simplicial $nc$-pairs $(X_\bullet, \overline{X}_\bullet)_{\bullet \leqslant i}$ 
over $(X,\overline{X})$ such that $X_{\bullet \leqslant i} \rightarrow X$ is an 
 $i$-truncated split hypercovering with respect to the $\cdp$-topology in $\Var_k$ and that 
the pullback 
\begin{equation}\label{eq:truncatedsquare}
\xymatrix{(Y'_\bullet,\overline{Y}'_\bullet)_{\bullet \leqslant i} \ar[r] \ar[d] & (X'_\bullet,\overline{X}'_\bullet)_{\bullet \leqslant i} \ar[d]\\
	(Y_\bullet,\overline{Y}_\bullet)_{\bullet \leqslant i} \ar[r] & (X_\bullet,\overline{X}_\bullet)_{\bullet \leqslant i}}
\end{equation}
of the square \eqref{eq:geocptZariskiSquare} to $(X_\bullet,\overline{X}_\bullet)_{\bullet \leqslant i}$ 
is a square of simplicial geometric pairs such that, for each 
$0 \leqslant j \leqslant i$, the induced square 
\begin{equation*}
\xymatrix{(Y'_j,\overline{Y}'_j) \ar[r] \ar[d] & (X'_j,\overline{X}'_j) \ar[d]\\
	(Y_j,\overline{Y}_j) \ar[r] & (X_j,\overline{X}_j) }
\end{equation*}
satisfies the condition $(N)$.
The case $i=0$ follows from Claim 2. Suppose that we have constructed 
an $i$-truncated simplicial $nc$-pairs $(X_\bullet, \overline{X}_\bullet)_{\bullet \leqslant i}$ 
over $(X,\overline{X})$ as above. Let 

\begin{equation}\label{eq:cosksquare}
\xymatrix{
((\cosk_i Y'_\bullet)_{i+1},  (\cosk_i \overline{Y}'_\bullet)_{i+1})
\ar[r] \ar[d] & 
((\cosk_i X'_\bullet)_{i+1},  (\cosk_i \overline{X}'_\bullet)_{i+1})
\ar[d] \\ 
((\cosk_i Y_\bullet)_{i+1},  (\cosk_i \overline{Y}_\bullet)_{i+1})
\ar[r] & 
((\cosk_i X_\bullet)_{i+1},  (\cosk_i \overline{X}_\bullet)_{i+1})
}
\end{equation}
be the square we obtain by taking 
the degree $i+1$ part of the pullback of the square \eqref{eq:geocptZariskiSquare} to $(\cosk_i X_\bullet,  \cosk_i \overline{X}_\bullet)$
in $\Var_k^{geo}$. By applying Claim 2 to the square \eqref{eq:cosksquare}, 
we find a strict morphism 
\[ 
g \colon (X''_{i+1},\overline{X}''_{i+1}) \rightarrow((\cosk_i X_\bullet)_{i+1},(\cosk_i \overline{X}_\bullet)_{i+1})
\]  
with $(X''_{i+1},\overline{X}''_{i+1})$ an $nc$-pair and $X''_{i+1} \rightarrow (\cosk_i X_\bullet)_{i+1}$ 
a $\cdp$-covering such that the pullback of the square 
\eqref{eq:geocptZariskiSquare} by $g$ in the category $\Var_k^{geo}$ satisfies the condition $(N)$.
This map $g$ induces an $(i+1)$-truncated 
simplicial $nc$-pairs $(X_\bullet, \overline{X}_\bullet)_{\bullet \leqslant i+1}$ 
over $(X,\overline{X})$ with desired properties by the recipe explained 
in  \cite[Prop.\,5.1.3]{S_1972} (see also \cite[6.2]{D1974}). 
So the proof of the proposition is finished. 
\end{proof}

%%%%%
\subsection{Integral $p$-adic cohomology for smooth and open varieties}
%%%%%

In this subsection we will define a candidate for an integral $p$-adic cohomology theory on $\Sm_k$.
In addition to Hypotheses  \ref{strong resolutions}, \ref{embedded resolutions} and \ref{embedded resolution with boundary} we also assume Hypothesis \ref{weak factorisation} which assures weak factorisation.

In the following, we regard $\Spec k, \Spec W_n(k)$ and $\Spf W(k)$ as fine log (formal) 
schemes endowed with trivial log structure.
For a log smooth fine log scheme $X$ over $\Spec k$,
let $W\omega^{\kr}_{X/k}:= \varprojlim W_n\omega^{\kr}_{X/k }$ 
be the \emph{logarithmic de~Rham--Witt complex} \cite[\S\,4]{HK_1994}. 
(In \cite[\S\,4]{HK_1994} the structure morphism is supposed to be of Cartier type, 
but this is automatically satisfied in our case (see the paragraph after \cite[Def.\,(4.8)]{K_1989}.) It is a complex of \'etale sheaves on $X$, and
it computes log crystalline cohomology, that is, 
\begin{align*}
   & R\Gamma_{\cris}(X/W_n(k) ) \cong R\Gamma_{\et}(X,W_n\omega^{\kr}_{X/k}), \\ 
	& R\Gamma_{\cris}(X/W(k) ) \cong R\varprojlim R\Gamma_{\cris}(X/W_n(k) ) 
\cong R\varprojlim R\Gamma_{\et}(X,W_n\omega^{\kr}_{X/k}) \cong R\Gamma_{\et}(X,W\omega^{\kr}_{X/k}). 
\end{align*}
In particular, we have the isomorphisms 
\begin{align*}
H^\ast_{\cris}(X/W_n(k) ) \cong H^\ast_{\et}(X,W_n\omega^{\kr}_{X/k }), \quad 
H^\ast_{\cris}(X/W(k) ) \cong H^\ast_{\et}(X,W\omega^{\kr}_{X/k }). 
	\end{align*}
	
\begin{remark}
It is possible to find an explicit complex representing 
$R\Gamma_{\cris}(X/W_n(k) ) \cong R\Gamma_{\et}(X,W_n\omega^{\kr}_{X/k})$
in the derived category in such a way that it is functorial in $X$. 
This can be accomplished by taking for example the Godement resolution of $W_n\omega^{\kr}_{X/k }$ on the \'etale site. 
The same holds for 
$R\Gamma_{\cris}(X/W (k) ) \cong R\Gamma_{\et}(X,W \omega^{\kr}_{X/k})$.
\end{remark}

\begin{example}
A normal crossing pair $(X,\overline{X})$ gives rise to a log scheme which we denote again by $(X,\overline{X})$
where the underlying scheme is $\overline{X}$ and the log structure is induced by the divisor $D= \overline{X}\backslash X$.
Such a log scheme is fine and proper log smooth over $\Spec k$.
In this case, the truncated (resp.\, untruncated) logarithmic de Rham--Witt complex 
$W_n\omega^{\kr}_{(X,\overline{X})/k}$ (resp.\, $W\omega^{\kr}_{(X,\overline{X})/k}$) is also denoted by $W_n\Omega^{\kr}_{\overline{X}/k}(\log D)$ (resp.\,$ W\Omega^{\kr}_{\overline{X}/k}(\log D)$). 
So we have the isomorphisms
	$$
	H^\ast_{\cris}((X,\overline{X})/W_n(k)) \cong H^\ast_{\et}(\overline{X},W_n\Omega^{\kr}_{\overline{X}/k }(\log D)), 
\quad 	
H^\ast_{\cris}((X,\overline{X})/W(k)) \cong H^\ast_{\et}(\overline{X},W\Omega^{\kr}_{\overline{X}/k }(\log D)).
	$$
\end{example}

\begin{remark}\label{rem:matsuue}
Another definition of the logarithmic de Rham--Witt complex $W_n\omega^{\kr}_{(X,\overline{X})/k}, W\omega^{\kr}_{(X,\overline{X})/k}$ is given by Matsuue \cite{M_2017}, and the compatibility of 
the two definitions are proven in \cite[A]{HS_2018}. 
In particular, $W_n\omega^{\kr}_{(X, \overline{X})/k}$ is a quotient of 
the log de Rham complex $\omega^{\kr}_{W_n(X,\overline{X})/W_n}$ of the (log) Witt scheme 
$W_n(X,\overline{X})$ \cite[Def.~2.5]{M_2017} of $(X,\overline{X})$, by the construction given in \cite[3.4]{M_2017}. 
Hence, for $u \in {\mathcal{M}}_{\overline{X}}$, the section 
$d\log u \in W_n\omega^{\kr}_{(X, \overline{X})/k}$ is well-defined, and 
it is equal to $\frac{d[u]}{[u]}$ if $u \in {\mathcal{O}}_{\overline{X},\et}^{\times}$. 
\end{remark}

\begin{definition}\label{def: quasi-iso}
For any normal crossing pair $(X,\overline{X})$  and any $n\geqslant 1$ let $A^{\kr}_n(X,\overline{X})$ be an explicit complex functorial in the log scheme associated to $(X,\overline{X})$ representing $R\Gamma_{\et}(\overline{X},W_n\omega^{\kr}_{(X,\overline{X})/k })$. 
\end{definition}

\begin{proposition}\label{prop: quasi-iso}
Let us assume Hypotheses \ref{strong resolutions}, \ref{embedded resolutions} and 
\ref{weak factorisation}. Then, for a fixed $X \in \Sm_k$, the complexes in the family
$\{A^{\kr}_n(X,\overline{X})\}_{\overline{X}}$  are all quasi-isomorphic, 
where $\overline{X}$ runs through all normal crossing compactifications of $X$.
\end{proposition}

\begin{proof}
By Lemma \ref{lem: nc compactification} and weak factorisation it suffices to prove that
\[
	A^{\kr}_n(X,\overline{X}) \rightarrow A^{\kr}_n(X,\overline{X}')
\]
is a quasi-isomorphism 	when $(X,\overline{X}')\rightarrow(X,\overline{X})$ is a blow-up 
with smooth centre $\overline{Z}$ which is normal-crossing with the boundary divisor $D=\overline{X}\backslash X$. 
Let $\overline{U}_\bullet\rightarrow \overline{X}$ be a Zariski \v{C}ech hypercovering 
induced by an affine covering of $\overline{X}$ and 
denote the pullback of it to $\overline{X}'$
by $\overline{U}'_\bullet \rightarrow \overline{X}'$. 
All of the $\overline{U}_i$ and $\overline{U}'_i$ are smooth 
(but in general not proper) and may be endowed with the log structure associated 
to the simple normal crossing divisors $D_i:=\overline{U}_i\cap D$ and $D'_i:=\overline{U}'_i\cap D'$. 
In analogy to the notation above, we denote the respective log schemes 
by $(U_i,\overline{U}_i)$ and $(U'_i,\overline{U}'_i)$, 
where $U_i:=\overline{U}_i\backslash D_i$ and $U'_i:=\overline{U}'_i\backslash D'_i$.
In particular we have a morphism of simplicial log schemes
$$(U'_\bullet,\overline{U}'_\bullet)\rightarrow (U_\bullet,\overline{U}_\bullet).$$
By Zariski descent, the vertical maps in the diagram
	\begin{equation}\label{equ: descent diag}
	\xymatrix{H^\ast A^{\kr}_n(X,\overline{X}) \cong H^\ast_{\et}(\overline{X},W_n\Omega_{\overline{X}/k}(\log D)) \ar[d]^{\sim} \ar[r] & H^\ast A^{\kr}_n(X',\overline{X}') \cong H^\ast_{\et}(\overline{X}',W_n\Omega_{\overline{X}'/k}(\log D')) \ar[d]^{\sim}\\
	H^\ast_{\et}(\overline{U}_\bullet,W_n\Omega_{\overline{U}_\bullet/k}(\log D_\bullet)) \ar[r] & H^\ast_{\et}(\overline{U}'_\bullet,W_n\Omega_{\overline{U}'_\bullet/k}(\log D'_\bullet))}
	\end{equation}
are isomorphisms. So it suffices to prove that the lower horizontal map is an isomorphism, 
and this is reduced to proving the isomorphism 
\[ 
H^\ast_{\et}(\overline{U}_i,W_n\Omega_{\overline{U}_i/k}(\log D_i)) \rightarrow 
H^\ast_{\et}(\overline{U}'_i,W_n\Omega_{\overline{U}'_i/k}(\log D'_i))
\] 
for each $i$, by standard spectral sequence argument. 
Hence it suffices to work Zariski locally on $\overline{X}$. 
Thus, we may assume that $\overline{X}$ is affine and that there exists an \'etale morphism $\overline{X}\rightarrow \mathbb{A}^d_k$ such that $D$ and $\overline{Z}$ are defined as zeros of some coordinates of $\mathbb{A}^d_k$.
More precisely, there are integers $a,b,c,d$ with $ 0\leqslant a\leqslant b\leqslant c\leqslant d$ and an \'etale morphism
\[
	\overline{X} \rightarrow \Spec k[x_1,\dots,x_a,\dots,x_b,\dots, x_c,\dots,x_d]
\]
such that
\begin{align*}
	\overline{Z} &=\{ x_1=\cdots=x_b=0 \}, \\
	D&=\{ x_1x_2\cdots x_a x_{b+1}\cdots x_c=0 \}.
\end{align*}
Take an \'etale morphism 
\[ 
\overline{\mathfrak{X}} \rightarrow \mathbb{A}_{W_n(k)}^d 
= \Spec W_n(k)[x_1,\dots,x_a,\dots,x_b,\dots, x_c,\dots,x_d]
\] 
which lifts the morphism 
$\overline{X} \rightarrow \mathbb{A}^d_k = \Spec k[x_1,\dots,x_a,\dots,x_b,\dots, x_c,\dots,x_d]$ above, 
define the closed subschemes $\overline{\mathfrak{Z}}$, $\fD$  of $\overline{\mathfrak{X}}$ by 
\begin{align*}
	\overline{\mathfrak{Z}} &:=\{ x_1=\cdots=x_b=0 \},\\
	\fD&:=\{ x_1x_2\cdots x_a x_{b+1}\cdots x_c=0 \},
\end{align*}
and set $\mathfrak{X}:=\overline{\mathfrak{X}} \backslash \fD$.
The log scheme $(\mathfrak{X},\overline{\fX})$,  whose underlying scheme is $\overline{\mathfrak{X}}$ and whose log structure is induced by the divisor $\fD$, is a lift to $W_n(k)$ of the log scheme $(X,\overline{X})$. 
The pair $\{(X,\overline{X}) , (\mathfrak{X},\overline{\fX})\}$ is an embedding system in the sense of \cite[(2.18)]{HK_1994}.
Let
\[
	F\colon  \overline{\mathfrak{X}} \rightarrow \overline{\mathfrak{X}}
\]
be the Frobenius lift on $\overline{\mathfrak{X}}$ induced by $x_i\mapsto x_i^p$ and the canonical Frobenius on $W_n(k)$. 
The map $F$ induces a morphism of schemes
\[
	W_n(\overline{X})  \xrightarrow{t_F} \overline{\mathfrak{X}},
\]
where $t_F$ is induced by the chosen Frobenius lift $F$ (see \cite[0.\,(1.3.20)]{I_1979}). 
Let us write $W_n(X,\overline{X})$ for the $n$-th Witt lift of the log scheme associated to $(X,\overline{X})$ (see \cite[Def.\,3.1]{HK_1994}). 
With the above choice of $F$, the morphism $t_F$ induces a morphism of log schemes
\[
	W_n(X,\overline{X})\rightarrow (\mathfrak{X},\overline{\fX})
\]
and so we obtain a morphism of complexes of sheaves
\[
	\Omega^{\kr}_{\overline{\mathfrak{X}}/W_n(k)}(\log \fD) \rightarrow 
	 W_n\Omega^{\kr}_{\overline{X}/k}(\log D), 
\]
where $\Omega^{\kr}_{\overline{\mathfrak{X}}/W_n(k)}(\log \fD)$ denotes the logarithmic de Rham complex of $(\mathfrak{X},\overline{\mathfrak{X}})$ over $W_n(k)$. 
Note that $\overline{\mathfrak{X}}$ has already divided powers (cf. proof of \cite[II.\,Thm.\,1.4]{I_1979}) and hence $\Omega^{\kr}_{\overline{\mathfrak{X}}/W_n(k)}(\log \fD)$ is the crystalline complex denoted by $C_n$ in the proof of \cite[Thm.\,4.19]{HK_1994}.
According to\textit{ loc.\,cit.} the above morphism of complexes is a quasi-isomorphism.

Similarly, one obtains a quasi-isomorphism
\[
	\Omega^{\kr}_{\overline{\mathfrak{X}}'/W_n(k)}(\log \fD') \xrightarrow{\sim}  W_n\Omega^{\kr}_{\overline{X}'/k}(\log D')
\]
where
\[
	\pi \colon \overline{\mathfrak{X}}':=\Bl_{\overline{\mathfrak{Z}}} \overline{\mathfrak{X}} \rightarrow \overline{\mathfrak{X}}
\]
denotes the blow-up of $\overline{\mathfrak{X}}$ along $\overline{\mathfrak{Z}}$ and $\fD':=(\pi^{-1}\fD)_{\mathrm{red}}$. 
Thus it suffices to prove the quasi-isomorphism 
\begin{equation}\label{eq:prop_quasiiso:reduction}
	R\Gamma_\et(\overline{\mathfrak{X}},\Omega^{\kr}_{\overline{\mathfrak{X}}/W_n(k)}(\log \fD)  ) \cong R\Gamma_\et(\overline{\mathfrak{X}}, R\pi_*\Omega^{\kr}_{\overline{\mathfrak{X}}'/W_n(k)}(\log \fD') ).
\end{equation}

In the first step, we reduce the general case to the case $a=b$.
Thus assume $b>a$. We define a new divisor 
$\widetilde{\fD}:=\{x_1x_2\cdots x_ax_{a+1}x_{b+1}\cdots x_c=0\}\subseteq 
\overline{\mathfrak{X}}$ and consider the closed immersion
\[
	i\colon \overline{\mathfrak{X}}_1:=\{x_{a+1}=0\}\hookrightarrow \overline{\mathfrak{X}}. 
\]
We will write $\fD_1:=\fD\cap \overline{\mathfrak{X}}_1$ and 
$\overline{\mathfrak{Z}}_1:=\overline{\mathfrak{Z}}\cap  \overline{\mathfrak{X}}_1$ for the intersections with $\fD$ and $\overline{\mathfrak{Z}}$. 
The blow-up
\[
	\pi\colon  \overline{\mathfrak{X}}':=\Bl_{\overline{\mathfrak{Z}}}  \overline{\mathfrak{X}}\rightarrow  \overline{\mathfrak{X}}
\]
admits a standard affine open covering
\[
	 \overline{\mathfrak{X}}'=\bigcup_{r=1}^b \overline{\mathfrak{X}}'_r
\]
such that $\pi$ on $\overline{\mathfrak{X}}'_r$ is given by 
the Cartesian diagram 
\[ 
\xymatrix{ \overline{\mathfrak{X}}'_r \ar[r] \ar[d] &  \overline{\mathfrak{X}} \ar[d] \\ 
\Spec W_n(k)[y_1,\dots,  \check{y}_r,\dots, y_b,x_r,x_{b+1},\dots,x_d] \ar[r] & 
\Spec W_n(k)[x_1,\dots,x_d], 
}\] 
where the variable $y_r$ is omitted on the lower left-hand side and where the lower horizontal morphism is
induced by $x_i\mapsto x_r y_i$ for  $1\leqslant i\leqslant b, i\not=r$ and $x_i\mapsto x_i$  otherwise. 
 The divisors on $\overline{\mathfrak{X}}'$ induced by $\fD$ and $\widetilde{\fD}$ 
will be denoted by $\fD':=(\pi^{-1}\fD)_{\text{red}}$ and $\widetilde{\fD}':=(\pi^{-1}\widetilde{\fD})_{\text{red}}$. 

Similarly, the blow-up
\[
	\pi_1 \colon  \overline{\mathfrak{X}}'_1:=\Bl_{\overline{\mathfrak{Z}}_1}  \overline{\mathfrak{X}}_1 \rightarrow  \overline{\mathfrak{X}}_1
\]
admits a standard affine open covering
\[
	 \overline{\mathfrak{X}}'_1=\bigcup_{\substack{1\leqslant r \leqslant b \\ r\neq a+1}} \overline{\mathfrak{X}}'_{1,r}
\]
such that  $\pi_1$ on $\overline{\mathfrak{X}}'_{1,r}$ is given by 
the Cartesian diagram 
\[ 
\xymatrix{ \overline{\mathfrak{X}}'_{1,r} \ar[r] \ar[d] &  \overline{\mathfrak{X}}_1 \ar[d] \\ 
\Spec W_n(k)[y_1,\dots, \check{y}_r,\dots,\check{y}_{a+1},\dots,y_b,x_r,x_{b+1},\dots,x_d] \ar[r] & 
\Spec W_n(k)[x_1,\dots\check{x}_{a+1,\dots},x_d] 
}\]
with $x_i\mapsto x_r y_i$ for $1\leqslant i\leqslant b, i\neq  r, a+1$ and $x_i\mapsto x_i$  otherwise. 
The divisor on $\overline{\mathfrak{X}}'_1$ induced by $\fD$ will be denoted by $\fD'_1:=(\pi_1^{-1}\fD)_{\text{red}}$. 
Let $i':\overline{\mathfrak{X}}'_1 \hookrightarrow \overline{\mathfrak{X}}'$ be the closed immersion induced by $i:\overline{\mathfrak{X}}_1 \hookrightarrow \overline{\mathfrak{X}}$. 
\bigskip 

\noindent
{\bf Claim.} There are short exact sequences of complexes
\begin{align*}
& \xymatrix{
	0 \ar[r] & \Omega^{\kr}_{\overline{ \mathfrak{X}}/W_n(k)}(\log \fD) \ar[r] & \Omega^{\kr}_{\overline{ \mathfrak{X}} /W_n(k)}(\log \widetilde{\fD}) \ar[r] & i_* \Omega^{\kr-1}_{ \overline{\mathfrak{X}}_1/W_n(k)}(\log \fD_1) \ar[r] & 0, 
} \\ & 
\xymatrix{
	0 \ar[r] & \Omega^{\kr}_{ \overline{\mathfrak{X}}'/W_n(k)}(\log \fD')  \ar[r] & \Omega^{\kr}_{ \overline{\mathfrak{X}}'/W_n(k)}(\log \widetilde{\fD}')  \ar[r] & i'_* \Omega^{\kr-1}_{ \overline{\mathfrak{X}}'_1/W_n(k)}(\log \fD'_1) \ar[r] & 0.
}
\end{align*}

\begin{proof}[Proof of Claim]
The sheaf $\Omega^{1}_{\overline{ \mathfrak{X}} /W_n(k)}(\log \widetilde{\fD})$ is  freely generated by
  \[
 	\frac{d x_i}{x_i} \text{ for } 1\leqslant i \leqslant a+1,\quad dx_i \text{ for } a+2\leqslant i \leqslant b,\quad \frac{d x_i}{x_i} \text{ for } b+1\leqslant i \leqslant c,\quad dx_i \text{ for } c+1\leqslant i \leqslant d,
 \]
 while the sheaf $\Omega^{1}_{\overline{ \mathfrak{X}}/W_n(k)}(\log \fD)$ is  freely generated by
 \[
 	\frac{d x_i}{x_i} \text{ for } 1\leqslant i \leqslant a,\quad dx_i \text{ for } a+1\leqslant i \leqslant b,\quad \frac{d x_i}{x_i} \text{ for } b+1\leqslant i \leqslant c,\quad dx_i \text{ for } c+1\leqslant i \leqslant d. 
 \]
 So the quotient 
$\Omega^{1}_{\overline{ \mathfrak{X}}/W_n(k)}(\log \widetilde{\fD}) /\Omega^{1}_{\overline{\mathfrak{X}}/W_n(k)}(\log \fD)$ 
is the free $i_*\co_{ \overline{\mathfrak{X}}_1}$-module generated by the class of $\dfrac{dx_{a+1}}{x_{a+1}}$. 
Passing to exterior powers proves the exactness of the first short exact sequence.
 
 We prove the exactness of the second short exact sequence on each of the standard open sets $\overline{\fX}'_{r}$.
For  $1 \leqslant r \leqslant a$ the sheaf $\Omega^1_{\overline{ \mathfrak{X}}'/W_n(k)} (\log \widetilde{\fD}')|_{\overline{\fX}'_{r}}$ is  freely generated by
   \[
 	\frac{d y_i}{y_i} \text{ for } 1\leqslant i \leqslant a, i \not= r,\quad \frac{dx_r}{x_r}, \quad \frac{dy_{a+1}}{y_{a+1}}, \quad dy_i \text{ for } a+2\leqslant i \leqslant b,\quad \frac{d x_i}{x_i} \text{ for } b+1\leqslant i \leqslant c,\quad dx_i \text{ for } c+1\leqslant i \leqslant d,
 \]
 while the sheaf $\Omega^{1}_{\overline{\mathfrak{X}}'/W_n(k)}(\log \fD') |_{\overline{\fX}'_{r}}$ is freely generated by
    \[
 	\frac{d y_i}{y_i} \text{ for }  1\leqslant i \leqslant a, i \not= r,\quad \frac{dx_r}{x_r}, \quad dy_i \text{ for } a+1\leqslant i \leqslant b,\quad \frac{d x_i}{x_i} \text{ for } b+1\leqslant i \leqslant c,\quad dx_i \text{ for } c+1\leqslant i \leqslant d.
 \] 
For $a+2 \leqslant r \leqslant b$ the sheaf $\Omega^1_{\overline{ \mathfrak{X}}'/W_n(k)} (\log \widetilde{\fD}')|_{\overline{\fX}'_{r}}$ is freely generated by
\[
\frac{d y_i}{y_i} \text{ for } 1\leqslant i \leqslant a,\quad \frac{dx_r}{x_r}, \quad \frac{dy_{a+1}}{y_{a+1}}, \quad dy_i \text{ for } a+2\leqslant i \leqslant b, i \not= r,\quad \frac{d x_i}{x_i} \text{ for } b+1\leqslant i \leqslant c,\quad dx_i \text{ for } c+1\leqslant i \leqslant d,
\]
while the sheaf $\Omega^{1}_{\overline{\mathfrak{X}}'/W_n(k)}(\log \fD') |_{\overline{\fX}'_{r}}$ is  freely generated by
\[
\frac{d y_i}{y_i} \text{ for } 1\leqslant i \leqslant a,\quad \frac{dx_r}{x_r}, \quad dy_i \text{ for } a+1\leqslant i \leqslant b,i \not= r,\quad \frac{d x_i}{x_i} \text{ for } b+1\leqslant i \leqslant c,\quad dx_i \text{ for } c+1\leqslant i \leqslant d.
\]
In both cases, the quotient $\Omega^1_{\overline{ \mathfrak{X}}'/W_n(k)}(\log \widetilde{\fD}') / \Omega^{1}_{\overline{\mathfrak{X}}'/W_n(k)}(\log \fD')|_{\overline{\fX}'_{r}}$ is the free $i_*\co_{\overline{\fX}'_{1,r}}$-module generated by  the class of $\dfrac{dy_{a+1}}{y_{a+1}}$. 
Passing to exterior powers proves the exactness of the second short exact sequence on all standard open subsets 
$\overline{\fX}'_{r}$ for  $1\leqslant r\leqslant b, i \not= a+1$. 
It remains to prove the exactness on $\overline{\fX}'_{a+1}$. 
The sheaves $\Omega^1_{\overline{ \mathfrak{X}}'/W_n(k)}(\log \widetilde{\fD}')|_{\overline{\fX}'_{a+1}}$ and $\Omega^1_{\overline{ \mathfrak{X}}'/W_n(k)}(\log \fD')|_{\overline{\fX}'_{a+1}}$  are both  freely generated by
   \[
 	\frac{d y_i}{y_i} \text{ for } 1\leqslant i \leqslant a,\quad \frac{dx_{a+1}}{x_{a+1}},\quad dy_i \text{ for } a+2\leqslant i \leqslant b,\quad \frac{d x_i}{x_i} \text{ for } b+1\leqslant i \leqslant c,\quad dx_i \text{ for } c+1\leqslant i \leqslant d.
 \]
 So the quotient $\Omega^1_{\overline{ \mathfrak{X}}'/W_n(k)}(\log \widetilde{\fD}') / \Omega^{1}_{\overline{ \mathfrak{X}}' /W_n(k)}(\log \fD')|_{\overline{\fX}'_{a+1}}$ is zero, but $\overline{\fX}'_{a+1}$ is the complement of the image of $i'$, 
so $\left(i'_*\Omega^{\kr}_{\overline{ \mathfrak{X}}'_1 /W_n(k)}(\log \fD'_1)\right) |_{\overline{\fX}'_{a+1}}=0$. 
This proves the exactness of the second short exact sequence on $\overline{\fX}'_{a+1}$ and concludes the proof of the claim.
\end{proof}

By the above claim, we obtain a morphism of triangles
\[
	\xymatrix{
		\Omega^{\kr}_{\overline{ \mathfrak{X}} /W_n(k)}(\log \fD) \ar[r]\ar[d] &  \Omega^{\kr}_{\overline{ \mathfrak{X}} /W_n(k)}(\log \widetilde{\fD}) \ar[r]\ar[d]  &  \Omega^{\kr}_{\overline{ \mathfrak{X}}_1/W_n(k)}(\log  \fD_1)[-1] \ar[r]^-{+1}\ar[d]   & \\
		R\pi_*\Omega^{\kr}_{\overline{\mathfrak{X}}' /W_n(k)}(\log \fD') \ar[r] &  R\pi_*\Omega^{\kr}_{\overline{ \mathfrak{X}}' /W_n(k)}(\log \widetilde{\fD}') \ar[r] &  R\pi_*i'_*\Omega^{\kr}_{\overline{ \mathfrak{X}}'_1 /W_n(k)}(\log \fD'_1)[-1] \ar[r]^-{+1} & .
	}
\]
In particular, the claimed quasi-isomorphism \eqref{eq:prop_quasiiso:reduction} for the triple 
$( \overline{\mathfrak{X}},\overline{\mathfrak{Z}},\fD)$ follows if one proves 
the quasi-isomorphism \eqref{eq:prop_quasiiso:reduction} for the triples  
$(\overline{ \mathfrak{X}}_1,\overline{\mathfrak{Z}}_1,\fD_1)$ and $( \overline{\mathfrak{X}},\overline{\mathfrak{Z}},\widetilde{\fD})$. 
For a moment, let us call the parameters $(a,b,c,d)$ the type of the triple $( \overline{\mathfrak{X}},\overline{\mathfrak{Z}},\fD)$. 
With this terminology, the triple $( \overline{\mathfrak{X}}_1,\overline{\mathfrak{Z}}_1,\fD_1)$ is of 
type $(a,b-1,c-1,d-1)$ and $( \overline{\mathfrak{X}},\overline{\mathfrak{Z}},\widetilde{\fD})$ is of type $(a+1,b,c,d)$. 
Thus, we can inductively reduce the proof to the case $a=b$.

From now on we assume that $a=b$. 
Let us first verify the isomorphism
\begin{equation}\label{eq:prop_quasiiso:pullback}
	\pi^*\Omega^i_{\overline{ \mathfrak{X}} /W_n(k)}(\log \fD) \cong \Omega^i_{ \overline{ \mathfrak{X}}' /W_n(k)}(\log \fD').
\end{equation}
This is easily checked locally on the open charts $\overline{\fX}'_{r}$: 
the sheaf $\pi^*\Omega^1_{\overline{\mathfrak{X}}/W_n(k)}(\log  \fD)$ is freely generated by
\[
	\frac{dx_i}{x_i} \text{ for } 1\leqslant i \leqslant a, \quad \frac{dx_i}{x_i} \text{ for }b+1\leqslant i\leqslant c, \quad dx_i \text{ for } c+1\leqslant i\leqslant d,
\]
while $\Omega^1_{\overline{\mathfrak{X}}'/W_n(k)}( \log \fD')|_{\overline{\mathfrak{X}}'_{r}}$ is  freely generated by 
 \[
	\frac{dy_i}{y_i} \text{ for } 1\leqslant i \leqslant a, i\not=r, \quad \frac{dx_r}{x_r}, \quad \frac{dx_i}{x_i} \text{ for }b+1\leqslant i\leqslant c, \quad dx_i \text{ for } c+1\leqslant i\leqslant d.
\]
Also, the map $\pi^*\Omega^1_{\overline{ \mathfrak{X}} /W_n(k)}(\log \fD)\rightarrow \Omega^1_{\overline{ \mathfrak{X}}'  /W_n(k)}(\log \fD')|_{\overline{\fX}'_{r}}$ is given by
\begin{align*}
	\frac{dx_i}{x_i} & \mapsto \frac{dy_i}{y_i}+\frac{dx_r}{x_r} \text{ for } 1\leqslant i\leqslant a  \text{ and }i \not=r,\\
	\frac{dx_i}{x_i} & \mapsto \frac{dx_i}{x_i} \text{ for } b+1\leqslant i\leqslant c  \text{ or }i=r,\\
	dx_i & \mapsto dx_i \text{ for } c+1\leqslant i\leqslant d.
\end{align*}
This proves the isomorphism \eqref{eq:prop_quasiiso:pullback} for $i=1$ and the general case follows by passing to exterior powers.
 
Using the isomorphism \eqref{eq:prop_quasiiso:pullback}, the projection formula and 
the well-known quasi-isomorphism $R\pi_*\co_{ \overline{\mathfrak{X}}'} \cong 
\co_{ \overline{\mathfrak{X}}}$, we obtain  the quasi-isomorphisms
\[
	R\pi_*\Omega^i_{\overline{ \mathfrak{X}}' /W_n(k)}(\log \fD') \cong R\pi_*\co_{ \overline{\mathfrak{X}}'}\otimes_{\co_{ \overline{\mathfrak{X}}}} \Omega^i_{\overline{ \mathfrak{X}} /W_n(k)}(\log \fD) \cong \Omega^i_{\overline{ \mathfrak{X}} /W_n(k)}(\log \fD),
\]
and finally  the quasi-isomorphism
\[
	R\Gamma_{\et}( \overline{\mathfrak{X}},\Omega^{\kr}_{\overline{ \mathfrak{X}} /W_n(k)}(\log \fD)) \cong R\Gamma_{\et}( \overline{\mathfrak{X}},R\pi_*\Omega^{\kr}_{\overline{ \mathfrak{X}}'/W_n(k)}(\log  \fD') ).
\]
\end{proof}

\begin{definition}
Assume Hypotheses  \ref{strong resolutions} and \ref{embedded resolutions}.
For $X\in\Sm_k$ and $n\geqslant 1$ set
	$$
	A^{\kr}_n(X) := \varinjlim_{(X,\overline{X})/X} A^{\kr}_n(X,\overline{X})
	$$
where the limit runs through the filtered category of 
normal crossing pairs extending $X$ (the opposite category of the one in \,Lemma \ref{lem: nc compactification}).
Let $A^{\kr}(X):= R\varprojlim A^{\kr}_n(X)$.
\end{definition}

\begin{corollary}\label{cor: iso}
Under Hypotheses  \ref{strong resolutions}, \ref{embedded resolutions} and \ref{weak factorisation}, 
$A^{\kr}_n(X)$ is quasi-isomorphic to $A^{\kr}_n(X,\overline{X})$ for $X\in\Sm_k$ and 
any normal crossing compactification $\overline{X}$ of $X$.
\end{corollary}

\begin{proposition}\label{prop: functorial}
Under Hypotheses  \ref{strong resolutions}, \ref{embedded resolutions} and \ref{weak factorisation}, $A^{\kr}_n(X)$ 
(and hence $A^{\kr}(X)$) is functorial in $X$.
\end{proposition}

\begin{proof}
Let $f: X\rightarrow Y$ be a morphism in $\Sm_k$. 
We consider the category
	$$
	\{\overline{f}:(X,\overline{X}) \rightarrow (Y,\overline{Y})\;\vert\; (X,\overline{X}),(Y,\overline{Y}) \in \Var_k^{nc}\}
	$$
of morphisms of normal crossing pairs extending $f$.
For simplicity, we denote it by $\{\overline{f}/f\}$.
We first observe that this forms a non-empty cofiltered category.

Choose normal crossing pairs  $(X,\overline{X})$ and $(Y,\overline{Y})$ extending $X$ and $Y$, 
which exist by Lemma \ref{lem: nc compactification}.
Let $\overline{X}'$ be the closure of the graph of $f$ in $\overline{X}\times_k \overline{Y}$.
Then $(X,\overline{X}')$ is a geometric pair, and the projection to $\overline{Y}$ 
induces a morphism of geometric pairs $(X,\overline{X}') \rightarrow (Y,\overline{Y})$ extending $f$. 
By Hypotheses \ref{strong resolutions} and \ref{embedded resolutions} there exists a birational morphism $(X,\overline{X}'') \rightarrow (X,\overline{X}')$ from a normal crossing pair which is an isomorphism on $X$. 
The composition $\overline{f}: (X,\overline{X}'') \rightarrow (Y,\overline{Y})$ is a morphism of normal crossing pairs extending $f$ as desired.
With similar methods as in Lemma \ref{lem: nc compactification} it can be seen that the category $\{\overline{f}/f\}$ is cofiltered. 

There are the two projection functors
	$$
	\xymatrix{ &\{\overline{f}/f\} \ar[dl]_{P_1} \ar[dr]^{P_2}&\\
	\{(X,\overline{X}) /X\} && \{(Y,\overline{Y})/Y\},}
	$$
	where $\{(X,\overline{X})/X\}$ denotes the category of normal crossing compactifications introduced in Lemma \ref{lem: nc compactification}.
The argument above shows in fact that $P_2$ is surjective. 
Clearly, any extension $\overline{f}$ of $f$ to normal crossing pairs induces a natural morphism of complexes
	$$
	A_n^{\kr}(P_2(\overline{f})) \rightarrow A_n^{\kr}(P_1(\overline{f})).
	$$
Hence, we obtain a zigzag of morphisms of complexes
	$$
	\varinjlim_{(Y,\overline{Y})/Y} A^{\kr}_n(Y,\overline{Y}) \xleftarrow{\sim} 
	\varinjlim_{\overline{f}/f} A^{\kr}_n(P_2(\overline{f})) \rightarrow \varinjlim_{\overline{f}/f} A_n^{\kr}(P_1(\overline{f})) \rightarrow \varinjlim_{(X,\overline{X})/X} A^{\kr}_n(X,\overline{X})
	$$
where the left pointing map is an isomorphism (not just a quasi-isomorphism) by the surjectivity of the projection $P_2$.
This provides a morphism of complexes $A^{\kr}_n(Y)\rightarrow A^{\kr}_n(X)$ induced by $f$ and it is now easy to check the functoriality.
\end{proof}

As a consequence, we may regard $A^{\kr}_n$  as a complex of presheaves on $\Sm_k$.
\begin{definition}\label{def:acdparhsm}
We define complexes of sheaves on $\Sm_{k,\cdp}$ and $\Sm_{k,\cdh}$
	$$a_{\cdp}^\ast A_n^{\kr} \quad\text{ and }\quad a_{\cdh}^\ast A_n^{\kr}$$
as the sheafifications with respect to the $\cdp$- and cdh-topology of the complex of presheaves on $\Sm_k$ given by
	$$
	X \mapsto A_n^{\kr}(X).
	$$
We denote the complexes of sheaves on $\Var_{k,\cdp}$ and $\Var_{k,\cdh}$
we obtain from 
$a_{\cdp}^\ast A_n^{\kr}$ and $a_{\cdh}^\ast A_n^{\kr}$ respectively using the equivalences 
$\Sh(\Sm_{k,\cdp}) \cong \Sh(\Var_{k,\cdp})$ and $\Sh(\Sm_{k,\cdh}) \cong \Sh(\Var_{k,\cdh})$ by 
the same symbol. Furthermore, we set 
	\begin{align*}
	a_{\cdp}^\ast A^{\kr} :=R\varprojlim a_{\cdp}^\ast A_n^{\kr}, \quad 
	a_{\cdh}^\ast A^{\kr}  := R\varprojlim a_{\cdh}^\ast A_n^{\kr}. 
	\end{align*}
Also, we denote by $R\Gamma_{\cdp}(X,-)$ and $R\Gamma_{\cdh}(X,-)$ the derived global section functor 
for $X \in \Var_k$ with respect to $\cdp$- and $\cdh$-topology. 
\end{definition}

Note that $R\Gamma_{\cdp}(-,a^\ast_{\cdp} A_n^{\kr})$ and $R\Gamma_{\cdp}(-,a^\ast_{\cdp} A^{\kr})$   
(resp. $R\Gamma_{\cdh}(-,a^\ast_{\cdh} A_n^{\kr})$ and $R\Gamma_{\cdh}(-,a^\ast_{\cdh} A^{\kr})$)  
satisfy (cohomological) $\cdp$-descent (resp. $\cdh$-descent) in $\Var_k$ by definition.

We show now that $A^{\kr}_n(X)$
 satisfies the so-called \textit{Mayer--Vietoris property} (compare \cite{CHSW_2008}) 
for smooth blow-up squares and smooth Nisnevich squares. 
 We will see that this implies that $A^{\kr}_n$ and  $A^{\kr}$ satisfy cohomological $\cdp$- and $\cdh$-descent.

\begin{proposition}\label{prop: MV blow-up}
Assume Hypotheses \ref{strong resolutions}, \ref{embedded resolutions}, \ref{embedded resolution with boundary} and \ref{weak factorisation}.
For a smooth blow-up square
	$$
	\xymatrix{Z' \ar[r] \ar[d] & X' \ar[d]\\ Z \ar[r] & X}
	$$
the induced diagram
	$$
	\xymatrix{A^{\kr}_n(Z')& A^{\kr}_n(X') \ar[l]\\ A^{\kr}_n(Z) \ar[u] & A^{\kr}_n(X)\ar[u]\ar[l]}
	$$
is homotopy co-Cartesian.
\end{proposition}

\begin{proof}
The statement is clear if $Z=\emptyset$, so we may assume that $Z$ is non-empty.
By Proposition \ref{prop:goodsquares} there is 
a good compactification
\begin{equation}\label{eq:diagBU}
	\xymatrix{(Z',\overline{Z}') \ar[r] \ar[d] & (X',\overline{X}') \ar[d]\\ (Z,\overline{Z}) \ar[r] & (X,\overline{X})}
\end{equation}
of the given square where all morphisms are strict. 
By Corollary \ref{cor: iso} it suffices to prove that 
\[
	\xymatrix{A^{\kr}_n(Z',\overline{Z}')& A^{\kr}_n(X',\overline{X}') \ar[l]\\ A_n^{\kr}(Z,\overline{Z}) \ar[u] & A^{\kr}_n(X,\overline{X}) \ar[u]\ar[l]}
\]
is homotopy co-Cartesian. 
Note that the diagram \eqref{eq:diagBU} is uniquely determined by $\overline{X}$, $D=\overline{X}\backslash X$ and $\overline{Z}$. 
As in the proof of Proposition \ref{prop: quasi-iso}, using a suitable Zariski \v{C}ech hypercovering, 
we can reduce to the situation where there exists an \'etale morphism 
\[ \overline{X} \rightarrow \Spec k[x_1,\dots,x_b,\dots,x_c,\dots,x_d] \] 
 for some $0 \leqslant b \leqslant c \leqslant d$ such that 
\begin{align*}
	\overline{Z}&=\{x_1=\dots=x_b=0\},\\
	D&=\{x_{b+1}\cdots x_c=0\}.
\end{align*}
Note that the parameter $a$, which appears in the proof of Proposition \ref{prop: quasi-iso}, is zero since $Z$ is assumed to be non-empty. Again we choose lifts 
\begin{align*}
	\overline{\fX}& \,\,  \rightarrow \,  \Spec W_n(k)[x_1,\dots,x_b,\dots,x_c,\dots,x_d],\\
	\overline{\mathfrak{Z}}&=\{x_1=\dots=x_b=0\},\\
	\fD&=\{x_{b+1}\cdots x_c=0\}
\end{align*}
to $W_n(k)$ with the Frobenius lift given by $x_i\mapsto x_i^p$. 
Define $\fX:=\overline{\fX}\backslash\fD$, $\mathfrak{Z}=\overline{\mathfrak{Z}}\cap \fX$ and set
$(\fX',\overline{\fX}'):=(\Bl_{\mathfrak{Z}}\fX,\Bl_{\overline{\mathfrak{Z}}}\overline{\fX})$ and $(\mathfrak{Z}',\overline{\mathfrak{Z}}'):=(\mathfrak{Z}\times_{\fX} \fX',\overline{\mathfrak{Z}}\times_{\overline{\fX}} \overline{\fX}')$. 
As in the proof of Proposition \ref{prop: quasi-iso}, the Frobenius lift induces quasi-isomorphisms
\begin{align*}
	\Omega^{\kr}_{\overline{\fX}/W_n(k)}(\log \mathfrak{D}) &\xrightarrow{\sim} W_n\Omega^{\kr}_{\overline{X}/k}(\log D), \\
	\Omega^{\kr}_{\overline{\fX}'/W_n(k)}(\log (\overline{\mathfrak{X}}'\backslash \mathfrak{X}')) &\xrightarrow{\sim} W_n\Omega^{\kr}_{\overline{X}'/k}(\log (\overline{X}'\backslash X')), \\
	\Omega^{\kr}_{\overline{\mathfrak{Z}}/W_n(k)}(\log (\overline{\mathfrak{Z}}\backslash\mathfrak{Z})) &\xrightarrow{\sim} W_n\Omega^{\kr}_{\overline{Z}/k}(\log (\overline{Z}\backslash Z)), \\
	\Omega^{\kr}_{\overline{\mathfrak{Z}}'/W_n(k)}(\log (\overline{\mathfrak{Z}}'\backslash\mathfrak{Z}'))  &\xrightarrow{\sim} W_n\Omega^{\kr}_{\overline{Z}'/k}(\log (\overline{Z}'\backslash Z')),
\end{align*}
and hence it suffices to show that 
\[
	\xymatrix{R\Gamma_\et(\overline{\mathfrak{Z}}',\Omega^{\kr}_{\overline{\mathfrak{Z}}'/W_n(k)}(\log (\overline{\mathfrak{Z}}'\backslash \mathfrak{Z}'))) & R\Gamma_\et(\overline{\fX}',\Omega^{\kr}_{\overline{\fX}'/W_n(k)}(\log (\overline{\mathfrak{X}}'\backslash \mathfrak{X}'))) \ar[l]\\ R\Gamma_\et(\overline{\mathfrak{Z}},\Omega^{\kr}_{\overline{\mathfrak{Z}}/W_n(k)}(\log (\overline{\mathfrak{Z}} \backslash \mathfrak{Z})))  \ar[u] & R\Gamma_\et(\overline{\fX},\Omega^{\kr}_{\overline{\fX}/W_n(k)}(\log \mathfrak{D} )) \ar[u]\ar[l]}
\]
is homotopy co-Cartesian. Now $\overline{\fX}'$ admits a standard affine open covering 
\[
	 \overline{\mathfrak{X}}'=\bigcup_{r=1}^b \overline{\mathfrak{X}}'_r
\]
such that the restriction of the blow-up 
$\pi \colon \overline{\fX}'\rightarrow \overline{\fX}$ 
to $\overline{\mathfrak{X}}'_r$, $1\leqslant r\leqslant d$, is given by 
the Cartesian diagram
\[ 
\xymatrix{ \overline{\mathfrak{X}}'_r \ar[r] \ar[d] &  \overline{\mathfrak{X}} \ar[d] \\ 
\Spec W_n(k)[y_1,\dots, \check{y}_r,\dots,y_b,x_r,x_{b+1},\dots,x_d] \ar[r] & 
\Spec W_n(k)[x_1,\dots,x_d], 
}\] 
(here again the variable $y_r$ is skipped) 
where the lower horizontal morphism is induced by 
$x_i\mapsto x_r y_i$ for $1\leqslant i\leqslant b, i \not= r$ and $x_i\mapsto x_i$ otherwise. 
On these charts, we have the following explicit descriptions of the sheaves 
$\pi^* \Omega^{1}_{\overline{\fX}/W_n(k)}(\log \mathfrak{D} )$ and 
$ \Omega^{1}_{\overline{\fX}'/W_n(k)}(\log \overline{\mathfrak{X}}'\backslash \mathfrak{X}')$: 
The sheaf $\pi^* \Omega^{1}_{\overline{\fX}/W_n(k)}(\log \mathfrak{D} )|_{\overline{\mathfrak{X}}'_r}$ 
is freely generated by
    \[
 	dx_i \text{ for } 1\leqslant i \leqslant b,\quad \frac{dx_i}{x_i}\text{ for } b+1\leqslant i \leqslant c,\quad dx_i \text{ for } c+1\leqslant i \leqslant d,
 \]
 while $\Omega^{1}_{\overline{\fX}'/W_n(k)}(\log \overline{\mathfrak{X}}'\backslash 
\mathfrak{X}')|_{\overline{\mathfrak{X}}'_r}$ is freely generated by 
\[
 	dy_i \text{ for } 1\leqslant i \leqslant b,  i \not= r,\quad dx_r,\quad \frac{dx_i}{x_i}\text{ for } b+1\leqslant i \leqslant c,\quad dx_i \text{ for } c+1\leqslant i \leqslant d.
 \]
The logarithmic differentials $\frac{dx_i}{x_i}$ for $b+1\leqslant i\leqslant c$ induce compatible weight filtrations on the above complexes. In the following, let us describe the associated graded pieces of these filtrations in terms of the residue map. For $b+1\leqslant i\leqslant c$ let us write $\fD_i$ for the closed subscheme of $\fD$ given by $\fD_i=\{x_i=0\}$ and $\overline{\mathfrak{Z}}_i=\overline{\mathfrak{Z}}\cap \fD_i$. Furthermore, let us write $\fD_i'=\Bl_{\overline{\mathfrak{Z}}_i}\fD_i$ and $\overline{\mathfrak{Z}}_i'=\overline{\mathfrak{Z}}_i \times_{\fD_i}\fD_i'$. 
For a subset $L\subseteq \{ b+1,\dots,c \}$ let us introduce the notation
\begin{align*}
	\fD_L:=\bigcap_{i\in L} \fD_i, && \fD'_L:=\bigcap_{i\in L} \fD'_i\\
	\overline{\mathfrak{Z}}_L:=\bigcap_{i\in L} \overline{\mathfrak{Z}}_i, && \overline{\mathfrak{Z}}'_L:=\bigcap_{i\in L} \overline{\mathfrak{Z}}'_i,
\end{align*}
and for $1\leqslant l\leqslant c-b$
\begin{align*}
	\fD_{(l)}:=\coprod_{\substack{ L\subseteq \{b+1,\dots,c\} \\ |L|=l}} \fD_L, && \fD'_{(l)}:=\coprod_{\substack{ L\subseteq \{b+1,\dots,c\} \\ |L|=l}} \fD'_L\\
	\overline{\mathfrak{Z}}_{(l)}:=\coprod_{\substack{ L\subseteq \{b+1,\dots,c\} \\ |L|=l}} \overline{\mathfrak{Z}}_L, &&\overline{\mathfrak{Z}}'_{(l)}:=\coprod_{\substack{ L\subseteq \{b+1,\dots,c\} \\ |L|=l}} \overline{\mathfrak{Z}}'_L
\end{align*}
as well as
\begin{align*}
	\fD_{(0)}:= \overline{\mathfrak{X}}, && \fD'_{(0)}:=\overline{\mathfrak{X}}'\\
	\overline{\mathfrak{Z}}_{(0)}:=\overline{\mathfrak{Z}}, &&\overline{\mathfrak{Z}}'_{(0)}:=\overline{\mathfrak{Z}}'.
\end{align*}
With this notation, the residue map gives the following explicit description of the associated graded pieces for $0\leqslant l\leqslant c-b$:
\begin{align*}
	\mathrm{gr}_l \Omega^{\kr}_{\overline{\fX}/W_n(k)}(\log \mathfrak{D}) &= \Omega^{\kr}_{\fD_{(l)}/W_n(k)}[-l]\\
	\mathrm{gr}_l \Omega^{\kr}_{\overline{\fX'}/W_n(k)}(\log (\overline{\fX}'\backslash\fX'))&= \Omega^{\kr}_{\fD'_{(l)}/W_n(k)}[-l]\\
	\mathrm{gr}_l \Omega^{\kr}_{\overline{\mathfrak{Z}}/W_n(k)}(\log (\overline{\mathfrak{Z}}\backslash\mathfrak{Z})) &= \Omega^{\kr}_{\overline{\mathfrak{Z}}_{(l)}/W_n(k)}[-l]\\
	\mathrm{gr}_l \Omega^{\kr}_{\overline{\mathfrak{Z}}'/W_n(k)}(\log (\overline{\mathfrak{Z}}'\backslash\mathfrak{Z}'))  &= \Omega^{\kr}_{\overline{\mathfrak{Z}}'_{(l)}/W_n(k)}[-l].
\end{align*}
Hence, it is enough to prove that for $0\leqslant l\leqslant c-b$ the diagram
\[
	\xymatrix{R\Gamma_\et(\overline{\mathfrak{Z}}'_{(l)},\Omega^{\kr}_{\overline{\mathfrak{Z}}'_{(l)}/W_n(k)}) & R\Gamma_\et({\fD}'_{(l)},\Omega^{\kr}_{{\fD}'_{(l)}/W_n(k)}) \ar[l]\\ R\Gamma_\et(\overline{\mathfrak{Z}}_{(l)},\Omega^{\kr}_{\overline{\mathfrak{Z}}_{(l)}/W_n(k)}) \ar[u] & R\Gamma_\et({\fD}_{(l)},\Omega^{\kr}_{{\fD}_{(l)}/W_n(k)})\ar[u]\ar[l]},
\]
associated to the blow-up square
	$$
	\xymatrix{\overline{\mathfrak{Z}}'_{(l)} \ar[r] \ar[d] & \fD'_{(l)} \ar[d]\\ \overline{\mathfrak{Z}}_{(l)} \ar[r] & \fD_{(l)}}
	$$
is homotopy co-Cartesian. But this is the non-logarithmic case proved in \cite[IV, Thm.\,1.2.1.]{G_1985}.
\end{proof}

\begin{corollary}\label{cor: cdp-descent}
Assume Hypotheses \ref{strong resolutions}, \ref{embedded resolutions}, \ref{embedded resolution with boundary} and \ref{weak factorisation}. 
For $X \in \Sm_k$, the
natural morphisms 
	\begin{align*}
	A_n^{\kr}(X)  \rightarrow R\Gamma_{\cdp}(X, a^\ast_{\cdp} A_n^{\kr}),\\
	A^{\kr}(X)  \rightarrow R\Gamma_{\cdp}(X, a^\ast_{\cdp} A^{\kr})
	\end{align*}
are quasi-isomorphisms.
\end{corollary}

\begin{proof}
Under the hypotheses mentioned, we have seen in Proposition \ref{prop: MV blow-up} 
 that $A_n^{\kr}$ satisfies the Mayer--Vietoris property for smooth blow-up squares. 
As smooth blow-up squares generate the $\cdp$-topology on $\Sm_k$
by Proposition \ref{cor:cdprhtoponsm}, 
this implies the first quasi-isomorphism by \cite[Thm.\,3.7]{CHSW_2008}.
The second quasi-isomorphism follows by taking (homotopy) limits.
\end{proof}

\begin{proposition}\label{prop: MV Zariski}
Assume Hypotheses \ref{strong resolutions}, \ref{embedded resolutions}, \ref{embedded resolution with boundary} and \ref{weak factorisation}.
For a smooth Nisnevich square

	$$
	\xymatrix{Y' \ar[r] \ar[d] & X' \ar[d]\\ Y \ar[r] & X,}
	$$
the induced diagram
	$$
	\xymatrix{A^{\kr}_n(Y')& A^{\kr}_n(X') \ar[l]\\ A^{\kr}_n(Y) \ar[u] & A^{\kr}_n(X)\ar[u]\ar[l]}
	$$
is homotopy co-Cartesian.
\end{proposition}

\begin{proof}
Take a split hypercovering $X_{\bullet} \rightarrow X$  with respect to the $\cdp$-topology in $\Var_k$ with $X_{\bullet}$ smooth 
which satisfies the condition in the statement of Proposition \ref{prop: good simplicial Zariski}, 
and let 
\[ 
\xymatrix{ 
Y'_{\bullet} \ar[r] \ar[d] & X'_{\bullet} \ar[d] \\ 
Y_{\bullet} \ar[r] & X_{\bullet}}
\]
be the pullback of the given smooth Nisnevich square to $X_{\bullet}$. 
Then we have the quasi-isomorphisms 
\[ 
A_n^{\kr}(X) \xrightarrow{\simeq} 
R\Gamma_{\cdp}(X, a^{\ast}_{\cdp}A^{\kr}_n) 
\xrightarrow{\simeq} 
R\Gamma_{\cdp}(X_{\bullet}, a^{\ast}_{\cdp}A^{\kr}_n) 
\xleftarrow{\simeq} 
A_n^{\kr}(X_{\bullet}), 
\] 
where the first and the third quasi-ismorphisms follows from Corollary \ref{cor: cdp-descent} 
and the second quasi-isomorphism is the $\cdp$-descent in $\Var_k$. 
Similar quasi-isomorphisms hold also for $X', Y$ and $Y'$. So it suffices to prove that the 
diagram
$$
	\xymatrix{A^{\kr}_n(Y'_{\bullet})& A^{\kr}_n(X'_{\bullet}) \ar[l]\\ A^{\kr}_n(Y_{\bullet}) \ar[u] & A^{\kr}_n(X_{\bullet})
\ar[u]\ar[l]}
	$$
is homotopy co-Cartesian, and then it is further reduced to proving that the diagram 
$$
	\xymatrix{A^{\kr}_n(Y'_i)& A^{\kr}_n(X'_i) \ar[l]\\ A^{\kr}_n(Y_i) \ar[u] & A^{\kr}_n(X_i)
\ar[u]\ar[l]}
$$
is homotopy co-Cartesian for each $i$.
Take a factorisation 
	\[
	\xymatrix{Y'_i \ar@{=}[r] & X'_{i,n} \ar[r] \ar[d] & X'_{i,n-1} \ar[r] \ar[d] & \cdots \ar[r] & X'_{i,1} \ar[r] \ar[d] & X'_{i,0} \ar@{=}[r] \ar[d] & X'_i \\ 
	Y_i \ar@{=}[r] & X_{i,n} \ar[r]  & X_{i,n-1} \ar[r]  & \cdots \ar[r] & X_{i,1} \ar[r] & X_{i,0} \ar@{=}[r] & X_i}
	\]
such that each square 
	\[
	\xymatrix{X'_{i,l} \ar[r] \ar[d] & X'_{i,l-1} \ar[d]\\
	X_{i,l} \ar[r] & X_{i,l-1}}
	\]
is a good smooth Nisnevich square, which exists by the conclusion of Proposition \ref{prop: good simplicial Zariski}. Then it suffices to prove that the diagram 
\[
\xymatrix{
A^{\kr}_n(X'_{i,l}) & A^{\kr}_n(X'_{i,l-1}) \ar[l] \\
A^{\kr}_n(X_{i,l}) \ar[u] & A^{\kr}_n(X_{i,l-1}) \ar[u] \ar[l]
}\]
is homotopy co-Cartesian for each $i$ and $l$.
In particular, we may assume 
that the given smooth Nisnevich square is good. Then we may assume that $X$ is connected, 
because we can work on each connected component of $X$.
If the associated square on pairs is one of the Cartesian diagrams
	$$
	\xymatrix{\emptyset \ar[r] \ar[d] & \emptyset \ar[d] & \emptyset \ar[r] \ar[d] & (X,\overline{X}) \ar[d] \\
	(X,\overline{X}) \ar[r] & (X,\overline{X}) & \emptyset \ar[r] &(X,\overline{X}) }
	$$
the statement is obvious. 
Thus assume that the given Nisnevich square may be embedded into a Cartesian diagram of $nc$-pairs of the form 
	$$
	\xymatrix{(Y',\overline{X}') \ar@{^{(}->}[r] \ar[d] & (X',\overline{X}') \ar[d]\\
	(Y,\overline{X}) \ar@{^{(}->}[r] & (X,\overline{X})}
	$$
such that the closure $\overline{D}$ of $D := X \setminus Y$ is a smooth divisor of $\overline{X}$. Let $\overline{D}'$ be the closure of $D' := X' \setminus Y'$ in $\overline{X}'$. By definition of smooth Nisnevich square, 
$D'$ is isomorphic to $D$ and so it is smooth. Hence $\overline{D}'$ is a smooth divisor of $\overline{X}'$. Put 
\[
\overline{D}_X := \overline{X} \setminus X, \quad 
\overline{D}_Y := \overline{X} \setminus Y, \quad 
\overline{D}_{X'} := \overline{X}' \setminus X', \quad 
\overline{D}_{Y'} := \overline{X}' \setminus Y',
\]
so that $\overline{D}_Y = \overline{D} \cup \overline{D}_X, 
\overline{D}_{Y'} = \overline{D}' \cup \overline{D}_{X'}$. 
Then, by \cite[Cor.\,2.12.4]{NS_2008}, we have the canonical quasi-isomorphisms 
\begin{align*}
[A_n^{\kr}(X,\overline{X}) \rightarrow A_n^{\kr}(Y,\overline{X})] 
& \cong
R\Gamma_{\et}(\overline{X}, W_n\Omega^{\kr}_{\overline{X}}(\log \overline{D}_X) \to W_n\Omega^{\kr}_{\overline{X}}(\log (\overline{D} \cup \overline{D}_X))) \\ & \cong
R\Gamma_{\et}(\overline{X}, 
W_n\Omega^{\kr}_{\overline{X}}(\log (\overline{D} \cup \overline{D}_X))/ 
P_0^{\overline{D}}W_n\Omega^{\kr}_{\overline{X}}(\log (\overline{D} \cup \overline{D}_X))) \\
& \cong 
R\Gamma_{\et}(\overline{D}, 
W_n\Omega^{\kr}_{\overline{D}}(\log (\overline{D} \cap \overline{D}_X)))
\cong A_n^{\kr}(D,\overline{D}),
\end{align*}
where $P_0^{\overline{D}}$ denotes the $0$-th weight filtration with respect to $\overline{D}$ defined in \cite[Def.\,2.12.1]{NS_2008}. Similar quasi-isomorphisms exist also for 
the complex 
$[A_n^{\kr}(X',\overline{X}') \rightarrow A_n^{\kr}(Y',\overline{X}')].$
So it suffices to prove that the canonical map 
$A_n^{\kr}(D,\overline{D}) 
\rightarrow A_n^{\kr}(D',\overline{D}')
$
is a quasi-isomorphism. Since $D'$ is smooth and isomorphic to $D$, 
it follows from Proposition \ref{prop: quasi-iso}. 
\end{proof}

\begin{corollary}\label{cor: rh descent}
Let $X\in\Sm_k$ and assume Hypotheses  \ref{strong resolutions}, \ref{embedded resolutions}, \ref{embedded resolution with boundary} and \ref{weak factorisation}.
The natural morphisms
	\begin{align*}
	A_n^{\kr}(X)  \rightarrow R\Gamma_{\cdp}(X,a^\ast_{\cdp}A_n^{\kr}) \rightarrow R\Gamma_{\cdh}(X, a^\ast_{\cdh} A_n^{\kr}),\\
	A^{\kr}(X)  \rightarrow R\Gamma_{\cdp}(X,a^\ast_{\cdp}A^{\kr})   \rightarrow R\Gamma_{\cdh}(X, a^\ast_{\cdh} A^{\kr})
	\end{align*}
are quasi-isomorphisms.
\end{corollary}

\begin{proof}
Recall that the $\cdh$-topology on $\Sm_k$ is generated by smooth blow-up and smooth Nisnevich squares by Corollary \ref{cor:cdprhtoponsm}.
As $A_n^{\kr}$ satisfies the Mayer--Vietoris property for smooth blow-up squares and for smooth Nisnevich squares by the Propositions \ref{prop: MV blow-up} and \ref{prop: MV Zariski}, 
the first quasi-isomorphism follows again from \cite[Thm.\,3.7]{CHSW_2008}.
The second quasi-isomorphism follows by taking (homotopy) limits.
\end{proof}

\begin{corollary}\label{cor: cdp-rh}
For $(X,\overline{X})\in \Var_k^{nc}$, we have quasi-isomorphisms
	\begin{align*}
	A^{\kr}_n(X,\overline{X})\cong A^{\kr}_n(X) \cong R\Gamma_{\cdp}(X, a^\ast_{\cdp} A_n^{\kr})  \cong R\Gamma_{\cdh}(X, a^\ast_{\cdh} A_n^{\kr}),\\
	R\varprojlim A^{\kr}_n(X,\overline{X}) \cong A^{\kr}(X) \cong R\Gamma_{\cdp}(X, a^\ast_{\cdp} A^{\kr})  \cong R\Gamma_{\cdh}(X, a^\ast_{\cdh} A^{\kr})
	\end{align*}
if we assume Hypotheses  \ref{strong resolutions}, \ref{embedded resolutions}, \ref{embedded resolution with boundary} and \ref{weak factorisation}.
\end{corollary}

%%%%%
\subsection{Extension to $k$-varieties}
%%%%%

In this section we extend the integral $p$-adic cohomology theory defined above to general $k$-varieties via sheafification.
As in the previous subsection, we assume Hypotheses  \ref{strong resolutions}, \ref{embedded resolutions}, \ref{embedded resolution with boundary} and \ref{weak factorisation}.

From the descent results for smooth $k$-varieties obtained previously, we deduce the following statement:

\begin{proposition}\label{prop:cdprhvar}
Under Hypotheses  \ref{strong resolutions}, \ref{embedded resolutions}, \ref{embedded resolution with boundary} and \ref{weak factorisation}, for $X\in\Var_k$
	\begin{align*}
	R\Gamma_{\cdp}(X, a^\ast_{\cdp} A_n^{\kr}) & \cong R\Gamma_{\cdh}(X, a^\ast_{\cdh} A_n^{\kr}),\\
	R\Gamma_{\cdp}(X, a^\ast_{\cdp} A^{\kr})  & \cong R\Gamma_{\cdh}(X, a^\ast_{\cdh} A^{\kr}).
	\end{align*}
\end{proposition}

\begin{proof}
Choose a $\cdp$-hypercovering $X_{\bullet} \rightarrow X$ such that each $X_i$ is smooth over $k$. 
Then the canonical map 
\[R\Gamma_{\cdp}(X, a^\ast_{\cdp} A_n^{\kr})  \rightarrow R\Gamma_{\cdh}(X, a^\ast_{\cdh} A_n^{\kr})\] 
is identified with the map 
\[R\Gamma_{\cdp}(X_{\bullet}, a^\ast_{\cdp} A_n^{\kr}) \rightarrow 
R\Gamma_{\cdh}(X_{\bullet}, a^\ast_{\cdh} A_n^{\kr})\] 
by $\cdp$-descent, 
and the latter map is a quasi-isomorphism 
because, for each $i$, the map 
\[ R\Gamma_{\cdp}(X_i, a^\ast_{\cdp} A_n^{\kr}) \rightarrow 
R\Gamma_{\cdh}(X_i, a^\ast_{\cdh} A_n^{\kr}) \] 
is a quasi-isomorphism by Corollary \ref{cor: cdp-rh}. So we have shown the former quasi-isomorphism 
of the proposition. 
The latter quasi-isomorphism can be proven in the same way.
\end{proof}

For a $k$-variety $X$, we define 
\begin{align*}
& H^\ast_{\cdp}(X,a^\ast_{\cdp} A^{\kr}):= H^\ast R\Gamma_{\cdp}(X,a^\ast_{\cdp} A^{\kr}), \\ 
& H^\ast_{\cdh}(X,a^\ast_{\cdh} A^{\kr}):= H^\ast R\Gamma_{\cdh}(X,a^\ast_{\cdh} A^{\kr}). 
\end{align*}
They are isomorphic by Proposition \ref{prop:cdprhvar}:
\begin{equation}\label{equ: rh-cdp}
H^\ast_{\cdp}(X,a^\ast_{\cdp} A^{\kr}) \cong H^\ast_{\cdh}(X,a^\ast_{\cdh} A^{\kr}).
\end{equation}
Also, by Corollary \ref{cor: cdp-rh}, 
we have the functorial isomorphisms 
\begin{equation}\label{equ: cris-A}
H^\ast_{\cris}((X,\overline{X})/W(k)) \cong 
H^\ast_{\cdp}(X,a^\ast_{\cdp} A^{\kr}) \cong H^\ast_{\cdh}(X,a^\ast_{\cdh} A^{\kr})
\end{equation}
for an $nc$-pair $(X,\overline{X})$. 
\begin{remark}
While for several of the proofs in the remainder of this section and in the subsequent section,
it suffices to work in the $\cdp$-topology, 
there are some arguments where it is indispensable to work in the $\cdh$-topology. 
Therefore, as well as for philosophical reasons 
(the $\cdh$-topology is generated by the $\cdp$- and the Nisnevich topology, and Nisnevich descent should be quite essential for a `good' cohomology theory in this context),
we formulate the statements in terms of the $\cdh$-topology.
\end{remark}

We show that this cohomology gives a good integral $p$-adic cohomology theory, 
by proving the following two theorems. 

\begin{theorem}\label{thm: fin gen}
Assume Hypotheses  \ref{strong resolutions}, \ref{embedded resolutions}, \ref{embedded resolution with boundary} and \ref{weak factorisation}.
Then, for a $k$-variety $X$ of dimension $d$, the cohomology groups $H^\ast_{\cdh}(X,a^\ast_{\cdh} A^{\kr})$ are finitely generated $W(k)$-modules and we have 
\[
	H^n_{\cdh}(X,a^\ast_{\cdh} A^{\kr}) =0,
\]
for $n<0$ or $n>2d$.
\end{theorem}

\begin{proof}
Let us first prove that the cohomology groups are finitely generated. 
Choose a $\cdp$-hypercovering $X_{\bullet} \rightarrow X$ such that each $X_i$ is smooth over $k$ and has a normal crossing compactification $\overline{X}_i$ which exists by our hypothesis on embedded resolutions of singularities. (This is a special case of Proposition \ref{prop: good simplicial Zariski}.)
By definition of $a^\ast_{\cdh} A^{\kr}$  one obtains a descent spectral sequence
	$$
	E_1^{ij}=H^j_{\cdh}(X_i, a^\ast_{\cdh} A^{\kr}) \Longrightarrow H^{i+j}_{\cdh}(X_{\bullet}, a^\ast_{\cdh} A^{\kr}) \cong H^{i+j}_{\cdh}(X, a^\ast_{\cdh} A^{\kr}).
	$$
For each $i,j$, we have 
the natural isomorphism $H^j_{\cris}((X_i,\overline{X}_i)/W(k) ) \xrightarrow{\cong} 
H^j_{\cdh}(X_i,a^\ast_{\cdh} A^{\kr})$ by Corollary \ref{cor: cdp-rh}.
This shows that all the $E_1$-terms are finitely generated over $W(k)$, and hence the same is true for the abutment $ H^\ast_{\cdh}(X, a^\ast_{\cdh} A^{\kr})$. 

Let us now prove that the cohomology is concentrated in degrees $0\leqslant n\leqslant 2d$, by induction on the dimension. 
For $d=0$, $X$ is a finite union of points. In this case, the vanishing for $n\neq 0$ follows from the comparison with crystalline cohomology $H^n_{\cdh}(X,a^\ast_{\cdh} A^{\kr}) \cong H^n_{\mathrm{cris}}(X/W(k))$. For $d>0$, let us write $X$ as the union of its irreducible components $X=\bigcup_{i=1}^m X_i$. The blow-up along $X_1$ yields the blow-up square
\[
	\xymatrix{
		X'\times_X X_1 \ar[d] \ar[r] & X' \ar[d] \\
		X_1 \ar[r] & X,
	}
\] 
with $X'=\bigcup_{i=2}^m X'_i$, where $X_i'$ is the blow-up of $X_i$ along $X_i\cap X_1$. By $\cdp$-descent, we obtain the following distinguished triangle:
\[
	R\Gamma_{\cdh}(X,a^{\ast}_{\cdh} A^{\kr}) \rightarrow R\Gamma_{\cdh}(X',a^{\ast}_{\cdh} A^{\kr})\oplus R\Gamma_{\cdh}(X_1,a^{\ast}_{\cdh} A^{\kr}) \rightarrow R\Gamma_{\cdh}(X'\cap X_1,a^{\ast}_{\cdh} A^{\kr}) \rightarrow
\]
Since $\dim X'\cap X_1\leqslant d-1$ we may reduce by induction to the case $m=1$, i.e., we may without loss of generality reduce to the case that $X$ is integral of dimension $d$. By resolution of singularity and $\cdp$-descent, we may assume that $X$ is smooth. Indeed, each blow-up square appearing in such a resolution
\[
	\xymatrix{
		Z' \ar[d] \ar[r] & X' \ar[d] \\
		Z \ar[r] & X
	}
\] 
induces a distinguished triangle
\[
	R\Gamma_{\cdh}(X,a^{\ast}_{\cdh} A^{\kr}) \rightarrow R\Gamma_{\cdh}(X',a^{\ast}_{\cdh} A^{\kr})\oplus R\Gamma_{\cdh}(Z,a^{\ast}_{\cdh} A^{\kr}) \rightarrow R\Gamma_{\cdh}(Z',a^{\ast}_{\cdh} A^{\kr}) \rightarrow.
\]
with $\dim Z,\dim Z'\leqslant d-1$. Thus, we may assume that $X$ is smooth of dimension $d$. Let us choose a normal crossing compactification $(X,\overline{X})$. We obtain by Corollary \ref{cor: cdp-rh} an isomorphism
\[
	H^n_{\cdh}(X, a^\ast_\cdh A^{\kr})\cong H^n_{\mathrm{cris}}((X,\overline{X})/W(k)).
\]
For the log crystalline cohomology, we have the weight spectral sequence, see \cite[(2.6.0.1)]{N_2012},
\[
	E_1^{-k,n+k}=H_{\mathrm{cris}}^{n-k}(D^{(k)}) \Rightarrow H_{\mathrm{cris}}^n((X,\overline{X})/W(k)),
\]
where $D^{(k)}$ denotes the disjoint union of all $k$-fold intersections of distinct irreducible components of $D$. By the induction hypothesis, we have $H^{n-k}_{\mathrm{cris}}(D^{(k)}/W(k))=0$ for $2(d-k)<n-k$. Now, the weight spectral sequence implies the desired vanishing $H^n_{\cdh}(X, a^\ast_\cdh A^{\kr})\cong H^n_{\mathrm{cris}}((X,\overline{X})/W(k))=0$ for $n>2d$.
\end{proof}

\begin{theorem}\label{thm:comparison}
Assume Hypotheses  \ref{strong resolutions}, \ref{embedded resolutions}, \ref{embedded resolution with boundary} and \ref{weak factorisation}. Then, 
for a $k$-variety $X$, there is a canonical quasi-isomorphism
	$$
	R\Gamma_{\rig}(X/K)\cong R\Gamma_{\cdh}(X,a^\ast_{\cdh} A^{\kr}) \otimes_{\Z} \Q,
	$$
hence the isomorphism 
	$$
    H_{\rig}^\ast(X/K)\cong H_{\cdh}^{\ast}(X,a^\ast_{\cdh} A^{\kr}) \otimes_{\Z} \Q.
    $$
\end{theorem}

\begin{proof}
Take a compactification $X \hookrightarrow \overline{X}$ of $X$ and take a simplicial $nc$-pair $(X_{\bullet},\overline{X}_{\bullet}) \to (X,\overline{X})$ such that $X_{\bullet} \to X$ is a split $\cdp$-hypercovering. The existence of such a simplicial $nc$-pair follows from 
Lemma \ref{lem:domhypercover1'} (or the nonemptyness of the category 
${\rm Ho}\{(X_{\bullet},\overline{X}_{\bullet})/(X,\overline{X})\}$ in 
Definition \ref{def:homotopycat} proven in Corollary \ref{cor:homotopycat}) 
and Example \ref{example:varkgeocdp}. (The proof in Proposition \ref{prop: good simplicial Zariski} also implies it.) 
Then, by \cite[Cor.\,7.7, Cor.\,11.7]{N_2012}, we have the canonical quasi-isomorphism 
\begin{equation*}
R\Gamma_{\rig}(X/K)\cong 
R\Gamma_{\et} (\overline{X}_{\bullet}, W\Omega^{\kr}_{\overline{X}_\bullet}(\log D_{\bullet})) 
\otimes_{\Z} \Q,  
\end{equation*}
where $D_{\bullet} = \overline{X}_{\bullet} \setminus X_{\bullet}$. 
Because we have the quasi-isomorphisms 
\begin{align*}
R\Gamma_{\et} (\overline{X}_{\bullet}, W\Omega^{\kr}_{\overline{X}_\bullet}(\log D_{\bullet}))
& \cong A(X_{\bullet}, \overline{X}_{\bullet}) 
\cong A(X_{\bullet}) 
\cong R\Gamma_{\cdh}(X_{\bullet}, a_{\cdp}^*A^{\kr}) 
\cong R\Gamma_{\cdh}(X, a_{\cdh}^*A^{\kr}), 
\end{align*}
we obtain the 
quasi-isomorphism 
\begin{equation}\label{eq:rigid-comparison}
	R\Gamma_{\rig}(X/K)\cong R\Gamma_{\cdh}(X,a^\ast_{\cdh} A^{\kr}) \otimes_{\Z} \Q
\end{equation}
as desired. By construction, this quasi-isomorphism is functorial with respect to the diagram 
$(X_{\bullet},\overline{X}_{\bullet}) \to (X,\overline{X})$. 

We prove that the quasi-isomorphism \eqref{eq:rigid-comparison} depends only on $X$. If we take another simplicial $nc$-pair $(X'_{\bullet},\overline{X}'_{\bullet}) \to (X,\overline{X})$ as above, we have the third one $(X''_{\bullet},\overline{X}''_{\bullet}) \to (X,\overline{X})$ covering the two by Proposition \ref{prop:domhypercover'} (or the connectedness  of the category ${\rm Ho}\{(X_{\bullet},\overline{X}_{\bullet})/(X,\overline{X})\}$ in Definition \ref{def:homotopycat} proven in Corollary \ref{cor:homotopycat}) 
and Example \ref{example:varkgeocdp}. Then, by functoriality,  the quasi-isomorphism 
\eqref{eq:rigid-comparison} induced by $(X_{\bullet},\overline{X}_{\bullet}) \to (X,\overline{X})$ is the same as that induced by $(X''_{\bullet},\overline{X}''_{\bullet}) \to (X,\overline{X})$ and then the same as that for $(X'_{\bullet},\overline{X}'_{\bullet}) \to (X,\overline{X})$. 
Thus the quasi-isomorphism \eqref{eq:rigid-comparison} is independent of the simplicial 
 $nc$-pair $(X_{\bullet},\overline{X}_{\bullet}) \to (X,\overline{X})$ once we fix $(X,\overline{X})$. Next, if we take another compactification $\overline{X}'$ of $X$, we have the third compactification $\overline{X}''$ of $X$ covering the two. If we take a simplicial $nc$-pair $(X_{\bullet},\overline{X}_{\bullet}) \to (X,\overline{X}'')$ such that $X_{\bullet} \to X$ is a split $\cdp$-hypercovering, it is regarded also as a simplicial $nc$-pair over $(X,\overline{X})$ and over $(X,\overline{X}')$. Hence  the quasi-isomorphism \eqref{eq:rigid-comparison} is independent of the compactification $\overline{X}$ of $X$ as well. 

We prove the functoriality of the quasi-isomorphism \eqref{eq:rigid-comparison} with respect to a morphism $f:X' \to X$ in $\Var_k$. Take a morphism $\overline{f}:(X',\overline{X}') \to (X,\overline{X})$ in $\Var_k^{geo}$ extending $f$, 
and take a commutative diagram
\begin{equation}\label{eq:funcoriality-rigid-comparison}
\xymatrix{
(X'_{\bullet},\overline{X}'_{\bullet}) \ar[r]^{\overline{f}_{\bullet}} \ar[d] & 
(X_{\bullet},\overline{X}_{\bullet}) \ar[d] \\ 
(X',\overline{X}') \ar[r]^{\overline{f}} & (X,\overline{X})
}
\end{equation}
such that $(X_{\bullet},\overline{X}_{\bullet}), (X'_{\bullet},\overline{X}'_{\bullet})$ are simplicial $nc$-pairs and $X_{\bullet} \to X$, $X'_{\bullet} \to X'$ are split $\cdp$-hypercoverings. The existence of such a diagram follows from Lemma \ref{lem:domhypercover1'}, Proposition \ref{prop:domhypercover'} (or the nonemptyness of the category ${\rm Ho}\{\overline{f}_{\bullet}/\overline{f}\}$ in Definition \ref{def:homotopycat} proven in Corollary \ref{cor:homotopycat}) and Example \ref{example:varkgeocdp}. By construction, the quasi-isomorphism \eqref{eq:rigid-comparison} is functorial with respect to the diagram 
\eqref{eq:funcoriality-rigid-comparison}. We need to prove that this functoriality 
depends only on $f$. 
For a fixed $\overline{f}$, if we take another diagram as in \eqref{eq:funcoriality-rigid-comparison}, we have the third one covering the two up to homotopy by Proposition \ref{prop:domhypercover'} (or the connectedness of the category ${\rm Ho}\{\overline{f}_{\bullet}/\overline{f}\}$ in Definition \ref{def:homotopycat} proven in Corollary \ref{cor:homotopycat}) and Example \ref{example:varkgeocdp}. Thus the functoriality of 
the quasi-isomorphism \eqref{eq:rigid-comparison} depends only on $\overline{f}$. 
Moreover, if we take another morphism $\overline{f}'$ in $\Var_k^{geo}$ extending $f$, we have the third one $\overline{f}''$ covering the two. Then, since a diagram \eqref{eq:funcoriality-rigid-comparison} for $\overline{f}''$ induces that for $\overline{f}$ and that for $\overline{f}'$ with the same top horizontal arrow, we see that the functoriality of 
the quasi-isomorphism \eqref{eq:rigid-comparison} is independent of the choice of $\overline{f}$ and so depends only on $f$. So we are done. 
\end{proof}

\begin{remark}\label{rem:rhexplicit}
It is possible to find an explicit complex $A^{\kr}_{\cdh}(X)$ representing 
$$R\Gamma_{\cdh}(X,a_{\cdh}^{\ast}A^{\kr}) =  
R\varprojlim_n R\Gamma_{\cdh}(X,a_{\cdh}^{\ast}A_n^{\kr})$$ in such a way that 
it is functorial in $X \in \Var_k$, because we can take Godement resolution 
on the $\cdh$-site. (For  a geometric description 
of the points of the $\cdh$-site, see \cite[Thm.\,1.6, 2.3, 2.6]{GK_2015}.)

Then, if we denote by $C^{\geqslant 0}(W(k))$ 
the category of complexes of $W(k)$-modules of non-negative degree, 
we obtain a functor 
$$ A_{\cdh}^{\kr}: \Var_k \rightarrow C^{\geqslant 0}(W(k))$$
satisfying the following conditions: 
\begin{enumerate}
\item
It gives a good integral $p$-adic cohomology theory in the following sense: 
\begin{enumerate}
\item
For any $X \in \Var_k$, $A^{\kr}_{\cdh}(X)$ has finitely generated cohomologies over $W(k)$. 
\item 
For any $nc$-pair $(X, \overline{X})$, there exists a functorial quasi-isomorphism 
$$A^{\kr}_{\cdh}(X) \simeq R\Gamma_{\cris}((X,\overline{X})/W(k)).$$ 
\item 
For any $X \in \Var_k$, there exists a functorial quasi-isomorphism 
$$A^{\kr}_{\cdh}(X)_{\Q} \simeq R\Gamma_{\rig}(X/K).$$ 
\end{enumerate}
\item
It satisfies  $\cdh$-descent, that is, for any $\cdh$-hypercovering
$X_{\bullet} \rightarrow X$, the induced morphism 
$$A^{\kr}_{\cdh}(X) \rightarrow A^{\kr}_{\cdh}(X_{\bullet})$$ 
is a quasi-isomorphism. 
\end{enumerate}
\end{remark}

The torsion in our integral $p$-adic cohomology behaves differently from $l$-adic cohomology, as the following example shows.
\begin{example}
	Let $C\subseteq \mathbb{P}_{\mathbb{F}_p}^2$ be a smooth curve of genus $g>0$. We define $X$ as the push-out along the Frobenius morphism in the following diagram 
	\begin{equation}\label{eq:ex_torsion}
		\xymatrix{
			C \ar[r]\ar[d]^{F} & \mathbb{P}_{\mathbb{F}_p}^2 \ar[d]\\
			C \ar[r] & X.
		}
	\end{equation}
	It follows from  
	\cite[\href{https://stacks.math.columbia.edu/tag/0ECH}{Tag 0ECH}]{StacksProject} that this diagram is a blow-up square, hence it induces a distinguished triangle
	\[
		R\Gamma_{
\cdh}(X,a^{\ast}_{\cdh} A^{\kr}) \rightarrow R\Gamma_{\cdh}(C,a^{\ast}_{\cdh} A^{\kr})\oplus R\Gamma_{\cdh}(\mathbb{P}_{\mathbb{F}_p}^2,a^{\ast}_{\cdh} A^{\kr}) \rightarrow R\Gamma_{\cdh}(C,a^{\ast}_{\cdh} A^{\kr}) \rightarrow.
	\]
	Using the comparison to crystalline cohomology gives
	\[
		\xymatrix{ \dots \ar[r] & H^1_{\mathrm{cris}}(C) \ar[r]^{F^*} & H^1_{\mathrm{cris}}(C) \ar[r] & H^2_{\cdh}(X,a^\ast_\cdh A^{\kr}) \ar[r] & H^2_{\mathrm{cris}}(\mathbb{P}^2)\oplus H^2_{\mathrm{cris}}(C)  \ar[r] & \dots,}
	\]
	and we deduce $ H^2_{\cdh}(X, a^\ast_{\cdh} A^{\kr})_{\mathrm{tors}}\cong \mathrm{coker}( F^* \colon H^1_{\mathrm{cris}}(C) \to H^1_{\mathrm{cris}}(C))$. Since $C$ is a smooth curve of genus $g>0$, the Frobenius endomorphism has a positive slope and we deduce $H^2_{\cdh}(X, a^\ast_{\cdh} A^{\kr})_{\mathrm{tors}}\neq 0$. 
On the other hand, a similar argument implies that $H^2_{\et}(X_{\overline{\mathbb{F}}_p}, {\mathbb{Z}}_l)_{\mathrm{tors}}=0$ for a prime $l \not= p$ 
because the Frobenius 
$F^* \colon H^1_{\et}(C_{\overline{\mathbb{F}}_p}, {\mathbb{Z}}_l) \to 
H^1_{\et}(C_{\overline{\mathbb{F}}_p}, {\mathbb{Z}}_l)$ is an isomorphism.
\end{example}

%%%%
\subsection{Further properties}
%%%

In this part we want to discuss further properties of the integral $p$-adic 
cohomology theory introdued for $k$-varieties in the last subsection. 
We continue to assume Hypotheses  \ref{strong resolutions}, \ref{embedded resolutions}, \ref{embedded resolution with boundary} and \ref{weak factorisation}.

Let us start with the K\"unneth formula:
\begin{proposition}\label{prop: kuenneth}
	Let $X_1,X_2\in \Var_k$ and assume the Hypotheses \ref{strong resolutions}, \ref{embedded resolutions}, \ref{embedded resolution with boundary} and \ref{weak factorisation}, then
	\[
		R\Gamma_{\cdh}(X_1, a_{\cdh}^\ast A^{\kr})\otimes_{W(k)}^L R\Gamma_{\cdh}(X_2, a_{\cdh}^\ast A^{\kr})\cong R\Gamma_{\cdh}(X_1\times X_2, a_{\cdh}^\ast A^{\kr}).
	\]
\end{proposition}
\begin{proof}
	Let us choose compactifications $(X_1,\overline{X}_1),(X_2,\overline{X}_2)\in \Var_k^{geo}$. 
	By our assumptions on resolution of singularities, we can find
simplicial $nc$-pairs $(X_{i,\bullet}, \overline{X}_{i,\bullet})$ over $(X_i,\overline{X}_i)$ for $i=1,2$ such that $X_{i,\bullet} \to X_i \,(i=1,2)$ are split $\cdp$-hypercoverings, as explained in the proof of Theorem \ref{thm:comparison}.
Let us denote by 
	\[
		(X_{3,\bullet}, \overline{X}_{3,\bullet})\rightarrow (X_{3}, \overline{X}_{3})
	\]
their product. 
	By $\cdp$-descent and Corollary \ref{cor: cdp-rh}, 
	we have for $i=1,2,3$ the following isomorphism in the derived category of $W_n(k)$-modules
	\[
		R\Gamma_{\cdh}(X_i,a_{\cdh}^\ast A^{\kr}_n)\cong R\Gamma_{\mathrm{cris}}((X_{i,\bullet},\overline{X}_{i,\bullet})/W_n(k)):=R\Gamma_{\et}(\overline{X}_{i,\bullet}, W_n\Omega^{\kr}_{\overline{X}_{i,\bullet}}(\log D_{i,\bullet})),
	\]
	where $D_{i,\bullet}:=\overline{X}_{i,\bullet}\setminus X_{i,\bullet}$.
	These complexes are perfect complexes, since they are bounded and their cohomology groups are finitely generated $W_n(k)$-modules. 
	By \cite[Thm.\,5.10, 2)]{N_2012}, 
	we have the following K\"unneth formula for the log crystalline cohomology of simplicial schemes:
	\begin{equation}\label{eq:Kuenneth_n}
		R\Gamma_{\mathrm{cris}}((X_{1,\bullet},\overline{X}_{1,\bullet})/W_n(k))\otimes_{W_n(k)}^L R\Gamma_{\mathrm{crys}}((X_{2,\bullet},\overline{X}_{2,\bullet})/W_n(k))\cong R\Gamma_{\mathrm{cris}}((X_{3,\bullet},\overline{X}_{3,\bullet})/W_n(k)).
	\end{equation}
	Furthermore, we have the isomorphism 
	\begin{multline*}
	R\Gamma_{\mathrm{cris}}((X_{1,\bullet},\overline{X}_{1,\bullet})/W_n(k))\otimes_{W_n(k)}^L R\Gamma_{\mathrm{cris}}((X_{2,\bullet},\overline{X}_{2,\bullet})/W_n(k))\\
	\cong R\Gamma_{\mathrm{cris}}((X_{1,\bullet},\overline{X}_{1,\bullet})/W(k))\otimes_{W(k)}^L R\Gamma_{\mathrm{cris}}((X_{2,\bullet},\overline{X}_{2,\bullet})/W(k))\otimes^L_{W(k)} W_n(k).
	\end{multline*}
	Applying $R\varprojlim_{n}$ to \eqref{eq:Kuenneth_n} gives
	\[
		R\Gamma_{\mathrm{cris}}((X_{1,\bullet},\overline{X}_{1,\bullet})/W(k)) \otimes_{W(k)}^L R\Gamma_{\mathrm{cris}}((X_{2,\bullet},\overline{X}_{2,\bullet})/W(k))\cong R\Gamma_{\mathrm{cris}}((X_{3,\bullet},\overline{X}_{3,\bullet})/W(k)).
	\]
	By $\cdp$-descent and Corollary \ref{cor: cdp-rh}, we have
	\[
		R\Gamma_{\mathrm{cris}}((X_{i,\bullet},\overline{X}_{i,\bullet})/W(k))\cong R\Gamma_{\cdh}(X_i, a^\ast_\cdh A^{\kr})
	\]	
	and we obtain the desired K\"unneth formula
	\[
		R\Gamma_{\cdh}(X_1, a_\cdp^\ast A^{\kr})\otimes_{W(k)}^L R\Gamma_{\cdh}(X_2, a_\cdp^\ast A^{\kr})\cong R\Gamma_{\cdh}(X_1\times X_2, a^\ast_\cdh A^{\kr}).
	\]
\end{proof}

\begin{corollary}\label{cor:homotopyinvariance}
Let $X\in\Var_k$ and $f:\mathbb{A}^1_X\rightarrow X$ the natural projection from the affine line over $X$ to $X$. 
Then the pull-back map
$$
f^\ast: R\Gamma_{\cdh}(X,a^\ast_{\cdh}A^{\kr}) \rightarrow R\Gamma_{\cdh}(\mathbb{A}^1_X,a^\ast_{\cdh}A^{\kr})
$$
is a quasi-isomorphism.
\end{corollary}
Our next aim is to prove the invariance of $H^{\ast}_{\cdh}(X, a^{\ast}_{\cdh}A^{\kr})$ under semi-normalisation. 
Let us first recall the definition of the semi-normalisation:
\begin{definition}
	The semi-normalisation $X^{\mathrm{sn}}\rightarrow X$ of $X\in \Var_k$ is the initial object in the category of universal homeomorphisms $Y\rightarrow X$ inducing isomorphisms on residue fields. 
\end{definition}
\begin{remark}
The semi-normalisation exists by \cite[\href{https://stacks.math.columbia.edu/tag/0EUS}{Lemma 0EUS}]{StacksProject}. 
Since the morphism 
$X^{\mathrm{sn}}\rightarrow X$ is integral by 
\cite[\href{https://stacks.math.columbia.edu/tag/04DF}{Lemma 04DF}]{StacksProject} and birational, it is covered by the normalisation $X^{\mathrm{n}} \to X$. 
Because the normalisation of a variety is a finite morphism, 
$X^{\mathrm{sn}}$ is finite over $X$.
\end{remark}

\begin{proposition}
	Let us assume the Hypotheses \ref{strong resolutions}, \ref{embedded resolutions}, \ref{embedded resolution with boundary} and \ref{weak factorisation}.
	Let $X,X'\in \Var_k$ and $f\colon X'\to X$ be a universal homeomorphism inducing isomorphisms on residue fields, then
	\[
		H^\ast_{\cdh}(X,a^\ast_{\cdh} A^{\kr})\cong H^\ast_{\cdh}(X',a^\ast_{\cdh} A^{\kr}).
	\]
	In particular, $H^\ast_{\cdh}(-,a^\ast_{\cdh} A^{\kr})$ is invariant under passing to the semi-normalisation.
\end{proposition}
\begin{proof}
	By the universal property of the semi-normalisation, the upper horizontal arrow in the diagram
	\[
		\xymatrix{
			(X')^{\mathrm{sn}} \ar[r] \ar[d] &  X^{\mathrm{sn}}\ar[d] \\
			X' \ar[r] & X
		}
	\]
	is an isomorphism. Thus, it is enough to prove the statement for $f\colon X^{\mathrm{sn}}\to X$, which is a finite birational morphism. In particular, it is proper and there exists a closed subset $Z\subseteq X$ such that $f$ is an isomorphism on $X\setminus Z$. Thus, we obtain a blow-up square
	\[
		\xymatrix{
			Z'\ar[r]\ar[d] & X^\mathrm{sn}\ar[d] \\
			Z \ar[r] & X,
		}
	\]
	with $Z':=Z\times_{X} X^{\mathrm{sn}}$. By $\cdp$-descent, we have a distinguished triangle
	 \[
	R\Gamma_{\cdh}(X,a^{\ast}_{\cdh} A^{\kr}) \rightarrow R\Gamma_{\cdh}(X^{\mathrm{sn}},a^{\ast}_{\cdh} A^{\kr})\oplus R\Gamma_{\cdh}(Z,a^{\ast}_{\cdh} A^{\kr}) \rightarrow R\Gamma_{\cdp}(Z',a^{\ast}_{\cdp} A^{\kr}) \rightarrow,
	\]
	and it suffices to prove the claim for $Z'\rightarrow Z$. 
	The statement follows now by Noetherian induction. 
\end{proof}

Finally, we construct a theory of Chern classes for our integral $p$-adic cohomology theory which are compatible with that for log crystalline cohomology and that for rigid cohomology.

 First we recall the Chern classes for rigid cohomology constructed by Petrequin \cite{P_2003}.

\begin{proposition}\label{prop: rig Chern}
For $X\in \Var_k$ there exists a unique theory of rigid Chern classes
$$
c_i^{\rig}: K_0(X) \rightarrow H^{2i}_{\rig}(X/K). 
$$
More precisely, for any (Zariski) locally free $\mathcal{O}_X$-module $\mathcal{E}$ of finite rank there are elements $c_i^{\rig}(\mathcal{E})\in H^{2i}_{\rig}(X/K)$, $i\geqslant 0$, characterised by the following properties:
\begin{enumerate}
\item \textbf{Normalisation.}
We have $c_0^{\rig}(\mathcal{E})=1$, $c_i^{\rig}(\mathcal{E})=0$ for $i>\mathrm{rk}(\mathcal{E})$, and if $\mathcal{E}$ is invertible, $c^{\rig}_1(\mathcal{E})\in H^2_{\rig}(X/K)$ can be described in terms of \v{C}ech cocycles associated to the image of the class of $\mathcal{E}$ in $\mathrm{Pic}(X)$. 

\item \textbf{Functoriality.}
For every morphism $f:X'\rightarrow X$ in $\Var_k$, 
we have $c_i^{\rig}(f^\ast\mathcal{E})=f^\ast c_i^{\rig}(\mathcal{E})$ for all $i\geqslant 0$.

\item \textbf{Whitney sum formula.}
For a short exact sequence 
$0\rightarrow \mathcal{E}' \rightarrow \mathcal{E} \rightarrow \mathcal{E}''\rightarrow 0$
 of locally free $\mathcal{O}_X$-modules of finite rank, 
we have
$$
c_i^{\rig}(\mathcal{E}) = \sum_{j+k=i}c_j^{\rig}(\mathcal{E}') c_k^{\rig}(\mathcal{E}'').
$$
\end{enumerate}
\end{proposition}

\begin{proof}
The construction of the first Chern class is given in \cite[Def.\,4.1]{P_2003} based on computations in \cite[\S\,3]{P_2003}. 
The total Chern class is then defined in \cite[Def.\,4.5]{P_2003} via the projective bundle formula \cite[Prop.\,4.3]{P_2003}. 
Normalisation follows by definition \cite[Rem.\,4.6]{P_2003}. 
Functoriality is proved in \cite[Prop.\,4.7]{P_2003}. 
Finally, the Whitney sum formula is shown in \cite[Cor.\,5.24]{P_2003}. 
Now the projective bundle formula \cite[Cor.\,4.4]{P_2003} implies the first splitting principle \cite[Exp.\,XVI, Prop.\,1.4]{ILO_2014}. (We note that the second splitting principle \cite[Exp.\,XVI, Prop.\,1.5]{ILO_2014}
 also holds for rigid cohomology by the fact that rigid cohomology satisfies Zariski descent %\cite[Thm.\,10.6.1]{CT_2003} 
and homotopy invariance.)  
Thus by a standard argument  \cite[Thm.\,1]{G_1958} the three properties (i)--(iii) characterise the theory of Chern classes entirely, proving therefore uniqueness.
\end{proof}

The existence of the theory of log crystalline Chern classes
$$
c_i^{\cris}: K_0(X) \rightarrow H^{2i}_{\cris}((X,\overline{X})/W(k)), 
i\geqslant 0,
$$
for an $nc$-pair $(X,\overline{X})$ over $k$ 
seems to be well-known to the community 
but not well-documented in the literature. 
For the benefit of the reader we write this out here.

\begin{construction}
The construction of Chern classes that we explain here is built upon results 
in \cite[\S\,12 and \S\,13]{AS_2019}. 
We first remark that the category $\Var_k$ clearly satisfies the condition ($\ast_3$) at the beginning of \cite[\S\,13]{AS_2019}, that is, the blow-up of a smooth scheme along a smooth centre is a morphism in $\Var_k$ \cite[Ex.\,13.3]{AS_2019}. 
Let $\Var_k^{\log}$ be the category consisting of pairs $(U,X)$ of $k$-varieties such that $X$ is smooth, $j:U\hookrightarrow X$ is an open immersion such that $D:=X\backslash U$ is a simple normal crossing divisor.
Again every pair $(U,X)$ gives rise to a log scheme which we denote again by $(U,X)$ where the underlying scheme is $X$ and the log structure is induced by the divisor $D$. 
A morphism $f:(U_1,X_1)\rightarrow (U_2,X_2)$ in $\Var_k^{\log}$  
is a morphism $f:X_1\rightarrow X_2$ in $\Var_k$ such that $f(U_1)\subset U_2$.
(Beware that in the notation of \cite[Def.\,12.1]{AS_2019} $\Var_k^{\log}$ would rather designate the category of pairs $(X,D)$, 
such that $X$ is smooth and $D\hookrightarrow X$ a simple normal crossing divisor.
But these two categories are clearly equivalent.)
By definition the category $\Var_k^{nc}$ is a full subcategory of $\Var_k^{\log}$. 
Thus if we define Chern classes on $\Var_k^{\log}$ we automatically obtain Chern classes on $\Var_k^{nc}$. 

%%%%%%

According to \cite[Ex.\,12.6(2)]{AS_2019} logarithmic Hodge--Witt cohomology 
can be regarded as an admissible cohomology theory with logarithmic poles on $\Var_k^{\log}$ in the sense of \cite[Def.\,12.5]{AS_2019} and therefore admits a theory of Chern classes.
We recall the key points here.
Let $(U,X) \in\Var_k^{\log}$. 
For $n\in\N$, consider the natural morphisms of abelian sheaves  on the \'etale site 
$$
d\log: \co^\times_{U, \et} \rightarrow W_n\Omega^1_{U/k}
$$
defined by
$d\log(u) =\frac{d[u]}{[u]}$, where $[u]$ denotes the Teichm\"uller lift of $u$  
\cite[I.(1.1.1)]{G_1985}. 
Then for $i\geqslant 1$, define $W_n\Omega^i_{U,\log}$ as the abelian subsheaf of $W_n\Omega^i_{U}$ generated \'etale locally by sections of the form
$d\log u_1\cdots d\log u_i$ with $u_1,\ldots,u_i\in \co^\times_{U,\et}$. 
Also, one sets $W_n\Omega^0_{U,\log}=\Z/p^n \Z$. 
For $n\in\N$ fixed, the graded cohomology theory $\Gamma(\ast)^{\log}$ on $\Var_k^{\log}$, defined for  $(U,X)\in \Var_k^{\log}$ as a complex of Zariski sheaves by
$$
\Gamma(i)_{(U,X)}^{\log} :=\begin{cases} \varepsilon_\ast\mathrm{Gd}_{\et} (j_\ast W_n\Omega_{U,\log}^i)[-i] & \text{ for } i\geqslant 0,\\ 0 & \text{ otherwise,}\end{cases}
$$
where $j$ denotes the open immersion $U \hookrightarrow X$,
$\mathrm{Gd}_{\et}$ denotes the Godement resolution with respect to \'etale points and $\varepsilon: X_{\et}\rightarrow X_{\Zar}$ is the natural morphism of sites, 
is an admissible cohomology theory with logarithmic poles. 
In particular, the map $ j_*\co^\times_U \rightarrow R\varepsilon_*j_*W_n\Omega^1_{U,\log}$ for $(U,X)\in \Var_k^{\log}$ induced by $d\log$ plays the role of the first Chern class map of \cite[Def.\,12.5(1)]{AS_2019} and induces a morphism
$$
c_1^{HW}: R\Gamma_{\Zar}(U,\co^\times_{U})\rightarrow R\Gamma_{\Zar}(X,\Gamma(1)^{\log}_{(U,X)}[1]) = R\Gamma_{\et}(X,j_\ast W_n\Omega^1_{U,\log}) 
$$
in the derived category. 
Therefore, as carried out in \cite[\S\,13]{AS_2019} 
(see in particular \cite[Ex.\,13.3]{AS_2019}), 
it admits a theory of Chern classes 
\begin{equation}\label{equ: Chern HW}
c_i^{HW}:K_0(U) \rightarrow H_{\Zar}^{2i}(X,\Gamma(i)^{\log}_{(U,X)})=H^{2i}_{\et}(X,j_\ast W_n\Omega^i_{U,\log}[-i])= H^i_{\et}(X,j_\ast W_n\Omega^i_{U,\log}), i\geqslant 0,
\end{equation}
satisfying the properties (L1) -- (L4) of \cite[\S\,13]{AS_2019} of which (L1) -- (L3) correspond to normalisation, functoriality and the Whitney sum formula. 
As $n\in\N$ varies these maps are naturally compatible. 

Now recall that, for $u \in {\mathcal{M}}_X = 
j_\ast({\mathcal{O}}_{U,\et}^{\times} \cap {\mathcal{O}}_{X,\et})$, 
$d\log(u)=\frac{d[u]}{[u]} \in j_\ast W_n\Omega^1_{U/k}$ is in fact contained in 
$W_n\Omega^1_{X/k}(\log D)$ by Remark \ref{rem:matsuue}, and that 
$d(d\log u) = 0$. So one obtains a morphism of complexes
\begin{equation}\label{equ: log to D}
j_\ast W_n\Omega^i_{U,\log}[-i]\rightarrow W_n\Omega^{\kr}_{X/k}(\log D).
\end{equation}
They induce morphisms in the derived category  
\begin{equation}\label{equ: log to D on coh}
R\Gamma_{\et}(X,j_\ast W_n\Omega^i_{U,\log}[-i]) \rightarrow  
R\Gamma_{\et}(X,W_n\Omega^{\kr}_{X/k}(\log D)) \cong R\Gamma_{\cris}((U,X)/W_n(k)),
\end{equation}
which are (contravariantly) functorial in $(U,X)$. 
Composing these in the correct degree with the Chern class maps (\ref{equ: Chern HW}) we obtain maps
$$
c_i^{\cris}: K_0(U) \rightarrow H^{2i}_{\cris}((U,X)/W_n(k)), i\geqslant 0,
$$
which are again compatible as $n\in\N$ varies, and hence induce maps
$$
c_i^{\cris}: K_0(U) \rightarrow H^{2i}_{\cris}((U,X)/W(k)), i\geqslant 0.
$$
\begin{remark}\label{rem: first log crys Chern class}
Note that the first Chern class is given explicitly by the morphism
$$
j_\ast \co^\times_U \rightarrow 
 R\varepsilon_\ast
j_\ast W_n\Omega^1_{U,\log} \rightarrow 
 R\varepsilon_\ast
W_n\Omega^{\kr}_X(\log D)[1]
$$
induced by \eqref{equ: log to D}, which we also denote by $d\log$.
\end{remark}
\end{construction}

\begin{proposition}\label{prop: cris Chern}
Let $(U,X)\in\Var_k^{\log}$, $D:= X\backslash U$ and $j:U\hookrightarrow X$.
There exists a unique theory of log-crystalline Chern classes
$$
c_i^{\cris}: K_0(U) \rightarrow H^{2i}_{\cris}((U,X)/W(k)). 
$$
More precisely, for any (Zariski) locally free $\mathcal{O}_U$-module $\mathcal{E}$ of finite rank there are elements $c_i^{\cris}(\mathcal{E})\in H^{2i}_{\cris}((U,X)/W(k))$, $i\geqslant 0$, characterised by the following properties:
\begin{enumerate}
\item \textbf{Normalisation.}
We have $c_0^{\cris}(\mathcal{E})=1$, $c_i^{\cris}(\mathcal{E})=0$ for $i>\mathrm{rk}(\mathcal{E})$, and if $\mathcal{E}$ is invertible, $c^{\cris}_1(\mathcal{E})\in H^2_{\cris}((U,X)/W(k))\cong H^2_{\et}(X,W\Omega^{\kr}_{X/k}(\log D))$  is the image of the class of $\mathcal{E}$ in $\mathrm{Pic}(U)$ under the morphism induced on cohomology by the map 
$ d\log:  j_\ast \mathcal{O}_U^\times \rightarrow  
R\varepsilon_*W\Omega^{\kr}_{X/k}(\log D)[1]$.

\item \textbf{Functoriality.}
For every morphism $f:(U',X')\rightarrow (U,X)$ in $\Var_k^{nc}$, 
we have $c_i^{\cris}(f^\ast\mathcal{E})=f^\ast c_i^{\cris}(\mathcal{E})$ for all $i\geqslant 0$.

\item \textbf{Whitney sum formula.}
For a short exact sequence 
$0\rightarrow \mathcal{E}' \rightarrow \mathcal{E} \rightarrow \mathcal{E}''\rightarrow 0$
 of locally free $\mathcal{O}_U$-modules of finite rank, 
we have
$$
c_i^{\cris}(\mathcal{E}) = \sum_{j+k=i}c_j^{\cris}(\mathcal{E}') c_k^{\cris}(\mathcal{E}'').
$$
\end{enumerate}
\end{proposition}

\begin{proof}
The normalisation is satisfied by Remark \ref{rem: first log crys Chern class}.
It is also clear that the maps $c_i^{\cris}$ are contravariantly functorial for morphisms in $\Var_k^{\log}$ as this is true for Hodge--Witt Chern classes (\ref{equ: Chern HW}) and for the morphism (\ref{equ: log to D on coh}). 
Similarly, the Whitney sum formula for $c_i^{\cris}$ follows from the Whitney sum formula for $c_i^{HW}$.

To show uniqueness, we first show a logarithmic version of the projective bundle formula. 
Let again $(U,X)\in\Var_k^{\log}$ and $\mathcal{E}$ a locally free $\co_X$-module of rank $r\geqslant 1$. 
Consider the associated projective bundle $\pi:\mathbb{P}(\mathcal{E}) \rightarrow X$, and $\pi: \mathbb{P}(\mathcal{E}|_U) \rightarrow U$ its restriction to $U$, which together give a morphism $\pi: (\mathbb{P}(\mathcal{E}|_U), \mathbb{P}(\mathcal{E})) \rightarrow (U,X)$ in $\Var_k^{\log}$. 
The first Chern class map defined above induces a natural morphism
\begin{equation}\label{equ: bundle cris}
\bigoplus_{i=0}^{r-1} c_1^{\cris}(\co_{\mathbb{P}(\mathcal{E})}(1))^i\cup \pi^\ast: \bigoplus_{i=0}^{r-1}R\Gamma_{\et}(X,W\Omega^{\kr}_X(\log D))[-2i] \rightarrow R\pi_\ast 
R\Gamma_{\et}(\mathbb{P}(\mathcal{E}),W\Omega^{\kr}_{\mathbb{P}(\mathcal{E})}(\log D')),
\end{equation}
with $D'=\pi^{-1}(D)$.
We claim that this is a quasi-isomorphism. 

Working Zariski locally to prove the claim, 
we may assume that $\mathcal{E}\vert_{X}$ is trivial. 
and it therefore suffices to show that
the natural morphism
\begin{equation}\label{equ: proj bundle formula}
\bigoplus_{i=0}^{r-1} c_1^{\cris}(\co_{\mathbb{P}_X^{r-1}}(1))^i\cup \pi^\ast: \bigoplus_{i=0}^{r-1} H^{\ast-2i}_{\et}(X,W\Omega^{\kr}_{X/k}(\log D)) \rightarrow H^\ast_{\et}(\mathbb{P}^{r-1}_X,W\Omega^{\kr}_{\mathbb{P}^{r-1}_X/k}(\log \mathbb{P}^{r-1}_{D})),
\end{equation}
is an isomorphism. 
But by the Künneth formula for log crystalline cohomology (see the proof of  Proposition \ref{prop: kuenneth}), the term on the right hand side is given by 
\begin{equation}\label{equ: kuenneth P}
H^\ast_{\et}(\mathbb{P}^{r-1}_X,W\Omega^{\kr}_{\mathbb{P}^{r-1}_X/k}(\log \mathbb{P}^{r-1}_{D}))\cong
 \bigoplus_{m=0}^{2(r-1)}  H^{m}_{\et}(\mathbb{P}^{r-1}_k,W\Omega^{\kr}_{\mathbb{P}^{r-1}_k})  \otimes H^{\ast-m}_{\et}(X,W\Omega^{\kr}_{X/k}(\log D)).
\end{equation}
It follows from the computation of finite length Hodge--Witt cohomology in \cite[I, Cor.\,4.2.15]{G_1985} that $H^i_{\et}(\mathbb{P}_k^{r-1},W\Omega^j_{\mathbb{P}_k^{r-1}})$ is non-zero if and only if $0\leqslant j=i\leqslant r-1$ and in that case it is a free $W(k)$-module of rank 1 generated by the ($i$\textsuperscript{th} product of the) first Hodge--Witt Chern class $c_1^{HW}(\co_{\mathbb{P}^{r-1}_k}(1))^i$. 
It follows that the slope spectral sequence degenerates at $E_1$ \cite[II, Thm.\,3.7]{I_1979}, and we see that $H^{m}_{\et}(\mathbb{P}^{r-1}_k,W\Omega^{\kr}_{\mathbb{P}^{r-1}_k})$ is a free $W(k)$-module of rank 1 if and only if $m=2i$, $0\leqslant i\leqslant r-1$ and zero otherwise. 
Thus (\ref{equ: kuenneth P}) becomes
$$
H^\ast_{\et}(\mathbb{P}^{r-1}_X,W\Omega^{\kr}_{\mathbb{P}^{r-1}_X/k}(\log \mathbb{P}^{r-1}_{D}))\cong
 \bigoplus_{r=0}^{r-1}  H^{2i}_{\et}(\mathbb{P}^{r-1}_k,W\Omega^{\kr}_{\mathbb{P}^{r-1}_k})  \otimes H^{\ast-2i}_{\et}(X,W\Omega^{\kr}_{X/k}(\log D)). 
$$
Moreover, by the compatibility of the first Hodge--Witt Chern class with the crystalline Chern class \cite[I, Thm.\,5.1.5]{G_1985}  and the compatibility of the crystalline Chern class and the log crystalline Chern class which is implicit in \cite[\S\,13]{AS_2019}, it follows that $c_1^{\cris}(\co_{\mathbb{P}_k^{r-1}}(1))^i$ is a generator of $H^{2i}_{\et}(\mathbb{P}^{r-1}_k,W\Omega^{\kr}_{\mathbb{P}^{r-1}_k})$. 
And this shows that (\ref{equ: proj bundle formula}) is indeed an isomorphism. 

Now the logarithmic projective bundle formula implies a logarithmic version of the first splitting principle \cite[Exp.\,XVI, Props.\,1.4]{ILO_2014}, because the flag bundle can be written as an iteration of projective bundles:
Let $\mathcal{E}$ be a locally free $\mathcal{O}_X$-module of finite rank 
and consider the flag bundle $\pi: \mathrm{Fl}(\mathcal{E}) \rightarrow X$ and its restriction $\pi:\mathrm{Fl}|(\mathcal{E}|_U ) \rightarrow U$ to $U$, which together give a morphism in $\Var^{\log}_k$ 
$(\mathrm{Fl}|(\mathcal{E}|_U ), \mathrm{Fl}(\mathcal{E})) \rightarrow (U,X)$. 
Then we have injections
$$
H^{2\ast}_{\et}(X,W\Omega^{\kr}_X(\log D)) \hookrightarrow H^{2\ast}_{\et}(\mathrm{Fl}(\mathcal{E}), W\Omega^{\kr}_{\mathrm{Fl}(\mathcal{E})}( \log D'))
$$
with $D'=\pi^{-1}(D)$. 
Using this, the same argument as in \cite[Thm.\,1]{G_1958} mutatis mutandis shows that (i)--(iii) characterise the theory of Chern classes.
\end{proof}

\begin{remark}
It would be also possible to construct log crystalline Chern classes more directly, 
i.e.,\,without passing through logarithmic Hodge--Witt cohomology. 
To this end, one can for example verify that it is an 
admissible cohomology theory with logarithmic poles
on $\Var_k^{\log}$ in the sense of \cite[Def.\,12.5]{AS_2019}.
Then \cite[\S\,13]{AS_2019} implies that we obtain a theory of Chern classes 
for log crystalline cohomology on $\Var_k^{\log}$, 
and hence by restriction also on $\Var_k^{nc}$.
While the properties (0), (1), and (2) of \cite[Def.\,12.5]{AS_2019} are rather straight forward to verify, 
the properties (3) and (4) are more elaborate.
\end{remark}

Now we give the definition of Chern classes for our integral $p$-adic cohomology theory.

\begin{theorem}\label{theorem: A-Chern clases}
For $X\in \Var_k$, there exists a unique theory of Chern classes
$$
c_i^A: K_0(X) \rightarrow H^{2i}_{\cdh}(X,a_{\cdh}^\ast A^{\kr}). 
$$
More precisely, for any (Zariski) locally free $\mathcal{O}_X$-module $\mathcal{E}$ of finite rank there are elements $c_i^{A}(\mathcal{E})\in H^{2i}_{\cdh}(X,a_{\cdh}^\ast A^{\kr})$, $i\geqslant 0$, characterised by the following properties:
\begin{enumerate}
\item \textbf{Normalisation.}
We have $c_0^{A}(\mathcal{E})=1$, $c_i^{A}(\mathcal{E})=0$ for $i>\mathrm{rk}(\mathcal{E})$, and if $\mathcal{E}$ is invertible, $c^{A}_1(\mathcal{E})\in 
H^{2}_{\cdh}(X,a_{\cdh}^\ast A^{\kr})$ is the image of the class of $\mathcal{E}$ in $\mathrm{Pic}(X)$ under the morphism induced on cohomology by the map 
\[
 R\Gamma(X,\co_X^\times) \rightarrow R\Gamma(X_\bullet,\co^\times_{X_\bullet})  
\rightarrow 
R\Gamma_{\et}(\overline{X}_{\bullet}, W\Omega^{\kr}_{\overline{X}_{\bullet}}(\log D_{\bullet})) 
\cong A^{\kr}(X_\bullet)[1] 
\]
induced by $d\log$ on $X_{\bullet}$, where 
$(X_\bullet,\overline{X}_\bullet)$ is a simplicial $nc$-pair 
 over $(X,\overline{X})$ with $\overline{X}$ a compactification of $X$ 
such that $X_\bullet \rightarrow X$ is a split $\cdp$-hypercovering in $\Var_k$ and 
$D_{\bullet} = \overline{X}_{\bullet} \setminus X_{\bullet}$.
\item \textbf{Functoriality.}
For every morphism $f:X'\rightarrow X$ in $\Var_k$, 
we have $c_i^{A}(f^\ast\mathcal{E})=f^\ast c_i^{A}(\mathcal{E})$ for all $i\geqslant 0$.

\item \textbf{Whitney sum formula.}
For a short exact sequence 
$0\rightarrow \mathcal{E}' \rightarrow \mathcal{E} \rightarrow \mathcal{E}''\rightarrow 0$
of locally free $\mathcal{O}_X$-modules of finite rank, 
we have
$$
c_i^{A}(\mathcal{E}) = \sum_{j+k=i}c_j^{A}(\mathcal{E}') c_k^{A}(\mathcal{E}'').
$$
\end{enumerate}
\end{theorem}

\begin{proof}
To obtain Chern classes for our integral $p$-adic cohomology theory 
we follow the argument in \cite[Expos\'e XVI]{ILO_2014}, where the \'etale Chern classes are defined.

We start by constructing the first Chern class map of a line bundle. 
For $X\in\Var_k$, take $\overline{X}$, 
$(X_\bullet,\overline{X}_\bullet)$, $D_{\bullet}$ as in the statement of (i) and  
denote by $j_\bullet: X_\bullet \hookrightarrow \overline{X}_\bullet$ the simplicial open immersion. Then the maps $j_{l,\ast} \mathcal{O}_{X_l}^\times \rightarrow 
R\varepsilon_* W\Omega^{\kr}_{\overline{X}_l/k}(\log D_l)[1]$ in Proposition \ref{prop: cris Chern} for $X_l \, (l \geqslant 0)$ induce a map 
$ j_{\bullet,\ast} \mathcal{O}_{X_\bullet}^\times \rightarrow R\varepsilon_* W\Omega^{\kr}_{\overline{X}_\bullet/k}(\log D_\bullet)[1]$ on $X_{\bullet}$ and so a map
\begin{equation}\label{equ: Chern 1 simplicial}
R\Gamma(X_\bullet,\co^\times_{X_\bullet}) = R\Gamma(\overline{X}_\bullet,j_{\bullet,\ast}\co^\times_{\overline{X}_\bullet}) \rightarrow 
R\Gamma_{\et}(\overline{X}_{\bullet}, W\Omega^{\kr}_{\overline{X}_{\bullet}}(\log D_{\bullet})).
\end{equation}
As $(X_\bullet,\overline{X}_{\bullet})$ is a (simplicial) $nc$-pair we have by definition
$R\Gamma_{\et}(\overline{X}_{\bullet}, W\Omega^{\kr}_{\overline{X}_{\bullet}}(\log D_{\bullet})) \cong A^{\kr}(X_\bullet).$
Thus by Corollary \ref{cor: cdp-rh} and cohomological $\cdp$-descent the above map induces the first Chern class map
\begin{equation}\label{eq:Chern 1 A}
c_1^A: \mathrm{Pic}(X)=H^1_{\Zar}(X,\co^\times_X)  \rightarrow H^1_{\Zar}(X_\bullet,\co^\times_{X_\bullet}) \rightarrow H^2_{\cdh}(X_\bullet,a_{\cdh}^\ast A^{\kr}) \cong H^2_{\cdh}(X,a_{\cdh}^\ast A^{\kr}).
\end{equation}
By construction, this map is functorial with respect to the diagram $(X_{\bullet},\overline{X}_{\bullet}) \to (X,\overline{X})$. 

To see that the map \eqref{eq:Chern 1 A} depends only on $X$ and is functorial with respect to morphisms $f:X' \to X$ in $\Var_k$, 
one repeats verbatim the argument given in the proof of Theorem \ref{thm:comparison}. 

Next we show the projective bundle formula.
Let again $X\in\Var_k$ and $\mathcal{E}$ a locally free $\co_X$-module of rank $r\geqslant 1$. 
Consider the associated projective bundle $\pi:\mathbb{P}(\mathcal{E}) \rightarrow X$.
Then the first Chern class map above induces a natural morphism
\begin{equation}\label{equ: bundle}
\bigoplus_{i=0}^{r-1} c_1^A(\co_{\mathbb{P}(\mathcal{E})}(1))^i\cup \pi^\ast: \bigoplus_{i=0}^{r-1}R\Gamma_{\cdh}(X,a^\ast_{\cdh}A^{\kr})[-2i] \rightarrow R\pi_\ast R\Gamma_{\cdh}(\mathbb{P}(\mathcal{E}),a^\ast_{\cdh}A^{\kr}).
\end{equation}
We claim that this is a quasi-isomorphism.

Since we may work Zariski locally on $X$ to prove the claim, 
we may assume that $\mathcal{E}\vert_{X}$ is trivial. 
Furthermore, since we may work $\cdp$-locally to prove the claim, we may assume that 
$X$ can be compactified into an $nc$-pair 
$(X,\overline{X})$, and the statement follows from the 
the natural isomorphism
$$
\bigoplus_{i=0}^{r-1} c_1^{\cris}(\co_{\mathbb{P}^{r-1}_{\overline{X}}}(1))^i\cup \pi^\ast: \bigoplus_{i=0}^{r-1} H^{\ast-2i}_{\et}({\overline{X}},W\Omega^{\kr}_{{\overline{X}}/k}(\log D)) \xrightarrow{\sim} H^\ast_{\et}(\mathbb{P}^{r-1}_{\overline{X}},W\Omega^{\kr}_{\mathbb{P}^{r-1}_{\overline{X}}/k}(\log \mathbb{P}^{r-1}_{D}))
$$
proved in Proposition \ref{prop: cris Chern}.

By the quasi-isomorphism \eqref{equ: bundle}, we have the isomorphism 
\begin{equation*}
\bigoplus_{i=0}^{r-1} \xi^i\cup \pi^\ast: \bigoplus_{i=0}^{r-1} H^{2(r-i)}_{\cdh}(X,a^\ast_{\cdh}A^{\kr}) \overset{\cong}{\rightarrow} H^{2r}_{\cdh}(\mathbb{P}(\mathcal{E}),a^\ast_{\cdh}A^{\kr}) 
\end{equation*}
for a locally free $\co_X$-module $\mathcal{E}$ of rank $r$, where 
$\xi := c_1^A(\co_{\mathbb{P}(\mathcal{E})}(1))$.
Using this, we define the $i$-th Chern class $c_i^A(\mathcal{E}) \in H^{2i}_{\cdh}(X,a^{\ast}_{\cdh}A^{\kr})$ of $\mathcal{E}$ for $1 \leqslant i \leqslant r$ as the unique element satisfying the equality 
$$ \xi^r - c_1^A({\mathcal{E}}) \xi^{r-1} + c_2^A({\mathcal{E}}) \xi^{r-2} + \cdots + (-1)^rc_r({\mathcal{E}}) = 0 $$ 
in $H^{2r}_{\cdh}({\mathbb{P}}({\mathcal{E}}),a^{\ast}_{\cdh}A^{\kr}).$ 
Also, we put $c_0^A(\mathcal{E}) = 1, c_i^A({\mathcal{E}}) = 0$ for $i > r$. 

When $\mathcal{E}$ is of rank $1$, we have ${\mathbb{P}}({\mathcal{E}}) \cong X$, $\co_{\mathbb{P}(\mathcal{E})}(1) \cong {\mathcal{E}}$ and so the definition of $c_1^A$ for such $\mathcal{E}$ is compatible with the one defined in the beginning. In particular, our Chern classes satisfy the normalisation. Also, the functoriality of our Chern classes follows from that for the first Chern classes and the functoriality of the formation of the projective bundle. 

We prove that our Chern classes satisfy the Whitney sum formula. Note that we have the following two splitting principles (cf. \cite[Expos\'e XVI, 1.4, 1.5]{ILO_2014}): 
\begin{enumerate}
\item[(A)] For a locally free $\co_X$-module $\mathcal{E}$ of finite rank on $X$, 
the pullback of $\mathcal{E}$ to the flag bundle ${\rm Fl}({\mathcal{E}}) \to X$ is by definition an iterated extension of locally free $\co_X$-modules of rank $1$, and we have injections 
$$  H^{2\ast}_{\cdh}(X,a^{\ast}_{\cdh}A^{\kr}) \hookrightarrow 
H^{2\ast}_{\cdh}({\rm Fl}({\mathcal{E}}),a^{\ast}_{\cdh}A^{\kr}) $$
because the flag bundle can be written as an iteration of projective bundles. 
\item[(B)] When we are given an exact sequence 
$ {\rm E} := [0 \to \mathcal{E}' \to \mathcal{E} \to \mathcal{E}'' \to 0]$ 
of locally free $\co_X$-modules of finite rank, we can define `the variety of sections of ${\rm E}$' ${\rm Sect}({\rm E}) \to X$, which is an affine space bundle. By definition, the pullback of the exact sequence ${\rm E}$ to 
${\rm Sect}({\rm E})$ splits, and we have  isomorphisms  
$$ H^{2\ast}_{\cdh}(X,a^{\ast}_{\cdh}A^{\kr}) \overset{\cong}{\to} 
H^{2\ast}_{\cdh}({\rm Sect}({\rm E}),a^{\ast}_{\cdh}A^{\kr}), $$
which follows fom the Zariski descent and the homotopy invariance 
of our integral $p$-adic cohomology (Corollary \ref{cor:homotopyinvariance}). 
\end{enumerate}

Using the splitting principles (A) and (B), the proof of Whitney sum formula is reduced to the case where ${\mathcal{E}} = {\mathcal{E}}' \oplus {\mathcal{E}}''$ and ${\mathcal{E}}', {\mathcal{E}}''$ are direct sums of locally free $\co_X$-module of rank $1$. Hence it suffices to prove that, for 
${\mathcal{E}} = \bigoplus_{i=1}^r {\mathcal{L}}_i$ with each 
${\mathcal{L}}_i$ invertible, the equality 
$$ \prod_{i=1}^r(\xi - c_1^A({\mathcal{L}}_i)) = 0 $$
holds in $H^{2r}_{\cdh}({\mathbb{P}}({\mathcal{E}}),a^{\ast}_{\cdh}A^{\kr})$, 
where $\xi = c_1^A(\co_{{\mathbb{P}}({\mathcal{E}})}(1))$. 
If we denote the projection ${\mathbb{P}}({\mathcal{E}}) \to X$ by $\pi$ and the relative hyperplane in ${\mathbb{P}}({\mathcal{E}})$ defined by the inclusion ${\mathcal{L}}_i \hookrightarrow {\mathcal{E}}$ by $H_i$, the injection 
$\pi^*{\mathcal{L}}_i \hookrightarrow \co_{{\mathbb{P}}({\mathcal{E}})}(1)$ is an isomorphism on ${\mathbb{P}}({\mathcal{E}}) \setminus H_i$. Thus 
the image of $\xi - c_1^A({\mathcal{L}}_i)$ in 
$H^{2}_{\cdh}({\mathbb{P}}({\mathcal{E}}) \setminus H_i,a^{\ast}_{\cdh}A^{\kr})$ is zero and so it comes from some element in 
$H^{2}_{\cdh,H_i}({\mathbb{P}}({\mathcal{E}}),a^{\ast}_{\cdh}A^{\kr})$. 
Therefore the product $\prod_{i=1}^r(\xi - c_1^A({\mathcal{L}}_i))$ comes from 
an element in $H^{2r}_{\cdh,\bigcap_{i=1}^rH_i}({\mathbb{P}}({\mathcal{E}}),a^{\ast}_{\cdh}A^{\kr})$ and it is zero since $\bigcap_{i=1}^rH_i$ is empty. 
Hence the proof of the Whitney sum formula is finished. 

Finally, by the same standard argument explained in  \cite[Thm.\,1]{G_1958} using the splitting principle (A) the three properties (i)--(iii) characterise the theory of Chern classes entirely, proving therefore uniqueness.
\end{proof}

\begin{corollary}
Let $(X,\overline{X})\in\Var_k^{nc}$. 
Then the Chern class maps $c_i^{\cris}$ and $c_i^A$ are compatible
in the sense that there are commutative diagrams
$$
\xymatrix{
 &  H^{2i}_{\cris}((X,\overline{X})/W(k)) \ar@{=}[dd]^\sim\\
K_0(X) \ar[ru]^{c_i^{\cris}} \ar[rd]_{c_i^A} &\\
 & H^{2i}_{\cdh}(X,a_{\cdh}^\ast A^{\kr})
 }
$$
for all $i\geqslant 0$.
\end{corollary}
\begin{proof}
Both theories of Chern classes are unique with respect to the three defining properties 
-- normalisation, functoriality, Whitney sum formula.
Thus it suffices to compare the first Chern class maps for line bundles. 
In other words, we have to show that the diagram
$$
\xymatrix{
 &  H^2_{\cris}((X,\overline{X})/W(k)) \ar@{=}[dd]^\sim\\
\mathrm{Pic}(X)=H^1_{\Zar}(X,\co^\times_X) \ar[ru]^{c_1^{\cris}} \ar[rd]_{c_1^A} &\\
 & H^2_{\cdh}(X,a_{\cdh}^\ast A^{\kr})
}
$$
commutes. 
But this is clear by the definition of (\ref{equ: Chern 1 simplicial}).
\end{proof}

\begin{proposition}
Let  $X\in\Var_k$. 
Then the Chern class maps $c_i^{\rig}$ and $c_i^A$ are compatible 
in the sense that there are commutative diagrams
$$
\xymatrix{
 &  H^{2i}_{\rig}(X/K) \ar@{=}[dd]^\sim\\
K_0(X) \ar[ru]^{c_i^{\rig}} \ar[rd]_{c_i^A} &\\
 & H^{2i}_{\cdh}(X,a_{\cdh}^\ast A^{\kr}) \otimes_{\Z}\Q
 }
$$
for all $i\geqslant 0$.
\end{proposition}

\begin{proof}
As in the above corollary, it suffices to compare the first Chern class maps of line bundles, 
that means we have to show that the diagram
\begin{equation}\label{commutative diagram c_1}
\xymatrix{
 &  H^2_{\rig}(X/K) \ar@{=}[dd]^\sim\\
\mathrm{Pic}(X)=H^1_{\Zar}(X,\co^\times_X) \ar[ru]^{c_1^{\rig}} \ar[rd]_{c_1^A} &\\
 & H^2_{\cdh}(X,a_{\cdh}^\ast A^{\kr}) \otimes_{\Z}\Q
}
\end{equation}
commutes.

We first treat the case $X=\overline{X}$ is already proper.
In this case Petrequin showed in \cite[Cor.\,5.29]{P_2003}
that the first rigid Chern class and the first crystalline Chern class $c_1^{\cris,BI}$ of Berthelot--Illusie \cite{BI_1970}  (compare \cite[Proof of I.Thm.\,5.1.5]{G_1985} for the construction) 
are compatible, namely, there is a commutative diagram 
\begin{equation}\label{commutative diagram c_1 a}
\xymatrix{
 &  H^2_{\rig}(X/K) \ar[dd]\\
\mathrm{Pic}(X)=H^1_{\Zar}(X,\co^\times_X) \ar[ru]^{c_1^{\rig}} \ar[rd]_{c_1^{\cris,BI}} &\\
 & H^2_{\cris}(X/W(k)) \otimes_{\Z}\Q, 
}
\end{equation}
where the vertical map is induced by the canonical map $R\Gamma_{\rig}(X/K) \rightarrow R\Gamma_{\cris}(X/W(k)) \otimes_{\Z}\Q$ constructed in the proof of \cite[Prop.\,1.9]{B_1997}. 

Now take a split $\cdp$-hypercovering $X_\bullet\rightarrow X$ by proper smooth $k$-varieties. 
By functoriality of the first crystalline Chern class map $c_1^{\cris,BI}$ %\cite{BI_1970}, 
there is a commutative diagram
$$
\xymatrix{
R\Gamma(X,\co^\times_{X}) \ar[d] \ar[r]^-{c_1^{\cris,BI}}&  R\Gamma_{\cris}(X/W(k))[1] \ar[d]\\
R\Gamma(X_\bullet,\co^\times_{X_\bullet})  \ar[r]^-{c_1^{\cris,BI}} & R\Gamma_{\cris}(X_\bullet/W(k))[1]. 
}
$$

Note that, for smooth $k$-varieties, the crystalline first Chern class map
$c_1^{\cris,BI}$ defined by Berthelot--Illusie via the crystalline site 
coincides with  the crystalline first Chern class map
$c_1^{\cris}$ defined by Gros and Asakura--Sato via the de Rham--Witt complex by the proof of \cite[I, Thm.\,5.1.5]{G_1985}: it follows from the commutative diagram \cite[(5.1.10)]{G_1985}. The same holds for split simplicial smooth $k$-varieties via the canonical comparison map \cite[Cor.\,7.7]{N_2012} between crystalline and de Rham--Witt cohomology in the split simplicial case, because the commutative diagram \cite[(5.1.10)]{G_1985} implies the same diagram in the simplicial case by functoriality.

As a consequence we obtain a commutative diagram of the form
\begin{equation}\label{commutative diagram c_1 b}
\xymatrix{
 & H^2_{\cris}(X/W(k))  \ar[d]\\
\mathrm{Pic}(X)=H^1_{\Zar}(X,\co^\times_X) \ar[ru]^{c_1^{\cris,BI}} \ar[r]^{c_1^{\cris,BI}} 
\ar[rd]^{c_1^{\cris}} \ar[rdd]_{c_1^A} & H^2_{\cris}(X_\bullet/W(k))   \ar@{=}[d]^\sim\\
&  H^2_{\et}(X_\bullet, W\Omega^{\kr}_{X_\bullet}) \ar@{=}[d]^\sim\\
 & H^2_{\cdh}(X,a_{\cdh}^\ast A^{\kr}) 
}
\end{equation}
whose rational version we combine with (\ref{commutative diagram c_1 a}). 
Thus to show the commutativity of (\ref{commutative diagram c_1}), it remains to show that the non-logarithmic version of the isomorphism from Theorem \ref{thm:comparison} factors as 
$$
H^2_{\rig}(X/K)  \rightarrow H^2_{\cris}(X/W(k)) \otimes_{\Z}\Q  \rightarrow  H^2_{\cris}(X_\bullet/W(k)) \otimes_{\Z}\Q \xleftarrow{\sim} H^2_{\et}(X_\bullet, W\Omega^{\kr}_{X_\bullet})\otimes_{\Z}\Q \cong H^2_{\cdh}(X,a_{\cdh}^\ast A^{\kr}) \otimes_{\Z}\Q.
$$
But the third morphism is the isomorphism from \cite[Cor.\,7.7]{N_2012} and the last morphism is the isomorphism induced by construction of our cohomology, both of which have been used in Theorem \ref{thm:comparison}. 
Hence we have to see that the isomorphism $H^2_{\rig}(X/K) \xrightarrow{\sim} H^2_{\cris}(X_\bullet/W(k)) \otimes_{\Z}\Q$ from \cite[Cor.\,11.7]{N_2012} in the case that (simplicial) horizontal divisor is empty factors through the canonical morphism $H^2_{\rig}(X/K) \rightarrow H^2_{\cris}(X/W(k)) \otimes_{\Z}\Q$. 
By functoriality there is a commutative diagram
$$
\xymatrix{
H^2_{\rig}(X/K)\cong H^2_{\mathrm{conv}}(X/K)  \ar[r] \ar[d]^\sim & H^2_{\cris}(X/W(k)) \otimes_{\Z}\Q \ar[d]\\
H^2_{\rig}(X_\bullet/K) \cong H^2_{\mathrm{conv}}(X_\bullet/K)  \ar[r]^\sim & H^2_{\cris}(X_\bullet/W(k)) \otimes_{\Z}\Q
}
$$
where rigid and convergent cohomology are identified by construction since $X$ and $X_\bullet$ are proper \cite[Thm.\,0.6.7]{O_1990}. 
That the lower horizontal map is an isomorphism follows from \cite[Thm.\,0.7.7]{O_1990}. 
But the isomorphism from \cite[Cor.\,11.7]{N_2012}, which was originally constructed in \cite[Cor.\,2.3.9, Thm.\,3.1.1]{S_2002}, is precisely the log version of that one. 
Since the log structure in the case at hand is trivial (as the horizontal divisor is empty), these morphisms coincide and this shows the desired factorisation. 
This finishes the proof of the compatibility of Chern classes in the proper case.

Next we show how to deduce the general case from the proper case.
Let $X\in\Var_k$ and and $\zeta\in \mathrm{Pic}(X)= H^1_{\Zar}(X,\co^\times_X)$.  
According to \cite[Prop.\,3.21]{P_2003} one may find a compactification 
$j:X\rightarrow \overline{X}$ such that
$\overline{X}\backslash X$ is a divisor and $\zeta$ extends to $\overline{X}$ 
in the sense that there is $\overline{\zeta}\in\mathrm{Pic}(\overline{X})$
with $j^\ast(\overline{\zeta})=\zeta$. 
For this compactification consider the diagram
$$
\xymatrix{
& \mathrm{Pic}(\overline{X})\ar[dr]^{c_1^A} \ar[ddl]^{c_1^{\rig}} \ar[rr] && \mathrm{Pic}(X) \ar[dr]^{c_1^A} \ar[ddl]^{c_1^{\rig}} & \\
&& H^2_{\cdh}(\overline{X},a_{\cdh}^\ast A^{\kr}) \otimes_{\Z}\Q \ar[rr]&& H^2_{\cdh}(X,a_{\cdh}^\ast A^{\kr}) \otimes_{\Z}\Q \\
H^2_{\rig}(\overline{X}/K) \ar[rr] \ar@{=}[urr]^\sim && H^2_{\rig}(X/K) \ar@{=}[urr]^\sim &&
}
$$
where all three squares commute by functoriality.
By what we have seen in the first part of the proof, 
the left triangle commutes. 
As $\zeta$ has a preimage $\overline{\zeta}$ in $\mathrm{Pic}(\overline{X})$, 
a diagram chase shows that
$$
c_1^A(\zeta) =c_1^A(j^\ast \overline{\zeta})= 
j^\ast c_1^A(\overline{\zeta})    = 
j^\ast c_1^{\rig}(\overline{\zeta}) =
c_1^{\rig}(\zeta).
$$
As we can find such a compactification $\overline{X}$ for each $\zeta\in\mathrm{Pic}(X)$, we finally obtain a commutative diagram of the form (\ref{commutative diagram c_1}) as desired.
\end{proof}

%%%%%%%%%%%%%%%%%%
%%%
\section{The case without resolution of singularities}\label{sec: without resolution}
%%%
%%%%%%%%%%%%%%%%%%

In this section, we discuss integral $p$-adic cohomology theories without assuming 
resolution of singularities. In this case, instead of resolutions to arrive at a situation of normal crossing pairs, one could envision to use de Jong's alteration theorem. 
In fact, for a $k$-variety $X$ and a  
compactification $X \subseteq \overline{X}$, we can form 
by \cite[Prop.\,9.2]{N_2012} a commutative diagram 

\begin{equation}\label{eq:2-1}
\xymatrix{
X_{\bullet} \ar[r]^{\pi} \ar[d]^{\cap} & X \ar[d]^{\cap} \\ 
\overline{X}_{\bullet} \ar[r]^{\overline{\pi}} & \overline{X},   
}
\end{equation}
where the square is Cartesian, $\pi$ is a split proper hypercovering, $\overline{\pi}$ is proper, and 
$(X_{\bullet}, \overline{X}_{\bullet})$ is a simplicial normal crossing pair which gives rise to a simplicial fine log scheme.
With $D_{\bullet} := \overline{X}_{\bullet} \backslash X_{\bullet}$, its crystalline cohomology $H^i_{\rm cris}((X_\bullet,\overline{X}_{\bullet})/W(k)) \cong 
H^i_{\et}(\overline{X}_{\bullet}, W\Omega^{\kr}_{\overline{X}_\bullet/k}(\log D_{\bullet}))$ is 
a finitely generated $W(k)$-module. 
Moreover, since we have the isomorphism 
$$ 
H^i_{\rm rig}(X/K) \cong H^i_{\rm cris}((X_\bullet,\overline{X}_{\bullet})/W(k)) 
\otimes_{\Z} \Q$$
by \cite[(1.0.17), Cor.\,11.7 1)]{N_2012}, the cohomology 
$H^i_{\rm cris}((X_\bullet,\overline{X}_{\bullet})/W(k))$ is a candidate for a good integral 
$p$-adic cohomology theory. Hence one is tempted to ask whether this is independent 
of the choice of the diagram \eqref{eq:2-1}. In particular, when $(X,\overline{X})$ is also 
a normal crossing pair with $D := \overline{X} \backslash X$, 
we ask whether the natural map 
\begin{equation}\label{eq:2-2}
\overline{\pi}^*: H^i_{\rm cris}((X,\overline{X})/W(k)) \to 
H^i_{\rm cris}((X_\bullet,\overline{X}_{\bullet})/W(k)) 
\end{equation}
is an isomorphism. However, we have the following easy 
counterexample:

\begin{example}
Let $X$ be an elliptic curve over $k = {\bold F}_p$ and let $F: X \to X$ be the absolute 
Frobenius morphism. If we denote the \v{C}ech hypercovering associated to $F$ by 
$\pi: X'_{\bullet} \to X$ and if we set $X_{\bullet} := (X'_{\bullet})_{\rm red}$, 
we see that each $X_i \,(i \in \N)$ is equal to $X$. Moreover, 
$X_{\bullet} \to X$ forms a proper hypercovering: Indeed, for each $i \in \N$, 
we have a commutative diagram 

\begin{equation}\label{eq:cosk}
\xymatrix{
X_{i+1} \ar[r] \ar[d] & X'_{i+1} \ar[d] \\ 
{\rm cosk}_i(X_{\bullet \leqslant i}) \ar[r] & 
{\rm cosk}_i(X'_{\bullet \leqslant i})}
\end{equation}
such that the horizontal arrows are homeomorphic closed immersions 
(because so are the morphisms $X_j \to X'_j \, (j \in \N)$) 
and the right vertical arrow is an isomorphism (because $X'_{\bullet} \to X$ 
is a \v{C}ech hypercovering). So the left vertical arrow is also a 
 homeomorphic closed immersion and so it is proper and surjective. 
Thus $X_{\bullet} \to X$ is a proper hypercovering, as required. 
Also, it is obviously split.  
If we set $\overline{X}_{\bullet} := X_{\bullet}$ and $\overline{X}:=X$, the pair 
$(X_{\bullet}, \overline{X}_{\bullet})$ over $(X,\overline{X})$ forms a 
commutative diagram \eqref{eq:2-1} satisfying the conditions above. 

In this case, we have maps 
$$ H^1_{\rm cris}(X/W(k)) \xrightarrow{\pi^*} 
H^1_{\rm cris}(X_{\bullet}/W(k)) \xrightarrow{\cong} H^1_{\rm cris}(X/W(k)), $$
where the second map is the edge map of the spectral sequence 
$$ 
E^{ij}_1 = H^j_{\rm cris}(X_i/W(k)) \Rightarrow H^{i+j}_{\rm cris}((X_{\bullet}/W(k)), 
$$
such that the composite is equal to $F^*: H^1_{\rm cris}(X/W(k)) \to 
H^1_{\rm cris}(X/W(k))$. Since $H^1_{\rm cris}(X/W(k))$ has non-trivial  
positive slope part, $F^*$ is not an isomorphism and so $\pi^*$ is not 
an isomorphism either. 
\end{example}

Thus we should assume some generic \'etaleness assumption on the hypercovering 
$\pi$.  (Note that, for a given $(X,\overline{X})$, there exists a diagram \eqref{eq:2-1} such that the split proper hypercovering $\pi$ is additionally generically \'etale by Lemma \ref{lem:domhypercover2'} and Corollary \ref{cor:domc}. 
So we would like to consider the following question. 

\begin{question}\label{q:2-1}
Let $X$ be a $k$-variety and let $X \subseteq \overline{X}$ be its 
compactification. Suppose we are given a commutative diagram \eqref{eq:2-1} such that 
the square is Cartesian, $\pi$ is a split proper generically \'etale hypercovering, $\overline{\pi}$ is proper, 
and $(X_{\bullet}, \overline{X}_{\bullet})$ is a simplicial normal crossing pair.
Then, is the crystalline cohomology 
 $H^i_{\rm cris}((X_\bullet,\overline{X}_{\bullet})/W(k) )$ 
independent of the choice of the diagram \eqref{eq:2-1}? 
In particular, when 
$(X,\overline{X})$ is also 
a normal crossing pair,
is the natural map \eqref{eq:2-2}
an isomorphism? 
\end{question}

%%%%%%
\subsection{Independence of $H^0$}
%%%%%%
In this subsection, we give an affirmative answer to Question \ref{q:2-1} for the zeroth  
crystalline cohomology group. It is certainly well-known but we write down a proof for the completeness of the paper.

Let $X$ be a $k$-variety and let $X \subseteq \overline{X}$ be its 
compactification. Suppose we are given a commutative diagram \eqref{eq:2-1} such that 
the square is Cartesian, $\pi$ is a split proper hypercovering, $\overline{\pi}$ is proper, and 
$(X_{\bullet}, \overline{X}_{\bullet})$ is a simplicial normal crossing pair.  
(We do not assume that $\pi$ is generically \'etale. 
This suffices for us because we consider the zeroth cohomology group here.)

Then we have the following:

\begin{theorem}\label{thm:H^0}
	$H^0_{\rm cris}((X_{\bullet}, \overline{X}_{\bullet})/W(k))$ is independent of the choice 
	of the pair $(X_{\bullet}, \overline{X}_{\bullet})$ as above. 
\end{theorem} 

\begin{lemma}\label{lem:H0}
\begin{enumerate}
	\item Let $Y \hookrightarrow \overline{Y}$ be an open immersion of smooth varieties over $k$ such that $\overline{Y} \setminus Y$ is a simple normal crossing divisor (but $\overline{Y}$ not necessarily proper) and let 
	$(Y,\overline{Y})$ be the resulting log scheme. Also, let $U$ be an open dense subscheme of $Y$. Then the induced morphism $H^0_{\cris}((Y,\overline{Y})/W(k)) \rightarrow H^0_{\cris}(U/W(k))$ is an isomorphism. 
	\item Let $U$ be a connected smooth variety over $k$ with $U(k) \not= \emptyset$. Then $H^0_{\cris}(U/W(k)) = W(k)$.  
\end{enumerate}
\end{lemma}

\begin{proof}
First we prove (i). 
Set $D := \overline{Y} \setminus Y$. 
By Zariski descent for crystalline cohomology, we may shrink $\overline{Y}$  
and so we may suppose the following: $\overline{Y}$ is affine, connected and liftable to a smooth $p$-adic formal scheme $\overline{\mathfrak{Y}}$ 
over $\Spf W(k)$, and $D$ is liftable to a relative simple normal crossing divisor $\mathfrak{D}$ of $\overline{\mathfrak{Y}}$. 
Then $U$ admits a lift $\mathfrak{U}$ which is an open formal subscheme of $\overline{\mathfrak{Y}}$. 
Also, by the base change property of crystalline cohomology, we may replace $k$ by a finite extension and so we may assume that 
$U$ admits a $k$-rational point. Then $\mathfrak{U}$ admits a $\Spf W(k)$-valued point $x$. 
Let $t_1,\dots,t_d$ be local coordinates of $\mathfrak{U}$ at $x$. 
Then the map $H^0_{\cris}((Y,\overline{Y})/W(k)) \rightarrow H^0_{\cris}(U/W(k))$ is equal to the map 
\begin{equation}\label{eq:H0}
	{\rm Ker}(\Gamma(\overline{\mathfrak{Y}},\mathcal{O}_{\overline{\mathfrak{Y}}}) \to 
	\Gamma(\overline{\mathfrak{Y}},\Omega^1_{\overline{\mathfrak{Y}}}(\log \mathfrak{D}))) 
	\longrightarrow 
	{\rm Ker}(\Gamma(\mathfrak{U},\mathcal{O}_{\mathfrak{U}}) \to 
	\Gamma(\mathfrak{U},\Omega^1_{\mathfrak{U}})). 
\end{equation}
Both sides of the map contain $W(k)$. Also, there exist injective maps 
\begin{align*}
& {\rm Ker}(\Gamma(\overline{\mathfrak{Y}},\mathcal{O}_{\overline{\mathfrak{Y}}}) \to 
\Gamma(\overline{\mathfrak{Y}},\Omega^1_{\overline{\mathfrak{Y}}}(\log \mathfrak{D}))) 
= {\rm Ker}(\Gamma(\overline{\mathfrak{Y}},\mathcal{O}_{\overline{\mathfrak{Y}}}) \to 
\Gamma(\overline{\mathfrak{Y}},\Omega^1_{\overline{\mathfrak{Y}}})) \\ 
& \phantom{{\rm Ker}(\Gamma(\overline{\mathfrak{Y}},\mathcal{O}_{\overline{\mathfrak{Y}}}) \to 
	\Gamma(\overline{\mathfrak{Y}},\Omega^1_{\overline{\mathfrak{Y}}}(\log \mathfrak{D})))} 
\hookrightarrow 
{\rm Ker}\left(W(k)[[t_1,\dots,t_d]] \to \bigoplus_{i=1}^d W(k)[[t_1,\dots,t_d]]dt_i \right) = W(k), \\
& {\rm Ker}(\Gamma(\mathfrak{U},\mathcal{O}_{\mathfrak{U}}) \to 
\Gamma(\mathfrak{U},\Omega^1_{\mathfrak{U}})) 
\hookrightarrow 
{\rm Ker}\left(W(k)[[t_1,\dots,t_d]] \to \bigoplus_{i=1}^d W(k)[[t_1,\dots,t_d]]dt_i\right) = W(k) 
\end{align*}
induced by injective maps 
\[ \Gamma(\overline{\mathfrak{Y}},\mathcal{O}_{\overline{\mathfrak{Y}}}) 
\hookrightarrow \Gamma(\mathfrak{U},\mathcal{O}_{\mathfrak{U}})  
\hookrightarrow \hat{\mathcal{O}}_{\mathfrak{U},x} = W(k)[[t_1,\dots,t_d]]. 
\]
Thus both sides of the map \eqref{eq:H0} are equal to $W(k)$ and so the assertion (i) is proved. 

Next we prove (ii). Take a finite Zariski covering $U = \bigcup_{i \in I} U_i$ such that each 
$U_i$ is affine, connected and liftable to a smooth $p$-adic formal scheme. Also, since we may replace $k$ by a finite extension to prove the assertion, we may assume that 
$V := \bigcap_{i \in I}U_i$ has a $k$-rational point. 
 (Note that $U$ is geometrically connected by the assumption $U(k) \not= \emptyset$ \cite[\href{https://stacks.math.columbia.edu/tag/056R}{Lemma 056R}]{StacksProject}.)
Then 
\begin{align*}
H^0_{\cris}(U/W(k)) & = 
{\rm Ker}\left(\bigoplus_{i \in I}H^0_{\cris}(U_i/W(k)) \to \bigoplus_{i,j \in I} H^0_{\cris}(U_i \cap U_j/W(k)) \right)	\\
& = 
{\rm Ker}\left(\bigoplus_{i \in I}H^0_{\cris}(V/W(k)) \to \bigoplus_{i,j \in I} H^0_{\cris}(V/W(k)) \right) = H^0_{\cris}(V/W(k)) = W(k), 
\end{align*}
where the second equality follows from (i), the third equality follows from the connectedness of $U$ and the last equality follows from the computation in the proof of (i). 
\end{proof}

\begin{proof}[Proof of Theorem \ref{thm:H^0}]
Take an open dense subscheme $U \subseteq X$ such that $U$ is smooth over $k$, and set $U_{\bullet} := X_{\bullet} \times_X U$. Then the induced map $U_{\bullet} \to U$ is a proper hypercovering. To prove the theorem, we may replace $k$ by a finite extension and so we may assume that each connected component of $U, U_0, U_1$ has a $k$-rational point. 
	
By Lemma \ref{lem:H0}(i), we have isomorphisms 
\[ 
H^0_{\cris}((X_{\bullet}, \overline{X}_{\bullet})/W(k)) \cong H^0_{\cris}(U_{\bullet}/W(k)) \cong {\rm Ker}(H^0_{\cris}(U_0/W(k)) \to H^0_{\cris}(U_1/W(k))). 
\]
So, to prove the theorem, it suffices to prove the exactness of the sequence 
\[
0 \rightarrow H^0_{\cris}(U/W(k)) \rightarrow H^0_{\cris}(U_0/W(k)) \rightarrow H^0_{\cris}(U_1/W(k)). 
\]
By Lemma \ref{lem:H0}(ii), the above sequence is equal to the sequence 
\[
0 \rightarrow W(k)^{\oplus \pi_0(U)} \rightarrow W(k)^{\oplus \pi_0(U_0)} \rightarrow W(k)^{\oplus \pi_0(U_1)}, 
\]
hence the sequence 
\[
0 \rightarrow H^0_{\et}(U,\Z_p) \otimes_{\Z_p} W(k) \to  H^0_{\et}(U_0,\Z_p) \otimes_{\Z_p} W(k) \to 
 H^0_{\et}(U_1,\Z_p) \otimes_{\Z_p} W(k). 
\]
The exactness of this sequence follows from the cohomological descent of \'etale cohomology for proper hypercoverings. 
\end{proof}

%%%%%%
\subsection{Independence of $H^1$}
%%%%%%

In this subsection, we prove our first main result in this section, 
which gives an affirmative answer to Question \ref{q:2-1} for the first 
crystalline cohomology group. 

We fix our setting. 
Let $X$ be a $k$-variety and let $X \subseteq \overline{X}$ be its 
compactification. Suppose we are given a commutative diagram \eqref{eq:2-1} such that 
the square is Cartesian, $\pi$ is a split proper hypercovering, $\overline{\pi}$ is proper, and 
$(X_{\bullet}, \overline{X}_{\bullet})$ is a simplicial normal crossing pair 
with $D_{\bullet} := \overline{X}_{\bullet} \backslash X_{\bullet}$. 
Suppose moreover that $X_0 \to X$ is generically \'etale. 
(This assumption is slightly weaker than that in Question \ref{q:2-1}, but it is the natural one 
because we only consider the first cohomology.) 

Then we have the following:

\begin{theorem}\label{thm:H^1}
$H^1_{\rm cris}((X_{\bullet}, \overline{X}_{\bullet})/W(k))$ is independent of the choice 
of the pair $(X_{\bullet}, \overline{X}_{\bullet})$ as above. 
\end{theorem} 

When $p \geqslant 3$, this is already proven by  
Andreatta--Barbieri-Viale \cite[Cor.\,7.5.4]{ABV_2005}. 
We will give a different proof which works also in the case $p=2$. 
Moreover, in Appendix \ref{appendixB}, we will give a description of the cohomology $H^1_{\rm cris}((X_{\bullet}, \overline{X}_{\bullet})/W(k))$ in the style of the previous section,
and discuss the functoriality of it.\\

First we explain briefly the proof of Andreatta--Barbieri-Viale. 
We take another diagram 
\begin{equation*}
\xymatrix{
X'_{\bullet} \ar[r] \ar[d]^{\cap} & X \ar[d]^{\cap} \\ 
\overline{X}'_{\bullet} \ar[r] & \overline{X}  
}
\end{equation*}
with $D'_{\bullet} := \overline{X}'_{\bullet} \backslash X'_{\bullet}$
satisfying the same condition as $(X_{\bullet}, \overline{X}_{\bullet})$. 
To prove Theorem \ref{thm:H^1}, it suffices to prove that, when 
we have a morphism $(X_{\bullet}, \overline{X}_{\bullet}) \to 
(X'_{\bullet}, \overline{X}'_{\bullet})$, the induced morphism of crystalline cohomologies 
\begin{equation}\label{eq:indep?}
H^1_{\rm cris}((X_{\bullet}, \overline{X}_{\bullet})/W(k)) \to 
H^1_{\rm cris}((X'_{\bullet}, \overline{X}'_{\bullet})/W(k))
\end{equation} 
is an isomorphism,  by Proposition \ref{prop:domhypercover'} \ref{prop:domhypercover_iii'} and Corollary \ref{cor:domc}. 

In this situation, we have the morphism between Picard $1$-motives  
\begin{equation}\label{eq:1-motive}
{\rm Pic}^+(X_{\bullet}, \overline{X}_{\bullet}) \to 
 {\rm Pic}^+(X'_{\bullet}, \overline{X}'_{\bullet})
\end{equation}
associated to $(X_{\bullet}, \overline{X}_{\bullet})$ and 
$(X'_{\bullet}, \overline{X}'_{\bullet})$. Then Theorem \ref{thm:H^1} 
in the case $p \geqslant 3$ is a consequence of the following two theorems.

\begin{theorem}[{\cite[Prop.\,A.2.8]{ABV_2005}}] \label{thm:ABV1}
The morphism \eqref{eq:1-motive} is an isomorphism \textup{(}for any $p)$. 
\end{theorem}

\begin{theorem}[{\cite[Thm.\,B']{ABV_2005}}]\label{thm:ABV2}
If $p \geqslant 3$, there exists a functorial isomorphism 
${\bold T}_{\rm cris}({\rm Pic}^+(X_{\bullet}, \overline{X}_{\bullet})) \xrightarrow{\cong}
H^1_{\rm cris}((X_{\bullet}, \overline{X}_{\bullet})/W(k))$ which is compatible with 
weight filtration and Frobenius action, where ${\bold T}_{\rm cris}$ 
is the covariant Dieudonn\'e functor \textup{(}the contravariant Dieudonn\'e functor of 
the $p$-divisible group associated to the dual\textup{)}. 
\end{theorem} 

In our proof of Theorem \ref{thm:H^1}, we will use Theorem \ref{thm:ABV1} but 
we will not use Theorem \ref{thm:ABV2}. Thus our proof is different from that in 
\cite{ABV_2005} even when $p \geqslant 3$.

\begin{remark}\label{rem: ABV2}
In Appendix \ref{appendixA} we give a proof of Theorem \ref{thm:ABV2} for any prime $p$, based on the results in this section.
\end{remark}

First we recall the definition of the
Picard $1$-motive ${\rm Pic}^+(X_{\bullet}, \overline{X}_{\bullet})$. 
Recall that a $1$-motive over $k$ is a complex 
${\mathbb{M}} := [{\mathbb{X}} \xrightarrow{u} {\mathbb{G}}]$ sitting in degrees $-1, 0$ consisting of 
a finite free ${\bold Z}$-module $\mathbb{X}$ endowed with continuous ${\rm Gal}(\overline{k}/k)$-action 
(regarded as a locally constant group scheme over $k$), 
a semiabelian variety ${\mathbb{G}}$ over $k$ and a homomorphism of $k$-group schemes $u: {\mathbb{X}} \to {\mathbb{G}}$. 

Let $(X_{\bullet}, \overline{X}_{\bullet})$ be a simplicial proper normal crossing 
pair with $D_{\bullet} = \overline{X}_{\bullet} \backslash X_{\bullet}$. 
Then the Picard fppf-sheaf 
$$ T \mapsto {\rm Pic}_{\overline{X}_{\bullet}}(T) := H^0_{\rm fppf}(T,
R^1f_{T,*}{\mathbb{G}}_{m,\overline{X}_{\bullet} \times_k T}), $$
where $f_T: \overline{X}_{\bullet,T} := \overline{X}_{\bullet} \times_{\Spec k} T \to T$ is the projection, 
is representable by a group scheme ${\rm Pic}_{\overline{X}_{\bullet}}$ over $k$ 
 \cite[Lem.\,4.1.2]{BVS_2001} and its reduced identity component 
${\rm Pic}^{0,{\rm red}}_{\overline{X}_{\bullet}}$ is a semiabelian variety over $k$  (see Remark \ref{rem:ABV} below).
\begin{remark}[{cf. \cite[p.54]{BVS_2001}}]\label{rem:BVS}
For a $k$-variety $T$ the Leray spectral sequence 
\[ E_2^{i,j} = H^i_{\rm fppf}(T, R^jf_{T*}\mathbb{G}_m) \,\Longrightarrow\, H^{i+j}_{\rm fppf}(\overline{X}_{\bullet,T},\mathbb{G}_m) \]
induces the exact sequence 
\begin{align*}
0 & \to H^1_{\rm fppf}(T,f_{T*}\mathbb{G}_m) \to H^1_{\rm fppf}(\overline{X}_{\bullet,T}, \mathbb{G}_m) 
\to H^0_{\rm fppf}(T, R^1f_{T*}\mathbb{G}_m) \to H^2_{\rm fppf}(T, f_{T*}\mathbb{G}_m). 
\end{align*}
Suppose $k = \overline{k}$. Then we have an isomorphism 
$f_*\mathcal{O}_{\overline{X}_{\bullet}} \cong k^N$ for some $N \geqslant 1$ and then 
$f_{T*}\mathbb{G}_m \cong (\mathcal{O}_T^{\times})^N$. Hence, if $T = \Spec R$ for some Artinian local $k$-algebra $R$, $H^i_{\rm fppf}(T, f_{T*}\mathbb{G}_m) = 0$ for $i=1,2$ and so 
\[ {\rm Pic}_{\overline{X}_{\bullet}}(T) \cong H^1_{\rm fppf}(\overline{X}_{\bullet,T}, \mathbb{G}_m) 
\cong H^1_{\et}(\overline{X}_{\bullet,T}, \mathbb{G}_m). \]
\end{remark}

Next, let ${\rm Div}_{D_{\bullet}}(\overline{X}_{\bullet})$ be the group of Weil divisors $E$ on  $\overline{X}_0 \times_{\Spec k} \Spec \overline{k}$ supported on $D_0 \times_{\Spec k} \Spec \overline{k}$ with $d_1^*E = d_0^*E$, where $d_i: \overline{X}_1 \times_{\Spec k} \Spec \overline{k} \to \overline{X}_0 \times_{\Spec k} \Spec \overline{k} \,(i=0,1)$ are projections. This is a finite free $\Z$-module  
endowed with continuous ${\rm Gal}(\overline{k}/k)$-action, and there is the canonical 
morphism 
${\rm Div}_{D_{\bullet}}(\overline{X}_{\bullet}) \to {\rm Pic}_{\overline{X}_{\bullet}}$. 
If we define ${\rm Div}^0_{D_{\bullet}}(\overline{X}_{\bullet})$ to 
be the kernel of the composition 
$$ {\rm Div}_{D_{\bullet}}(\overline{X}_{\bullet}) \to {\rm Pic}_{\overline{X}_{\bullet}} 
\twoheadrightarrow \pi_0({\rm Pic}_{\overline{X}_{\bullet}} ) =: 
{\rm NS}(\overline{X}_{\bullet}), $$
we have the induced morphism 
$u: {\rm Div}^0_{D_{\bullet}}(\overline{X}_{\bullet}) \to 
{\rm Pic}^{0,{\rm red}}_{\overline{X}_{\bullet}}.$ 
Then the Picard $1$-motive ${\rm Pic}^+(X_{\bullet}, \overline{X}_{\bullet})$ is defined by 
$$ {\rm Pic}^+(X_{\bullet}, \overline{X}_{\bullet}) := [ 
{\rm Div}^0_{D_{\bullet}}(\overline{X}_{\bullet}) \to {\rm Pic}^{0,{\rm red}}_{\overline{X}_{\bullet}}].  $$
By Theorem \ref{thm:ABV1}, we obtain the following corollary.

\begin{corollary}\label{cor:ABV}
The natural morphisms ${\rm Div}^0_{D_{\bullet}}(\overline{X}_{\bullet}) \to 
{\rm Div}^0_{D'_{\bullet}}(\overline{X}'_{\bullet})$, 
${\rm Pic}^{0,{\rm red}}_{\overline{X}_{\bullet}} \to 
{\rm Pic}^{0,{\rm red}}_{\overline{X}'_{\bullet}}$ induced by 
\eqref{eq:1-motive} are isomorphisms \textup{(}for any $p)$. 
\end{corollary}

\begin{remark}\label{rem:ABV}
We have by a standard argument the exact sequence 
\begin{equation}\label{eq:rem1}
0 \to {\mathbb{T}} \to {\rm Pic}_{\overline{X}_{\bullet}} \to {\mathbb{K}} 
\to {\mathbb{T}}' 
\end{equation}
with
\begin{align*}
& {\mathbb{T}} := \frac{{\rm Ker}\left((f_1)_*{\mathbb{G}}_{m,\overline{X}_1} \to 
(f_2)_*{\mathbb{G}}_{m,\overline{X}_2}\right)}{{\rm Im}\left((f_0)_*{\mathbb{G}}_{m,\overline{X}_0} \to 
(f_1)_*{\mathbb{G}}_{m,\overline{X}_1}\right)}, 
\quad 
{\mathbb{T}}' := \frac{{\rm Ker}\left((f_2)_*{\mathbb{G}}_{m,\overline{X}_2} \to 
(f_3)_*{\mathbb{G}}_{m,\overline{X}_3}\right)}{{\rm Im}\left((f_1)_*{\mathbb{G}}_{m,\overline{X}_1} \to 
(f_2)_*{\mathbb{G}}_{m,\overline{X}_2}\right)}, \\ 
& {\mathbb{K}} := {\rm Ker}({\rm Pic}_{\overline{X}_0} \to {\rm Pic}_{\overline{X}_1}),  
\end{align*} 
where $f_i\colon\overline{X}_i \to {\rm Spec}\,k \, (i=0,1,2)$ are structure morphisms. 
The reduced identity components ${\mathbb{T}}^{0, {\rm red}},  {{\mathbb{T}}'}^{0,{\rm red}}$ 
of ${\mathbb{T}},  {\mathbb{T}}'$ are tori, 
and the reduced identity component 
$$ {\mathbb{K}}^{0,{\rm red}} =
 {\rm Ker}^{0,{\rm red}}
({\rm Pic}_{\overline{X}_0} \to {\rm Pic}_{\overline{X}_1}) = 
 {\rm Ker}^{0,{\rm red}}
({\rm Pic}^{0,{\rm red}}_{\overline{X}_0} \to {\rm Pic}^{0,{\rm red}}_{\overline{X}_1})$$ of 
${\mathbb{K}}$ is an abelian variety. Therefore, the composite 
${\mathbb{K}}^{0,{\rm red}} \to {\mathbb{K}} \to  {\mathbb{T}}'$ is zero. 
So, if we set ${\rm Pic}'_{\overline{X}_{\bullet}} := 
{\rm Pic}_{\overline{X}_{\bullet}} \times_{{\mathbb{K}}} {\mathbb{K}}^{0,{\rm red}}$, 
we obtain an exact sequence 
\begin{equation}\label{eq:rem2}
0 \to {\mathbb{T}} \to {\rm Pic}'_{\overline{X}_{\bullet}} \to {\mathbb{K}}^{0,{\rm red}} 
\to 0 
\end{equation}
by pulling back the sequence \eqref{eq:rem1} by ${\mathbb{K}}^{0,{\rm red}} \to {\mathbb{K}}$.  
If we set ${\mathbb{F}} := {\mathbb{T}}/{\mathbb{T}}^{0, {\rm red}}$, 
it is a finite group scheme and if we denote the quotient 
${\rm Pic}'_{\overline{X}_{\bullet}} /{\mathbb{T}}^{0, {\rm red}}$ by 
$\widetilde{\mathbb{K}}$, we obtain exact sequences 
\begin{equation}\label{eq:rem3}
0 \to {\mathbb{T}}^{0, {\rm red}} \to {\rm Pic}'_{\overline{X}_{\bullet}} \to 
\widetilde{\mathbb{K}}
\to 0, \quad 0 \to {\mathbb{F}} \to \widetilde{\mathbb{K}} \to 
 {\mathbb{K}}^{0,{\rm red}} \to 0. 
\end{equation}
From the second sequence of \eqref{eq:rem3}, we see that 
$\widetilde{\mathbb{K}}$ is proper and so its reduced identity component $\widetilde{\mathbb{K}}^{0, {\rm red}}$ is an abelian variety. Thus, by 
pulling back the first sequence of \eqref{eq:rem3} by 
$\widetilde{\mathbb{K}}^{0, {\rm red}} \to \widetilde{\mathbb{K}}$, 
we obtain the exact sequence 
\begin{equation}\label{eq:rem4}
0 \to {\mathbb{T}}^{0, {\rm red}} \to {\rm Pic}^{0,{\rm red}}_{\overline{X}_{\bullet}} \to 
\widetilde{\mathbb{K}}^{0, {\rm red}} 
\to 0, 
\end{equation}
which describes ${\rm Pic}^{0,{\rm red}}_{\overline{X}_{\bullet}}$ 
as an extension of an abelian variety by a torus. 

By construction, there exists an exact sequence 
\[ 0 \to {\mathbb{F}}' \to \widetilde{\mathbb{K}}^{0, {\rm red}} \to 
{\mathbb{K}}^{0,{\rm red}} \to 0 \] 
for some closed subgroup scheme ${\mathbb{F}}'$ of ${\mathbb{F}}$. 
In \cite[Rem.\,5.1.2]{ABV_2005}, Andreatta--Barbieri-Viale claim 
that ${\mathbb{K}}^{0,{\rm red}}$ is the abelian part of 
${\rm Pic}^{0,{\rm red}}_{\overline{X}_{\bullet}}$ (hence ${\mathbb{F}}' = 0$), 
but it seems that there is no proof for it. Because they use this description of the abelian part 
in the proof of Theorem \ref{thm:ABV2}, some supplementary argument would be 
required. In this article, we do not use Theorem \ref{thm:ABV2}. 
\end{remark}

\begin{remark}
Let the notations be as above and assume that $k$ is algebraically closed. 
Then we have ${\rm NS}(\overline{X}_{\bullet}) := \pi_0({\rm Pic}_{\overline{X}_{\bullet}})
= {\rm Pic}_{\overline{X}_{\bullet}}(k)/ {\rm Pic}^{0,{\rm red}}_{\overline{X}_{\bullet}}(k)$. 
By the commutative diagrams with exact horizontal lines 
\[ 
\xymatrix{
0 \ar[r] & {\mathbb{T}}^{0,{\rm red}}(k) \ar[r] \ar@{^{(}->}[d] & 
{\rm Pic}^{0,{\rm red}}_{\overline{X}_{\bullet}}(k) \ar[r] \ar@{^{(}->}[d] & 
\widetilde{\mathbb{K}}^{0,{\rm red}}(k) \ar[r] \ar[d] & 0 \\ 
0 \ar[r] & {\mathbb{T}}(k) \ar[r] & 
{\rm Pic}_{\overline{X}_{\bullet}}(k) \ar[r] & 
 \mathbb{K}(k), 
}\] 
\[
\xymatrix{
0 \ar[r] & {\mathbb{F}}'(k) \ar[r] \ar[d] & 
\widetilde{\mathbb{K}}^{0,{\rm red}}(k) \ar[r] \ar[d] & 
{\mathbb{K}}^{0,{\rm red}}(k) \ar[r] \ar@{^{(}->}[d] & 0 \\   
0 \ar[r] &  0 \ar[r] & 
 \mathbb{K}(k) \ar[r] & 
\mathbb{K}(k)\ar[r] & 0,}  
\]
\[
\xymatrix{
0 \ar[r] & {\mathbb{K}}^{0,{\rm red}}(k) \ar[r] \ar@{^{(}->}[d] & 
{\mathbb{K}}'(k) \ar[r] \ar@{^{(}->}[d] & 
{\mathbb{K}}'(k)/{\mathbb{K}}^{0,{\rm red}}(k) \ar[r] \ar[d] 
& 0 \\ 
0 \ar[r] & {\mathbb{K}}(k) \ar[r] & 
{\mathbb{K}}(k) \ar[r] & 0 \ar[r] & 0 
}  
\]
(where ${\mathbb{K}}' := {\rm Ker}
({\rm Pic}^{0,{\rm red}}_{\overline{X}_0} \to {\rm Pic}^{0,{\rm red}}_{\overline{X}_1})$), 
\[
\xymatrix{
0 \ar[r] & {\mathbb{K}'(k)} \ar[r] \ar@{^{(}->}[d] & 
{\rm Pic}^{0,{\rm red}}_{\overline{X}_0}(k) \ar[r] \ar@{^{(}->}[d] & 
{\rm Pic}^{0,{\rm red}}_{\overline{X}_1}(k) \ar@{^{(}->}[d] \\ 
0 \ar[r] & {\mathbb{K}}(k) \ar[r] & 
{\rm Pic}_{\overline{X}_0}(k) \ar[r] & 
{\rm Pic}_{\overline{X}_1}(k),
}  
\]
we obtain the exact sequences 

\begin{align*}
& {\mathbb{T}}(k)/ {\mathbb{T}}^{0,{\rm red}}(k) \to 
{\rm Pic}_{\overline{X}_{\bullet}}(k)/
{\rm Pic}^{0,{\rm red}}_{\overline{X}_{\bullet}}(k) \to 
 \mathbb{K}(k)/{\rm Im}(\widetilde{\mathbb{K}}^{0,{\rm red}}(k) \to \mathbb{K}(k)), \\ 
&  0 \to \mathbb{K}(k)/{\rm Im}(\widetilde{\mathbb{K}}^{0,{\rm red}}(k) \to \mathbb{K}(k)) \to 
{\mathbb{K}}(k)/{\mathbb{K}}^{0,{\rm red}}(k) \to 0, \\ 
& 
{\mathbb{K}}'(k)/{\mathbb{K}}^{0,{\rm red}}(k) \to 
{\mathbb{K}}(k)/{\mathbb{K}}^{0,{\rm red}}(k) \to 
{\mathbb{K}}(k)/{\mathbb{K}}'(k),  \\ 
& 0 \to {\mathbb{K}(k)/{\mathbb{K}}'(k)} \to 
{\rm Pic}_{\overline{X}_0}(k)/{\rm Pic}^{0,{\rm red}}_{\overline{X}_0}(k). 
\end{align*}
Because ${\mathbb{T}}(k)/ {\mathbb{T}}^{0,{\rm red}}(k)$ is a finite group, ${\mathbb{K}}'(k)/{\mathbb{K}}^{0,{\rm red}}(k) = 
{\mathbb{K}}'(k)/{{\mathbb{K}}'}^{0,{\rm red}}(k)$ is a finite group 
(because ${\mathbb{K}}'$ is proper) and 
${\rm Pic}_{\overline{X}_0}(k)/{\rm Pic}^{0,{\rm red}}_{\overline{X}_0}(k) = 
{\rm NS}(\overline{X}_0)$ is a finitely generated $\Z$-module, 
we conclude that ${\rm NS}(\overline{X}_{\bullet})$ is also a finitely 
generated $\Z$-module. 
\end{remark}

Now we start the proof of Theorem \ref{thm:H^1}. 
Note that, to prove that the map \eqref{eq:indep?} is an isomorphism, 
we can reduce to the case where $k$ is algebraically closed. 
So we assume that $k$ is algebraically closed and 
we generalise a part of the theory of 
de Rham--Witt cohomology in \cite{I_1979} to the case of 
proper smooth simplicial schemes. 

First, we consider the slope spectral sequence 
\begin{equation}\label{eq:slspseq}
E_1^{ij} = 
H^j_{\et}(\overline{X}_{\bullet}, W\Omega^i_{\overline{X}_{\bullet}/k}) \Longrightarrow  
H^{i+j}_{\et}(\overline{X}_{\bullet}, W\Omega^{\kr}_{\overline{X}_{\bullet}/k}) 
\cong H^{i+j}_{\rm cris}(\overline{X}_{\bullet}/W(k)).  
\end{equation}
(The last isomorphism is proven in \cite[Cor.\,7.7]{N_2012}.) 
This degenerates at $E_1$ modulo torsion by \cite[Prop.\,7.3]{N_2012} 
(see also \cite[II.\,Thm.\,3.2]{I_1979}). 
Since the sheaf $W\Omega^i_{\overline{X}_0/k}$ is $p$-torsion free \cite[I.,Cor. 3.6]{I_1979}, 
so are $H^0_{\et}(\overline{X}_0, W\Omega^i_{\overline{X}_0/k})$ and its sub-module $H^0_{\et}(\overline{X}_{\bullet}, W\Omega^i_{\overline{X}_{\bullet}/k})$. Consequently, in the spectral sequence \eqref{eq:slspseq}, we obtain
\begin{equation}\label{eq:Ei0}
H^0_{\et}(\overline{X}_{\bullet}, W\Omega^i_{\overline{X}_{\bullet}/k}) 
= E^{i,0}_{\infty} \subseteq H^i_{\rm cris}(\overline{X}_{\bullet}/W(k)). 
\end{equation}
Because $H^i_{\rm cris}(\overline{X}_{\bullet}/W(k))$ is a finite $W(k)$-module, 
we see that $H^0_{\et}(\overline{X}_{\bullet}, W\Omega^i_{\overline{X}_{\bullet}/k})$ is a 
finite free $W(k)$-module. 

Next we study $H^1_{\et}(\overline{X}_{\bullet}, W{\mathcal{O}}_{\overline{X}_{\bullet}})$. 
We will need a simplicial version of several results in \cite{I_1979}.
For the convenience of the reader we provide the arguments which can be adapted from the non-simplicial case with minor modifications.

\begin{lemma}[{cf. \cite[II.\,(2.1.2)]{I_1979}}]\label{lem:projlimwomega}
For any $i,j$, there exists an isomorphism 
\begin{equation}\label{eq:projlim}
	H^i_{\et}(\overline{X}_{\bullet}, W\Omega^j_{\overline{X}_{\bullet}}) 
	\xrightarrow{\cong} \varprojlim_n H^i_{\et}(\overline{X}_{\bullet}, W_n\Omega^j_{\overline{X}_{\bullet}}). 
\end{equation}
\end{lemma}
\begin{proof}
	Because $W\Omega^j_{\overline{X}_{\bullet}} = R\varprojlim_n W_n\Omega^j_{\overline{X}_{\bullet}}$ (which follows from the case for usual 
	smooth schemes), we have the isomorphism 
	$$ R\Gamma_{\et}(\overline{X}_{\bullet}, W\Omega^j_{\overline{X}_{\bullet}}) 
	\xrightarrow{\cong} R\varprojlim_n 
	R\Gamma_{\et}(\overline{X}_{\bullet}, W_n\Omega^j_{\overline{X}_{\bullet}}). $$ 
	Also, since 
	$H^i_{\et}(\overline{X}_{\bullet}, W_n\Omega^j_{\overline{X}_{\bullet}})$ is 
	a $W(k)$-module  of finite length (because this is the case for usual proper smooth schemes), 
	$$ R^1\varprojlim_n H^i_{\et}(\overline{X}_{\bullet}, W_n\Omega^j_{\overline{X}_{\bullet}}) = 0. $$
	Thus we obtain the isomorphism \eqref{eq:projlim}. 
\end{proof}

\begin{lemma}\label{lem:v}
The group $H^1_{\et}(\overline{X}_{\bullet}, W{\mathcal{O}}_{\overline{X}_{\bullet}})$ 
is $V$-adically separated and complete. Also, for any $n$, 
$H^1_{\et}(\overline{X}_{\bullet}, W{\mathcal{O}}_{\overline{X}_{\bullet}})/V^n H^1_{\et}(\overline{X}_{\bullet}, W{\mathcal{O}}_{\overline{X}_{\bullet}})$ is a $W(k)$-module of finite length. 
\end{lemma}

\begin{proof}
 We follow the argument in \cite[pp.\,30--31]{S_1958}.
Because the inverse limit of a projective system of exact sequences of 
$W(k)$-modules of finite length is exact \cite[4.\,Lem.\,1]{S_1958}, the exact sequences 
$$ H^1_{\et}(\overline{X}_{\bullet}, W_m{\mathcal{O}}_{\overline{X}_{\bullet}}) 
\xrightarrow{V^n} H^1_{\et}(\overline{X}_{\bullet}, W_{n+m}{\mathcal{O}}_{\overline{X}_{\bullet}}) 
\to H^1_{\et}(\overline{X}_{\bullet}, W_n{\mathcal{O}}_{\overline{X}_{\bullet}}) \quad (m \in \N) $$
induce the diagram 
\begin{equation}\label{eq:vn}
H^1_{\et}(\overline{X}_{\bullet}, W{\mathcal{O}}_{\overline{X}_{\bullet}}) \to 
H^1_{\et}(\overline{X}_{\bullet}, W{\mathcal{O}}_{\overline{X}_{\bullet}})/
V^nH^1_{\et}(\overline{X}_{\bullet}, W{\mathcal{O}}_{\overline{X}_{\bullet}}) 
\hookrightarrow 
H^1_{\et}(\overline{X}_{\bullet}, W_n{\mathcal{O}}_{\overline{X}_{\bullet}}) 
\end{equation}
by Lemma \ref{lem:projlimwomega}. 
Taking the inverse limit with respect to $n$ and using  Lemma \ref{lem:projlimwomega} again, 
we obtain maps 
$$ 
H^1_{\et}(\overline{X}_{\bullet}, W{\mathcal{O}}_{\overline{X}_{\bullet}}) \to 
\varprojlim_n (H^1_{\et}(\overline{X}_{\bullet}, W{\mathcal{O}}_{\overline{X}_{\bullet}})/
V^nH^1_{\et}(\overline{X}_{\bullet}, W{\mathcal{O}}_{\overline{X}_{\bullet}}) )
\hookrightarrow 
H^1_{\et}(\overline{X}_{\bullet}, W{\mathcal{O}}_{\overline{X}_{\bullet}})
$$ 
whose composition is the identity. 
Thus the first morphism is an isomorphism 
$$ 
H^1_{\et}(\overline{X}_{\bullet}, W{\mathcal{O}}_{\overline{X}_{\bullet}}) \xrightarrow{\cong} 
\varprojlim_n (H^1_{\et}(\overline{X}_{\bullet}, W{\mathcal{O}}_{\overline{X}_{\bullet}})/
V^nH^1_{\et}(\overline{X}_{\bullet}, W{\mathcal{O}}_{\overline{X}_{\bullet}}) ).
$$
In particular, $H^1_{\et}(\overline{X}_{\bullet}, W{\mathcal{O}}_{\overline{X}_{\bullet}})$ 
is $V$-adically separated and complete. 
Moreover, the inclusion in the diagram 
\eqref{eq:vn} implies that 
$H^1_{\et}(\overline{X}_{\bullet}, W{\mathcal{O}}_{\overline{X}_{\bullet}})/V^n H^1_{\et}(\overline{X}_{\bullet}, W{\mathcal{O}}_{\overline{X}_{\bullet}})$ is of finite length. 
\end{proof}

\begin{remark}\label{rem:topH1WO}
The inclusion in the diagram \eqref{eq:vn} shows that the projective limit topology on 
$H^1_{\et}(\overline{X}_{\bullet}, W{\mathcal{O}}_{\overline{X}_{\bullet}}) \allowbreak = \varprojlim_n 
H^1_{\et}(\overline{X}_{\bullet}, W_n{\mathcal{O}}_{\overline{X}_{\bullet}})$
coincides with the $V$-adic topology. 
\end{remark}

\begin{proposition}[{cf.\,\cite[II.\,Prop.\,2.19]{I_1979}, \cite[7.\,Prop.\,4]{S_1958}}]\label{prop:finfree}
The cohomology group $H^1_{\et}(\overline{X}_{\bullet}, W{\mathcal{O}}_{\overline{X}_{\bullet}})$ 
is a finite free $W(k)$-module. 
\end{proposition}

\begin{proof}
We follow the proof in \cite[7.\,Prop.\,4]{S_1958}. 
Since there exists an injection $F^{n-1}d: W_n{\mathcal{O}}_{\overline{X}_{\bullet}}/ FW_n{\mathcal{O}}_{\overline{X}_{\bullet}} \hookrightarrow \Omega^1_{\overline{X}_{\bullet}/k}$ \cite[I.\,(3.11.4)]{I_1979}, we have 
$$\dim_k H^0_{\et}( \overline{X}_{\bullet}, 
W_n{\mathcal{O}}_{\overline{X}_{\bullet}}/ FW_n{\mathcal{O}}_{\overline{X}_{\bullet}}) \leqslant 
\dim_k H^0_{\et}(\overline{X}_{\bullet}, 
\Omega^1_{\overline{X}_{\bullet}/k}) =: c. $$
Then, since we have an exact sequence 
$$ 0 \to H^0_{\et}( \overline{X}_{\bullet}, 
W_n{\mathcal{O}}_{\overline{X}_{\bullet}}/ FW_n{\mathcal{O}}_{\overline{X}_{\bullet}}) \to 
H^1_{\et}(\overline{X}_{\bullet}, W_n{\mathcal{O}}_{\overline{X}_{\bullet}}) 
\xrightarrow{F} 
H^1_{\et}(\overline{X}_{\bullet}, W_n{\mathcal{O}}_{\overline{X}_{\bullet}}) $$
and $H^1_{\et}(\overline{X}_{\bullet}, W_n{\mathcal{O}}_{\overline{X}_{\bullet}})$ is a 
$W(k)$-module of finite length, we obtain the inequality
$$ 
\dim_k (H^1_{\et}(\overline{X}_{\bullet}, W_n{\mathcal{O}}_{\overline{X}_{\bullet}}) /
F H^1_{\et}(\overline{X}_{\bullet}, W_n{\mathcal{O}}_{\overline{X}_{\bullet}}) ) 
= \dim_k H^0_{\et}( \overline{X}_{\bullet}, 
W_n{\mathcal{O}}_{\overline{X}_{\bullet}}/ FW_n{\mathcal{O}}_{\overline{X}_{\bullet}}) 
\leqslant c. 
$$
Taking inverse limit, we obtain the inequality 
$ 
\dim_k (H^1_{\et}(\overline{X}_{\bullet}, W{\mathcal{O}}_{\overline{X}_{\bullet}}) /
F H^1_{\et}(\overline{X}_{\bullet}, W{\mathcal{O}}_{\overline{X}_{\bullet}}) ) 
\leqslant c$. 
Together with Lemma \ref{lem:v} using the equality $p=FV$, 
we see that 
$\dim_k (H^1_{\et}(\overline{X}_{\bullet}, W{\mathcal{O}}_{\overline{X}_{\bullet}}) /
p H^1_{\et}(\overline{X}_{\bullet}, W{\mathcal{O}}_{\overline{X}_{\bullet}})) $ is finite. 

To ease the notation, we denote $H^1_{\et}(\overline{X}_{\bullet}, W{\mathcal{O}}_{\overline{X}_{\bullet}})$ 
by $H$ in what follows. 
Take a finite $W(k)$-submodule $N$ of $H$ which surjects to $H/pH$. 
Then $N$ is dense in 
$H$ with respect to the $V$-adic topology: 
Indeed, if we set $M := H/(N + V^nH)$, then $M = pM$ by definition of $N$, 
and $M$ is of finite length by the latter assertion of Lemma \ref{lem:v}. 
Thus $M=0$ and hence $N$ is dense in $H$. 
On the other hand, $N$ is closed in 
$H$ with respect to $V$-adic topology: Indeed, 
if we set $N_n := N \cap V^nH$ for $n \in \N$, $\{N_n\}_n$ is 
a family of submodules of a finite $W(k)$-module $N$ with 
$N/N_n$ finite length and $\bigcap_n N_n = 0$. 
By the argument of 
\cite[Ch.\,1, 3. Prop.\,2]{S_1953}, these conditions imply that 
the topology defined by the family $\{N_n\}_n$ is the same as 
the $p$-adic topology. 
Thus $N$ is complete with respect to the topology defined by $\{N_n\}_n$ and so $N$ is closed in $H$. 
Hence $N = H$ and thus $H$ is a finite $W(k)$-module. 

Finally we prove the freeness of $H$.  
Since the map 
$W_{n+1}\mathcal{O}_{\overline{X}_{\bullet}} \to \mathcal{O}_{\overline{X}_{\bullet}}$ 
admits a section $a \mapsto (a,0,\dots,0)$ as a map of sheaves of sets, it induces the surjection 
$H^0_{\et}(\overline{X}_{\bullet}, W_{n+1}\mathcal{O}_{\overline{X}_{\bullet}}) 
\to H^0_{\et}(\overline{X}_{\bullet}, \mathcal{O}_{\overline{X}_{\bullet}})$. 
Hence  the exact sequence 
$$ 0 \to W_n{\mathcal{O}}_{\overline{X}_{\bullet}}  \xrightarrow{V}
W_{n+1}{\mathcal{O}}_{\overline{X}_{\bullet}} \to 
{\mathcal{O}}_{\overline{X}} \to 0 $$ 
induces the injection 
$V: H^1_{\et}(\overline{X}_{\bullet}, W_n{\mathcal{O}}_{\overline{X}_{\bullet}}) 
\hookrightarrow H^1_{\et}(\overline{X}_{\bullet}, W_{n+1}{\mathcal{O}}_{\overline{X}_{\bullet}})$, thus 
the injection $V: H \hookrightarrow H$. If we denote the torsion submodule of $H$ 
by $T$, $V: T \to T$ is injective and $T$ is of finite length. Thus 
$V(T) = T$ and so $T = \bigcap_n V^n(T) \subseteq \bigcap_n V^n(H) = 0$. 
Hence $H$ is free, as required. 
\end{proof}

\begin{remark}\label{rem:topH1WO-2}
The above proof shows that the $V$-adic topology on $H^1_{\et}(\overline{X}_{\bullet}, W{\mathcal{O}}_{\overline{X}_{\bullet}})$ coincides with the $p$-adic topology. In particular, 
the action of $V$ on $H^1_{\et}(\overline{X}_{\bullet}, W{\mathcal{O}}_{\overline{X}_{\bullet}})$ 
is topologically nilpotent with respect to the $p$-adic topology. 
\end{remark}

\begin{proposition}[{cf. \cite[II.\,Prop.\,3.11]{I_1979}}]\label{prop:sl01}
In the slope spectral sequence \eqref{eq:slspseq}, we have the equality 
$H^1_{\et}(\overline{X}_{\bullet}, W{\mathcal{O}}_{\overline{X}_{\bullet}}) = 
E^{0,1}_{\infty}$. 
Consequently, there is an exact sequence 
of finite free $W(k)$-modules 
\begin{equation}\label{eq:slexseq}
0 \to H^0_{\et}(\overline{X}_{\bullet}, W\Omega^1_{\overline{X}_{\bullet}/k}) 
\to H^1_{\rm cris}(\overline{X}_{\bullet}/W(k)) \to 
H^1_{\et}(\overline{X}_{\bullet}, W{\mathcal{O}}_{\overline{X}_{\bullet}})
\to 0. 
\end{equation}
\end{proposition}

\begin{proof}
Consider the map 
$d: H^1_{\et}(\overline{X}_{\bullet}, W{\mathcal{O}}_{\overline{X}_{\bullet}}) \to 
H^1_{\et}(\overline{X}_{\bullet}, W\Omega^1_{\overline{X}_{\bullet}/k})$. 
Note that $FdV = d$. So we have increasing sequences of $W(k)$-modules 
$\{{\rm Ker}(F^nd)\}_{n \in \N}, \{{\rm Im}(F^nd)\}_{n \in \N}$. 
In order to show that $d=0$, 
it suffices, by a lemma due to Nygaard \cite[II.\,Lem.\,3.8]{I_1979}, to 
prove that the above sequences are both stationary.
(Note that, to apply \cite[II.\,Lem.\,3.8]{I_1979}, we need to check that 
the map $d$ above is continuous with respect to (a topology coarser than) the $V$-adic topology on 
$H^1_{\et}(\overline{X}_{\bullet}, W{\mathcal{O}}_{\overline{X}_{\bullet}})$ and 
some separated topology 
on
$H^1_{\et}(\overline{X}_{\bullet}, W\Omega^1_{\overline{X}_{\bullet}/k})$. 
This is true if we take the projective limit topology in both cases defined by \eqref{eq:projlim}, which on 
$H^1_{\et}(\overline{X}_{\bullet}, W{\mathcal{O}}_{\overline{X}_{\bullet}})$ 
is equal to the $V$-adic topology by Remark \ref{rem:topH1WO}.)
The sequence $\{{\rm Ker}(F^nd)\}_{n \in \N}$ is stationary because 
$H^1_{\et}(\overline{X}_{\bullet}, W{\mathcal{O}}_{\overline{X}_{\bullet}})$ is 
a finite $W(k)$-module. On the other hand, 
$\{{\rm Im}(F^nd)/{\rm Im}(d)\}_{n \in \N}$ is an increasing sequence of 
$W(k)$-modules in $E_2^{1,1}$ of the spectral sequence
\eqref{eq:slspseq}. Because $E_2^{3,0} = E_{\infty}^{3,0}$ by 
\eqref{eq:Ei0}, the map $E_2^{1,1} \to E_2^{3,0}$ in the spectral sequence 
is zero, and so $E_2^{1,1} = E_{\infty}^{1,1}$. Thus $E_2^{1,1}$ is a subquotient of 
$H^2_{\rm cris}(\overline{X}_{\bullet}/W(k))$ and so a finite $W(k)$-module. 
Thus the sequence $\{{\rm Im}(F^nd)/{\rm Im}(d)\}_{n \in \N}$ is stationary and so 
$\{{\rm Im}(F^nd)\}_{n \in \N}$ is also stationary. Hence we have proved that $d=0$.
Then, using this fact and \eqref{eq:Ei0}, we see that 
$H^1(\overline{X}_{\bullet}, W{\mathcal{O}}_{\overline{X}_{\bullet}}) = 
E^{0,1}_{\infty}$. 
\end{proof}

We relate the cohomology groups in the exact sequence 
\eqref{eq:slexseq} to the reduced Picard scheme. 
To achieve this, but also at later parts in this section, we will need an isomorphism of certain
cohomology groups that we explain now: 
Recall that the exact sequence on the fppf-site of $\overline{X}_{\bullet}$ 
\[ 1 \to \mu_{p^n} \to \mathbb{G}_m \xrightarrow{p^n} \mathbb{G}_m \to 1\] 
and the exact sequence on the \'etale site of $\overline{X}_{\bullet}$ 
\[ 1 \to \mathcal{O}_{\overline{X}_{\bullet}}^{\times} \xrightarrow{p^n} 
\mathcal{O}_{\overline{X}_{\bullet}}^{\times} \to 
\mathcal{O}_{\overline{X}_{\bullet}}^{\times}/(\mathcal{O}_{\overline{X}_{\bullet}}^{\times})^{p^n} \to 1 
\]
induce isomorphisms 
\begin{align}
\varprojlim_nH^i_{\et}(\overline{X}_{\bullet}, 
\mathcal{O}_{\overline{X}_{\bullet}}^{\times}/(\mathcal{O}_{\overline{X}_{\bullet}}^{\times})^{p^n})
& \cong 
\varprojlim_nH^i_{\et}(\overline{X}_{\bullet}, \mathbb{G}_m \xrightarrow{p^n} \mathbb{G}_m) 
\label{eq:etfppf} \\ 
& \cong 
\varprojlim_nH^i_{\rm fppf}(\overline{X}_{\bullet}, \mathbb{G}_m \xrightarrow{p^n} \mathbb{G}_m) 
\cong 
H^{i+1}_{\rm fppf}(\overline{X}_{\bullet}, \Z_p(1)). \nonumber 
\end{align}

\begin{lemma}
Assume that $k$ is algebraically closed. Then there exists a canonical isomorphism 
$$ H^1_{\rm fppf}(\overline{X}_{\bullet}, \Z_p(1)) \otimes_{\Z_p} W(k)
\xrightarrow{\cong} H^0_{\et}(\overline{X}_{\bullet}, W\Omega^1_{\overline{X}_{\bullet}/k}). $$
\end{lemma}

\begin{proof}
The exact sequences 
$$ 0 \to 
{\mathcal{O}}^{\times}_{\overline{X}_{\bullet}}/
({\mathcal{O}}_{\overline{X}_{\bullet}}^{\times})^{p^n} \to 
W_n\Omega_{\overline{X}_{\bullet}/k}^1 \xrightarrow{1-F} 
W_n\Omega_{\overline{X}_{\bullet}/k}^1 \to 0  \quad (n \in \N) $$
on the \'etale site of $\overline{X}_{\bullet}$ (see \cite[I.\,(3.27.1)]{I_1979}) 
together with the isomorphism \eqref{eq:etfppf} imply the isomorphism
$$H^1_{\rm  fppf}(\overline{X}_{\bullet}, \Z_p(1)) \xrightarrow{\cong} 
{\rm Ker}(1-F: H^0_{\et}(\overline{X}_{\bullet}, W\Omega^1_{\overline{X}_{\bullet}/k}) \to 
H^0_{\et}(\overline{X}_{\bullet}, W\Omega^1_{\overline{X}_{\bullet}/k})).$$
On the other hand, 
because $H^0_{\et}(\overline{X}_{\bullet}, W\Omega^1_{\overline{X}_{\bullet}/k}) 
\subseteq  H^0_{\et}(\overline{X}_0, W\Omega^1_{\overline{X}_0/k})$ is pure of 
slope $1$ (with respect to the Frobenius on crystalline cohomology), 
the map 
$F: H^0_{\et}(\overline{X}_{\bullet}, W\Omega^1_{\overline{X}_{\bullet}/k}) \to 
H^0_{\et}(\overline{X}_{\bullet}, W\Omega^1_{\overline{X}_{\bullet}/k})$ 
induced by $F$ on $W\Omega^1_{\overline{X}_{\bullet}/k}$ is an automorphism. 
Since $H^0_{\et}(\overline{X}_{\bullet}, W\Omega^1_{\overline{X}_{\bullet}/k})$ is 
a finite $W(k)$-module, the required isomorphism follows from 
\cite[II.\,Lem.\,6.8.4]{I_1979}. 
\end{proof}

\begin{proposition}\label{prop:h0}
Assume that $k$ is algebraically closed. Then 
$H^0_{\et}(\overline{X}_{\bullet}, W\Omega^1_{\overline{X}_{\bullet}/k})$ is 
canonically isomorphic to the covariant Dieudonn\'e module of the \'etale part of 
the $p$-divisible group associated to 
${\rm Pic}^{0,{\rm red}}_{\overline{X}_{\bullet}}$. 
\end{proposition}

\begin{proof}
By \cite[Lem.\,A.2.1]{ABV_2005} (which is the simplicial version of \cite[III.\,Lem.\,4.17]{M_1980}), there is an isomorphism 
$H^1_{\rm fppf}(\overline{X}_{\bullet}, \Z_p(1)) \xrightarrow{\cong} 
{\rm Hom}(\Q_p/\Z_p, {\rm Pic}_{\overline{X}_{\bullet}})$.  Because
${\rm Hom}(\Q_p/\Z_p,G) = 0$ when $G$ is a finite group scheme or a finitely generated discrete group scheme the right hand side is equal to
$$
{\rm Hom}(\Q_p/\Z_p, {\rm Pic}^{0,{\rm red}}_{\overline{X}_{\bullet}})
= {\rm Hom}(\Q_p/\Z_p, {\rm Pic}^{0,{\rm red}}_{\overline{X}_{\bullet}}[p^{\infty}]).
$$ 
By duality, this is further equal to 
${\rm Hom}(({\rm Pic}^{0,{\rm red}}_{\overline{X}_{\bullet}})^{\vee}[p^{\infty}], \mu_{p^{\infty}}) 
= {\rm Hom}(({\rm Pic}^{0,{\rm red}}_{\overline{X}_{\bullet}})^{\vee}[p^{\infty}], {\mathbb{G}}_m).$ 
Thus
$$H^1_{\rm  fppf}(\overline{X}_{\bullet}, \Z_p(1)) \otimes_{\Z_p} W(k) 
\cong H^0_{\et}(\overline{X}_{\bullet}, W\Omega^1_{\overline{X}_{\bullet}/k})$$
is isomorphic to $${\rm Hom}(({\rm Pic}^{0,{\rm red}}_{\overline{X}_{\bullet}})^{\vee}[p^{\infty}], {\mathbb{G}}_{m}) 
\otimes_{\Z_p} W(k), $$ 
which is the contravariant 
Dieudonn\'e module of the multiplicative part of the dual of 
${\rm Pic}^{0,{\rm red}}_{\overline{X}_{\bullet}}$ (see \cite[Def.\,3.12, Def.\,3.17]{O_1969}),
 namely, 
the covariant Dieudonn\'e module of the \'etale part of 
the $p$-divisible group associated to 
${\rm Pic}^{0,{\rm red}}_{\overline{X}_{\bullet}}$. 
\end{proof}

\begin{proposition}\label{prop:h1}
Assume that $k$ is algebraically closed.
$H^1_{\et}(\overline{X}_{\bullet}, W{\mathcal{O}}_{\overline{X}_{\bullet}})$ is 
canonically isomorphic to the covariant Dieudonn\'e module of the connected part of 
the $p$-divisible group associated to 
${\rm Pic}^{0,{\rm red}}_{\overline{X}_{\bullet}}$. 
\end{proposition}

\begin{proof}
We follow the proof of \cite[Thm.\,4.4]{O_1969}. 

First we fix notations. Let ${\rm W}_n, {\rm W} := \varprojlim_n {\rm W}_n$ 
be the truncated and untruncated Witt ring schemes. We denote the sheaf on a simplicial $k$-scheme associated to 
${\rm W}$ by the same letter. In particular, 
$H^1_{\et}(\overline{X}_{\bullet}, {\rm W})$ is the same as $H^1_{\et}(\overline{X}_{\bullet}, W{\mathcal{O}}_{\overline{X}_{\bullet}})$. 

Let ${\rm C}_{-n} := {\rm W}_n$ and let ${\rm C} := \varinjlim_n {\rm C}_{-n}$, where the limit is 
taken with respect to the maps ${\rm C}_{-n} = {\rm W}_n \xrightarrow{V} {\rm W}_{n+1} = 
{\rm C}_{-(n+1)} \, (n \in \N)$. This is called the module scheme of Witt covectors. 
The operator $F:{\rm C} \to {\rm C}$ is defined as the inductive limit of 
$F:{\rm W}_{n} \to {\rm W}_n \,(n \in \N)$ and the operator $V: {\rm C} \to {\rm C}$ 
is defined as the inductive limit of ${\rm W}_n \xrightarrow{V} {\rm W}_{n+1} 
\xrightarrow{{\rm proj}} {\rm W}_n \, (n \in \N)$. 
$V$ is the same as the inductive limit of 
${\rm W}_{n+1} \xrightarrow{{\rm proj}} {\rm W}_n \, (n \in \N)$.  
We set $C := {\rm C}(k)$. 

For a group scheme $G$ over $k$ or  
an ind-object of finite group schemes $G$ over $k$ whose ranks are powers of $p$, its 
contravariant 
Dieudonn\'e module ${\bold M}(G)$ is defined by 
$$ {\bold M}(G) := {\rm Hom}(G,{\rm C}) \oplus  ({\rm Hom}(G, {\mathbb{G}}_m) \otimes_{\Z} W(k)), $$
endowed with the operators $F, V$ in a suitable way (see \cite[Def.\,3.12, Def.\,3.17]{O_1969}). 
On the first factor, the operators $F,V$ are the ones induced by those on ${\rm C}$. 
Also, note that the second factor vanishes when $G$ is unipotent. 

We have the following compatibilities of the functor ${\bold M}$ with duality: 
First, for a finite $k$-group scheme $G$ whose rank is a power of $p$, 
we have the canonical isomorphism 
${\bold M}{\bold D}(G) \cong {\rm Hom}_{W(k)}({\bold M}(G),C)$, 
where ${\bold D}$ is the Cartier dual and the operators $F,V$ on the right hand side 
are defined by 
	\[ F(\varphi)(x) = F(\varphi(Vx)), \quad V(\varphi)(x) = F^{-1}(\varphi(Fx)) 
	\quad (\varphi \in {\rm Hom}_{W(k)}({\bold M}(G),C), \, x \in {\bold M}(G)) \]
(see \cite[Thm.\,3.19]{O_1969}). 
Second, for a $p$-divisible group $G$, we have the canonical isomorphism 
${\bold M}{\bold D}(G) \cong {\rm Hom}_{W(k)}({\bold M}(G), W(k))$, where 
${\bold D}$ is the Cartier dual of $p$-divisible groups  and the definition of 
$F,V$ on the right hand side is similar to the previous case 
(see \cite[Prop.\,3.22]{O_1969}). 

Now we start the proof. We need to prove that 
$H^1_{\et}(\overline{X}_{\bullet}, {\rm W})$ is 
canonically isomorphic to the contravariant Dieudonn\'e module of the unipotent part of  
$({\rm Pic}^{0,{\rm red}}_{\overline{X}_{\bullet}})^{\vee}[p^{\infty}]$. 
By duality, this is the $W(k)$-linear dual of 
the contravariant Dieudonn\'e module of the connected part of  
${\rm Pic}^{0,{\rm red}}_{\overline{X}_{\bullet}}[p^{\infty}]$. 
Because the connected part of  
${\rm Pic}^{0,{\rm red}}_{\overline{X}_{\bullet}}[p^{\infty}]$ is 
${\rm Pic}^{0,{\rm red}}_{\overline{X}_{\bullet}}[F^{\infty}]$, what we should prove is the 
equality ${\bold M}({\rm Pic}^{0,{\rm red}}_{\overline{X}_{\bullet}}[F^{\infty}]) = 
{\rm Hom}_{W(k)}(H^1_{\et}(\overline{X}_{\bullet}, {\rm W}),W(k))$. 
Applying the Cartier dual, it suffices to prove the equality 
\[\varinjlim_n {\bold M}{\bold D}({\rm Pic}^{0,{\rm red}}_{\overline{X}_{\bullet}}[F^n]) 
= C \otimes_{W(k)} H^1_{\et}(\overline{X}_{\bullet}, {\rm W}).\]

First, we prove that the above equality follows from the equality 
\begin{equation}\label{eq:dm}
\varinjlim_n {\bold M}{\bold D}({\rm Pic}_{\overline{X}_{\bullet}}[F^n]) 
= H^1_{\et}(\overline{X}_{\bullet}, {\rm C}). 
\end{equation}
We have an exact sequence 
$$ 0 \to {\rm Pic}^{0,{\rm red}}_{\overline{X}_{\bullet}}[F^{\infty}] 
\to {\rm Pic}_{\overline{X}_{\bullet}}[F^{\infty}] 
\to N \to 0 $$
for a finite group scheme $N$ killed by a power of $F$. 
In the associated exact sequence of Dieudonn\'e modules 
$$ 0 \to {\bold M}(N) \to {\bold M}({\rm Pic}_{\overline{X}_{\bullet}}[F^{\infty}])
\to {\bold M}({\rm Pic}^{0,{\rm red}}_{\overline{X}_{\bullet}}[F^{\infty}]) \to 0, $$
${\bold M}(N)$ is killed by a power of $F$ by \cite[Prop.\,3.13]{O_1969}. 
On the other hand, ${\bold M}({\rm Pic}^{0,{\rm red}}_{\overline{X}_{\bullet}}[F^{\infty}])$
is a finite free $W(k)$-module because it is the Dieudonn\'e module 
of a $p$-divisible group, and so $F$ is injective on it. 
Thus 
$$ {\bold M}({\rm Pic}^{0,{\rm red}}_{\overline{X}_{\bullet}}[F^{\infty}]) 
= \varinjlim_m \left({\bold M}({\rm Pic}_{\overline{X}_{\bullet}}[F^{\infty}])/{\rm Ker}\left(F^m: 
{\bold M}({\rm Pic}_{\overline{X}_{\bullet}}[F^{\infty}]) \to 
{\bold M}({\rm Pic}_{\overline{X}_{\bullet}}[F^{\infty}])\right)\right). $$ 
Applying the Cartier dual, we see that 
\begin{align*}
\varinjlim_n {\bold M}{\bold D}({\rm Pic}^{0,{\rm red}}_{\overline{X}_{\bullet}}[F^n]) 
& = \bigcap_m {\rm Im}\left(V^m: 
\varinjlim_n {\bold M}{\bold D}({\rm Pic}_{\overline{X}_{\bullet}}[F^n]) 
\to 
\varinjlim_n {\bold M}{\bold D}({\rm Pic}_{\overline{X}_{\bullet}}[F^n])\right). 
\end{align*}
If we have the equality \eqref{eq:dm}, the above module is equal to 

\begin{align*}
 \bigcap_m {\rm Im}\left(V^m: 
H^1_{\et}(\overline{X}_{\bullet}, {\rm C}) \to H^1_{\et}(\overline{X}_{\bullet}, {\rm C})\right) 
\,=\, & 
\bigcap_m \varinjlim_n {\rm Im}\left(V^m: 
H^1_{\et}(\overline{X}_{\bullet}, {\rm C}_{-n-m}) \to H^1_{\et}(\overline{X}_{\bullet}, {\rm C}_{-n})\right) \\ 
\,=\, & 
\bigcap_m \varinjlim_{n,V} {\rm Im}\left({\rm proj}: 
H^1_{\et}(\overline{X}_{\bullet}, {\rm W}_{n+m}) \to H^1_{\et}(\overline{X}_{\bullet}, {\rm W}_{n})\right) \\
\, \xleftarrow{\cong} \, & 
\varinjlim_{n,V} \bigcap_m {\rm Im}\left({\rm proj}: 
H^1_{\et}(\overline{X}_{\bullet}, {\rm W}_{n+m}) \to H^1_{\et}(\overline{X}_{\bullet}, {\rm W}_{n})\right) \\ 
\, = \, & 
\varinjlim_{n,V} {\rm Im}\left({\rm proj}: 
H^1_{\et}(\overline{X}_{\bullet}, {\rm W}) \to H^1_{\et}(\overline{X}_{\bullet}, {\rm W}_{n})\right) \\
\, \xrightarrow{\cong} \, & 
\varinjlim_{n,V} H^1_{\et}(\overline{X}_{\bullet}, {\rm W})/V^n H^1_{\et}(\overline{X}_{\bullet}, {\rm W}) 
\,\cong\,  C \otimes_{W(k)}  H^1_{\et}(\overline{X}_{\bullet}, {\rm W}),
\end{align*}
where $\varinjlim_{n,V}$ denotes the colimit induced by the maps ${\rm W}_n \xrightarrow{V} {\rm W}_{n+1} \, (n \in \N)$.
Here, the isomorphism in the third line follows from the equality
\begin{multline*}
 {\rm Im}\left({\rm proj}: H^1_{\et}(\overline{X}_{\bullet}, {\rm W}_{n+m}) \to H^1_{\et}(\overline{X}_{\bullet}, {\rm W}_{n})\right)\cap 
  {\rm Im}\left({\rm V^{n-n'}}: H^1_{\et}(\overline{X}_{\bullet}, {\rm W}_{n'}) \to H^1_{\et}(\overline{X}_{\bullet}, {\rm W}_{n})\right)\\
  =V^{n-n'}{\rm Im}\left({\rm proj}: H^1_{\et}(\overline{X}_{\bullet}, {\rm W}_{n'+m}) \to H^1_{\et}(\overline{X}_{\bullet}, {\rm W}_{n'})\right)
\end{multline*}
for $m,n,n' \in \N$ with $n \geqslant n'$, which is implied by the exact sequences 
$$ H^1_{\et}(\overline{X}_{\bullet}, {\rm W}_{n'+m}) \xrightarrow{(V^{n-n'}, {\rm proj})} 
H^1_{\et}(\overline{X}_{\bullet}, {\rm W}_{n+m}) \oplus H^1_{\et}(\overline{X}_{\bullet}, {\rm W}_{n'}) 
\xrightarrow{({\rm proj} - V^{n-n'})} H^1_{\et}(\overline{X}_{\bullet}, {\rm W}_{n}). $$
The isomorphism in the fifth line follows from the exact sequence 
$$ H^1_{\et}(\overline{X}_{\bullet}, {\rm W}) \xrightarrow{V^n} 
 H^1_{\et}(\overline{X}_{\bullet}, {\rm W}) 
\xrightarrow{{\rm proj}} H^1_{\et}(\overline{X}_{\bullet}, {\rm W}_{n}) $$
and the last isomorphism follows from  \cite[Lem.\,4.6]{O_1969}. 
(Note that, to apply \cite[Lem.\,4.6]{O_1969}, we use the finite freeness of 
$H^1_{\et}(\overline{X}_{\bullet}, {\rm W})$ as $W(k)$-module which is shown in 
Proposition \ref{prop:finfree} and the topological nilpotence of 
the action of $V$ on $H^1_{\et}(\overline{X}_{\bullet}, {\rm W})$ 
with respect to the $p$-adic topology which is shown in Remark \ref{rem:topH1WO-2}.)
Thus it suffices to prove the equality \eqref{eq:dm} to prove the theorem. 

Now we prove the equality \eqref{eq:dm}. 
We set 
$H_n := ({\mathcal{O}}_{{\rm Pic}_{\overline{X}_{\bullet}},0}/
(F^*)^n({\mathfrak{m}}_{{\rm Pic}_{\overline{X}_{\bullet}},0}){\mathcal{O}}_{{\rm Pic}_{\overline{X}_{\bullet}},0})^*$, 
where ${\mathcal{O}}_{{\rm Pic}_{\overline{X}_{\bullet}},0}$ is the local ring of 
${\rm Pic}_{\overline{X}_{\bullet}}$ at the identity, 
${\mathfrak{m}}_{{\rm Pic}_{\overline{X}_{\bullet}},0}$ is its maximal ideal
and $^*$ denotes the $k$-linear dual. Then 
${\bold D}({\rm Pic}_{\overline{X}_{\bullet}}[F^n]) = {\rm Spec}\,H_n$. 
Since each 
${\bold D}({\rm Pic}_{\overline{X}_{\bullet}}[F^n])$ is unipotent,
\eqref{eq:dm} is equivalent to the equality 
$$ H^1_{\et}(\overline{X}_{\bullet}, {\rm C}) = \varinjlim_n {\rm Hom}({\rm Spec}\,H_n, {\rm C}) $$
and to prove the latter, it suffices to prove the equality 
\begin{equation}\label{eq:h}
H^1_{\et}(\overline{X}_{\bullet}, {\rm W}_n) = {\rm Hom}({\rm Spec}\,H, {\rm W}_n), 
\end{equation} 
where $H = \bigcup_n H_n$. 

We prove the equality \eqref{eq:h}. 
Note that, for an augmented Artinian local $k$-algebra $(R, {\mathfrak{m}}, \pi:R \to k)$, 
$R^* = k \oplus {\mathfrak{m}}^*$ is a $k$-coalgebra and so 
${\rm Sym}({\mathfrak{m}}^*)$ is a $k$-bialgebra. 
If we set ${\rm S}_R := {\rm Spec}\,{\rm Sym}({\mathfrak{m}}^*)$, 
we have the following for a $k$-algebra $B$: 
\begin{align*}
{\rm S}_R(B) & = {\rm Hom}_{\text{$k$-alg.}} ({\rm Sym}({\mathfrak{m}}^*), B) \\ 
& = \{ \varphi \in {\rm Hom}_{\text{$k$-vec.\,sp.}} (k \oplus {\mathfrak{m}}^*, B) \,|\, 
\varphi(1)=1\} \\ 
& = 
\{ \varphi \in {\rm Hom}_{\text{$k$-vec.\,sp.}} (R^*, B) \,|\, 
\varphi({\rm id})=1\} \\ 
& = \{ a \in R \otimes_k B \,|\, 
(\pi \otimes {\rm id})(a) = 1\} = 1 + {\mathfrak{m}} \otimes_k B.  
\end{align*}
We compute ${\rm Ker}\left({\rm Pic}_{\overline{X}_{\bullet}}(R) \to {\rm Pic}_{\overline{X}_{\bullet}}(k)\right)$ 
in two ways. First, we have the equalities 
\begin{align*}
& {\rm Ker}\left({\rm Pic}_{\overline{X}_{\bullet}}(R) \to {\rm Pic}_{\overline{X}_{\bullet}}(k)\right)
= {\rm Ker}\left(H^1_{\et}(\overline{X}_{\bullet}, ({\mathcal{O}}_{\overline{X}_{\bullet}} \otimes_k R)^{\times}) 
\to H^1_{\et}(\overline{X}_{\bullet}, {\mathcal{O}}_{\overline{X}_{\bullet}}^{\times})\right) 
= H^1_{\et}(\overline{X}_{\bullet}, {\rm S}_R).
\end{align*}
(The first equality follows from Remark \ref{rem:BVS}.)
Second, we have the equalities 
\begin{align*}
{\rm Ker}\left({\rm Pic}_{\overline{X}_{\bullet}}(R) \to {\rm Pic}_{\overline{X}_{\bullet}}(k)\right) 
& = {\rm Hom}_{\text{aug.\,$k$-alg.}}({\mathcal{O}}_{{\rm Pic}_{\overline{X}_{\bullet}},0}, R)  
= {\rm Hom}_{\text{aug.\,$k$-coalg.}}(R^*, H) \\ 
& = {\rm Hom}_{\text{aug.\,$k$-bialg.}}({\rm Sym}({\mathfrak{m}}^*), H )
= {\rm Hom}({\rm Spec}\,H, {\rm S}_R). 
\end{align*}
Thus we have 
\begin{equation}\label{eq:sr}
H^1_{\et}(\overline{X}_{\bullet}, {\rm S}_R) =  {\rm Hom}({\rm Spec}\,H, {\rm S}_R).
\end{equation}
It is known (see \cite[p.\,112]{O_1969}) that, when $R = k[t]/(t^{p^n})$, 
$$ {\rm S}_R = {\rm W}_n \times \prod_{1 < i < p^n \atop (i,p)=1} {\rm W}_{r_i}, $$
where $r_i = \min\{r \,|\, p^r \geqslant p^n/i\}$. Using this, we see that 
\eqref{eq:sr} implies the required equality \eqref{eq:h}. So the proof is finished. 
\end{proof}

	The last ingredient for the proof of our first main theorem in this section is the following exact sequence:
	
	\begin{proposition}\label{prop: exact div}
		Assume that $k$ is algebraically closed. Then there exists an exact sequene
		\begin{equation}\label{eq:exseqh2-3}
			0 \to H^1_{\rm cris}(\overline{X}_{\bullet}/W(k)) 
			\to H^1_{\rm cris}((X_{\bullet}, \overline{X}_{\bullet})/W(k))
			\to {\rm Div}^0_{D_{\bullet}}(\overline{X}_{\bullet}) \otimes W(k) \to 0. 
		\end{equation}
	\end{proposition}
	
	\begin{proof}
		Consider the weight filtration 	$P_j W\Omega^{i}_{\overline{X}_\bullet/k}(\log D_{\bullet})$ 
		defined in \cite[1.4]{M_1993}. (One has also the description 
		$$
		P_j W\Omega^{i}_{\overline{X}_\bullet/k}(\log D_{\bullet}) = 
		\left\{
		\begin{aligned}
		& \mathrm{Im}\left( W\Omega^{j}_{\overline{X}_\bullet/k}(\log D_{\bullet})\otimes W\Omega^{i-j}_{\overline{X}_\bullet/k}
		\rightarrow W\Omega^{i}_{\overline{X}_\bullet/k}(\log D_{\bullet})\right) \quad (j \leq i), \\
		&  W\Omega^{i}_{\overline{X}_\bullet/k}(\log D_{\bullet}) \quad (j \geq i), 
		\end{aligned}
		\right. 
		$$
		which is given in \cite[8.2, 9]{M_2017}.) The Poincar\'e residue isomorphism (\cite[1.4.5]{M_1993} or \cite[9]{M_2017}) 
		induces the exact sequence 
		$$ 0 \to W\Omega^{\kr}_{\overline{X}_i/k} \to P_1 W\Omega^{\kr}_{\overline{X}_i/k}(\log D_i) \to 
		a_{i*}^{(1)}W\Omega^{\kr}_{D_i^{(1)}/k}[-1] \to 0 $$
		for each $i$, where $D^{(1)}_i$ is the disjoint union of the irreducible components of $D_i$ and 
		$a^{(1)}_i: D^{(1)}_i \to \overline{X}_i$ is the canonical morphism. Since the map 
		$W\Omega^{\kr}_{\overline{X}_i/k} \to P_1 W\Omega^{\kr}_{\overline{X}_i/k}(\log D_i)$ is compatible with respect 
		to $i$, the above exact sequences for $i \in \N$ induce the exact sequence of the form 
		\begin{equation}\label{eq:drwsimpsheaf}
		0 \to W\Omega^{\kr}_{\overline{X}_\bullet/k} \to P_1 W\Omega^{\kr}_{\overline{X}_\bullet/k}(\log D_\bullet) \to 
		a_{\bullet *}^{(1)}W\Omega^{\kr}_{D_\bullet^{(1)}/k}[-1] \to 0.
		\end{equation}
		Here $a_{\bullet *}^{(1)}W\Omega^{\kr}_{D_\bullet^{(1)}/k}[-1]$ denotes the complex of sheaves on $\overline{X}_{\bullet}$ 
		defined on each $\overline{X}_i$ by $a_{i*}^{(1)}W\Omega^{\kr}_{D_i^{(1)}/k}[-1]$ with the transition maps induced 
		by those on $W\Omega^{\kr}_{\overline{X}_\bullet}$ and $P_1 W\Omega^{\kr}_{\overline{X}_\bullet}(\log D_\bullet)$. 
		(Beware that the schemes $D^{(1)}_i \, (i \in \N)$ do not form a simplicial scheme and the maps $a^{(1)}_i \, (i \in \N)$ do not form a map of simplicial schemes.) This exact sequence induces the long exact sequence 
		\begin{align}
	        0  
			& \rightarrow H^1_{\et}(\overline{X}_{\bullet}, W\Omega^{\kr}_{\overline{X}_{\bullet}/k}) \rightarrow 
			H^1_{\et}(\overline{X}_{\bullet}, P_1W\Omega^{\kr}_{\overline{X}_\bullet/k}(\log D_{\bullet})) \label{eq:longexseqcrys} \\  & \rightarrow 
			H^0_{\et}(\overline{X}_{\bullet}, a^{(1)}_{\bullet *}W\Omega^{\kr}_{D^{(1)}_\bullet/k}) \rightarrow H^2_{\et}(\overline{X}_{\bullet}, W\Omega^{\kr}_{\overline{X}_{\bullet}/k}). \nonumber  
		\end{align}	
		
		We rewrite the second and the third nonzero terms of the sequence. First, by using the spectral sequences 
		\begin{align*}
			& E_i^{ij} = H^j_{\et}(\overline{X}_{\bullet}, P_1W\Omega^{i}_{\overline{X}_{\bullet}/k}(\log D_{\bullet})) 
			\Longrightarrow H^{i+j}_{\et}(\overline{X}_{\bullet}, P_1W\Omega^{\kr}_{\overline{X}_{\bullet}/k}(\log D_{\bullet})), \\
			& E_i^{ij} = H^j_{\et}(\overline{X}_{\bullet}, W\Omega^{i}_{\overline{X}_{\bullet}/k}(\log D_{\bullet})) 
			\Longrightarrow H^{i+j}_{\et}(\overline{X}_{\bullet}, W\Omega^{\kr}_{\overline{X}_{\bullet}/k}(\log D_{\bullet})) 
		\end{align*}
		and noting the facts 
		\begin{align*}
			& H^j_{\et}(\overline{X}_{\bullet}, P_1W\Omega^{i}_{\overline{X}_{\bullet}/k}(\log D_{\bullet})) =  H^j_{\et}(\overline{X}_{\bullet}, W\Omega^{i}_{\overline{X}_{\bullet}/k}(\log D_{\bullet})) \text{ for } i \leqslant 1, \\ 
		    &  H^0_{\et}(\overline{X}_{\bullet}, P_1W\Omega^{i}_{\overline{X}_{\bullet}/k}(\log D_{\bullet})) \subseteq   H^0_{\et}(\overline{X}_{\bullet}, W\Omega^{i}_{\overline{X}_{\bullet}/k}(\log D_{\bullet})) \text{ for } i \in \N, 
		\end{align*}
		we see the equality 
		$H^1_{\et}(\overline{X}_{\bullet}, P_1W\Omega^{\kr}_{\overline{X}_{\bullet}/k}(\log D_{\bullet})) = 
		H^1_{\et}(\overline{X}_{\bullet}, W\Omega^{\kr}_{\overline{X}_{\bullet}/k}(\log D_{\bullet})).$ 
		Next, we have the equality 
		\begin{equation}\label{eq:logcrysH0D}
			H^0_{\et}(\overline{X}_{\bullet}, a^{(1)}_{\bullet *}W\Omega^{\kr}_{D_\bullet/k})
			= {\rm Ker}(d_0^* - d_1^*: H^0_{\et}(\overline{X}_0, a^{(1)}_{0*}W\Omega^{\kr}_{D_0^{(1)}}) \to 
			H^0_{\et}(\overline{X}_1, a^{(1)}_{1*}W\Omega^{\kr}_{D_1^{(1)}})), 
		\end{equation}
		where the maps $d_i^* \, (i=0,1)$ are the ones induced by the maps 
		\[ d_i^*: W\Omega^{\kr}_{\overline{X}_0/k} \to W\Omega^{\kr}_{\overline{X}_1/k}, \quad 
		d_i^*: P_1W\Omega^{\kr}_{\overline{X}_0/k}(\log D^{(1)}_0) \to 
		P_1W\Omega^{\kr}_{\overline{X}_1/k}(\log D^{(1)}_1)
		\]
		which are defined by the projections $d_i: \overline{X}_1 \to \overline{X}_0$. 
		Since $k$ is algebraically closed, $H^0_{\et}(\overline{X}_i, a^{(1)}_{0*}W\Omega^{\kr}_{D_i^{(1)}}) \,\,\allowbreak  (i=0,1)$ is identified with ${\rm Div}_{D_i}(\overline{X}_i) \otimes_{\Z} W(k)$ 
		(where ${\rm Div}_{D_i}(\overline{X}_i)$ is the group of Weil divisors on $\overline{X}_i$ supported on 
		$D_i$).
		With this identification the map $d_i^*$ is, by definition of the Poincar\'e 
		residue isomorphism in \cite[1.4]{M_1993} or \cite[8.2, 9]{M_2017},  
		equal to the pullback map of divisors by the projection $d_i$. 
		Thus the group in \eqref{eq:logcrysH0D} is identified with 
		${\rm Div}_{D_{\bullet}}(\overline{X}_{\bullet})\otimes_{\Z} W(k)$.  Consequently, we have shown that the long exact sequence \eqref{eq:longexseqcrys} is rewritten as follows: 
		\begin{align}
			0  
			& \rightarrow H^1_{\et}(\overline{X}_{\bullet}, W\Omega^{\kr}_{\overline{X}_{\bullet}/k}) \rightarrow 
			H^1_{\et}(\overline{X}_{\bullet}, W\Omega^{\kr}_{\overline{X}_\bullet/k}(\log D_{\bullet})) \label{eq:longexseqcrys2} \\  & \rightarrow 
			{\rm Div}_{D_{\bullet}}(\overline{X}_{\bullet})\otimes_{\Z} W(k) 
			\rightarrow H^2_{\et}(\overline{X}_{\bullet}, W\Omega^{\kr}_{\overline{X}_{\bullet}/k}). \nonumber 
		\end{align}	
		
		We prove that the last map factors through ${\rm NS}(\overline{X}_\bullet) \otimes_{\Z} W(k)$. 
		Consider the non-logarithmic and logarithmic
		first Chern class maps of the de Rham--Witt complexes  
		$$d\log: \mathcal{O}_{\overline{X}_{\bullet}}^\times \rightarrow W\Omega^{\kr}_{\overline{X}_{\bullet}/k}[1], 
		\quad 
		d \log: j_{\bullet \ast}\mathcal{O}_{\overline{X}_{\bullet}\backslash D_{\bullet}}^\times \rightarrow  
		W\Omega^{\kr}_{\overline{X}_\bullet/k}(\log D_{\bullet})[1],$$
		where $j_{\bullet}: \overline{X}_{\bullet}\backslash D_{\bullet} \rightarrow \overline{X}_{\bullet}$ is 
		the natural open immersion (compare \cite{G_1985}). The exact sequences 
		$$
		0 \rightarrow \mathcal{O}_{\overline{X}_i}^\times \rightarrow j_{i\ast}\mathcal{O}^\times_{\overline{X}_i \setminus D_i } \rightarrow a_{i*}^{(1)}\Z_{D^{(1)}_i}\rightarrow 0 
		$$
		for $i \in \N$ induce the exact sequence of the form 
		$$
		0 \rightarrow \mathcal{O}_{\overline{X}_{\bullet}}^\times \rightarrow j_{\bullet \ast}\mathcal{O}^\times_{\overline{X}_{\bullet} \setminus D_\bullet } \rightarrow a_{\bullet *}^{(1)}\Z_{D^{(1)}_\bullet} \rightarrow 0 
		$$
		which makes the following diagram commutative: 
		\[
		\xymatrix{
		0 \ar[r] & \mathcal{O}^{\times}_{\overline{X}_{\bullet}} \ar[r] \ar[d]^{d\log} & 
		j_{\bullet \ast} \mathcal{O}^{\times}_{\overline{X}_\bullet \setminus D_{\bullet}} \ar[r] \ar[d]^{d\log} & 
		a_{\bullet *}^{(1)}\Z_{D^{(1)}_\bullet} \ar[d] \ar[r] & 0 \\
		0 \ar[r] & W\Omega^{\kr}_{\overline{X}_\bullet}[1] \ar[r] & P_1W\Omega^{\kr}_{\overline{X}_\bullet}(\log D_{\bullet})[1] 
		\ar[r] & a^{(1)}_{\bullet \ast}W\Omega^{\kr}_{D^{(1)}_{\bullet}} \ar[r] & 0. 	
		}
		\]
		Considering the connecting homomorphisms, we obtain the commutative diagram 
		$$
		\xymatrix{
			{\rm Div}_{D_{\bullet}}(\overline{X}_{\bullet}) \ar[d] \ar[r] & 
			H^1_{\et}(\overline{X}_\bullet, \mathcal{O}_{\overline{X}_{\bullet}}^\times) \ar[d] \\
			{\rm Div}_{D_{\bullet}}(\overline{X}_{\bullet}) \otimes_{\Z} W(k) \ar[r] & 
			H^2_{\et}(\overline{X}_{\bullet}, W\Omega^{\kr}_{\overline{X}_{\bullet}/k}).  
		}
		$$
		Thus the map ${\rm Div}_{D_{\bullet}}(\overline{X}_{\bullet}) \otimes_{\Z} W(k) \to  
		H^2_{\et}(\overline{X}_{\bullet}, W\Omega^{\kr}_{\overline{X}_{\bullet}/k})$ factors through 
		$(\varprojlim_n H^1_{\et}(\overline{X}_{\bullet}, \mathcal{O}_{\overline{X}_{\bullet}}^\times)/p^n) 
		\otimes_{\Z_p} W(k).$ Since we have isomorphisms 
		\[\varprojlim_n H^1_{\et}(\overline{X}_{\bullet}, \mathcal{O}_{\overline{X}_{\bullet}}^\times)/p^n 
		= \varprojlim_n {\rm Pic}(\overline{X}_{\bullet})(k)/p^n 
		\cong \varprojlim_n {\rm NS}(\overline{X}_{\bullet})/p^n
		\cong {\rm NS}(\overline{X}_{\bullet}) \otimes_{\Z} \Z_p, 
		\]
		we see that the map ${\rm Div}_{D_{\bullet}}(\overline{X}_{\bullet}) \otimes_{\Z} W(k) \to  
		H^2_{\et}(\overline{X}_{\bullet}, W\Omega^{\kr}_{\overline{X}_{\bullet}/k})$ 
		factors as 
		\begin{equation}\label{eq:h2-3}
			{\rm Div}_{D_{\bullet}}(\overline{X}_{\bullet}) \otimes_{\Z} W(k) \to 
			{\rm NS}(\overline{X}_{\bullet}) \otimes_{\Z} W(k) \to 
			H^2_{\et}(\overline{X}_{\bullet}, W\Omega^{\kr}_{\overline{X}_{\bullet}/k}),
		\end{equation}
		as required. 
		
		In the subsequent lemma, we will see that the latter map in \eqref{eq:h2-3} is injective.
		Hence we obtain from \eqref{eq:longexseqcrys2} the exact sequence
		$$
		0 \to H^1_{\rm cris}(\overline{X}_{\bullet}/W(k)) 
		\to H^1_{\rm cris}((X_{\bullet}, \overline{X}_{\bullet})/W(k))
		\to {\rm Div}^0_{D_{\bullet}}(\overline{X}_{\bullet}) \otimes_{\Z} W(k) \to 0, 
		$$
		as desired.
	\end{proof}

\begin{lemma}[{cf. \cite[II.\,Rem.\,6.8.5]{I_1979}}]\label{lem:nsinj}
 Keep the assumption that $k$ is algebraically closed. Then the
map ${\rm NS}(\overline{X}_{\bullet}) \otimes_{\Z} W(k) \to 
H^2_{\rm cris}(\overline{X}_{\bullet}/W(k))$ in \eqref{eq:h2-3} 
is injective. 
\end{lemma}

\begin{proof}
We follow the proof of \cite[II.\,Rem.\,6.8.5]{I_1979}. 
In the proof, we will identify 
$\varprojlim_nH^1_{\et}(\overline{X}_{\bullet}, 
\mathcal{O}_{\overline{X}_{\bullet}}^{\times}/(\mathcal{O}_{\overline{X}_{\bullet}}^{\times})^{p^n})$ and 
$H^2_{\rm fppf}(\overline{X}_{\bullet}, \Z_p(1))$ via \eqref{eq:etfppf}. 
First, the exact sequences on the \'etale site 
\[ 
0 \to \mathcal{O}^{\times}_{\overline{X}_{\bullet}} \xrightarrow{p^n} 
\mathcal{O}^{\times}_{\overline{X}_{\bullet}} \to 
\mathcal{O}^{\times}_{\overline{X}_{\bullet}}/(\mathcal{O}^{\times}_{\overline{X}_{\bullet}})^{p^n} \to 0
\]
for $n \in \N$ induce the injection 
${\rm NS}(\overline{X}_{\bullet}) \otimes_{\Z} \Z_p 
\hookrightarrow H^2_{\rm fppf}(\overline{X}_{\bullet}, \Z_p(1))$, hence the injection 
\begin{equation}\label{eq:ns1}
	{\rm NS}(\overline{X}_{\bullet}) \otimes_{\Z} W(k) 
	\hookrightarrow H^2_{\rm fppf}(\overline{X}_{\bullet}, \Z_p(1)) \otimes_{\Z_p} W(k). 
\end{equation}
Next, note that we have the following exact sequence on the  \'etale site 
\begin{equation}\label{eq:oxoxp}
0 \to 
{\mathcal{O}}^{\times}_{\overline{X}_{\bullet}}/
({\mathcal{O}}_{\overline{X}_{\bullet}}^{\times})^{p^n}[-1]  \xrightarrow{d\log} 
W_n\Omega_{\overline{X}_{\bullet}/k}^{\geqslant 1} \xrightarrow{1-F'} 
W_n\Omega_{\overline{X}_{\bullet}/k}^{\geqslant 1} \to 0  \quad (n \in \N) 
\end{equation}
(where  $d\log$ is the map induced by $d\log: {\mathcal{O}}^{\times}_{\overline{X}_{\bullet}} 
\to W_n\Omega_{\overline{X}_{\bullet}/k}^{\kr}[-1]$ and 
$F'$ is the map of the complex
$ W_n\Omega_{\overline{X}_{\bullet}/k}^{\geqslant 1} \to W_n\Omega_{\overline{X}_{\bullet}/k}^{\geqslant 1} $ 
whose degree $i$ part is $p^{i-1}F$), by \cite[I.\,(3.29.2)]{I_1979}.  
From this, we obtain the long exact sequence 
\begin{equation}\label{eq:lesf'}
\cdots \to H^i_{ \rm fppf}(\overline{X}_{\bullet}, \Z_p(1)) \to 
H^i_{\et}(\overline{X}_{\bullet}, W\Omega_{\overline{X}_{\bullet}/k}^{\geqslant 1} ) 
\xrightarrow{1-F'} H^i_{\et}(\overline{X}_{\bullet}, W\Omega_{\overline{X}_{\bullet}/k}^{\geqslant 1} ) \to \cdots 
\end{equation}
(cf. \cite[II.\,(5.5.2)]{I_1979}). By the long exact sequence 
$$ 0 \to H^1_{\et}(\overline{X}_{\bullet}, W\Omega^{\geqslant 1}_{\overline{X}_{\bullet}/k}) 
\to H^1_{\rm cris}(\overline{X}_{\bullet}/W(k)) \to 
 H^1_{\et}(\overline{X}_{\bullet}, W{\mathcal{O}}_{\overline{X}_{\bullet}}) \to 
H^2_{\et}(\overline{X}_{\bullet}, W\Omega^{\geqslant 1}_{\overline{X}_{\bullet}/k}) 
\to H^2_{\rm cris}(\overline{X}_{\bullet}/W(k)) $$
associated to the exact sequence 
$$ 0 \to W\Omega^{\geqslant 1}_{\overline{X}_{\bullet}/k} \to W\Omega^{\kr}_{\overline{X}_{\bullet}/k} 
\to W{\mathcal{O}}_{\overline{X}_{\bullet}} \to 0 $$
and Proposition \ref{prop:sl01}, 
we obtain the equality 
$H^1_{\et}(\overline{X}_{\bullet}, W\Omega^{\geqslant 1}_{\overline{X}_{\bullet}/k}) = 
H^0_{\et}(\overline{X}_{\bullet}, W\Omega^{1}_{\overline{X}_{\bullet}/k})$ and the injection 
\begin{equation}\label{eq:ns2}
H^2_{\et}(\overline{X}_{\bullet}, W\Omega^{\geqslant 1}_{\overline{X}_{\bullet}/k}) 
\hookrightarrow H^2_{\rm cris}(\overline{X}_{\bullet}/W(k)). 
\end{equation}
In particular, $H^i_{\et}(\overline{X}_{\bullet}, W\Omega^{\geqslant 1}_{\overline{X}_{\bullet}/k}) \, (i=1,2)$ 
are finite $W(k)$-modules. Then, by \cite[II.\,Lem.\,5.3]{I_1979}, the long exact sequence 
\eqref{eq:lesf'} induces the exact sequence 
\begin{equation}\label{eq:ns3}
0 \to H^{2}_{ \rm fppf}(\overline{X}_{\bullet}, \Z_p(1)) \to 
H^2_{\et}(\overline{X}_{\bullet}, W\Omega_{\overline{X}_{\bullet}/k}^{\geqslant 1} ) 
\xrightarrow{1-F'} H^2_{\et}(\overline{X}_{\bullet}, W\Omega_{\overline{X}_{\bullet}/k}^{\geqslant 1} ) \to 0.
\end{equation}
Note that $H^{2}_{ \rm fppf}(\overline{X}_{\bullet}, \Z_p(1))$ is a finite $\Z_p$-module 
because so are $H^{j}_{ \rm fppf}(\overline{X}_{i}, \Z_p(1))$ for $j \leqslant 2$ by 
\cite[II.\,(5.8.1), (5.8.2), 5.9]{I_1979}. Let 
$H^{2}_{ \rm fppf}(\overline{X}_{\bullet}, \Z_p(1))_{\rm tor}$, 
$H^2_{\et}(\overline{X}_{\bullet}, W\Omega_{\overline{X}_{\bullet}/k}^{\geqslant 1})_{\rm tor}$ be
the subgroup of torsion elements of 
$H^{2}_{\et}(\overline{X}_{\bullet}, \Z_p(1))$, 
$H^2_{\et}(\overline{X}_{\bullet}, W\Omega_{\overline{X}_{\bullet}/k}^{\geqslant 1})$ respectively and 
set 
\begin{align*}
& H^{2}_{ \rm fppf}(\overline{X}_{\bullet}, \Z_p(1))_{\rm fr} := H^{2}_{ \rm fppf}(\overline{X}_{\bullet}, \Z_p(1))/H^{2}_{ \rm fppf}(\overline{X}_{\bullet}, \Z_p(1))_{\rm tor}, \\  
& H^2_{\et}(\overline{X}_{\bullet}, W\Omega_{\overline{X}_{\bullet}/k}^{\geqslant 1})_{\rm fr} := 
H^2_{\et}(\overline{X}_{\bullet}, W\Omega_{\overline{X}_{\bullet}/k}^{\geqslant 1})
/H^2_{\et}(\overline{X}_{\bullet}, W\Omega_{\overline{X}_{\bullet}/k}^{\geqslant 1})_{\rm tor}. 
\end{align*}
Then we have the following commutative diagram with exact rows: 
\begin{equation}\label{eq:torfr}
\xymatrix{
0 \ar[r] & H^{2}_{ \rm fppf}(\overline{X}_{\bullet}, \Z_p(1))_{\rm tor} \ar[d] \ar[r] & 
H^{2}_{ \rm fppf}(\overline{X}_{\bullet}, \Z_p(1)) \ar[d] \ar[r] & 
H^{2}_{ \rm fppf}(\overline{X}_{\bullet}, \Z_p(1))_{\rm fr} \ar[d] \ar[r] & 0 \\ 
0 \ar[r] & 
H^2_{\et}(\overline{X}_{\bullet}, W\Omega_{\overline{X}_{\bullet}/k}^{\geqslant 1})_{\rm tor} \ar[r] & 
H^2_{\et}(\overline{X}_{\bullet}, W\Omega_{\overline{X}_{\bullet}/k}^{\geqslant 1}) \ar[r] & 
H^2_{\et}(\overline{X}_{\bullet}, W\Omega_{\overline{X}_{\bullet}/k}^{\geqslant 1})_{\rm fr} \ar[r] & 0. 
}
\end{equation}
Moreover, we see by \eqref{eq:ns3} that the middle vertical arrow is injective and 
the left square is Cartesian. Hence the right vertical arrow is an injective morphism such that 
the image is contained in the kernel of $1-F'$ on 
$H^2_{\et}(\overline{X}_{\bullet}, W\Omega_{\overline{X}_{\bullet}/k}^{\geqslant 1})_{\rm fr}$. 
Then we see that this induces the injective morphism 
\begin{equation}\label{eq:ns4}
H^{2}_{\mathrm{fppf}}(\overline{X}_{\bullet}, \Z_p(1))_{\rm fr} \otimes_{ \Z_p}  W(k) \hookrightarrow 
H^2_{\et}(\overline{X}_{\bullet}, W\Omega_{\overline{X}_{\bullet}/k}^{\geqslant 1})_{\rm fr}. 
\end{equation}
Indeed, assume there is an element 
$0 \not= \sum_{i=1}^n a_ie_i$ in the kernel 
with $a_i \in W, e_i \in H^{2}_{\mathrm{fppf}}(\overline{X}_{\bullet}, \Z_p(1))_{\rm fr}$ 
such that $e_1, ..., e_n$ are linearly independent over $\Z_p$. 
Take such an element with minimal $n$. Then we may assume that 
$a_1 = p^m$ for some $m$, and then we see that 
$\sum_{i=1}^n (F(a_i) - a_i)e_i = \sum_{i=2}^n (F(a_i) - a_i)e_i$ is also in the kernel. 
By minimality of $n$, we conclude that all the coefficients $F(a_i) - a_i$ should be zero, but this 
contradicts the linear independence of $e_1, ..., e_n$ over $\Z_p$. 

Next we prove that that the left vertical map in \eqref{eq:torfr} induces an isomorphism 
\begin{equation}\label{eq:ns5}
H^{2}_{\mathrm{fppf}}(\overline{X}_{\bullet}, \Z_p(1))_{\rm tor} \otimes_{ \Z_p} W(k) \xrightarrow{\cong} 
H^2_{\et}(\overline{X}_{\bullet}, W\Omega_{\overline{X}_{\bullet}/k}^{\geqslant 1})_{\rm tor}. 
\end{equation}
Let us consider the short exact sequence
\[
	0\to W\Omega^{\geqslant 1}_{\overline{X}_i/k} \xrightarrow{F'} W\Omega^{\geqslant 1}_{\overline{X}_i/k} \to W\Omega^{\geqslant 1}_{\overline{X}_i/k}/F'W\Omega^{\geqslant 1}_{\overline{X}_i/k}\to 0.
\]
Note that the injectivity of $F'$ in this sequence follows from $V\circ F=p$ and the fact that $W\Omega^n_{\overline{X}_i/k}$ is $p$-torsion free.
By \cite[II.\,Lem.\,6.8.2]{I_1979}, the complex of sheaves 
$W\Omega^{\geqslant 1}_{\overline{X}_i/k}/F'W\Omega^{\geqslant 1}_{\overline{X}_i/k}$ has 
zero cohomology in degree $0, 1$ for any $i$. Thus 
$H^1_{\et}(\overline{X}_{\bullet}, W\Omega^{\geqslant 1}_{\overline{X}_{\bullet}/k}/F'W\Omega^{\geqslant 1}_{\overline{X}_{\bullet}/k}) = 0$ 
and so the map 
$$ F': H^2_{\et}(\overline{X}_{\bullet}, W\Omega^{\geqslant 1}_{\overline{X}_{\bullet}/k}) \to 
H^2_{\et}(\overline{X}_{\bullet}, W\Omega^{\geqslant 1}_{\overline{X}_{\bullet}/k}) $$ 
is injective. Because 
$H^2_{\et}(\overline{X}_{\bullet}, W\Omega_{\overline{X}_{\bullet}/k}^{\geqslant 1})_{\rm tor}$ is of finite length, 
$F'$  induces the automorphism 
$$ F': H^2_{\et}(\overline{X}_{\bullet}, W\Omega^{\geqslant 1}_{\overline{X}_{\bullet}/k})_{\rm tor} \to 
H^2_{\et}(\overline{X}_{\bullet}, W\Omega^{\geqslant 1}_{\overline{X}_{\bullet}/k})_{\rm tor}. $$ 
Then the isomorphism \eqref{eq:ns5} follows from \cite[II.\,(6.8.4)]{I_1979}. 
Now, the injection \eqref{eq:ns4} and the isomorphism \eqref{eq:ns5} imply that 
the middle vertical arrow in \eqref{eq:torfr} induces an injection 
$$ 
H^{2}_{\mathrm{fppf}}(\overline{X}_{\bullet}, \Z_p(1))\otimes_{ \Z_p} W \hookrightarrow 
H^2_{\et}(\overline{X}_{\bullet}, W\Omega_{\overline{X}_{\bullet}/k}^{\geqslant 1}). 
$$ 
Combining it with \eqref{eq:ns1} and \eqref{eq:ns2}, 
we obtain the required injection 
$$ {\rm NS}(\overline{X}_{\bullet}) \otimes_{ \Z} W(k) \hookrightarrow H^2_{\rm cris}(\overline{X}_{\bullet}/W(k)), $$
which by definition is equal to the map in \eqref{eq:h2-3}.
\end{proof}

Now we give a proof of Theorem \ref{thm:H^1}.

\begin{proof}[{Proof of Theorem \ref{thm:H^1}}]
Let $(X_{\bullet}, \overline{X}_{\bullet})$ be as in Theorem \ref{thm:H^1}, 
let $(X'_{\bullet}, \overline{X}'_{\bullet})$  be another simplicial normal crossing pair 
satisfying the same assumption. By Proposition \ref{prop:domhypercover'} \ref{prop:domhypercover_ii'} and Corollary \ref{cor:domc},  
it suffices to prove that, when there is a morphism 
$f: (X_{\bullet}, \overline{X}_{\bullet}) \to 
(X'_{\bullet}, \overline{X}'_{\bullet})$, the induced morphism of crystalline cohomologies 
\eqref{eq:indep?} is an isomorphism. Moreover, we may assume that $k$ is algebraically 
closed. 

Note that the morphism $f$ induces the morphism 
\begin{equation}\label{eq:indep??}
H^1_{\rm cris}(\overline{X}_{\bullet}/W(k)) \to 
H^1_{\rm cris}(\overline{X}'_{\bullet}/W(k)). 
\end{equation} 
Moreover, we have a morphism from the exact sequence 
\eqref{eq:slexseq} to the analogous sequence for $\overline{X}'_{\bullet}$. 
By the description of the terms of the exact sequence \eqref{eq:slexseq} 
given in Propositions \ref{prop:h0} and \ref{prop:h1}, the morphism 
\eqref{eq:indep??} is an isomorphism, because 
the map of reduced Picard schemes 
${\rm Pic}^{0,{\rm red}}_{\overline{X}_{\bullet}} \to {\rm Pic}^{0,{\rm red}}_{\overline{X}'_{\bullet}}$ 
is an isomorphism by Corollary \ref{cor:ABV}. 

Next,  note that the morphism $f$ induces the morphism from the exact sequence 
\eqref{eq:exseqh2-3} to the similar exact sequence for 
 $(X'_{\bullet}, \overline{X}'_{\bullet})$. This is an isomorphism because 
the map \eqref{eq:indep??} is an isomorphism and the map of divisor groups 
${\rm Div}^0_{D_{\bullet}}(\overline{X}_{\bullet}) \to 
{\rm Div}^0_{D'_{\bullet}}(\overline{X}'_{\bullet})$ 
 (where $D'_{\bullet} := \overline{X}'_{\bullet} \setminus X'_{\bullet}$)
is an isomorphism by Corollary \ref{cor:ABV}. 
So we have proven that the morphism \eqref{eq:indep?} is an isomorphism. 
\end{proof}

%%%%%%%%%
\subsection{Non-independence of $H^i, i \geqslant 2$}
%%%%%%%%%

In this subsection, we give a negative answer to Question 
\ref{q:2-1} for $H^i, i \geqslant 2$ by providing 
counterexamples. 

Let $k$ be a perfect field of characteristic $p>0$ and set  
$\overline{X} := {\mathbb{P}}^1_k$. 
Denote the coordinate of ${\mathbb{A}}^1_k \subseteq {\mathbb{P}}^1_k = \overline{X}$ 
by $x$. Let $r \geqslant 1$, let $a_1, \dots, a_r$ be distinct elements of $k^\times$ and 
let $n_1, \dots, n_r$ be positive integers prime to $p$. 

Let $f: \overline{X}_0 \to \overline{X}$ be the morphism between proper smooth 
curves over $k$ induced by the field extension 
$$ k(\overline{X}) = k(x) \subseteq k(x)[y]/(y^p-y-\dfrac{x^{\sum_{i=1}^r n_i}}{\prod_{i=1}^r (x-a_i)^{n_i}}) =:  
k(\overline{X}_0). $$
This is a finite flat morphism of degree $p$ 
between proper smooth curves such that $k(\overline{X}_0)/k(\overline{X})$
is a Galois extension with Galois group $G = \langle g \rangle \cong \Z/p\Z$. 
If we denote by $P_i \, (1 \leqslant i \leqslant r)$ 
the closed point of $\overline{X}$ 
defined by $x = a_i$, the ramification locus 
$D \subseteq \overline{X}$ of $f$ is equal to the union of 
$P_i$'s.

For $i \in \N$, let $\overline{X}_i$ be the normalisation of 
$\overline{X}_0 \times_{\overline{X}} \overline{X}_0 \cdots 
\times_{\overline{X}} \overline{X}_0$ of $(i+1)$-copies of 
$\overline{X}_0$ over $\overline{X}$. 
Note that $\overline{X}_i$ is equal to $\coprod^{G^i} \overline{X}_0$. 
Then $\overline{X}_{\bullet}$ forms a simplicial scheme over $\overline{X}$. 
We set $X := \overline{X} \backslash D$, $X_{\bullet} := X \times_{\overline{X}} 
\overline{X}_{\bullet}$. 
Then we obtain the commutative diagram \eqref{eq:2-1} 
and it satisfies the assumption imposed in Question \ref{q:2-1}. 
We prove the following.

\begin{proposition}\label{prop:negative}
With the above notation, the map 
$$ H^2_{\rm cris}((X,\overline{X})/W(k)) \to H^2_{\rm cris}((X_\bullet,\overline{X}_{\bullet})/W(k))$$ 
is not an isomorphism. In particular, 
Question \ref{q:2-1} has a negative answer for $H^2$. 
\end{proposition}

\begin{proof}
We compute $H^2_{\rm cris}((X_{\bullet}, \overline{X}_{\bullet})/W(k))$
by using the spectral sequence 
\begin{equation}\label{eq:desspseq}
E_1^{i.j} = 
H^j_{\rm cris}((X_i, \overline{X}_i)/W(k)) \Longrightarrow  
H^{i+j}_{\rm cris}((X_{\bullet}, \overline{X}_{\bullet})/W(k)). 
\end{equation}
The complex $E_1^{\bullet,j}$ has the form 
$$ \cdots \to H^{j}_{\rm cris}((X_{i-1}, \overline{X}_{i-1})/W(k)) 
\to H^j_{\rm cris}((X_i, \overline{X}_i)/W(k)) 
\to H^{j}_{\rm cris}((X_{i+1}, \overline{X}_{i+1})/W(k)) \to \cdots, $$
which can be rewritten as 
\begin{align*}
\cdots & \to {\rm Hom}_{\rm sets}(G^{i-1}, H^{j}_{\rm cris}((X_0, \overline{X}_0)/W(k))) 
\to {\rm Hom}_{\rm sets}(G^{i}, H^{j}_{\rm cris}((X_0, \overline{X}_0)/W(k))) \\ & 
\to  {\rm Hom}_{\rm sets}(G^{i+1}, H^{j}_{\rm cris}((X_0, \overline{X}_0)/W(k))) 
\to \cdots. 
\end{align*}
Since this complex is the one computing the group cohomology of $G$, 
we have 
\begin{align}
E_2^{i,j} & = H^i(E_1^{\bullet,j}) = H^i(G,  H^{j}_{\rm cris}((X_0, \overline{X}_0)/W(k))) \label{eq:gpcoh} \\ 
& = \begin{cases}
{\rm Ker}(1-g \text{ on } H^{j}_{\rm cris}((X_0, \overline{X}_0)/W(k))), & (i=0) \\ 
{\rm Ker}(\sum_{i=0}^{p-1}g^i \text{ on } H^{j}_{\rm cris}((X_0, \overline{X}_0)/W(k)))/ 
{\rm Im}(1-g \text{ on } H^{j}_{\rm cris}((X_0, \overline{X}_0)/W(k))), & (\text{$i$: odd}) \\
{\rm Ker}(1-g \text{ on } H^{j}_{\rm cris}((X_0, \overline{X}_0)/W(k)))/
{\rm Im}(\sum_{i=0}^{p-1}g^i \text{ on } H^{j}_{\rm cris}((X_0, \overline{X}_0)/W(k))), & 
(\text{$i > 0$: even})
\end{cases} \nonumber 
\end{align}
where the action of $G$ on $H^{j}_{\rm cris}((X_0, \overline{X}_0)/W(k))$ is the one 
induced by the geometric action of $G$ on $(X_0, \overline{X}_0)$. 

When $j=0$, the action of $G$ on $H^0_{\rm cris}((X_0, \overline{X}_0)/W(k)) = W(k)$ is trivial. 
Thus we see from \eqref{eq:gpcoh} that 
\begin{align}
E_2^{i,0} & = H^i(G,  H^0_{\rm cris}((X_0, \overline{X}_0)/W(k))) = \begin{cases}
W(k), & (i=0) \\ 
0,  & (\text{$i$: odd}) \\
k. & (\text{$i > 0$: even})
\end{cases} \label{eq:gpcoh0} 
\end{align}
Then we have the isomorphisms  
$$ W(k)^{r-1} \cong H^1_{\rm cris}((X, \overline{X})/W(k)) 
\xrightarrow{\cong} H^1_{\rm cris}((X_{\bullet}, \overline{X}_{\bullet})/W(k)) 
= E_{\infty}^{0,1} = E_3^{0,1} = {\rm Ker}(E_2^{0,1} \to E_2^{2,0} = k), $$
where the first isomorphism  
follows from the assumption that 
$X$ is equal to $\overline{X} = {\mathbb{P}}^1_k$ minus $r$ points with $r > 0$ 
and the second isomorphism follows from Theorem \ref{thm:H^1}. 
Since $E_2^{0,1} \subseteq H^1_{\rm cris}((X_0, \overline{X}_0)/W(k))$ 
is a finite free $W(k)$-module, we conclude that 
\begin{equation}\label{eq:e201}
{\rm Ker}(1-g \text{ on } H^{ 1}_{\rm cris}((X_0, \overline{X}_0)/W(k))) = E_2^{0,1} \cong W(k)^{r-1},
\end{equation}
and when $r=1$, 
\begin{equation}\label{eq:e20}
E_{\infty}^{2,0} = E_3^{2,0} = {\rm Coker}(E_2^{0,1} \to E_2^{2,0} = k) = k. 
\end{equation}

We estimate the length of 
$$ E_{\infty}^{1,1} = E_2^{1,1} \cong {\rm Ker}(\textstyle\sum_{i=0}^{p-1}g^i)/{\rm Im}(1-g) \text{ on } 
H^1_{\cris}((X_0, \overline{X}_0)/W(k)). $$
First, since $g^p-1 = 0$, 
$\sum_{i=0}^{p-1}g^i$ acts as zero on 
$H^1_{\rm cris}((X_0, \overline{X}_0)/W(k))/{\rm Ker}(1-g)$.
Also, $\sum_{i=0}^{p-1}g^i$ acts as the multiplication by $p$ on ${\rm Ker}(1-g)$, which is injective. 
So, 
by considering the action of $\sum_{i=0}^{p-1}g^i$ on the exact sequence 
\[ 0 \to {\rm Ker}(1-g) \to H^1_{\rm cris}((X_0, \overline{X}_0)/W(k)) \to 
H^1_{\rm cris}((X_0, \overline{X}_0)/W(k))/{\rm Ker}(1-g) \to 0,\]
we obtain the exact sequence 
\[ 0 \to {\rm Ker}(\textstyle\sum_{i=0}^{p-1}g^i) \to 
H^1_{\rm cris}((X_0, \overline{X}_0)/W(k))/{\rm Ker}(1-g) 
\to {\rm Ker}(1-g)/p{\rm Ker}(1-g),\]
and hence the exact sequence 
\begin{equation}\label{eq:e211exseq}
	0 \to \frac{{\rm Ker}(\textstyle\sum_{i=0}^{p-1}g^i)}{{\rm Im}(1-g)} \to 
	\frac{H^1_{\rm cris}((X_0, \overline{X}_0)/W(k))/{\rm Ker}(1-g)}{(1-g)(H^1_{\rm cris}((X_0, \overline{X}_0)/W(k))/{\rm Ker}(1-g))}  
	\to \frac{{\rm Ker}(1-g)}{p{\rm Ker}(1-g)}. 
\end{equation}
Since $H^1_{\rm cris}((X_0, \overline{X}_0)/W(k))$ is a finite free $W(k)$-module, 
we see that 
$ 
H^1_{\rm cris}((X_0, \overline{X}_0)/W(k))/{\rm Ker}(1-g)
$
is also a finite free $W(k)$-module on which 
$\sum_{i=0}^{p-1}g^i$ acts as zero. Thus 
the action of $g$ defines 
a finite free $W(k)[\zeta_p]$-module structure on $H^1_{\rm cris}((X_0, \overline{X}_0)/W(k))/{\rm Ker}(1-g)$ 
(where $\zeta_p$ is a primitive $p$-th root of unity), and its rank as a $W(k)[\zeta_p]$-module is equal to 
$$\left({\rm rk}_{W(k)} H^1_{\rm cris}((X_0, \overline{X}_0)/W(k))/{\rm Ker}(1-g)\right)/(p-1).$$
Hence the length of 
\[ 
\frac{H^1_{\rm cris}((X_0, \overline{X}_0)/W(k))/{\rm Ker}(1-g)}{(1-g)(H^1_{\rm cris}((X_0, \overline{X}_0)/W(k))/{\rm Ker}(1-g))} =   
\frac{H^1_{\rm cris}((X_0, \overline{X}_0)/W(k))/{\rm Ker}(1-g)}{(1-\zeta_p)(H^1_{\rm cris}((X_0, \overline{X}_0)/W(k))/{\rm Ker}(1-g))}  
\]
is equal to 
$$\left({\rm rk}_{W(k)} H^1_{\rm cris}((X_0, \overline{X}_0)/W(k))/{\rm Ker}(1-g)\right)/(p-1).$$
So, by \eqref{eq:e201} and \eqref{eq:e211exseq}, we obtain an estimate of the length of 
$$
E_{\infty}^{1,1} = E_2^{1,1} \cong {\rm Ker}(\textstyle\sum_{i=0}^{p-1}g^i)/{\rm Im}(1-g)$$
in the form of the following inequality:
\begin{align}
{\rm length} E_{\infty}^{1,1} & \geqslant 
\left({\rm rk}_{W(k)} H^1_{\rm cris}((X_0, \overline{X}_0)/W(k))/{\rm Ker}(1-g)\right)/(p-1) 
- (r-1) \label{eq:e11lengthestimate} \\ 
& = \left({\rm rk}_{W(k)} H^1_{\rm cris}((X_0, \overline{X}_0)/W(k)) - r + 1\right)/(p-1) - (r-1). \nonumber
\end{align}
Also, the above inequality is in fact an equality when $r=1$, because the term 
 $\displaystyle\frac{{\rm Ker}(1-g)}{p{\rm Ker}(1-g)}$ in the short exact sequence \eqref{eq:e211exseq} vanishes in this case by \eqref{eq:e201}.
If we denote by 
${\rm Fil}^iH^2_{\rm cris}((X_{\bullet}, \overline{X}_{\bullet})/W(k))$ 
the filtration induced by the spectral sequence 
\eqref{eq:desspseq}, we deduce from \eqref{eq:e20} and \eqref{eq:e11lengthestimate}  
\begin{align}
& {\rm length}( {\rm Fil}^1H^2_{\rm cris}((X_{\bullet}, \overline{X}_{\bullet})/W(k)) ) \geqslant
({\rm rk}_{W(k)} H^1_{\rm cris}((X_0, \overline{X}_0)/W(k)) - r + 1)/(p-1)  -(r-1), \label{eq:length1} \\ 
& {\rm length}( {\rm Fil}^1H^2_{\rm cris}((X_{\bullet}, \overline{X}_{\bullet})/W(k)) ) = 
({\rm rk}_{W(k)} H^1_{\rm cris}((X_0, \overline{X}_0)/W(k)))/(p-1) + 1 \quad \text{if $r=1$}. 
\label{eq:length1-1}
\end{align}
Let $g_0$ be the genus of $\overline{X}_0$. Then, 
by the Riemann--Hurwitz formula, we have the equality 
$$ 2g_0 - 2 = -2p + \sum_{i=1}^r (n_i+1)(p-1), $$
hence $2g_0 = (p-1)\left( \sum_{i=1}^r (n_i+1) - 2 \right)$. 
Hence 
$$ {\rm rk}_{W(k)} H^1_{\rm cris}((X_0, \overline{X}_0)/W(k))) = (p-1)\left( \sum_{i=1}^r (n_i+1) - 2 \right) 
+r-1 $$ 
and therefore 
\eqref{eq:length1}, \eqref{eq:length1-1} are rewritten as 
\begin{align}
& {\rm length}\left( {\rm Fil}^1H^2_{\rm cris}((X_{\bullet}, \overline{X}_{\bullet})/W(k)) \right)
 \geqslant \sum_{i=1}^r (n_i+1) - 2  - (r-1) \geqslant r-1, \label{eq:length2}
\\ 
& {\rm length}\left( {\rm Fil}^1H^2_{\rm cris}((X_{\bullet}, \overline{X}_{\bullet})/W(k)) \right) = 
(n_1+1) - 1 \geqslant 1 \quad \text{if $r=1$.} 
\label{eq:length2-1}
\end{align}
By \eqref{eq:length2} and \eqref{eq:length2-1}, we conclude that 
$H^2_{\rm cris}((X_\bullet,\overline{X}_{\bullet})/W(k)) \not= 0.$ 
Since $H^2_{\rm cris}((X,\overline{X})/W(k))$ is zero and 
 $H^2_{\rm cris}((X_\bullet,\overline{X}_{\bullet})/W(k))$ is non-zero, 
the map 
$$ H^2_{\rm cris}((X,\overline{X})/W(k)) \to H^2_{\rm cris}((X_\bullet,\overline{X}_{\bullet})/W(k))$$ 
is not an isomorphism. 
\end{proof}

\begin{remark}\label{rem:etexplicit}
The above proof shows that there does not exist a functor 
$$ A_{\et}^{\kr}: \Sm_k \rightarrow C^{\geqslant 0}(W(k)),$$
where $C^{\geqslant 0}(W(k))$ denotes again the category of complexes of $W(k)$-modules of non-negative degree, 
satisfying the following conditions (compare with Remark \ref{rem:rhexplicit}): 
\begin{enumerate}
\item\label{rem:etexplicit (i)}
It gives a good integral $p$-adic cohomology theory in the sense that,  
for any $nc$-pair $(X, \overline{X})$, there exists a functorial quasi-isomorphism 
$$A^{\kr}_{\et}(X) \simeq R\Gamma_{\cris}((X,\overline{X})/W(k)).$$ 
\item\label{rem:etexplicit (ii)}
It satisfies Galois descent in the sense that, for any \v{C}ech hypercovering 
$X_{\bullet} \rightarrow X$ associated to a finite \'etale Galois covering 
$X_0 \rightarrow X$, 
the induced morphism 
$$A^{\kr}_{\et}(X) \rightarrow A^{\kr}_{\et}(X_{\bullet})$$ 
is a quasi-isomorphism. 
\end{enumerate}
Indeed, for the hypercovering $X_{\bullet} \to X$ in the above proof, 
if the map $H^1(A^{\kr}_{\et}(X)) \rightarrow H^1(A^{\kr}_{\et}(X_{\bullet}))$ is not an isomorphism, 
it violates the condition \ref{rem:etexplicit (ii)}, and if it is an isomorphism, the argument of the above proof 
works for the spectral sequence 
$$ E_1^{i,j} =  H^{j}(A^{\kr}_{\et}(X_{i})) \, \Longrightarrow \, H^{i+j}(A^{\kr}_{\et}(X_{\bullet}))$$ 
and we conclude that $H^2(A^{\kr}_{\et}(X)) \rightarrow H^2(A^{\kr}_{\et}(X_{\bullet}))$ is not an isomorphism, 
which again violates the condition \ref{rem:etexplicit (ii)}. 
\end{remark}

\begin{remark}
In the above remark, if we assume moreover 
that $A^{\kr}_{\et}(X_0)$ is a perfect complex of 
$W(k)[G]$-modules (where $G$ is the Galois group of $X_0$ over $X$), 
one can prove the nonexistence of the functor $A^{\kr}_{\et}$ in a simpler way,  which is due to Crew:\\
For the hypercovering $X_{\bullet} \to X$ in the proof of 
Proposition \ref{prop:negative}, $G$ is equal to $\Z/p\Z$. In this case, 
$W(k)[G]$ is a local ring and so every finitely generated projective module is free
(see \cite[Lem.\,1.6]{C_1984}).  Thus $A^{\kr}_{\et}(X_0)$ is quasi-isomorphic to 
a bounded complex of finitely generated free $W(k)[G]$-modules. 
Then, by comparing the cohomologies of $A^{\kr}_{\et}(X_0)  \otimes_{\Z} \Q$ and those of 
$$ A^{\kr}_{\et}(X) \otimes_{\Z} \Q \simeq A^{\kr}_{\et}(X_{\bullet})  \otimes_{\Z} \Q = 
R\Gamma(G, A^{\kr}_{\et}(X_0)  \otimes_{\Z} \Q) = \Gamma(G, A^{\kr}_{\et}(X_0) \otimes_{\Z} \Q), $$
we obtain the equality 
$$ \sum_i (-1)^i \dim_K (H^i(A^{\kr}_{\et}(X_0))  \otimes_{\Z} \Q) = 
p \left(\sum_i (-1)^i \dim_K ( H^i(A^{\kr}_{\et}(X))  \otimes_{\Z} \Q) \right) $$
(see \cite[Thm.\,1.5, Cor.\,1.7]{C_1984}). But this is absurd because the  right hand side 
is equal to 
$$p\left(\sum_i (-1)^i \dim (H^i_{\cris}((X,\overline{X})/W(k))  \otimes_{\Z} \Q) \right) = p(2-r), $$ 
while the left hand side, which is equal to 
$\sum_i (-1)^i \dim (H^i_{\cris}((X_0,\overline{X}_0)/W(k))  \otimes_{\Z} \Q)$, 
can be estimated by
$$ (2-r)-(p-1)(\sum_{i=1}^r(n_i+1) - 2) \leqslant
(2-r)-(p-1)(2r-2) < p(2-r). $$
\end{remark}

\begin{remark}
Let us observe that the above example $(X_{\bullet},\overline{X}_\bullet)\to (X,\overline{X})$ with $r=1$ also shows that the independence of $H^1$ fails if we take mod $p^n$ coefficients. Indeed, we claim that the map 
\[
	H^1_{\rm cris}((X,\overline{X})/W_n(k)) \to H^1_{\rm cris}((X_\bullet,\overline{X}_{\bullet})/W_n(k))
\]
is not an isomorphism. As in the proof of Proposition \ref{prop:negative}, we use the spectral sequence
\[
E_1^{i.j} = 
H^j_{\rm cris}((X_i, \overline{X}_i)/W_n(k)) \Longrightarrow  
H^{i+j}_{\rm cris}((X_{\bullet}, \overline{X}_{\bullet})/W_n(k))
\]
to compute $H^1_{\rm cris}((X_\bullet,\overline{X}_{\bullet})/W_n(k))$. For $j=0$, we get from \eqref{eq:gpcoh}
\begin{equation*}
E_2^{i,0}  = H^i(G,  H^0_{\rm cris}((X_0, \overline{X}_0)/W_n(k)))= k, \quad (i> 0).
\end{equation*}
Hence $H^1_{\rm cris}((X_\bullet,\overline{X}_{\bullet})/W_n(k)) \supseteq E_{\infty}^{1,0} = E_2^{1,0} \not= 0$. 
On the other hand, $H^1_{\rm cris}((X,\overline{X})/W_n(k)) = 0$ because 
$(X,\overline{X}) = ({\mathbb{A}}^1_k, {\mathbb{P}}^1_k)$. This 
proves the non-independence of $H^1$ for mod $p^n$ coefficients.
\end{remark}

We can extend the above counterexample easily to 
the case of $H^i \, (i \geqslant 2)$.

\begin{corollary}\label{cor:negativehighercoh}
For any $i \geqslant 2$, Question 
\ref{q:2-1} has a negative answer for $H^i$. 
\end{corollary}

\begin{proof}
Let $(X_{\bullet}, \overline{X}_{\bullet}) \to (X, \overline{X})$, $D_{\bullet}, D$ 
be as in Proposition \ref{prop:negative}. 
Let $Y$ be a proper smooth connected variety over $k$ of dimension $d$ 
such that, for any $0 \leqslant n \leqslant 2d$, 
$H^n_{\rm cris}(Y/W(k))$ is a non-zero, finite free $W(k)$-module. 
(For example, one can take as $Y$ a $d$-fold fibre product of a proper smooth 
curve of positive genus 
over $k$.) 
Then we have isomorphisms  
\begin{align*}
& H^i_{\rm cris}((X \times_k Y,\overline{X} \times_k Y)/W(k)) \cong \bigoplus_{a+b=i} H^a_{\rm cris}((X,\overline{X})/W(k)) \otimes_{W(k)} H^b_{\rm cris}(Y/W(k)), \\ 
& 
H^i_{\rm cris}((X_{\bullet} \times_k Y, \overline{X}_{\bullet} \times_k Y)/W(k)) \cong \bigoplus_{a+b=i} H^a_{\rm cris}((X_\bullet, \overline{X}_{\bullet})/W(k)) \otimes_{W(k)} 
H^b_{\rm cris}(Y/W(k)). 
\end{align*} 
Consider the morphism $(X_{\bullet} \times_k Y, \overline{X}_{\bullet} \times_k Y) \to 
(X \times_k Y, \overline{X} \times_k Y)$.
Then 
the induced morphism 
$$H^i_{\rm cris}((X \times_k Y,\overline{X} \times_k Y)/W(k)) \to 
H^i_{\rm cris}((X_{\bullet} \times_k Y,\overline{X}_{\bullet} \times_k Y)/W(k)) $$
is not an isomorphism for $2 \leqslant i \leqslant 2(d+1)$, because one of its direct factor 
$$ 
H^2_{\rm cris}((X,\overline{X})/W(k)) \otimes_{W(k)} 
H^{i-2}_{\rm cris}(Y/W(k)) \to 
H^2_{\rm cris}((X_\bullet,\overline{X}_{\bullet})/W(k)) \otimes_{W(k)} 
H^{i-2}_{\rm cris}(Y/W(k))
$$ is not an isomorphism by Proposition \ref{prop:negative}. So we are done. 
\end{proof}

%%%%%%
\appendix

%%%%%%
\section{The first crystalline cohomology group and the Picard $1$-motive}\label{appendixA}
%%%%%%

As for the comparison of the first log crystalline cohomology and 
the Dieudonn\'e module of the associated $1$-motive, we have the 
following result mentioned in Remark \ref{rem: ABV2}.

\begin{theorem}
Let $(X_{\bullet}, \overline{X}_{\bullet})$ be a split simplicial proper strict normal crossing 
pair over $k$. 
Then there exists 
a functorial isomorphism 
$${\bold T}_{\rm cris}({\rm Pic}^+(X_{\bullet}, \overline{X}_{\bullet})) \xrightarrow{\cong}
H^1_{\rm cris}((X_\bullet,\overline{X}_{\bullet})/W(k))$$ which is compatible with 
weight filtration and Frobenius action, where ${\bold T}_{\rm cris}$ 
is the covariant Dieudonn\'e functor \textup{(}the contravariant Dieudonn\'e functor of 
the $p$-divisible group associated to the dual\textup{)}. 
\end{theorem}

\begin{remark}
Compared to Theorem \ref{thm:ABV2} of Andreatta--Barbieri-Viale, our result includes the case $p=2$.
\end{remark}

\begin{proof} 

We denote the covariant Dieudonn\'e functor for $p$-divisible groups 
also by ${\bold T}_{\rm cris}$. 
The natural connected-étale split exact sequence of $p$-divisible groups
$$
0 \rightarrow {\rm Pic}^{0,{\rm red}}_{\overline{X}_{\bullet}}[p^{\infty}]^0 \rightarrow 
{\rm Pic}^{0,{\rm red}}_{\overline{X}_{\bullet}}[p^{\infty}] \rightarrow
{\rm Pic}^{0,{\rm red}}_{\overline{X}_{\bullet}}[p^{\infty}]^{\rm et} \rightarrow 0
$$
induces a similar split exact sequence on the level of Dieudonné modules.
In particular,
${\bold T}_{\rm cris}({\rm Pic}^{0,{\rm red}}_{\overline{X}_{\bullet}}[p^{\infty}]) $ 
 can be written as a direct sum
\begin{equation}\label{eq:tt1}
{\bold T}_{\rm cris}({\rm Pic}^{0,{\rm red}}_{\overline{X}_{\bullet}}[p^{\infty}]) \cong
{\bold T}_{\rm cris}({\rm Pic}^{0,{\rm red}}_{\overline{X}_{\bullet}}[p^{\infty}]^0) \oplus {\bold T}_{\rm cris}({\rm Pic}^{0,{\rm red}}_{\overline{X}_{\bullet}}[p^{\infty}]^{\rm et})
\end{equation}
of $F$-crystals.  Note that the expression \eqref{eq:tt1} as direct sum is unique 
because the slopes of the direct summands are distinct.
On the other hand, by Proposition 
\ref{prop:sl01}, we have a short exact sequence 
\begin{equation}\label{eq:tt2}
0 \to H^0_{\et}(\overline{X}_{\bullet}, W\Omega^1_{\overline{X}_{\bullet}/k}) 
\to H^1_{\rm cris}(\overline{X}_{\bullet}/W) \to 
H^1_{\et}(\overline{X}_{\bullet}, W{\mathcal{O}}_{\overline{X}_{\bullet}})
\to 0. 
\end{equation}

We prove that \eqref{eq:tt2} is also a split exact sequence of $F$-crystals. 
Let $ZW\Omega^1_{\overline{X}_{\bullet}/k}$ be the kernel of the map 
$d: W\Omega^1_{\overline{X}_{\bullet}/k} \to W\Omega^2_{\overline{X}_{\bullet}/k}$, and let 
$\tau_{\leqslant 1}W\Omega^{\kr}_{\overline{X}_{\bullet}/k}$ be the complex 
$W{\mathcal{O}}_{\overline{X}_{\bullet}} \to ZW\Omega^1_{\overline{X}_{\bullet}/k}$, where 
$W{\mathcal{O}}_{\overline{X}_{\bullet}}$ is at degree $0$. Then 
we have the commutative diagram 
\begin{equation}\label{eq:spl1}
\xymatrix{
& H^0_{\et}(\overline{X}_{\bullet},ZW\Omega^1_{\overline{X}_{\bullet}/k}) 
\ar[r] \ar[d] 
& H^1_{\et}(\overline{X}_{\bullet}, \tau_{\leqslant 1}W\Omega^{\kr}_{\overline{X}_{\bullet}/k}) 
\ar[r] \ar[d] & 
H^1_{\et}(\overline{X}_{\bullet}, W{\mathcal{O}}_{\overline{X}_{\bullet}}) 
\ar@{=}[d] \\ 
0 \ar[r] & H^0_{\et}(\overline{X}_{\bullet}, W\Omega^1_{\overline{X}_{\bullet}/k}) 
\ar[r] & H^1_{\rm cris}(\overline{X}_{\bullet}/W) \ar[r] & 
H^1_{\et}(\overline{X}_{\bullet}, W{\mathcal{O}}_{\overline{X}_{\bullet}})
\ar[r] & 0}
\end{equation}
where the horizontal lines are exact. Since 
$H^0_{\et}(\overline{X}_{\bullet}, W\Omega^1_{\overline{X}_{\bullet}/k}) = 
E^{1,0}_{\infty}$
in the spectral sequence \eqref{eq:slspseq} by Proposition 
\ref{prop:sl01}, the left vertical arrow in \eqref{eq:spl1} is an isomorphism. 
Then a diagram chase shows that the exact sequence \eqref{eq:tt2} can be identified with the 
exact sequence 
\begin{equation*}
0 \to 
H^0_{\et}(\overline{X}_{\bullet},ZW\Omega^1_{\overline{X}_{\bullet}/k}) 
\overset{\iota}{\to} 
H^1_{\et}(\overline{X}_{\bullet}, \tau_{\leqslant 1}W\Omega^{\kr}_{\overline{X}_{\bullet}/k}) 
\overset{\pi}{\to} 
H^1_{\et}(\overline{X}_{\bullet}, W{\mathcal{O}}_{\overline{X}_{\bullet}}) 
\to 0. 
\end{equation*}
Let $V':  \tau_{\leqslant 1}W\Omega^{\kr}_{\overline{X}_{\bullet}/k} \to \tau_{\leqslant 1}W\Omega^{\kr}_{\overline{X}_{\bullet}/k}$ 
be the map of complexes of sheaves on $\overline{X}_{\bullet}$ defined in
\cite[III (1.7.1)]{IR_1983}: It is induced by the map 
$V: W{\mathcal{O}}_{\overline{X}_{\bullet}} \to W{\mathcal{O}}_{\overline{X}_{\bullet}} $ and the automorphism  
$V': ZW\Omega^1_{\overline{X}_{\bullet}/k} \to ZW\Omega^1_{\overline{X}_{\bullet}/k}$ defined in 
\cite[III (1.3.2)]{IR_1983}. Also, take an increasing sequence $(N_n)_n$ in $\N$ with $V^{N_n}
H^1_{\et}(\overline{X}_{\bullet}, W{\mathcal{O}}_{\overline{X}_{\bullet}}) \subseteq 
p^nH^1_{\et}(\overline{X}_{\bullet}, W{\mathcal{O}}_{\overline{X}_{\bullet}})$ 
(such a sequence exists by Remark \ref{rem:topH1WO-2}). Then, for an element $a$ in 
$H^1_{\et}(\overline{X}_{\bullet}, \tau_{\leqslant 1}W\Omega^{\kr}_{\overline{X}_{\bullet}/k})$ and $n \in \N$,  
$\pi((V')^{N_n}a) = V^{N_n}\pi(a)\in V^{N_n}H^1_{\et}(\overline{X}_{\bullet}, W{\mathcal{O}}_{\overline{X}_{\bullet}})
\subseteq p^nH^1_{\et}(\overline{X}_{\bullet}, W{\mathcal{O}}_{\overline{X}_{\bullet}})$ and so 
$(V')^{N_n}a \in H^0_{\et}(\overline{X}_{\bullet},ZW\Omega^1_{\overline{X}_{\bullet}/k}) + 
p^n H^1_{\et}(\overline{X}_{\bullet}, \tau_{\leqslant 1}W\Omega^{\kr}_{\overline{X}_{\bullet}/k})$. 
Since $V'$ is an automorphism on $H^0_{\et}(\overline{X}_{\bullet},ZW\Omega^1_{\overline{X}_{\bullet}/k})$, 
we can find an element $b_n$ of $H^0_{\et}(\overline{X}_{\bullet},ZW\Omega^1_{\overline{X}_{\bullet}/k})$ 
such that 
$(V')^{N_n}a - (V')^{N_n}b_n \in p^n H^1_{\et}(\overline{X}_{\bullet}, \tau_{\leqslant 1}W\Omega^{\kr}_{\overline{X}_{\bullet}/k})$. 
Then we see that $(V')^{N_{n+1}}(b_{n+1}-b_n) 
\in 
p^n H^1_{\et}(\overline{X}_{\bullet}, \tau_{\leqslant 1}W\Omega^{\kr}_{\overline{X}_{\bullet}/k}) \cap 
H^0_{\et}(\overline{X}_{\bullet},ZW\Omega^1_{\overline{X}_{\bullet}/k}) 
= p^n H^0_{\et}(\overline{X}_{\bullet},ZW\Omega^1_{\overline{X}_{\bullet}/k}) 
$ and so $b_{n+1}-b_n \in p^n H^0_{\et}(\overline{X}_{\bullet},ZW\Omega^1_{\overline{X}_{\bullet}/k})$. 
Thus $b = \lim_n b_n$ is defined, and one can check that it is independent of the choice of $b_n$'s. Also, 
one sees easily that the map $a \mapsto b$ defines a section of $\iota$. Moreover, 
using the relation $FV' = V'F = p$ \cite[III (1.7.7)]{IR_1983}, we can check that this section is compatible with $F$. 
So we have proven that \eqref{eq:tt2} is a split exact sequence of $F$-crystals, thus we have the decomposition 
\begin{equation}\label{eq:spl2}
H^1_{\rm cris}(\overline{X}_{\bullet}/W) = 
H^1_{\et}(\overline{X}_{\bullet}, W{\mathcal{O}}_{\overline{X}_{\bullet}}) \oplus 
 H^0_{\et}(\overline{X}_{\bullet}, W\Omega^1_{\overline{X}_{\bullet}/k})
\end{equation}
as $F$-crystals. The decomposition \eqref{eq:spl2} is also unique by the slope reason.

Denote the base change of 
$\overline{X}_{\bullet}$ 
to $\overline{k}$ 
by  $\overline{X}'_{\bullet}$
and denote $W(\overline{k})$ 
simply by $W'$. 
By Propositions \ref{prop:h0} and 
\ref{prop:h1},  ${\bold T}_{\rm cris}({\rm Pic}^{0,{\rm red}}_{\overline{X}'_{\bullet}}[p^{\infty}]^{\rm et})$ and ${\bold T}_{\rm cris}({\rm Pic}^{0,{\rm red}}_{\overline{X}'_{\bullet}}[p^{\infty}]^0)$ 
are isomorphic to 
$H^0_{\et}(\overline{X}'_{\bullet}, W\Omega^1_{\overline{X}'_{\bullet}/\overline{k}})$ and $
H^1_{\et}(\overline{X}'_{\bullet}, W{\mathcal{O}}_{\overline{X}'_{\bullet}})$  
respectively, and these isomorphisms are compatible with the action of 
${\rm Gal}(\overline{k}/k)$. 
Thus the isomorphisms descent to the isomorphisms 
$$ 
{\bold T}_{\rm cris}({\rm Pic}^{0,{\rm red}}_{\overline{X}_{\bullet}}[p^{\infty}]^0) 
\cong H^1_{\et}(\overline{X}_{\bullet}, W{\mathcal{O}}_{\overline{X}_{\bullet}}), \quad 
{\bold T}_{\rm cris}({\rm Pic}^{0,{\rm red}}_{\overline{X}_{\bullet}}[p^{\infty}]^{\rm et})
\cong H^0_{\et}(\overline{X}_{\bullet}, W\Omega^1_{\overline{X}_{\bullet}/k}). 
$$

Taking into account the decompositions \eqref{eq:tt1}  and \eqref{eq:spl2},  
we have the canonical isomorphism 
\begin{equation}\label{eq:tt4}
{\bold T}_{\rm cris}({\rm Pic}^{0,{\rm red}}_{\overline{X}_{\bullet}}) 
\cong 
H^1_{\rm cris}(\overline{X}_{\bullet}/W(k)).  
\end{equation}
Next, by definition of the $1$-motive and its Dieudonn\'e module (see \cite{ABV_2005}), 
we have an exact sequence 
\begin{equation}\label{eq:tt5}
0 \to 
{\bold T}_{\rm cris}({\rm Pic}^{0,{\rm red}}_{\overline{X}_{\bullet}}) 
\to 
{\bold T}_{\rm cris}({\rm Pic}^+(X_{\bullet}, \overline{X}_{\bullet})) 
\to 
{\bold T}_{\rm cris}({\rm Div}^0_{D_{\bullet}}(\overline{X}_{\bullet}))  
\to 0. 
\end{equation}
Furthermore, by the proof of Proposition \ref{prop: exact div}
we have the exact sequence
\begin{equation}\label{eq:logcrys-ses}
 0 \rightarrow H^1_{\rm crys}(\overline{X}_{\bullet}/W(k)) \rightarrow 
H^1_{\rm crys}((X_{\bullet}, \overline{X}_{\bullet})/W(k))  
\rightarrow 
H^0_{\et}(\overline{X}_{\bullet}, a^{(1)}_{\bullet *}W\Omega^{\kr}_{D^{(1)}_\bullet/k})^0 
\rightarrow 0, 
\end{equation}
where $H^0_{\et}(\overline{X}_{\bullet}, a^{(1)}_{\bullet *}W\Omega^{\kr}_{D^{(1)}_\bullet/k})^0 := {\rm Ker}(
 H^0_{\et}(\overline{X}_{\bullet}, a^{(1)}_{\bullet *}W\Omega^{\kr}_{D^{(1)}_\bullet/k}) \rightarrow H^2_{\et}(\overline{X}_{\bullet}, W\Omega^{\kr}_{\overline{X}_{\bullet}/k}))$. 
After extension to $W'$, the sequence \eqref{eq:logcrys-ses} is canonically identical to 
\eqref{eq:exseqh2-3} (with $k$ replaced by $\overline{k}$ and 
$X_{\bullet}, \overline{X}_{\bullet}$ replaced by the base changes $X'_{\bullet}, \overline{X}'_{\bullet}$ 
of $X_{\bullet}, \overline{X}_{\bullet}$ here to 
$\overline{k}$) by the proof of Proposition \ref{prop: exact div} and
Lemma \ref{lem:nsinj}, in the sense that the identification 
$ H^0_{\et}(\overline{X}_{\bullet}, a^{(1)}_{\bullet *}W\Omega^{\kr}_{D^{(1)}_\bullet/k})^0 \otimes_{W(k)} W' \cong 
 {\rm Div}^0_{D_{\bullet}}(\overline{X}_{\bullet}) \otimes W'$ is compatible with the action of 
 the Galois group ${\rm Gal}(\overline{k}/k)$. We see from this fact that there is an isomorphism 
\begin{equation}\label{eq:divpartisom}
{\bold T}_{\rm cris}({\rm Div}^0_{D_{\bullet}}(\overline{X}_{\bullet})) \cong
H^0_{\et}(\overline{X}_{\bullet}, a^{(1)}_{\bullet *}W\Omega^{\kr}_{D_\bullet/k})^0 
\end{equation}
of $F$-crystals. To prove the theorem, we need to upgrade canonically the isomorphisms 
\eqref{eq:tt4} and \eqref{eq:divpartisom} to an isomorphism between the exact sequences 
\eqref{eq:tt5}, \eqref{eq:logcrys-ses} of $F$-crystals.

To do so, first we recall the definition of the exact sequence \eqref{eq:tt5} in detail, following 
\cite[\S\,1.3]{ABV_2005}. Put ${\mathbb{G}} := {\rm Pic}^{0,{\rm red}}_{\overline{X}_{\bullet}}$, 
${\mathbb{X}} := {\rm Div}^0_{D_{\bullet}}(\overline{X}_{\bullet})$, 
${\rm Pic}^+(X_{\bullet}, \overline{X}_{\bullet}) := [{\mathbb{X}} \overset{u}{\to} {\mathbb{G}}]$, and 
for $n \in \N$, define ${\mathbb{M}}[p^n]$ by 
$$ 
{\mathbb{M}}[p^n] := 
\frac{{\rm Ker}(u+p^n: {\mathbb{X}} \times_k {\mathbb{G}} \to {\mathbb{G}})}{{\rm Im}((p^n,-u):{\mathbb{X}} \to {\mathbb{X}} \times_k {\mathbb{G}})}.  
$$
Then there exists an exact sequence 
\begin{equation}\label{eq:dieudonne-n}
0 \to {\mathbb{G}}[p^n] \to {\mathbb{M}}[p^n] \to {\mathbb{X}}/p^n{\mathbb{X}} \to 0 
\end{equation}
in which the maps are defined by $g \mapsto (0,g)$, $(x,g) \mapsto x$. By identifying 
${\mathbb{X}}/p^n{\mathbb{X}}$ with $\frac{1}{p^n}{\mathbb{X}}/{\mathbb{X}}$ and taking the inductive limit 
with respect to $n$, we obtain an exact sequence of $p$-divisible groups 
\begin{equation}\label{eq:dieudonne-infty}
0 \to {\mathbb{G}}[p^{\infty}] \to {\mathbb{M}}[p^{\infty}] \to {\mathbb{X}} \otimes_{\Z} \Q_p/\Z_p \to 0. 
\end{equation}
By applying the covariant Dieudonn\'e functor to \eqref{eq:dieudonne-infty}, we obtain 
the exact sequence \eqref{eq:tt5}. 

Let ${\mathbb{X}}', {\mathbb{G}}', {\mathbb{M}}'[p^n]$ be the base change of 
${\mathbb{X}}, {\mathbb{G}}, {\mathbb{M}}[p^n]$ to $\overline{k}$ respectively. Then we can take a 
compatible family of 
homomorphisms
$\widetilde{u}:=\{\widetilde{u}_n: \frac{1}{p^n}{\mathbb{X}}' \to {\mathbb{G}}'\}_n$ which extends the base change 
${\mathbb{X}}' \to {\mathbb{G}}'$ of $u$. Then the map 
$$ \frac{1}{p^n}{\mathbb{X}}'/{\mathbb{X}}' \cong 
{\mathbb{X}}'/p^n{\mathbb{X}}' \to 
{\mathbb{M}}'; \quad x \mapsto p^nx \mapsto (p^nx, -\widetilde{u}(x)) $$
gives a compatible family of splittings of the exact sequences 
\eqref{eq:dieudonne-n} for $n \in \N$ over $\overline{k}$, hence a splitting 
$s: {\mathbb{X}}' \otimes \Q_p/\Z_p \to {\mathbb{M}}'[p^{\infty}]$  
of the exact sequence \eqref{eq:dieudonne-infty} over $\overline{k}$. 
Hence the exact sequence \eqref{eq:dieudonne-infty} is given by the $1$-cocycle
$c \in Z^1({\rm Gal}(\overline{k}/k), 
{\rm Hom}({\mathbb{X}}' \otimes \Q_p/\Z_p, \mathbb{G}'[p^{\infty}]))$ defined by 
$\sigma \mapsto \sigma(s) - s = - \sigma(\widetilde{u}) + \widetilde{u}$. 

Noting that 
${\rm Hom}({\mathbb{X}}' \otimes \Q_p/\Z_p, \mathbb{G}'[p^{\infty}]) = 
{\rm Hom}({\mathbb{X}}' \otimes \Q_p/\Z_p, \mathbb{G}'[p^{\infty}]^{\rm et})
$, $c$ is in fact an element in 
$Z^1({\rm Gal}(\overline{k}/k), 
{\rm Hom}({\mathbb{X}}' \otimes \Q_p/\Z_p, \mathbb{G}'[p^{\infty}]^{\rm et}))$. 
Hence the extension \eqref{eq:dieudonne-infty} is the push-out of the exact sequence  
\begin{equation}\label{eq:dieudonne-infty-et}
	0 \to {\mathbb{G}}[p^{\infty}]^{\rm et} 
	\to E \to {\mathbb{X}} \otimes_{\Z} \Q_p/\Z_p \to 0
\end{equation}
induced by $c$ by ${\mathbb{G}}[p^{\infty}]^{\rm et} \to 
{\mathbb{G}}[p^{\infty}]$ {where for simplicity $E$ denotes the extension of ${\mathbb{G}}[p^{\infty}]^{\rm et}$ by $ {\mathbb{X}} \otimes_{\Z} \Q_p/\Z_p$}.

Recall that, over $\overline{k}$, the covariant Dieudonn\'e functor 
for \'etale $p$-divisible groups is given by 
$$ G \mapsto {\rm Hom}(G^{\vee}, {\mathbb{G}}_m) \otimes_{\Z_p} W'
=  {\rm Hom}(G^{\vee}, \mu_{p^{\infty}}) \otimes_{\Z_p} W' 
= {\rm Hom}(\Q_p/\Z_p, G) \otimes_{\Z_p} W' $$
and over $k$, it is defined by taking the ${\rm Gal}(\overline{k}/k)$-invariant 
part  of this module. Let us consider the exact sequence 
\begin{equation}\label{eq:dieudonne-infty-et-qpzp}
	0 \to {\rm Hom}(\Q_p/\Z_p, {\mathbb{G}}'[p^{\infty}]^{\rm et}) 
	\to {\rm Hom}(\Q_p/\Z_p, E') \to 
	{\mathbb{X}}' \otimes_{\Z} \Z_p \to 0
\end{equation}
(where $E'$ is the base change of $E$ to $\overline{k}$) 
endowed with ${\rm Gal}(\overline{k}/k)$-action which 
we obtain by applying ${\rm Hom}(\Q_p/\Z_p, -)$ to the base change of 
\eqref{eq:dieudonne-infty-et} to $\overline{k}$. 
This corresponds to the $1$-cocycle $\widetilde{c} := {\rm Hom}(\Q_p/\Z_p, c)$ 
in 
\begin{align*}
	& Z^1({\rm Gal}(\overline{k}/k), 
	{\rm Hom}({\rm Hom}(\Q_p/\Z_p, {\mathbb{X}}' \otimes \Q_p/\Z_p), 
	{\rm Hom}(\Q_p/\Z_p, {\mathbb{G}}'[p^{\infty}])) \\ 
	& = 
	Z^1({\rm Gal}(\overline{k}/k), 
	{\rm Hom}({\mathbb{X}}' \otimes \Z_p, 
	\varprojlim_n{\mathbb{G}}'[p^n])).
\end{align*}
Then we see that we can obtain the exact sequence \eqref{eq:tt5} of Dieudonné modules by applying 
$- \otimes_{\Z_p} W'$ to the sequence \eqref{eq:dieudonne-infty-et-qpzp},  
taking ${\rm Gal}(\overline{k}/k)$-invariants and then pushing it out by 
$$ 
({\rm Hom}(\Q_p/\Z_p, {\mathbb{G}}'[p^{\infty}]^{\rm et}) \otimes_{\Z_p} W')^{{\rm Gal}(\overline{k}/k)}
= {\bold T}_{\rm cris}({\rm Pic}^{0,{\rm red}}_{\overline{X}_{\bullet}}[p^{\infty}]^{\rm et}) 
\to {\bold T}_{\rm cris}({\rm Pic}^{0,{\rm red}}_{\overline{X}_{\bullet}}). $$ 

Next we study the exact sequence \eqref{eq:logcrys-ses} in detail.  
It follows from the fact that $W\Omega^0_{\overline{X}'_\bullet/\overline{k}}= W\Omega^0_{\overline{X}'_\bullet/\overline{k}}(\log D'_\bullet)$ 
(see the proof of  Lemma \ref{lem:nsinj}, especially \eqref{eq:drwsimpsheaf}) and \eqref{eq:ns2}
that we have a diagram of short exact sequences 
\begin{equation}\label{eq:diag-logcrys}
\xymatrix{
	0 \ar[r]& 
	H^1_{\et}(\overline{X}'_{\bullet},W\Omega^{\geqslant 1}_{\overline{X}'_\bullet/\overline{k}}) \ar@{^{(}->}[d] \ar[r] & 
	H^1_{\et}(\overline{X}'_{\bullet},W\Omega^{\geqslant 1}_{\overline{X}'_\bullet/\overline{k}}(\log D'_\bullet)) \ar@{^{(}->}[d] \ar[r] &
	{\rm Div}^0_{D_{\bullet}}(\overline{X}'_{\bullet}) \otimes W' \ar@{=}[d] \ar[r] & 0\\
	0 \ar[r]& 
	H^1_{\cris}(\overline{X}'_\bullet/W') \ar[r] & H^1_{\cris}((X'_\bullet,\overline{X}'_\bullet)/W') \ar[r] & 
	{\rm Div}^0_{D'_{\bullet}}(\overline{X}'_{\bullet}) \otimes W' \ar[r] & 0, 
}
\end{equation}
where $D'_{\bullet}$ is the base change of $D_{\bullet}$ to $\overline{k}$. 
The maps of complexes $F': W\Omega^{\geqslant 1}_{\overline{X}'_{\bullet}/\overline{k}} 
\to W\Omega^{\geqslant 1}_{\overline{X}'_{\bullet}/\overline{k}}$, 
$F': W\Omega^{\geqslant 1}_{\overline{X}'_{\bullet}/\overline{k}}(\log D'_{\bullet}) 
\to W\Omega^{\geqslant 1}_{\overline{X}'_{\bullet}/\overline{k}}(\log D'_{\bullet})$
whose degree $i$ parts are $p^{i-1}F$ 
induce the endomorphisms $F'$ on 
$H^1_{\et}(\overline{X}'_{\bullet},W\Omega^{\geqslant 1}_{\overline{X}'_\bullet/\overline{k}})$ and on 
$H^1_{\et}(\overline{X}'_{\bullet},W\Omega^{\geqslant 1}_{\overline{X}'_\bullet/\overline{k}}(\log D'_\bullet))$, 
and we can check 
by the proof of Proposition \ref{prop: exact div}
that they induce the endomorphism $F'$ on  
${\rm Div}^0_{D'_{\bullet}}(\overline{X}'_{\bullet}) \otimes W'$ given by 
$F'=1\otimes F$, where $F$ denotes the Witt vector Frobenius. 

Take the kernels of $1-F'$ of the terms in the top horizontal line in \eqref{eq:diag-logcrys}. 
By the exact sequence \eqref{eq:oxoxp} and \cite[II.Lem.\,5.3]{I_1979}, 
we have the exact sequence 
\[ 
0 \to H^0_{\et}(\overline{X}'_{\bullet}, {\mathcal{O}}_{\overline{X}'_{\bullet}}^{\times}/ 
({\mathcal{O}}_{\overline{X}'_{\bullet}}^{\times})^{p^n}) 
\to  
H^1_{\et}(\overline{X}'_{\bullet},W_n\Omega^{\geqslant 1}_{\overline{X}'_\bullet/\overline{k}})
\overset{1-F'}{\to} 
H^1_{\et}(\overline{X}'_{\bullet},W_n\Omega^{\geqslant 1}_{\overline{X}'_\bullet/\overline{k}}) 
\to 0, 
\]
and so 
$$ {\rm Ker}(1-F') = \varprojlim_n 
H^0_{\et}(\overline{X}'_{\bullet}, {\mathcal{O}}_{\overline{X}'_{\bullet}}^{\times}/ 
({\mathcal{O}}_{\overline{X}'_{\bullet}}^{\times})^{p^n}) 
$$ 
for the first term. For the third term, 
${\rm Ker}(1-F') = {\rm Div}^0_{D'_{\bullet}}(\overline{X}'_{\bullet}) \otimes_{\Z} \Z_p$. 
So, since the map  
$1-F'$ is surjective on $H^1_{\et}(\overline{X}'_{\bullet},W\Omega^{\geqslant 1}_{\overline{X}'_\bullet/\overline{k}})$, 
we obtain the exact sequences 
\begin{equation}\label{eq:ker1-f'}
0 \to 
\varprojlim_n 
H^0_{\et}(\overline{X}'_{\bullet}, {\mathcal{O}}_{\overline{X}'_{\bullet}}^{\times}/ 
({\mathcal{O}}_{\overline{X}'_{\bullet}}^{\times})^{p^n}) 
\to 
{\rm Ker}(1-F') 
\to 
 {\rm Div}^0_{D'_{\bullet}}(\overline{X}'_{\bullet}) \otimes_{\Z} \Z_p 
 \to 0 \quad (n \in \N)
\end{equation}
consisting of the kernels of $1-F'$. On the other hand, we have exact sequences  
\[ 
0 \to 
{\mathcal{O}}_{\overline{X}'_{\bullet}}^{\times}/ 
({\mathcal{O}}_{\overline{X}'_{\bullet}}^{\times})^{p^n} 
\overset{\rm dlog}{\to} 
(j'_{\bullet *}{\mathcal{O}}_{\overline{X}'_{\bullet} \setminus D'_{\bullet}}^{\times})/ 
(j'_{\bullet *}{\mathcal{O}}_{\overline{X}'_{\bullet} \setminus D'_{\bullet}}^{\times})^{p^n} 
\to 
a_{\bullet *}^{(1)}(\Z/p^n\Z)_{D_{\bullet}^{(1)}}
\to 0 \quad (n \in \N)
\]
(where $j'_{\bullet}$ is the open immersion $\overline{X}'_{\bullet} \setminus D'_{\bullet} 
\hookrightarrow \overline{X}'_{\bullet}$) 
and it induces the long exact sequences 
\begin{align*}
0 \to H^0_{\et}(\overline{X}'_{\bullet}, 
{\mathcal{O}}_{\overline{X}'_{\bullet}}^{\times}/ 
({\mathcal{O}}_{\overline{X}'_{\bullet}}^{\times})^{p^n} ) 
& \to 
H^0_{\et}(\overline{X}'_{\bullet}, 
(j'_{\bullet *}{\mathcal{O}}_{\overline{X}'_{\bullet} \setminus D'_{\bullet}}^{\times})/ 
(j'_{\bullet *}{\mathcal{O}}_{\overline{X}'_{\bullet} \setminus D'_{\bullet}}^{\times})^{p^n}) \\ 
& \to {\rm Div}_{D'_{\bullet}}(\overline{X}'_{\bullet}) \otimes_{\Z} \Z/p^n\Z
\to 
H^1_{\et}(\overline{X}'_{\bullet}, 
{\mathcal{O}}_{\overline{X}'_{\bullet}}^{\times}/ 
({\mathcal{O}}_{\overline{X}'_{\bullet}}^{\times})^{p^n}) 
\quad (n \in \N). 
\end{align*}
Since the first term is equal to 
${\rm Ker}(H^1_{\et}(\overline{X}'_{\bullet}, 
{\mathcal{O}}_{\overline{X}'_{\bullet}}^{\times}) 
\overset{p^n}{\to} 
H^1_{\et}(\overline{X}'_{\bullet}, 
{\mathcal{O}}_{\overline{X}'_{\bullet}}^{\times}))
= {\rm Pic}_{\overline{X}'_{\bullet}}(\overline{k})[p^n]$, 
it is finite, and the third term is obviously finite. Thus, by taking inverse limits, we obtain the exact sequence 
\begin{align*}
0 \to \varprojlim_n H^0_{\et}(\overline{X}'_{\bullet}, 
{\mathcal{O}}_{\overline{X}'_{\bullet}}^{\times}/ 
({\mathcal{O}}_{\overline{X}'_{\bullet}}^{\times})^{p^n} ) 
& \to 
\varprojlim_n H^0_{\et}(\overline{X}'_{\bullet}, 
(j'_{\bullet *}{\mathcal{O}}_{\overline{X}'_{\bullet} \setminus D'_{\bullet}}^{\times})/ 
(j'_{\bullet *}{\mathcal{O}}_{\overline{X}'_{\bullet} \setminus D'_{\bullet}}^{\times})^{p^n}) \\ 
& \to {\rm Div}_{D'_{\bullet}}(\overline{X}'_{\bullet}) \otimes_{\Z} \Z_p
\to 
\varprojlim_n H^1_{\et}(\overline{X}'_{\bullet}, 
{\mathcal{O}}_{\overline{X}'_{\bullet}}^{\times}/ 
({\mathcal{O}}_{\overline{X}'_{\bullet}}^{\times})^{p^n}). 
\end{align*}
The last map is identified with 
${\rm Div}_{D'_{\bullet}}(\overline{X}'_{\bullet}) \otimes_{\Z} \Z_p
\to 
H^2_{\rm fppf}(\overline{X}'_{\bullet}, \Z_p(1))
$, which factors through ${\rm NS}(\overline{X}'_{\bullet}) \otimes_{\Z} \Z_p$. 
Hence the above exact sequence induces the exact sequence 
\begin{align}
	0 \to \varprojlim_n H^0_{\et}(\overline{X}'_{\bullet}, 
	{\mathcal{O}}_{\overline{X}'_{\bullet}}^{\times}/ 
	({\mathcal{O}}_{\overline{X}'_{\bullet}}^{\times})^{p^n} ) 
    & \to 
	\varprojlim_n H^0_{\et}(\overline{X}'_{\bullet}, 
	(j'_{\bullet *}{\mathcal{O}}_{\overline{X}'_{\bullet} \setminus D'_{\bullet}}^{\times})/ 
	(j'_{\bullet *}{\mathcal{O}}_{\overline{X}'_{\bullet} \setminus D'_{\bullet}}^{\times})^{p^n}) 
	\label{eq:ker1-f'2} \\ 
	& \to {\rm Div}^0_{D'_{\bullet}}(\overline{X}'_{\bullet}) \otimes_{\Z} \Z_p \to 0.  
	\nonumber 
\end{align}
Because there exists the map 
$$ 
{\rm dlog}: j'_{\bullet *}{\mathcal{O}}_{\overline{X}'_{\bullet} \setminus D'_{\bullet}}^{\times}/ 
(j'_{\bullet *}{\mathcal{O}}_{\overline{X}'_{\bullet} \setminus D'_{\bullet}}^{\times})^{p^n}
\to {\rm Ker}(1-F': W_n\Omega^{1}_{\overline{X}'_\bullet/\overline{k}}(\log D'_\bullet)) \to 
W_n\Omega^{1}_{\overline{X}'_\bullet/\overline{k}}(\log D'_\bullet))), 
$$
we have the canonical map of exact sequences from \eqref{eq:ker1-f'} to \eqref{eq:ker1-f'2} 
whose first and third terms are identity and so the two exact sequences canonically coincide. 
Moreover, if we apply $\otimes_{\Z_p} W'$ to \eqref{eq:ker1-f'2}, we obtain 
the top horizontal arrow in \eqref{eq:diag-logcrys}: Indeed, the first and the third terms are the 
covariant Dieudonn\'e modules associated to certain \'etale $p$-divisible groups and for such modules, 
${\rm Ker}(1-F') \otimes_{\Z_p} W'$ is identical with the original Dieudonné module. Then the isomorphism on the middle term follows from the five lemma. 

By the description of the exact sequences \eqref{eq:tt5} and \eqref{eq:logcrys-ses} above, 
we see that, to prove their coincidence in the category of Dieudonné modules, it suffices to prove the coincidence of 
\eqref{eq:dieudonne-infty-et-qpzp} and \eqref{eq:ker1-f'2} as exact sequences endowed with 
${\rm Gal}(\overline{k}/k)$-action, namely, it suffices to prove that 
the exact sequence \eqref{eq:ker1-f'2} is also constructed by the section 
induced by $-\widetilde{u}$ and the $1$-cocycle  
\begin{align*}
\widetilde{c} & \in 
Z^1({\rm Gal}(\overline{k}/k), 
{\rm Hom}({\mathbb{X}}' \otimes \Z_p, 
\varprojlim_n{\mathbb{G}}'[p^{n}])) \\ & = 
Z^1({\rm Gal}(\overline{k}/k), 
{\rm Hom}({\rm Div}^0_{D'_{\bullet}}(\overline{X}'_{\bullet}) \otimes \Z_p, 
\varprojlim_n{\rm Pic}^{0,{\rm red}}_{\overline{X}'_{\bullet}}(\overline{k})[p^n])) \\ 
& = 
Z^1({\rm Gal}(\overline{k}/k), 
{\rm Hom}({\rm Div}^0_{D'_{\bullet}}(\overline{X}'_{\bullet}) \otimes \Z_p, 
\varprojlim_n{\rm Pic}_{\overline{X}'_{\bullet}}(\overline{k})[p^{\infty}]))
\end{align*}
given by $\sigma \mapsto \varprojlim_n(-\sigma(\widetilde{u}'_n) + \widetilde{u}'_n)$, 
where $\widetilde{u}'_n$ is the composite 
${\mathbb{X}}'/p^n{\mathbb{X}}' \cong 
\frac{1}{p^n}{\mathbb{X}}'/{\mathbb{X}}' \overset{\widetilde{u}_n}{\to} {\mathbb{G}}'.$

Before the calculation, we prepare several notations on 
\v{C}ech complexes. Take a split simplicial scheme 
$\overline{X}''_{\bullet}$ over $\overline{X}'_{\bullet}$ such that, 
for each $n$, $\overline{X}''_{n}$ is a disjoint union 
$\coprod_{i \in I_n} U_{n,i}$ 
of connected affine open subschemes $U_{n,i} \, (i \in I_n)$ of 
$\overline{X}'_n$ which covers $\overline{X}'_{\bullet}$. 
(The existence of such a simplicial scheme follows from the proof of 
\cite[Lem.\,6.1]{N_2012}.) 
Concretely, for each map $\varphi: [n] \to [m]$ of simplicial sets, 
we have a map of index sets $\varphi^{\sharp}: I_m \to I_n$ such that 
$U_{m,i}$ is sent to $U_{n,\varphi^{\sharp}(i)}$ by 
$\varphi^*:\overline{X}''_m \to \overline{X}''_n$. 
If we denote the \v{C}ech hypercovering of $\overline{X}''_n \to \overline{X}'_n$ by 
$\overline{X}''_{n\bullet}$, then $\overline{X}''_{\bullet\bullet}$ forms a double simplicial scheme: 
Concretely,  $\overline{X}''_{nn'}$ is a disjoint union 
$\sqcup_{(i_0,\dots,i_{n'}) \in I_n^{n'}} U_{n,(i_0,\dots,i_{n'})}$, where 
$U_{n,(i_0,\dots,i_{n'})} := \bigcap_{j=0}^{n'} U_{n,i_j}$. 
Then the total complex 
${\rm Tot}(\Gamma(\overline{X}''_{\bullet\bullet}, 
{\mathcal{O}}_{\overline{X}''_{\bullet\bullet}}^{\times}))$
of the double complex $\Gamma(\overline{X}''_{\bullet\bullet}, 
{\mathcal{O}}_{\overline{X}''_{\bullet\bullet}}^{\times})$ is the \v{C}ech complex 
associated to $\overline{X}''_{\bullet}$ with coefficient 
${\mathcal{O}}^{\times}_{\overline{X}'_{\bullet}}$. 
If we take the first cohomology of this complex and take the direct limit over  
all the possible $\overline{X}''_{\bullet}$'s, we obtain the first cohomology 
$H^1(\overline{X}'_{\bullet}, {\mathcal{O}}^{\times}_{\overline{X}'_{\bullet}})$. 
(We worked with the Zariski topology, but the above cohomology is the same as 
the correponding \'etale cohomology.) We also note that we may assume 
each index set $I_n$ admits the action of ${\rm Gal}(\overline{k}/k)$ 
which are compatible with the maps $\varphi^{\sharp}: I_m \to I_n$ for any  
$\varphi: [n] \to [m]$. (Replace $I_n$ by $I_n \times {\rm Gal}(\overline{k}/k)$ and add 
Galois conjugates of $U_{n,i}$'s.) 

A $1$-cocycle of the complex 
${\rm Tot}(\Gamma(\overline{X}''_{\bullet\bullet}, 
{\mathcal{O}}_{\overline{X}''_{\bullet\bullet}}^{\times}))$ is given by 
the data $(v_{ij}, w_{i'})_{i,j \in I_0, i' \in I_1}$ 
with $v_{ij} \in \Gamma(U_{0,(i,j)}, {\mathcal{O}}^{\times})$, 
$w_{i'} \in \Gamma(U_{1,i'}, {\mathcal{O}}^{\times})$ 
which satisfy the compatibility conditions 
on $U_{2,i''} \, (i'' \in I_2)$, $U_{1,(i',j')} \, \allowbreak (i',j' \in I_1)$, 
$U_{0,(i,j,l)} \, (i,j,l \in I_0)$ which we omit to write the details.

Now take $C \in {\rm Div}^0_{D'_{\bullet}}(\overline{X}'_{\bullet})$: This is a divisor 
on $\overline{X}'_0$ with $d_1^*C = d_0^*C$, where $d_i: \overline{X}'_1 \to \overline{X}'_0$ are 
the projections, namely, the transition maps corresponding to 
$e_0: [0] \to [1]; 0 \mapsto 0$, $e_1: [0] \to [1]; 0 \mapsto 1$ respectively. 
Then we can take $\overline{X}''_{\bullet} \to \overline{X}'_{\bullet}$ as above and  
$t_{C,i} \in \Gamma(U_{0,i}, j'_{0*}{\mathcal{O}}_{\overline{X}'_0 \setminus D'_0}^{\times})$ 
for $i \in I_0$ such that 
$C \cap U_{0,i} = {\rm div}(t_{C,i}^{-1})$ on $U_{0,i}$. 
Then take sections 
$v_{C,ij}$ of ${\mathcal{O}}^{\times}$ on $U_{0,(i,j)} \, (i,j \in I_0)$ and sections 
$w_{C,i'}$ of ${\mathcal{O}}^{\times}$ on $U_{1,i'} \, (i' \in I_1)$ 
such that $v_{C,ij}t_{C,i} = t_{C,j}$ on $U_{0,(i,j)}$ and 
$w_{C,i'}d_0^*t_{C,e_0^{\sharp}(i')} = d_1^*t_{C,e_1^{\sharp}(i')}$ on $U_{1,i'}$. 
Note that $(v_{C,ij}, w_{C,i'})_{i,j,i'}$ defines a $1$-cocycle whose class in 
$H^1(\overline{X}'_{\bullet}, {\mathcal{O}}_{\overline{X}'_{\bullet}}^{\times}) = 
{\rm Pic}(\overline{X}'_{\bullet})(\overline{k})$ is equal to $-u(C)$ (the class of $-C$). 

Then we can take, after possibly taking a refinement of 
the morphism $\overline{X}''_{\bullet} \to \overline{X}'_{\bullet}$ and adjusting $t_{C,i}$'s 
by a suitable coboundary, sections 
$\alpha_{C,ij}$ of ${\mathcal{O}}^{\times}$ on $U_{0,(i,j)} \, (i,j \in I_0)$ and sections 
$\beta_{C,i'}$ of ${\mathcal{O}}^{\times}$ on $U_{1,i'} \, (i' \in I_1)$ 
such that $\alpha_{C,ij}^{p^n} = v_{C,ij}$, $\beta_{C,i'}^{p^n} = w_{i'}$ and that 
$(\alpha_{C,ij}, \beta_{C,i'})_{i,j,i'}$ induces the element $-\widetilde{u}_n(C)$ in 
$H^1(\overline{X}'_{\bullet}, {\mathcal{O}}^{\times}_{\overline{X}'_{\bullet}}) = 
{\rm Pic}(\overline{X}'_{\bullet})(\overline{k})$.

So we see that $(t_{C,i})_{i \in I_0}$ defines a well-defined element in 
$H^0(\overline{X}'_{\bullet}, 
(j'_{\bullet *}{\mathcal{O}}_{\overline{X}'_{\bullet} \setminus D'_{\bullet}}^{\times})/
(j'_{\bullet *}{\mathcal{O}}_{\overline{X}'_{\bullet} \setminus D'_{\bullet}}^{\times})^{p^n})$ for $n \in \N$, 
and the maps $C \mapsto (t_{C,i})_i$ for $n \in \N$ 
gives a (non-Galois-equivariant) 
section of 
the exact sequence \eqref{eq:ker1-f'2}. Thus  
\eqref{eq:ker1-f'2} is given by the $1$-cocycle 
\[ 
\sigma \mapsto (C \mapsto (\sigma(t_{\sigma^{-1}(C),\sigma^{-1}(i)})t_{C,i}^{-1})_{i \in I_0}). 
\]
The image of 
$(\sigma(t_{\sigma^{-1}(C),\sigma^{-1}(i)})t_{C,i}^{-1})_i \in 
H^0(\overline{X}'_{\bullet}, 
{\mathcal{O}}_{\overline{X}'_{\bullet}}^{\times}/
({\mathcal{O}}_{\overline{X}'_{\bullet}}^{\times})^{p^n})$ 
by the connecting map 
$$ \delta: H^0(\overline{X}'_{\bullet}, 
{\mathcal{O}}_{\overline{X}'_{\bullet}}^{\times}/
({\mathcal{O}}_{\overline{X}'_{\bullet}}^{\times})^{p^n}) \to 
 H^1(\overline{X}'_{\bullet}, 
{\mathcal{O}}_{\overline{X}'_{\bullet}}^{\times})$$
is computed as follows: first, by lifting sections 
$\sigma(t_{\sigma^{-1}(C),\sigma^{-1}(i)})t_{C,i}^{-1}$ of ${\mathcal{O}}_{\overline{X}'}^{\times}/
({\mathcal{O}}_{\overline{X}'}^{\times})^{p^n}$ to sections of 
${\mathcal{O}}_{\overline{X}'}^{\times}$ and 
applying the differential of \v{C}ech complex 
$$ d: \prod_{i\in I_0} \Gamma(U_{0,i}, {\mathcal{O}}_{\overline{X}'_0}^{\times}) \to 
\prod_{i,j \in I_0} \Gamma(U_{0,(i,j)}, {\mathcal{O}}_{\overline{X}'_0}^{\times}) 
\times 
\prod_{i' \in I_1} \Gamma(U_{1,i'}, {\mathcal{O}}_{\overline{X}'_1}^{\times}),$$ 
we obtain
\begin{align*}
& (\sigma(t_{\sigma^{-1}(C),\sigma^{-1}(i)})t_{C,i}^{-1})^{-1}\sigma(t_{\sigma^{-1}(C),\sigma^{-1}(j)})t_{C,j}^{-1}, 
\, d_0^*(\sigma(t_{\sigma^{-1}(C),\sigma^{-1}(e_0^{\sharp}(i'))})t_{C,e_0^{\sharp}(i')}^{-1})^{-1}
d_1^*(\sigma(t_{\sigma^{-1}(C),\sigma^{-1}(e_1^{\sharp}(i'))})t_{C,e_1^{\sharp}(i')}^{-1}) 
)_{i,j,i'} \\ & = 
((\sigma(\alpha_{\sigma^{-1}(C),\sigma^{-1}(i)\sigma^{-1}(j)})\alpha_{C,ij}, \, 
\sigma(\beta_{\sigma^{-1}(C),\sigma^{-1}(i')})\beta_{C,i'}
)_{i,j,i'})^{p^n}. 
\end{align*}
Then we see that the image of 
$(\sigma(t_{\sigma^{-1}(C),\sigma^{-1}(i)})t_{C,i}^{-1})_i$ by the connecting map 
$\delta$ is equal to 
$$(\sigma(\alpha_{\sigma^{-1}(C),\sigma^{-1}(i)\sigma^{-1}(j)})\alpha_{C,ij}, \, 
\sigma(\beta_{\sigma^{-1}(C),\sigma^{-1}(i')})\beta_{C,i'}
)_{i,j,i'},$$ namely, $-\sigma(\widetilde{u}'_n(C))+\widetilde{u}'_n(C)$. 
Thus the exact sequence \eqref{eq:ker1-f'2} is defined by the $1$-cocycle 
$\sigma \mapsto \varprojlim_n(-\sigma(\widetilde{u}'_n) + \widetilde{u}'_n)$, 
as required. 

Therefore, we obtained an isomorphism between the exact sequences 
\eqref{eq:tt5}, \eqref{eq:logcrys-ses} 
of $F$-crystals. 
Even if we start with another extension $\widetilde{u}$ of $u$, 
the $1$-cocycle changes according to this change and so  
the resulting isomorphism does not change. Thus the isomorphism we obtained is canonical and functorial. 
So the proof of the theorem is finished. 
\end{proof}

%%%%%%%
\section{Homotopy categories
of certain hypercoverings}\label{appendixB'}
%%%%%%%

In this appendix, we prove the cofilteredness of the homotopy category of 
certain diagrams involving hypercoverings. The results in this section 
should be well-known, but we review them for the convenience of the reader. 

Throughout this section, let $\cC$ be a category with finite direct sums, 
let $\tau$ be a pretopology on $\cC$ and denote the resulting site by $\cC_{\tau}$.  We call a morphism appearing in the pretopology $\tau$ 
a $\tau$-morphism, and call a morphism $X' \to X$ a $\tau$-covering if the singleton 
$\{X' \to X\}$ is a covering of $X$ with respect to $\tau$.
Let moreover $\cC^{ \mathrm{pf}}$ be a category which has the same objects as $\cC$, but morphisms restricted, such that the following conditions are satisfied:
\begin{itemize}
\item[(a)] if $Z\rightarrow X$ is a morphism in $\cC$ and $Y\rightarrow X$ is a morphism in $\cC^{\mathrm{pf}}$, then there exists a fibre product $Y\times^{\cC}_X Z$ in $\cC$;
\item[(b)] any base change of a $\cC^{\mathrm{pf}}$-morphism is a $\cC^{\mathrm{pf}}$-morphism;
\item[(c)] if a $\cC^{\mathrm{pf}}$-morphism $f$ has a factorisation $f=h\circ g$ in $\cC$, such that $h$ is a $\cC^{\mathrm{pf}}$-morphism, then $g$ is as well;
\item[(d)] $\tau$-morphisms are morphisms in $\cC^{\mathrm{pf}}$.
\end{itemize}
Note that (a), (b) and (c) imply the following property:
\begin{itemize}
\item[(e)] if both morphisms of (a) are in $\cC^{\mathrm{pf}}$, then  $Y\times^{\cC}_X Z \rightarrow X$ is a fibre product in $\cC^{\mathrm{pf}}$. 
\end{itemize}
Indeed, by the base change property (b)  
 the projection $Y\times^{\cC}_X Z \rightarrow Y$ is a 
$\cC^{\mathrm{pf}}$-morphism. Then, by (c) for any pair of $\cC^{\mathrm{pf}}$-morphisms $W\rightarrow Y$ and $W\rightarrow Z$  over $X$ the morphism $W \rightarrow Y\times^{\cC}_X Z$ induced by the universal property of the fibre product in $\cC$ is also a $\cC^{\mathrm{pf}}$-morphism  by (c).

Finally, let $\cC'$ be a full subcategory of $\cC$ closed under finite direct sums satisfying the following condition: 

\medskip
$(*)$: For any object $A \in \cC$, there exists a $\tau$-covering $A' \to A$ 
with $A' \in \cC'$. \label{ast}
\medskip 

\begin{lemma}\label{lem:domhypercover1'}
Let $A$ be an object in $\cC$, and let 
$A_{\bullet} \to A, A'_{\bullet} \to  A$ be hypercoverings in 
$\cC_{\tau}$. 
\begin{enumerate}
	\item The fibre product 
	 $A_{\bullet} \times_A A'_{\bullet}
	 \to A$ in the category of simplicial objects in $\cC$ over $A$ 
	 is also a hypercovering in $\cC_{\tau}$. 
	 \item If we are given two $\cC^{\mathrm{pf}}$-morphisms 
	 $\alpha, \beta: A'_{\bullet} \to A_{\bullet}$ over $A$, 
	 there exists another hypercovering 
     $A''_{\bullet} \to A$ in $\cC_{\tau}$ and 
     a $\cC^{\mathrm{pf}}$-morphism $\gamma: A''_{\bullet} \to A'_{\bullet}$ 
     over $A$ such that $\alpha \circ \gamma$ and $\beta \circ \gamma$ are homotopic.
\end{enumerate}
\end{lemma}	

\begin{proof}
Because $\cC$ and $\cC^{\mathrm{pf}}$ have the same objects and $\tau$-morphisms are $\cC^{\mathrm{pf}}$-morphisms by condition (c), 
any hypercovering in $\cC_\tau$ is a hypercovering in $\cC^{\mathrm{pf}}_\tau$, and we may work in $\cC^{\mathrm{pf}}$ which has fibre products by conditions (a) and (b). 
Then, 
(i) is proven in \cite[\href{https://stacks.math.columbia.edu/tag/01GI}{Tag 01GI}]{StacksProject} and (ii) is proven in 
\cite[\href{https://stacks.math.columbia.edu/tag/01GS}{Tag 01GS}]{StacksProject}. 
\end{proof}

\begin{lemma}\label{lem:domhypercover2'}
Let $A_{\bullet} \to A$ be a hypercovering in $\cC_{\tau}$. Then there exists a split hypercovering $A'_{\bullet} \to A$ in $\cC_{\tau}$ with 
each $A'_i$ in $\cC'$ and a $\cC^{\mathrm{pf}}$-morphism $A'_{\bullet} \to A_{\bullet}$ over $A$. 
\end{lemma}

\begin{proof}
Suppose that we have constructed an $i$-truncated 
split hypercovering 
 $A'_{\bullet \leqslant i} \to A$ in $\cC_{\tau}$ 
 with each $A'_j \, (j \leqslant i)$ in $\cC'$ and a $\cC^{\mathrm{pf}}$-morphism $A'_{\bullet \leqslant i} \to A_{\bullet \leqslant i}$ 
 over $A$. 
Consider the diagram 
 \[ 
 {\rm cosk}_i(A'_{\bullet \leqslant i})_{i+1} \rightarrow 
 {\rm cosk}_i(A_{\bullet \leqslant i})_{i+1} \leftarrow A_{i+1}
 \]
where both morphisms are $\cC^{\mathrm{pf}}$-morphisms, let $A'''_{i+1}$ be its fibre product
in $\cC^{\mathrm{pf}}$ which exists by conditions (a) and (e) and take a $\tau$-covering $A''_{i+1} \to A'''_{i+1}$
with $A''_{i+1} \in \cC'$, which exists by the condition \hyperref[ast]{$(*)$}. 
Then, by the recipe in \cite[Prop.\,5.1.3]{S_1972}, we can form from $A''_{i+i}$ an $(i+1)$-truncated split hypercovering 
 $A'_{\bullet \leqslant i+1} \to A$ in $\cC_{\tau}$ 
with each $A'_j \, (j \leqslant i+1)$ in $\cC'$ and a $\cC^{\mathrm{pf}}$-morphism $
A'_{\bullet \leqslant i+1} \to A_{\bullet \leqslant i+1}$ 
over $A$. (In this step, we use the assumption that $\cC'$ is closed under finite direct sums.) By induction we obtain the required assertion. 
\end{proof}

\begin{proposition}\label{prop:domhypercover'}
Let $f: A' \to A$ be a morphism in $\cC$. 
\begin{enumerate}
\item
For any split hypercovering $A_{\bullet} \to A$ in $\cC_{\tau}$ 
with each $A_i$ in $\cC'$, 
there exists a split hypercovering $A'_{\bullet} \to A'$ in $\cC_{\tau}$ 
with each $A'_i$ in $\cC'$ and a morphism $A'_{\bullet} \to A_{\bullet}$ 
over $f$. 
\item \label{prop:domhypercover_ii'}
If we are given split hypercoverings 
\[ A_{i,\bullet} \to A, \quad A'_{i,\bullet} \to A' \, (i=1,2) \] 
in $\cC_{\tau}$ with each $A_{i,j}, A'_{i,j}$ in $\cC'$ endowed with morphisms 
$f_i: A'_{i,\bullet} \to A_{i,\bullet} \, (i=1,2)$ over $f$, 
there exist split hypercoverings 
\[ A_{3,\bullet} \to A, \quad A'_{3,\bullet} \to A'  \]
in $\cC_{\tau}$ with each $A_{3,j}, A'_{3,j}$ in $\cC'$ endowed with a morphism $f_3: A'_{3,\bullet} \to A_{3,\bullet}$ over $f$, morphisms 
$g_i: A_{3,\bullet} \to A_{i,\bullet} \, (i=1,2)$ over $A$ and morphisms 
$g'_i: A'_{3,\bullet} \to A'_{i,\bullet} \, (i=1,2)$ over $A'$ such that, 
for $i=1,2$, $f_i \circ g'_i$ and $g_i \circ f_3$ are homotopic.  
\item \label{prop:domhypercover_iii'}
If we are given split hypercoverings 
\[ A_{i,\bullet} \to A, \quad A'_{i,\bullet} \to A' \, (i=1,2) \] 
in $\cC_{\tau}$ with each $A_{i,j}, A'_{i,j}$ in $\cC'$ endowed with morphisms 
$f_i: A'_{i,\bullet} \to A_{i,\bullet} \, (i=1,2)$ over $f$ and morphisms 
\[ \alpha, \beta: A_{2,\bullet} \to A_{1,\bullet}, \quad 
\alpha', \beta': A'_{2,\bullet} \to A'_{1,\bullet} \] 
such that $f_1 \circ \alpha'$ and $\alpha \circ f_2$ are homotopic and that 
$f_1 \circ \beta'$ and $\beta \circ f_2$ are homotopic, 
there exist split hypercoverings 
\[ A_{3,\bullet} \to A, \quad A'_{3,\bullet} \to A'  \]
in $\cC_{\tau}$ with each $A_{3,j}, A'_{3,j}$ in $\cC'$ endowed with a morphism $f_3: A'_{3,\bullet} \to A_{3,\bullet}$ over $f$, a morphism 
$\gamma: A_{3,\bullet} \to A_{2,\bullet}$ over $A$ and a morphism 
$\gamma': A'_{3,\bullet} \to A'_{2,\bullet}$ over $A'$ such that 
$f_2 \circ \gamma'$ and $\gamma \circ f_3$ are homotopic, 
$\alpha \circ \gamma$ and $\beta \circ \gamma$ are homotopic and 
$\alpha' \circ \gamma'$ and $\beta' \circ \gamma'$ are homotopic. 
\end{enumerate}
\end{proposition}

\begin{proof}
First we prove (i). By Lemma \ref{lem:domhypercover2'}, there exists 
a split hypercovering $A'_{\bullet} \to A'$ in $\cC_{\tau}$ with each 
$A'_i$ in $\cC'$ and a $\cC^{\mathrm{pf}}$-morphism 
$A'_{\bullet} \to A' \times_A A_{\bullet}$ over $A'$. The hypercovering 
$A'_{\bullet} \to A'$ and the composite $A'_{\bullet} \to A' \times_A A_{\bullet} \to A_{\bullet}$ satisfy the required properties. 

Next we prove (ii). 
Let $\underline{A}_{3,\bullet} := A_{1,\bullet} \times_A A_{2,\bullet}$, 
$\underline{A}'_{3,\bullet} := A'_{1,\bullet} \times_{A'} A'_{2,\bullet}$ which are by Lemma \ref{lem:domhypercover1'} again $\tau$-hypercoverings and in particular (simplicial) fibre products in $\cC^{\mathrm{pf}}$. 
Then we have natural morphisms 
\[ 
\underline{g}_i: \underline{A}_{3,\bullet} \to A_{i,\bullet} \, (i=1,2), 
\quad 
\underline{g}'_i: \underline{A}'_{3,\bullet} \to A'_{i,\bullet} \, (i=1,2), 
\quad 
\underline{f}_3: \underline{A}'_{3,\bullet} \to \underline{A}_{3,\bullet} 
\]
where the first two are $\cC^{\mathrm{pf}}$-morphisms (even $\tau$-morphisms)  with 
\begin{equation}\label{eq:b10'-1}
f_i \circ \underline{g}'_i = 
\underline{g}_i \circ \underline{f}_3 \,\, (i=1,2). 
\end{equation}
Let $\underline{f}_{3,A'}: \underline{A}'_{3,\bullet} \to A' \times_A 
\underline{A}_{3,\bullet}$ be the morphism of $\tau$-hypercoverings over $A'$ induced by $\underline{f}_3$, which by condition (c) is a $\cC^{\mathrm{pf}}$-morphism.

By Lemma \ref{lem:domhypercover2'}, there exists 
a split hypercovering $A_{3,\bullet} \to A$ in $\cC_{\tau}$ with each 
$A_{3,j}$ in $\cC'$ and a $\cC^{\mathrm{pf}}$-morphism 
$h: A_{3,\bullet} \to \underline{A}_{3,\bullet}$ over $A$. 
Put $g_i = \underline{g}_i \circ h$ for $i=1,2$, and let 
$h_{A'}:  A' \times_A A_{3,\bullet} \to A' \times_A \underline{A}_{3,\bullet}$ be
the base change of $h$, which is again a $\cC^{\mathrm{pf}}$-morphism by condition (b). 

Let $\underline{A}''_{3,\bullet} := 
\underline{A}'_{3,\bullet} \times_{A'} (A' \times_A A_{3,\bullet})$ 
which is again a $\tau$-hypercovering by Lemma \ref{lem:domhypercover1'}
 and let 
\[ 
q: \underline{A}''_{3,\bullet} \to \underline{A}'_{3,\bullet}, \quad 
r: \underline{A}''_{3,\bullet} \to A_{3,\bullet}, \quad 
r_{A'}: \underline{A}''_{3,\bullet} \to A' \times_A A_{3,\bullet} 
\] 
be projections where the first and the last are $\cC^{\mathrm{pf}}$-morphisms. Then we have two $\cC^{\mathrm{pf}}$-morphisms 
\[ 
\underline{f}_{3,A'} \circ q, \, h_{A'} \circ r_{A'}: 
\underline{A}''_{3,\bullet} \to A' \times_A \underline{A}_{3,\bullet}. 
\] 
By Lemma \ref{lem:domhypercover1'} and Lemma \ref{lem:domhypercover2'}, 
there exists a split hypercovering $A'_{3,\bullet} \to A'$ in $\cC_{\tau}$ with each $A'_{3,j}$ in $\cC'$ and a $\cC^{\mathrm{pf}}$-morphism $s: A'_{3,\bullet} \to 
\underline{A}''_{3,\bullet}$ over $A'$ such that 
\[ \underline{f}_{3,A'} \circ q \circ s, \, h_{A'} \circ r_{A'} \circ s: 
A'_{3,\bullet} \to A' \times_A \underline{A}_{3,\bullet}. 
\] 
are homotopic. Composing with the projection to $\underline{A}_{3,\bullet}$, 
we see that 
\[ \underline{f}_3 \circ q \circ s, \, h \circ r \circ s: 
A'_{3,\bullet} \to \underline{A}_{3,\bullet}. 
\] 
are homotopic. Now we put 
\[ g'_i := \underline{g}'_i \circ q \circ s: A'_{3,\bullet} \to 
A'_{i,\bullet} 
 \, (i=1,2), \quad 
f_3 := r \circ s: A'_{3,\bullet} \to A_{3,\bullet}. 
\]
Then, for $i=1,2$, 
\[ 
f_i \circ g'_i = f_i \circ \underline{g}'_i \circ q \circ s 
\overset{\text{\eqref{eq:b10'-1}}}{=} 
\underline{g}_i \circ \underline{f}_3 \circ q \circ s \] 
and 
\[ g_i \circ f_3 = \underline{g}_i \circ h \circ r \circ s \] 
are homotopic. So we have proved the required properties. 

Finally we prove (iii). 
By Lemma \ref{lem:domhypercover1'} and Lemma \ref{lem:domhypercover2'}, 
there exists a split hypercovering $A_{3,\bullet} \to A$ in $\cC_{\tau}$ 
with each $A_{3,j}$ in $\cC'$ and a $\cC^{\mathrm{pf}}$-morphism $\gamma: A_{3,\bullet} \to 
A_{2,\bullet}$ over $A$ 
such that $\alpha \circ \gamma$ and $\beta \circ \gamma$ 
are homotopic. Also, by Lemma \ref{lem:domhypercover1'}, 
there exists a hypercovering $\underline{A}'_{3.\bullet} \to A'$ in 
$\cC_{\tau}$ and a $\cC^{\mathrm{pf}}$-morphism $\underline{\gamma}': \underline{A}'_{3,\bullet} 
\to A'_{2,\bullet}$ such that $\alpha' \circ \underline{\gamma}'$ and 
$\beta' \circ \underline{\gamma}'$ are homotopic. 

Let $\underline{A}''_{3,\bullet} := 
\underline{A}'_{3,\bullet} \times_{A'} (A' \times_A A_{3,\bullet})$ which is again a $\tau$-hypercovering by Lemma \ref{lem:domhypercover1'} and let 
\[ 
q: \underline{A}''_{3,\bullet} \to \underline{A}'_{3,\bullet}, \quad 
r: \underline{A}''_{3,\bullet} \to A_{3,\bullet}, \quad 
r_{A'}: \underline{A}''_{3,\bullet} \to A' \times_A A_{3,\bullet} 
\] 
be projections where the first and the last are $\cC^{\mathrm{pf}}$-morphisms. Then we have two $\cC^{\mathrm{pf}}$-morphisms 
\[ 
f_{2,A'} \circ \underline{\gamma}' \circ q, \, 
\gamma_{A'} \circ r_{A'}: 
\underline{A}''_{3,\bullet} \to A' \times_A A_{2,\bullet},
\] 
where $f_{2,A'}: A'_{2,\bullet} \to A' \times_A A_{2,\bullet}$ 
is the morphism of $\tau$-hypercoverings of $A'$ induced by $f_2$ which by
condition (c) is a $\cC^{\mathrm{pf}}$-morphism, 
and $\gamma_{A'}: A' \times_A A_{3,\bullet} \to A' \times_A A_{2,\bullet}$ 
is the base change  of $\gamma$ which is a $\cC^{\mathrm{pf}}$-morphism by condition (b).
By Lemma \ref{lem:domhypercover1'} and Lemma \ref{lem:domhypercover2'}, 
there exists a split hypercovering $A'_{3,\bullet} \to A'$ in $\cC_{\tau}$ with each $A'_{3,j}$ in $\cC'$ and a $\cC^{\mathrm{pf}}$-morphism $s: A'_{3,\bullet} \to 
\underline{A}''_{3,\bullet}$ over $A'$ such that 
\[ f_{2,A'} \circ \underline{\gamma}' \circ 
q \circ s, \, \gamma_{A'} \circ r_{A'} \circ s: 
A'_{3,\bullet} \to A' \times_A A_{2,\bullet}. 
\] 
are homotopic. Composing with the projection to $A_{2,\bullet}$, 
we see that $f_2 \circ \underline{\gamma}' \circ q \circ s, \, 
\gamma \circ r \circ s$ 
are homotopic. If we put $\gamma' := \underline{\gamma}' \circ q \circ s:A'_{3,\bullet} \to A'_{2,\bullet}$ and 
$f_3 := r \circ s: A'_{3,\bullet} \to A_{3,\bullet}$, 
$f_2 \circ \gamma$ and $\gamma \circ f_3$ 
are homotopic. Also, since $\alpha' \circ \underline{\gamma}'$ and 
$\beta' \circ \underline{\gamma}'$ are homotopic, 
we see that $\alpha' \circ \gamma'$ and 
$\beta' \circ \gamma'$ are homotopic. So we are done. 
\end{proof}

\begin{definition}\label{def:homotopycat}
We define the homotopy category of certain hypercoverings as follows: 
\begin{enumerate}
\item 
For $A \in \cC$, we define the category ${\rm Ho}\{A_{\bullet}/A\}$ as the one in which an object is a split hypercovering $A_{\bullet} \to A$ in $\cC_{\tau}$ with each $A_i$ in $\cC'$ and a morphism from $A_{1,\bullet} \to A$ to $A_{2,\bullet} \to A$ is a homotopy class of morphisms $A_{1,\bullet} \to A_{2,\bullet}$ over $A$. 
\item 
For a morphism $f:A' \to A$ in $\cC$, we define the category 
${\rm Ho}\{f_{\bullet}/f\}$ as the one in which an object is 
a triple $(A_{\bullet} \to A, A'_{\bullet} \to A', 
f_{\bullet})$ consisting of 
split hypercoverings $A_{\bullet} \to A, A'_{\bullet} \to A'$ in $\cC_{\tau}$ with each $A_i, A'_i$ in $\cC'$ and a morphism $f_{\bullet}:A'_{\bullet} \to A_{\bullet}$ over $f$, and a morphism from 
$(A_{1,\bullet} \to A, A'_{1,\bullet} \to A', f_{1,\bullet})$ to 
$(A_{2,\bullet} \to A, A'_{2,\bullet} \to A', f_{2,\bullet})$
is a pair $(g,g')$ of homotopy classes of morphisms 
$g: A_{1,\bullet} \to A_{2,\bullet}$, $g':A'_{1,\bullet} \to A'_{2,\bullet}$ over $A$, $A'$ respectively such that $f_{2,\bullet} \circ g'$ and $g \circ f_{1,\bullet}$ are homotopic. 
\end{enumerate}
\end{definition}

Then Proposition \ref{prop:domhypercover'} immediately implies the following: 

\begin{corollary}\label{cor:homotopycat}
\begin{enumerate}
\item 
With the notation in Definition \ref{def:homotopycat}, the category 
${\rm Ho}\{A_{\bullet}/A\}$, ${\rm Ho}\{f_{\bullet}/f\}$ are cofiltered. 
\item
For a morphism $f:A' \to A$ in $\cC$, there exists a diagram of functors 
\[ {\rm Ho}\{A'_{\bullet}/A'\} \leftarrow 
{\rm Ho}\{f_{\bullet}/f\} \rightarrow {\rm Ho}\{A_{\bullet}/A\} \] 
in which the first arrow (resp. the second arrow) sends $f_{\bullet}: A'_{\bullet} \to A_{\bullet}$ to $A'_{\bullet}$ (resp. $A_{\bullet}$) and the second arrow is surjective on the set of objects. 
\end{enumerate}
\end{corollary}

\begin{example}\label{example:varkgeocdp}
In this example, we suppose Hypotheses 
\ref{strong resolutions}, \ref{embedded resolutions}. 
We define the $\cdp$-topology on $\Var_k^{geo}$ as the one induced from 
the $\cdp$-topology on $\Var_k$ by the forgetful functor 
\[ \psi: \Var_k^{geo} \to \Var_k; \,\, (X,\overline{X}) \mapsto X, \] 
and we denote the resulting site by $\Var^{geo}_{k,\cdp}$. 
Since $\Var_{k}$ and $\Var_{k}^{geo}$ admit fibre products and the forgetful functor $\psi$ preserves fibre products, a family of morphisms 
$\{(X_i,\overline{X}_i) \to (X,\overline{X})\}_i$ 
is a $\cdp$-covering in $\Var_k^{geo}$ if and only if the family of morphisms 
 $\{X_i \to X\}_i$ is a $\cdp$-covering in $\Var_k$ (\cite[Cor.\,3.3]{V_1972}). 
We prove that, if we put $\cC =\cC^{\mathrm{pf}}  = \Var_k^{geo}$, $\tau$ to be the $\cdp$-topology on $\Var_k^{geo}$ and $\cC' = \Var_k^{nc}$, the condition \hyperref[ast]{$(*)$} is satisfied: 
Indeed, for $(X,\overline{X}) \in \Var_k^{geo}$, Hypotheses 
\ref{strong resolutions}, \ref{embedded resolutions} imply the existence of strict birational morphism $f: (X',\overline{X}') \to (X,\overline{X})$ such that $(X',\overline{X}') \in \Var_k^{nc}$. Let $U \subset \overline{X}$ be an open dense subset on which $f$ is an isomorphism, $\overline{Y}' = \overline{X} \setminus U$ be its complement, $Y := X \cap \overline{Y}'$ and let $\overline{Y}$ be the closure of $Y$ in $\overline{Y}'$. Then $(Y,\overline{Y}) \in \Var_k^{geo}$ and $(X',\overline{X}') \coprod (Y,\overline{Y}) \to (X,\overline{X})$ is a $\cdp$-covering. By Noetherian induction, we may supppse the existence of a $\cdp$-covering $(Y',\overline{Y}') \to (Y,\overline{Y})$ with $(Y',\overline{Y}') \in \Var_k^{nc}$, and then $(X',\overline{X}') \coprod (Y',\overline{Y}') \to (X,\overline{X})$ gives a required $\cdp$-covering. 

So the results in this section is applicable to the case $\cC =\cC^{\mathrm{pf}}  = \Var_k^{geo}$,
$\cC_{\tau} = \Var_{k,\cdp}^{geo}$ and $\cC' = \Var_k^{nc}$. 
\end{example}

%%%%%%
\section{A site of separable alterations}\label{appendixB}
%%%%%%

The goal of this appendix is to reframe the simplicial construction used in Section \ref{sec: without resolution}
within a setting similar to the first section. 
To this end we consider a version of Gabber's alteration site \cite{O_2014}, \cite[Def.\,6.8]{HKK_2017}.

\begin{definition}
Let $\dom_k$ be the category whose objects are $k$-varieties $X$, but whose morphisms are the maximally dominant morphisms. Recall that a morphism of schemes is called \emph{maximally dominant} if it sends generic points of the source to generic points of the target.

Let moreover $\dom_{k,\mathrm{gf}}$ be the category whose objects are $k$-varieties $X$, and whose morphisms are the generically finite maximally dominant morphisms. 
Here we say that a maximally dominant morphism $f:Y\rightarrow X$ is generically finite if for every generic point $\eta$ of $Y$ the extension of residue fields $k(\eta)/k(f(\eta))$ is finite.
\end{definition}

\begin{lemma}\label{lem:domfiberproduct}
Let $g:X'\rightarrow X$ and $f:Y\rightarrow X$ be morphisms in $\dom_k$ and assume that $f$ is generically finite. 
The fibre product $X'\times_X^{\dom_k} Y$ exists in $\dom_k$.
\end{lemma}

\begin{proof}
Let $X'\times_X^{\dom_k} Y$ be the union of the reduced irreducible components of the fibre product $(Y\times_X Z)_{\mathrm{red}}$ in $\Var_k$ (see \hyperref[conventions]{\S\,Conventions}) which dominate an irreducible component of $X$. 

We first show that the morphisms induced by the projections $X'\times_X^{\dom_k} Y\rightarrow X'$ and $X'\times_X^{\dom_k} Y\rightarrow Y$ are maximally dominant. 
By definition, every generic point $\lambda\in X'\times_X^{\dom_k} Y$ is mapped to a generic point $\xi\in X$. 
After base change along $\xi \rightarrow X$, we reduce to the situation that $X=\Spec(k(\xi))$ is the spectrum of a field. 
In that case, the morphisms $g$ and $f$ are flat and $X'\times_X^{\dom_k} Y$ coincides with the fibre product $(Y\times_X Z)_{\mathrm{red}}$ in $\Var_k$. 
But according to \cite[Prop.\,1.1.5]{O_2014} the base change of a maximally dominant morphism along a flat morphism is again maximally dominant which shows the claim. 

Next we observe that if $\xi\in X$ is a generic point with generic points $\xi'\in X'$ and $\eta\in Y$ mapping to it, then $(\xi'\times_\xi \eta)_{\mathrm{red}}$ is a finite disjoint union of generic points. 
Indeed, by assumption on $f$, $k(\eta)/k(\xi)$ is a finite field extension and therefore  $(\xi'\times_\xi \eta)_{\mathrm{red}}=\bigsqcup_{i=1}^n \lambda_i $ is a finite disjoint union of points. 
We have to show that the $\lambda_i$ are generic points.
For $j$ fixed, let $\widetilde{\lambda}_j$ be a generic point of $(X'\times_X Y)_{\mathrm{red}}$ such that $\lambda_j\in\overline{\{\widetilde{\lambda}_j\}}$.
As $\lambda_j$ maps to $\xi'$ and $\eta$, the same is true for $\widetilde{\lambda}_j$. 
Therefore $\widetilde{\lambda}_j\in \bigsqcup_{i=1}^n \lambda_i$ and hence $\widetilde{\lambda}_j = \lambda_j$. 

We use this to show that $X'\times_X^{\dom_k} Y$ satisfies the universal property of a fibre product in $\dom_k$. 
Let $Z\in \Var_k$ such that there is a commutative diagram 
$$
\xymatrix{Z\ar[r] \ar[d] & Y\ar[d] \\ X'\ar[r]& X}
$$
in $\dom_k$. 
By the universal property of fibre products in $\Var_k$, 
there is a unique morphism $Z\rightarrow (X'\times_X Y)_{\mathrm{red}}$ and we have to show that it factors through $X'\times_X^{\dom_k} Y$ in $\dom_k$. 
Let $\zeta\in Z$ be a generic point, mapping to generic points $\xi'\in X'$ and $\eta\in Y$ over a generic point $\xi \in X$. 
As we have seen, $(\xi'\times_\xi \eta)_{\mathrm{red}}=\bigsqcup_{i=1}^n \lambda_i $ is a finite disjoint union of generic points, and thus
 $\zeta$ maps to one of the $\lambda_i$ under $Z\rightarrow (X'\times_X Y)_{\mathrm{red}}$. 
Consequently, the image of $Z$ is contained in $X'\times_X^{\dom_k} Y$.
\end{proof}

\begin{example}
Note that general fibre products do not exist in $\dom_k$. 
As a counter example, consider $X=\Spec(k)$, $X'=\Spec(k[x])= \mathbb{A}^1_k= Y=\Spec(k[y])$. 
Then clearly $X'\rightarrow X$ and $Y\rightarrow X$ are maximally dominant, but $k(x)/k$ and $k(y)/k$ are not finite. 

We first observe that the usual fibre product  $X'\times_X Y= \Spec(k[x,y])=\mathbb{A}^2_k$ (which coincides with the definition from Lemma \ref{lem:domfiberproduct} because the projections are maximally dominant) does not give a fibre product in $\dom_k$.  
Indeed, let $Z=\Spec(k[z])=\mathbb{A}^1_k$ with morphisms $Z\rightarrow Y, y\mapsto z$ and $Z\rightarrow X', x\mapsto z$ be given. 
Then the morphism $Z\rightarrow X'\times_X Y$  induced by the universal property is the diagonal morphism, which is not maximally dominant, and hence 
 $X'\times_X Y$  does not satisfy the universal property of fibre products in $\dom_k$.

On can even show that in $\dom_k$ a fibre product does not exist. 
Assume now that there exists a fibre product $M$ of $X'\rightarrow X$ and $Y\rightarrow X$ in $\dom_k$. 
Then by the universal property of the usual fibre product 
we have a commutative diagram
$$
\xymatrix{
& M \ar[ddl] \ar[ddr] \ar[d] & \\
&X'\times_X Y= \mathbb{A}^2_k \ar[dr] \ar[dl]& \\
X'= \mathbb{A}^1_k \ar[dr] && Y= \mathbb{A}^1_k \ar[dl]\\
&X= \Spec(k)&
}
$$
We can find infinitely many curves $C$  in $X' \times_X Y$
over $k$ with  the projections 
$C\rightarrow X'$ and $C\rightarrow Y$ maximally dominant. 
On the one hand the universal property $M$ in $\dom_k$ induces a maximally dominant morphism $C\rightarrow M$ 
which in particular maps the generic point of $C$ to a generic point of $M$.
On the other hand the universal property of  
$X'\times_X Y$ in $\Var_k$ induce 
 the canonical inclusion
$C\hookrightarrow X'\times_X Y$ 
which factors through the morphism $C\rightarrow M$. 
Consequently, the map $M\rightarrow  X'\times_X Y$ maps the generic point of $M$ corresponding to $C$ to the generic point of $C$ inside $X'\times_X Y$. 
Varying $C$, we see that the image of the generic points of $M$ under the morphism $M\rightarrow  X'\times_X Y$  contains infinitely many points of transcendental degree 1. This contradicts the fact the $M$ has only finitely many generic points.
\end{example}

\begin{lemma}\label{lem: dom base change}
Let $g:X'\rightarrow X$ and $f:Y\rightarrow X$ be morphisms in $\dom_k$.
Assume in addition that for every generic point $\xi$ of $Y$, the extension of residue fields $k(\xi)/k(f(\xi))$ is finite (respectively finite and separable). 
Then the morphism $f':Y'=Y\times_X^{\dom_k} X' \rightarrow X'$ satisfies the same property.
\end{lemma}

\begin{proof}
Let $\eta'\in Y'$ be a generic point.
Its images $\eta$ in $Y$, $f'(\eta')$ in $X'$, $g\left(f(\eta')\right)=f(\eta)$ in $X$ are generic points, because all morphisms in sight are maximally dominant by hypothesis.
The situation is illustrated by the diagram of field extensions
	$$
	\xymatrix{k(\eta')&  k(\eta) \ar@{_{(}->}[l]\\ k\left(f'(\eta')\right)  \ar@{_{(}->}[u] & k\left(f(\eta)\right)  \ar@{_{(}->}[l]  \ar@{_{(}->}[u]}
	$$
where the extension $k(\eta)/k\left(f(\eta)\right)$ is finite (respectively finite and separable), 
the extension $k\left(f'(\eta')\right)/k\left(f(\eta)\right) $ is arbitrary
and $k(\eta') = k(\eta)k\left(f'(\eta')\right)$. So we see that 
$k(\eta')/k\left(f'(\eta')\right)$ is also finite (respectively finite and separable).
\end{proof}

\begin{lemma}
Let $f:Y\rightarrow X$ be a morphism in $\dom_k$.
The following conditions are equivalent:
\begin{enumerate}
\item For every generic point $\xi$ of $Y$, the extension of residue fields $k(\xi)/k\left(f(\xi)\right)$ is finite and separable.
\item The morphism $f$ is generically \'etale, i.e.\,there is a dense open subset $U\subset X$ over which $f$ is \'etale. 
\end{enumerate}
\end{lemma}

\begin{proof}
The equivalence follows from 
\cite[\href{https://stacks.math.columbia.edu/tag/02GU}{Tag 02GU}]{StacksProject}. 
Note, in particular, the equivalence between $(6)$ and $(9)$ in loc.\,cit.
\end{proof}

\begin{definition}\label{def: salt}
We call the topology on $\dom_k$ generated by generically \'etale proper morphisms 
the \emph{$\salt$-topology} and we denote the resulting site by $\dom_{k,\salt}$.
\end{definition}

\begin{corollary}\label{cor: categroy properties}
The triple $(\cC:= \dom_{k}, \cC':=\dom_{k,\salt}, \cC^{\mathrm{fp}}:=\dom_{k,\mathrm{gf}})$ satisfies the properties (a)--(e) of Appendix \ref{appendixB'}. 
\end{corollary}

\begin{proof}
This follows from Lemmas \ref{lem:domfiberproduct}, \ref{lem: dom base change} and by definition.
\end{proof}

\begin{definition}
 Let $\dom_k^{geo}$ \textup{(}resp.\,$\dom_k^{nc})$\textup{)} be the category of geometric pairs 
\textup{(}resp.\,normal crossing pairs\textup{)}, with morphisms $(\pi,\overline{\pi})$ such that $\pi$ is maximally dominant. 

Let $\dom_{k,\mathrm{gf}}^{geo}$ \textup{(}resp.\,$\dom_{k,\mathrm{gf}}^{nc})$\textup{)} be the category of geometric pairs 
\textup{(}resp.\,normal crossing pairs\textup{)}, with morphisms $(\pi,\overline{\pi})$ such that $\pi$ is maximally dominant and generically finite. 
\end{definition}

\begin{definition}
We call the topology on 
$\dom_k^{geo}$ induced from $\dom_{k,\salt}$ by the forgetful functor 
\[ \psi: \dom_k^{geo} \rightarrow \dom_k, \quad (X,\overline{X}) \mapsto X \]
the \emph{$\salt$-topology on $\dom_k^{geo}$} and we denote the resulting site by $\dom_{k,\salt}^{geo}$.
\end{definition}

\begin{corollary}
The triple $(\cC:= \dom^{geo}_{k}, \cC':=\dom^{geo}_{k,\salt}, \cC^{\mathrm{fp}}:=\dom^{geo}_{k,\mathrm{gf}})$ satisfies the properties (a)--(e) of Appendix \ref{appendixB'}. 
\end{corollary}
\begin{proof}
This follows from Corollary \ref{cor: categroy properties} and 
the argument in Remark \ref{rem:finiteinverselimit}. 
\end{proof}

Since $\dom_{k}$ and $\dom_{k}^{geo}$ admit fibre products if one of the morphisms is an $\salt$-morphism and the forgetful functor 
$\psi$ preserves fibre products, a family of morphisms $\{(X_i,\overline{X}_i) \to (X,\overline{X})\}_i$ 
is an $\salt$-covering in $\dom_k^{geo}$ if and only if the family of morphisms 
 $\{X_i \to X\}_i$ is an $\salt$-covering in $\dom_k$ (\cite[Cor.\,3.3]{V_1972}). 

\begin{remark}
\begin{enumerate}
	\item Recall that an alteration is a surjective generically finite proper morphism.
	Therefore the covering morphisms in the $\salt$-topology are exactly the generically \'etale alterations.
	\item  
	A diagram \eqref{eq:2-1} with the assumption in Question \ref{q:2-1} is nothing but a 
	split hypercovering $(X_{\bullet}, \overline{X}_{\bullet}) \to (X,\overline{X})$ in $\dom_{k,\salt}^{geo}$ such that 
	each $(X_i,\overline{X}_i)$ belongs to $\dom_k^{nc}$.
\end{enumerate}
\end{remark}

\begin{corollary}\label{cor:domc}
The assumptions and the condition \hyperref[ast]{$(*)$} in Appendix \ref{appendixB'} are satisfied for $\cC_{\tau} = \dom_{k,\salt}^{geo}$ and $\cC' = \dom_{k}^{nc}$. 
\end{corollary}

\begin{proof}
By de Jong's alteration theorem, for any object $(X,\overline{X})$ in $\dom_{k}^{geo}$, there exists an $\salt$-covering $(X',\overline{X}') \to (X,\overline{X})$ with $(X',\overline{X}') \in \dom_{k}^{nc}$. From this fact and the results in this appendix up to here, 
we obtain the corollary. 
\end{proof}

Thus, for $(X,\overline{X}) \in \dom_k^{geo}$,  
if we define the category ${\rm Ho}\{(X_\bullet,\overline{X}_\bullet)/(X,\overline{X})\}$  as in Definition \ref{def:homotopycat}, it is cofiltered by 
Corollary \ref{cor:homotopycat}. 
Hence, for a fixed $(X, \overline{X})$, 
the inductive system 
\[ \left\{H^i_{\cris}((X_{\bullet}, \overline{X}_{\bullet})/W(k))\right\}_{(X_{\bullet}, \overline{X}_{\bullet}) \in {\rm Ho}\{(X_{\bullet},\overline{X}_{\bullet})/(X,\overline{X})\}} \]
is filtered and so the colimit 
$\varinjlim_{(X_{\bullet}, \overline{X}_{\bullet}) \in {\rm Ho}\{(X_{\bullet},\overline{X}_{\bullet})/(X,\overline{X})\}} 
H^i_{\cris}((X_{\bullet}, \overline{X}_{\bullet})/W(k))$ 
is defined. 
Also, for a morphism 
$\overline{f}: (X',\overline{X}') \to (X,\overline{X})$ in 
$\dom_k^{geo}$, 
if we define the category ${\rm Ho}\{\overline{f}_{\bullet}/\overline{f}\}$ as in Definition \ref{def:homotopycat}, it is cofiltered by 
Corollary \ref{cor:homotopycat} and we have a diagram of projections 
\[ {\rm Ho}\{(X'_{\bullet},\overline{X}'_{\bullet})/(X', \overline{X}')\} 
\leftarrow {\rm Ho}\{ \overline{f}_{\bullet}/\overline{f}\} 
\rightarrow {\rm Ho}\{(X_{\bullet},\overline{X}_{\bullet})/(X, \overline{X})\} 
\]
with the second arrow surjective on the set of objects. 
This diagram induces a morphism between colimits 
\[ \varinjlim_{(X_{\bullet}, \overline{X}_{\bullet}) \in {\rm Ho}\{(X_{\bullet},\overline{X}_{\bullet})/(X,\overline{X})\}} 
H^i_{\cris}((X_{\bullet}, \overline{X}_{\bullet})/W(k)) 
\to 
\varinjlim_{(X'_{\bullet}, \overline{X}'_{\bullet}) \in {\rm Ho}\{(X'_{\bullet},\overline{X}'_{\bullet})/(X',\overline{X}')\}} 
H^i_{\cris}((X'_{\bullet}, \overline{X}'_{\bullet})/W(k)). 
\] 

Next, for $X \in \dom_k$, denote by $\{(X,\overline{X})/X\}^{geo}$ the category of geometric pairs 
$(X,\overline{X})$ extending $X$. 
Then it is easy to see that this category is cofiltered (cf. Lemma \ref{lem: nc compactification}) and so the colimit 
\[ 
\varinjlim_{(X,\overline{X}) \in \{(X,\overline{X})/X\}^{geo}} 
\varinjlim_{(X_{\bullet}, \overline{X}_{\bullet}) \in {\rm Ho}\{(X_{\bullet},\overline{X}_{\bullet})/(X,\overline{X})\}} 
H^i_{\cris}((X_{\bullet}, \overline{X}_{\bullet})/W(k))
\] 
is defined. 
Also, for a morphism $f: X' \to X$ in $\dom_k$, denote by 
$\{ \overline{f}/f \}^{geo}$ the category of morphisms of geometric pairs 
$(X',\overline{X}') \to (X, \overline{X})$ extending $f$. It is easy to see that 
this category is also cofiltered and that there exists a diagram of projections 
\[ \{(X',\overline{X}')/X\}^{geo} \leftarrow \{ \overline{f}/f \}^{geo} 
\rightarrow \{(X,\overline{X})/X\}^{geo}
\]
where the second arrow is surjective on the set of objects (cf. Proposition \ref{prop: functorial}). This diagram induces a morphism 
\begin{align*}
& \varinjlim_{(X,\overline{X}) \in \{(X,\overline{X})/X\}^{geo}} 
\varinjlim_{(X_{\bullet}, \overline{X}_{\bullet}) \in {\rm Ho}\{(X_{\bullet},\overline{X}_{\bullet})/(X,\overline{X})\}} 
H^i_{\cris}((X_{\bullet}, \overline{X}_{\bullet})/W(k)) \\
\to & 
\varinjlim_{(X',\overline{X}') \in \{(X',\overline{X}')/X'\}^{geo}} 
\varinjlim_{(X'_{\bullet}, \overline{X}'_{\bullet}) \in {\rm Ho}\{(X'_{\bullet},\overline{X}'_{\bullet})/(X',\overline{X}')\}} 
H^i_{\cris}((X'_{\bullet}, \overline{X}'_{\bullet})/W(k)). 
\end{align*}

Therefore, we can make the following definition.

\begin{definition}
For $i=0,1$, we define the cohomology theory $X \mapsto H^i_{\salt}(X)$ on $\dom_k$ by 
\begin{equation}\label{eq:defH0H1}
	H^i_{\salt}(X) := 
\varinjlim_{(X,\overline{X}) \in \{(X,\overline{X})/X\}^{geo}} 
\varinjlim_{(X_{\bullet}, \overline{X}_{\bullet}) \in {\rm Ho}\{(X_{\bullet},\overline{X}_{\bullet})/(X,\overline{X})\}}
	H^i_{\cris}((X_{\bullet}, \overline{X}_{\bullet})/W(k)). 
\end{equation}

\end{definition}

The following is the reformulation of the results in Section 2.

\begin{theorem}\label{thm:reformulation}
For $i=0,1$, the inductive system on the right hand side of \eqref{eq:defH0H1} is constant. 
\end{theorem}

\begin{proof}
Theorems \ref{thm:H^0}, \ref{thm:H^1} imply that the inductive system on
the right hand side of \eqref{eq:defH0H1} is constant with respect to 
${\rm Ho}\{(X_{\bullet},\overline{X}_{\bullet})/(X,\overline{X})\}$. 
Note also that, for a morphism $(X,\overline{X}') \to (X,\overline{X})$ in $\dom_k^{geo}$ and 
an object $(Y_{\bullet}, \overline{Y}_{\bullet}) \to (X, \overline{X}')$ in  
${\rm Ho}\{(X_{\bullet},\overline{X}_{\bullet})/(X,\overline{X}')\}$, the composition 
$(Y_{\bullet}, \overline{Y}_{\bullet}) \to (X, \overline{X}') \to (X, \overline{X})$ 
is an object in ${\rm Ho}\{(X_{\bullet},\overline{X}_{\bullet})/(X,\overline{X})\}$. 
This fact and Theorems \ref{thm:H^0}, \ref{thm:H^1} imply that the 
inductive system on the right hand side of \eqref{eq:defH0H1} is constant also 
with respect to $\{(X,\overline{X})/X\}^{geo}$. 
\end{proof}

\begin{corollary}\label{cor:reformulation}
\begin{enumerate}
	\item For any $X \in \dom_k$, $H^i_{\salt}(X) \, (i=0,1)$ are finitely generated over $W(k)$. 
	\item For any $(X,\overline{X}) \in \dom_k^{nc}$, there exist functorial isomorphisms 
	$$H^i_{\salt}(X) \cong H^i_{\cris}((X,\overline{X})/W(k)) \quad (i=0,1).$$
	\item For any $X \in \dom_k$, there exist functorial isomorphisms 
	$$H^i_{\salt}(X) \otimes_{\Z} \Q \cong H^i_{\rig}(X/K) \quad (i=0,1).$$
\end{enumerate}
\end{corollary}

\begin{proof}
	Since any member in the limit on the right hand side of \eqref{eq:defH0H1} is finitely generated over $W(k)$, 
	(i) follows from Theorem \ref{thm:reformulation}. Noting the functoriality of $H^i_{\salt}(X)$ 
	with respect to $X \in \dom_k$, we see that (ii) follows also 
	from Theorem \ref{thm:reformulation}. (iii) follows from the canonical isomorphism 
	$H^{\ast}_{\cris}((X_{\bullet}, \overline{X}_{\bullet})/W(k)) \otimes_{\Z} \Q \cong H^\ast_{\rig}(X/K)$  
	in \cite[Cor.\,11.7 1)]{N_2012}. 
\end{proof}

\end{document}